\documentclass{amsart}
\usepackage{amsfonts, amsmath, amssymb, amscd, amsthm, graphicx,enumerate,float}
\usepackage{hyperref,xypic}
\usepackage[usenames, dvipsnames]{color}
\usepackage{oubraces}
\usepackage{todonotes}
\usepackage{overpic}
\usepackage{caption}
\usepackage{subcaption}

\usepackage{thmtools}
\usepackage{thm-restate}
\usepackage[normalem]{ulem}

\makeatletter
\def\blfootnote{\xdef\@thefnmark{}\@footnotetext}
\makeatother

\makeatletter 
\def\l@subsection{\@tocline{2}{0pt}{2pc}{6pc}{}} \makeatother

\newtheorem{thm}{Theorem}[section]
\newtheorem{cor}[thm]{Corollary}
\newtheorem{lem}[thm]{Lemma}
\newtheorem{prop}[thm]{Proposition}

\newtheorem{ques}[thm]{Question}

\theoremstyle{definition}
\newtheorem{defn}[thm]{Definition}
\theoremstyle{remark}
\newtheorem{rem}[thm]{Remark}
\newtheorem{ex}[thm]{Example}

\newtheorem{claim}[thm]{Claim}
\newtheorem{conv}[thm]{Convention}

\newfont{\eufm}{eufm10}

\renewcommand{\phi}{\varphi}

\newcommand{\N}{\mathbb N}
\newcommand{\Z}{\mathbb Z}
\newcommand{\R}{\mathbb R}

\newcommand{\scc}{\mathcal{C}}

\newcommand{\MCG}{\mathrm{Map}}
\newcommand{\intr}{\mathrm{int}}

\newcommand{\pmap}{\operatorname{PMap}}

\def\mc {\mathcal}

\begin{document}

\title{Infinite-type loxodromic isometries of the relative arc graph}
\author{Carolyn Abbott}
\address{Department of Mathematics\\Brandeis University\\Waltham, MA 02453}
\email{carolynabbott@brandeis.edu}

\author{Nicholas Miller}
\address{Department of Mathematics\\University of California, Berkeley\\Berkeley, CA 94720}
\email{nickmbmiller@berkeley.edu}

\author{Priyam Patel}
\address{Department of Mathematics\\University of Utah\\Salt Lake City, UT 84112}
\email{patelp@math.utah.edu}

\date{}

\begin{abstract}
An infinite-type surface $\Sigma$ is of type $\mathcal{S}$ if it has an isolated puncture $p$ and admits shift maps. This includes all infinite-type surfaces with an isolated puncture outside of two sporadic classes. Given such a surface, we construct an infinite family of intrinsically infinite-type mapping classes that act loxodromically on the relative arc graph $\mathcal{A}(\Sigma, p)$. J. Bavard produced such an element for the plane minus a Cantor set, and our result gives the first examples of such mapping classes for all other surfaces of type $\mathcal{S}$. The elements we construct are the composition of three shift maps on $\Sigma$, and we give an alternate characterization of these elements as a composition of a pseudo-Anosov on a finite-type subsurface of $\Sigma$ and a standard shift map. We then explicitly find their limit points on the boundary of $\mathcal{A}(\Sigma,p)$ and  their limiting geodesic laminations.
Finally, we show that these infinite-type elements can be used to prove that $\MCG(\Sigma,p)$ has an infinite-dimensional space of quasimorphisms.
\end{abstract}

\maketitle
\tableofcontents

\section{Introduction}

 A surface $\Sigma$ is of finite-type if $\pi_1(\Sigma)$ is finitely generated, and otherwise $\Sigma$ is of infinite-type. Recently, there has been a surge of interest in infinite-type surfaces and their mapping class groups $\MCG(\Sigma)$, which arise naturally in a variety of contexts in low-dimensional topology, dynamics, and even descriptive set theory. 
See \cite{survey} for a survey of recent results on infinite-type mapping class groups. 

For finite-type surfaces $\Sigma$,  Nielsen and Thurston \cite{Nielsen, Thurston} give a powerful classification of the elements of $\MCG(\Sigma)$: every element is  periodic, reducible, or pseudo-Anosov.
The action of $\MCG(\Sigma)$ by isometries on the (infinite-diameter and hyperbolic) curve graph $\mathcal{C}(\Sigma)$ captures a coarser classification of the elements of $\MCG(\Sigma)$ since elements are either \textit{elliptic} or \textit{loxodromic}.
These two classifications, which are both interesting in their own right, have a strong relationship; the loxodromic elements are exactly the pseudo-Anosovs. In this way, the most interesting and complex mapping classes correspond to the dynamically richest actions. 

The situation for infinite-type surfaces is more complicated for a few reasons. First, the exact analog of the Nielsen--Thurston classification is no longer valid in this setting since some elements are neither periodic, reducible, nor pseudo-Anosov in the traditional sense. Second, the curve graph of an infinite-type surface has finite diameter unlike for finite-type surfaces.
This paper is motivated by one of the biggest open problems for infinite-type surfaces, which is to give an analog of the Nielsen--Thurston classification for infinite-type mapping classes. We work towards this goal by studying the action of $\MCG(\Sigma)$ on a different hyperbolic graph. 

When $\Sigma$ is an infinite-type surface with at least one isolated puncture $p$, the \textit{relative arc graph}, $\mathcal{A}(\Sigma,p)$, plays the role of $\mathcal{C}(\Sigma)$ and is defined as follows: the vertices correspond to isotopy classes of simple arcs that begin and end at $p$ and edges connect vertices for arcs admitting disjoint representatives. The subgroup $\MCG(\Sigma, p)$ of $\MCG(\Sigma)$ that fixes the isolated puncture $p$ acts on $\mathcal{A}(\Sigma, p)$ by isometries. 
This graph was first defined by D. Calegari \cite{Calegari}, who initiated its study by asking whether, for the plane minus a Cantor set, this graph was infinite diameter and whether any element of $\MCG(\Sigma,p)$ acted loxodromically.
In \cite{Bavard}, J. Bavard carried out Caelgari's program for the plane minus a Cantor set and, for that surface, showed that $\mathcal{A}(\Sigma,p)$ is both infinite-diameter and hyperbolic.
Aramayona--Fossas--Parlier \cite{AFP} then showed that these properties for $\mathcal{A}(\Sigma, p)$ hold more generally for any infinite-type surface with at least one isolated puncture. 

Given that the trichotomy of the Nielsen--Thurston classification does not exactly hold for infinite-type surfaces, it is necessary to redefine reducible, and therefore irreducible, mapping classes in this setting. One of the most promising ways to motivate a new definition is to classify the elements of infinite-type mapping class groups that are loxodromic with respect to the action of $\MCG(\Sigma)$ on a hyperbolic graph since  these elements correspond to infinite-order irreducibles in the finite-type setting. In order to classify these elements, we must first construct them.

When $\Sigma$ is the sphere minus a Cantor set with an isolated puncture $p$ (i.e., $\Sigma$ is the plane minus a Cantor set), Bavard \cite{Bavard} constructed an \emph{intrinsically infinite-type} mapping class that is loxodromic with respect to the action of $\MCG(\Sigma)$ on $\mathcal{A}(\Sigma, p)$, and for several years, this was the only known such example. 
In this paper, we give a new construction of mapping classes that are loxodromic with respect to the action of $\MCG(\Sigma, p)$ on the relative arc graph $\mathcal{A}(\Sigma, p)$ for a large class of infinite-type surfaces. 

\begin{thm}\label{thm:loxomain}
For any surface $\Sigma$ of type $\mathcal{S}$, there is an infinite family of intrinsically infinite-type homeomorphisms $\{g_n\}_{n\in\N}$ in $\MCG(\Sigma,p)$ such that each $g_n$ is  loxodromic with respect to the action of $\MCG(\Sigma,p)$ on $\mathcal{A}(\Sigma,p)$. 
\end{thm}

\begin{figure}[h]
\begin{center}
\begin{overpic}[width=4in,trim = 2in 5.85in 1.25in 0.15in, clip=true, totalheight=0.4\textheight]{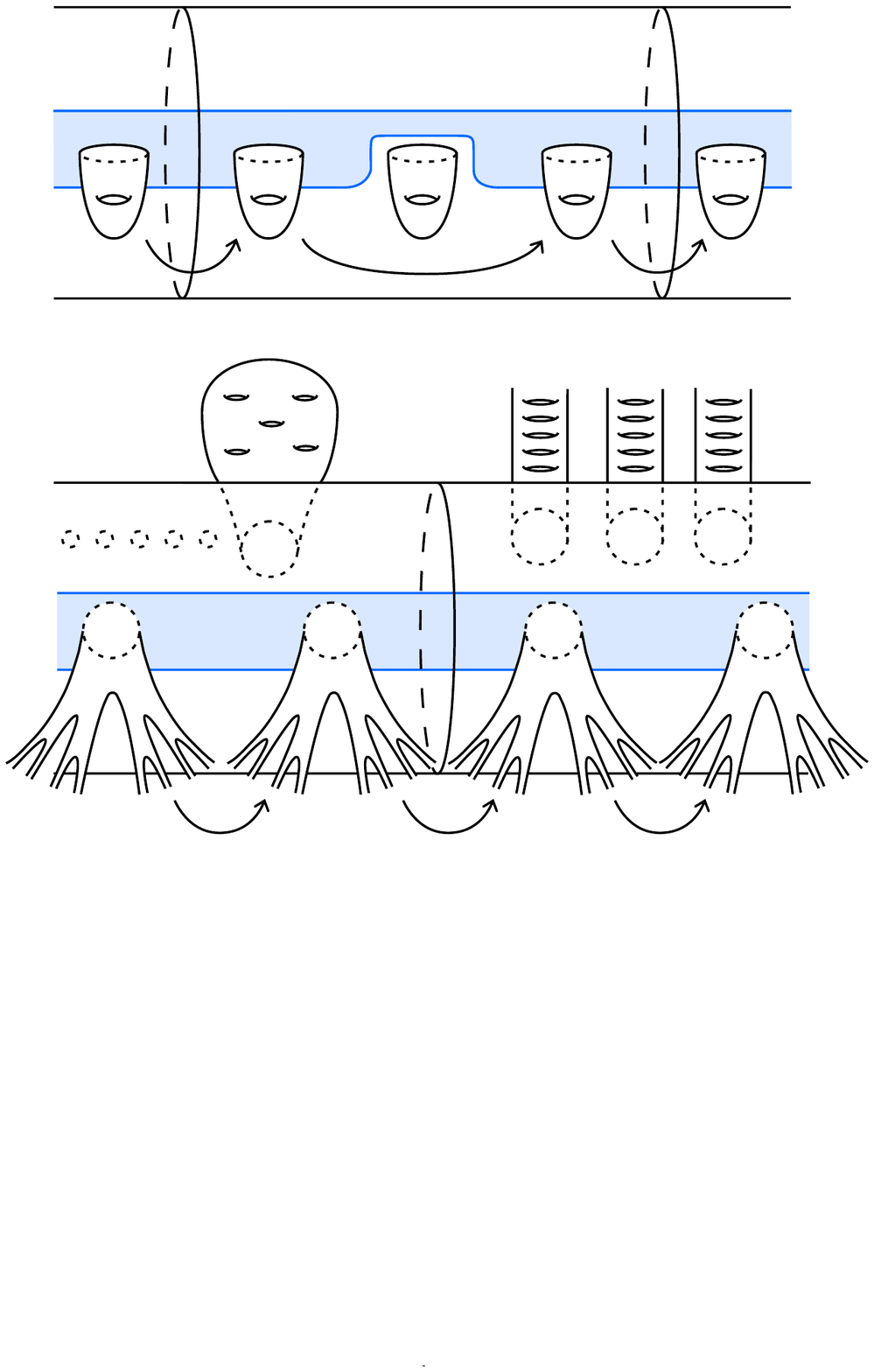}
\end{overpic}
\caption{A handleshift on an infinite-type surface (above) and a shift map on an infinite-type surface (below).}\label{fig:exs}
\end{center}
\end{figure}

 Each mapping class in our construction is the composition of three homeomorphisms called \textit{shift maps}. Shift maps are generalizations of the \textit{handleshift} homeomorphisms constructed by the third author and N. Vlamis in \cite{PatelVlamis} (see Figure~\ref{fig:exs} for examples of both). Roughly, an infinite-type surface $\Sigma$ with an isolated puncture $p$ is of type $\mathcal{S}$ if it there is a proper embedding of the biinfinite flute surface containing $p$ into $\Sigma$  such that certain shift maps on the flute surface induce shift maps on $\Sigma$. See Section~\ref{sec:backgroundshiftmaps} for more details and Figure~\ref{fig:typeS}  for some examples of surfaces of type $\mathcal{S}$. In Lemma~\ref{lem:equiv}, we show that a surface with an isolated puncture is of type $\mathcal{S}$ if and only if it admits shift maps. This set of surfaces consists of all infinite-type surfaces with an isolated puncture except a flute surface with finite (possibly zero) genus and a \textit{fluted Loch Ness monster}.  We call these two classes \textit{sporadic surfaces} in this context. See Figure~\ref{fig:sporadic} for examples of sporadic surfaces and Lemma~\ref{lem:sporadic} for a proof of this fact. Since sporadic surfaces are exactly the small class of surfaces with an isolated puncture that do not admit shift maps, different methods will need to be developed in order to  prove an analogue of Theorem~\ref{thm:loxomain} for these surfaces.  It would be interesting to understand how elements of the mapping class groups of  sporadic surfaces act on the relative arc graph.

\begin{figure}[h]
\begin{center}
\begin{overpic}[trim = 1.25in 7.5in .25in 1.5in, clip=true, width=4in]{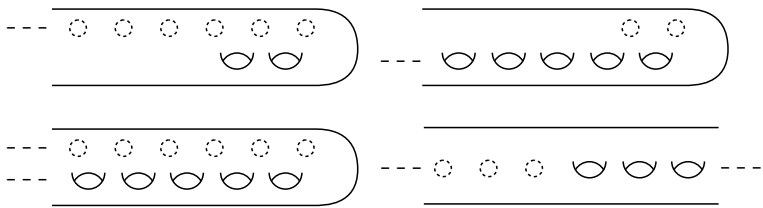}
\end{overpic}
\caption{Examples of sporadic infinite-type surfaces that are not of type $\mathcal{S}$. The first is a flute with finite genus, the other three are fluted Loch Ness Monster surfaces.}\label{fig:sporadic}
\end{center}
\end{figure}

\begin{figure}[h]
\begin{center}
\begin{overpic}[trim = 1.8in 3.5in 1.15in 0.15in, clip=true, totalheight=0.525\textheight,width=4in]{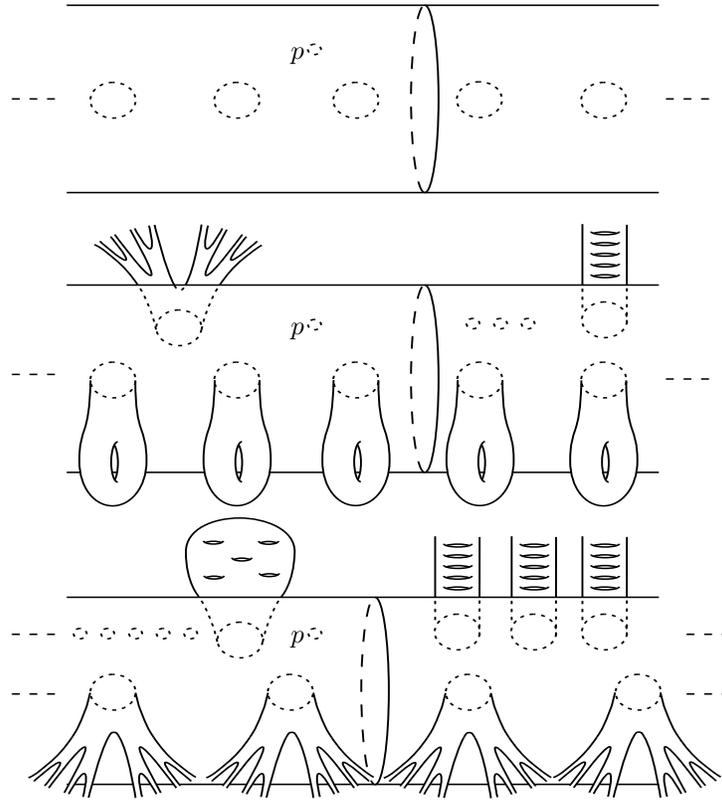}
\put(37,92){$p$}
\put(37,58){$p$}
\put(37,20){$p$}
\end{overpic}
\caption{Examples of surfaces of type $\mathcal{S}$, the first of which is the biinfinite flute surface $S$ itself.}\label{fig:typeS}
\end{center}
\end{figure}

The handleshift homeomorphisms mentioned above have proven to be crucial in understanding various aspects of infinite-type mapping class groups. For example, it is shown in \cite{PatelVlamis} that they are needed to topologically generate the \textit{pure mapping class group} whenever $\Sigma$ has at least two non-planar ends, and in \cite{APV} they are used to show that the pure mapping class groups of such surfaces surject onto $\mathbb{Z}$. With this paper, we emphasize the importance of more general shift maps to the theory of infinite-type mapping class groups. Inspired by Bavard's work in \cite{Bavard}, we choose the shift maps in our construction carefully so that their composition mimics some of the behavior of pseudo-Anosov maps in the finite-type setting. In fact, we show that there is an alternate description of our homeomorphisms as the composition of a pseudo-Anosov homeomorphism on a finite-type subsurface and a \textit{standard shift map} on $\Sigma$ in Theorem~\ref{thm:pAshift}. Additionally, in Section~\ref{sec:lamination}, we use the work of D. Saric \cite{Saric} to prove the following theorem regarding geodesic laminations for the mapping classes constructed in Theorem~\ref{thm:loxomain}.

\begin{restatable}{thm}{GL}\label{thm:geodlam}
If $\Sigma$ is a surface of type $\mathcal{S}$ equipped with its conformal hyperbolic metric that is equal to its convex core, then there exists a simple closed curve $c_0$ on $\Sigma$ such that the sequence $(g_n^i(c_0))_{i\in\mathbb N}$  converges to a geodesic lamination on $\Sigma$. 
\end{restatable}
 
 \noindent In particular, we produce a train track on $\Sigma$ and show the geodesic lamination from this theorem is weakly carried by this train track. 

We emphasize that the elements arising from our construction are of \textit{intrinsically infinite-type}, that is, they do not lie in the closure of the compactly supported mapping class group $\overline{\MCG_c(\Sigma)}$, where the closure is taken with respect to the compact-open topology on $\MCG(\Sigma)$. These are the first such examples for all surfaces of type $\mathcal{S}$ outside of the plane minus a Cantor set. Additionally, we emphasize that our construction does not rely on, and is not a generalization of, one of the few known methods for constructing pseudo-Anosov mapping classes for finite-type surfaces. 

The most obvious candidates for mapping classes that are loxodromic with respect to the action on $\mathcal{A}(\Sigma, p)$ are those that are pseudo-Anosov on a finite-type subsurface $\Sigma' \subset \Sigma$ containing the special puncture $p$, that extend via the identity map to the rest of $\Sigma$ (these are compactly supported mapping classes).
In \cite{BavardWalker}, Bavard and Walker prove that these types of mapping classes do indeed act loxodromically on a graph that is quasi-isometric to $\mathcal{A}(\Sigma, p)$. In that paper they point out that, though their class of examples is interesting, it will be even more important to construct mapping classes of intrinsically infinite-type that act loxodromically on $\mathcal{A}(\Sigma, p)$; this remark was one of the main points of inspiration for writing this paper. The intrinsically infinite-type elements of $\MCG(\Sigma)$ are more mysterious since tools from finite-type surface theory do not directly generalize when studying these elements. 

\begin{rem} Morales and Valdez \cite{MoralesValdez} have also produced non-compactly supported elements that are loxodromic, but their elements are in the closure of the compactly supported mapping class group. Their method is a generalization of the Thurston--Veech construction of pseudo-Anosovs in the finite-type setting.
\end{rem}

Aside from the motivation provided by a Nielsen--Thurston classification for infinite-type mapping classes, Bestvina–Fujiwara \cite{BestvinaFujiwara} show that constructing elements of $\MCG(\Sigma)$ that act loxodromically on hyperbolic graphs can be used to understand the second bounded cohomology $H^2_b(\MCG(\Sigma), \mathbb{R})$ of $\MCG(\Sigma)$. 
In particular, they show that, for a compact surface $\Sigma$, there exist elements acting loxodromically on $\mathcal{C}(\Sigma)$ that are \textit{weakly properly discontinuous} (WPD). These elements are used to prove that the space of quasimorphisms of $\MCG(\Sigma)$ is infinite-dimensional, which is sufficient to conclude that $H^2_b(\MCG(\Sigma), \mathbb{R})$ is, as well.

 Along these lines, M. Bestvina asked the following question at the AIM workshop on infinite-type surfaces \cite[Problem 4.7]{AIM}: ``For $\Sigma = \mathbb{R}^2-C$ (where $C$ is a Cantor set) is it true that every subgroup of $\MCG(\Sigma)$ has either infinite-dimensional space of quasimorphisms or is amenable?" More generally, we would like to characterize the infinite-type mapping classes that can be used to produce quasimorphisms of $\MCG(\Sigma)$. In Section~\ref{sec:qms}, we show that the elements constructed in Theorem~\ref{thm:loxomain} can be used to give a new proof of the following theorem, originally due to Bavard \cite{Bavard} in the case of a plane minus a Cantor set and Bavard and Walker \cite{BavardWalker} in the general case.

\begin{restatable}{thm}{QM}\label{thm:infiniteqms}
Let $\Sigma$ be a surface of type $\mathcal{S}$. The space of non-trivial quasimorphisms on $\MCG(\Sigma,p)$ is infinite dimensional.
\end{restatable}

In \cite{BavardWalker}, Bavard and Walker use a weaker condition on loxodromic isometries introduced by Bestvina--Bromberg--Fujiwara \cite{BBF}, called WWPD, to show that homeomorphisms that are pseudo-Anosov on finite-type subsurfaces $\Sigma'$ and extend via the identity to the rest of $\Sigma$ can be used to produce quasimorphisms of $\MCG(\Sigma)$. A. Rasmussen shows in \cite{Ras} that for a surface $\Sigma$ with an isolated puncture $p$, an element of $\MCG(\Sigma,p)$ is WWPD with respect to the action on $\mathcal{A}(\Sigma,p)$ if and only if it stabilizes a finite-type subsurface $\Sigma'$ containing the puncture $p$ and restricts to a pseudo-Anosov on $\Sigma'$.  The elements we construct in Theorem~\ref{thm:loxomain} do not fix any finite-type subsurface and thus are \textit{not} WWPD. Despite this, we are still able to build non-trivial quasimorphisms using a criterion of Bestvina and Fujiwara \cite{BestvinaFujiwara} and an approach similar to that of Bavard in \cite{Bavard} which involves defining an intersection pairings on a specific class of arcs on $\Sigma$. Our construction gives subgroups of $\MCG(\Sigma)$ that do not contain WWPD elements but do have an infinite dimensional space of quasimorphisms. \\

\noindent{\bf Plan of the paper:} 
In order to prove Theorem~\ref{thm:loxomain}, we explicitly compute the images of a particular arc on $\Sigma$ under iterates of each homeomorphism $g_n$ and prove that these images form a quasi-geodesic axis for the action of $\langle g_n\rangle $ on $\mathcal{A}(\Sigma,p)$. Though some of the methods in our paper are inspired by Bavard's work in \cite{Bavard}, we note that there are a variety of additional challenges in proving Theorem~\ref{thm:loxomain} for such a wide class of surfaces. In fact, we first prove the theorem for the biinfinite flute surface $S$ and then use the fact that the inclusion of $\mathcal A(S,p)$ into $\mathcal A(\Sigma,p)$ is a $(2,0)$--quasi-isometric embedding (see Lemma~\ref{lem:AsqiembsAsigma}) to extend the theorem to all surfaces of type $\mathcal{S}$. One of the first challenges in proving Theorem~\ref{thm:loxomain} is rigorously coding arcs on $S$, which we do in Section~\ref{sec:thecode}, in order to quantify how long two arcs on $S$ fellow travel. We then introduce \textit{standard position} for an arc on $S$ in Section~\ref{sec:standardposition} so that we can use the code for an arc to find its image under our shift maps in a well-defined way. Most importantly, we must understand when segments of arcs become trivial under our shift maps, and in Section~\ref{sec:cascading} we introduce a kind of cancellation in the image of the code for a segment which we call \textit{cascading cancellation}. This kind of cancellation will cause technical problems throughout the paper and much of Section~\ref{sec:arcsthatstartlike} is devoted to understanding how to control it.

The rest of Section~\ref{sec:loop} is devoted to proving Theorem~\ref{prop:nontrivialloop} (The Loop Theorem) which answers the question of when a segment in an arc becomes trivial under our shift maps. We define the homeomorphisms $g_n$ of Theorem~\ref{thm:loxomain} in Section~\ref{sec:arcinit}, show that we have \textit{``starts like" functions} in Section~\ref{sec:arcsthatstartlike}, and show that we have \textit{highways} in Section~\ref{sec:highways}. Finally, we prove Theorem~\ref{thm:loxomain} in Section~\ref{sec:loxo}, introduce an intersection pairing for arcs and prove Theorem~\ref{thm:infiniteqms} in Section~\ref{sec:qms}, and prove the convergence to a geodesic lamination from Theorem~\ref{thm:geodlam} in Section~\ref{sec:lamination}. \\

\noindent \textbf{Acknowledgements:} The authors are grateful to Chris Leininger for pointing out the work of Verberne and to Ekta Patel for helping us with computational work related to this project.
The authors would also like to thank Alden Walker and Nick Vlamis for helpful discussions as well as Elizabeth Field for helpful comments on an earlier draft. 
The first author was partially supported by NSF DMS--1803368 and DMS--2106906, the second author was partially supported by NSF DMS--1408458, and the third author was partially supported by NSF DMS--1840190 and DMS--2046889. The authors would also like to thank the American Institute of Mathematics and the organizers of the Infinite-type Surfaces Workshop that was held there, during which some of this work was completed.


\section{Background}\label{background}

\subsection{Space of ends and classification of infinite-type surfaces} 

Central to the classification of infinite-type surfaces is the definition of the space of ends $E(\Sigma)$ of an infinite-type surface $\Sigma$. Informally, an end of $\Sigma$ is a way to escape or go off to infinity in $\Sigma$. More formally we have: 

\begin{defn}
An \textit{exiting sequence} in \( \Sigma \) is a sequence \( \{U_n\}_{n\in\N} \) of connected open subsets of \( \Sigma \) satisfying:
\begin{enumerate}
\item \( U_{n} \subset U_m \) whenever $m<n$;
\item $U_n$ is not relatively compact for any $n \in \N$, that is, the closure of $U_n$ in $\Sigma$ is not compact;
\item the boundary of \(U_n \) is compact for each \( n \in \N \); and
\item any relatively compact subset of $\Sigma$ is disjoint from all but finitely many of the $U_n$’s.
\end{enumerate}
Two exiting sequences \( \{U_n\}_{n\in\N} \) and \( \{V_n\}_{n\in\N} \) are equivalent if for every \( n \in \N \) there exists \( m \in \N \) such that \( U_m \subset V_n \) and \( V_m \subset U_n \). An \textit{end} of \( \Sigma \) is an equivalence class of exiting sequences. 
\end{defn}

 The \textit{space of ends} $E(\Sigma)$, or simply $E$, of $\Sigma$ is the set of ends of $\Sigma$ equipped with a natural topology for which it is totally disconnected, Hausdorff, second countable, and compact. In particular, $E(\Sigma)$ is homeomorphic to a closed subset of the Cantor set. To describe the topology, let \( V \) be an open subset of \( \Sigma \) with compact boundary, define \( \widehat V = \left\{ \left[\{U_n\}_{n\in\N}\right]\in E : U_n \subset V \text{ for some } n\in \N\right\} \) and let \( \mathcal V = \{ \widehat V : V \subset \Sigma \text{ is open with compact boundary}\} \).
The set \( E \) becomes a topological space by declaring \( \mathcal V \) a basis for the topology.

We note that ends can be isolated or not and can be \textit{planar} (if there exists an $i$ such that $U_i$ is homeomorphic to an open subset of the plane $\mathbb{R}^2$) or \textit{nonplanar} (if every $U_i$ has infinite genus). The set of nonplanar ends of $\Sigma$ is a closed subspace of $E(\Sigma)$ and will be denoted by $E_g(\Sigma)$. 

Ker\'ekj\'art\'o \cite{Kerekjarto} and Richards \cite{Richards} showed that the homeomorphism type of an orientable infinite-type surface is determined by the quadruple $$(g, b, E_g(\Sigma), E(\Sigma))$$ where $g \in \mathbb{Z}_{\ge 0} \cup \{\infty\}$ is the genus of $\Sigma$ and $b \in \mathbb{Z}_{\ge 0}$ is the number of (compact) boundary components of $\Sigma$. 
 
 Of particular interest to us is the infinite-type surface called the \textit{biinfinite flute} obtained from an infinite cylinder by deleting a countable discrete sequence of points exiting both ends of the cylinder (see Figure~\ref{fig:typeS}). By the classification theorem of Ker\'ekj\'art\'o and Richards, this surface can also be obtained from $\mathbb{S}^2$ by deleting $\{x_i\}$, $\{y_i\}$, $x$, and $y$, where $\{x_i\}$ and $\{y_i\}$ are countable discrete sequences of points converging to distinct points $x$ and $y$, respectively. Note that $S$ has two special non-isolated ends.  
 
 \subsection{Mapping class groups and arc graphs}
 The \textit{mapping class group}, $\MCG(\Sigma)$, of a surface $\Sigma$ is the group of orientation-preserving homeomorphisms of $\Sigma$ up to isotopy. The natural topology on any group of homeomorphisms is the compact-open topology and $\MCG(\Sigma)$ is endowed with the quotient topology with respect to the compact-open topology on the space of homeomorphisms of $\Sigma$. When $\Sigma$ is a finite-type surface, this topology agrees with the discrete topology on $\MCG(\Sigma)$, but when $\Sigma$ is of infinite type it does not. There are several important subgroups of $\MCG(\Sigma)$: $\MCG_c(\Sigma)$ is the subgroup consisting of mapping classes with compact support, $\pmap(\Sigma)$ is the pure mapping class group consisting of mapping classes which fix the set of ends pointwise, $\overline{\MCG_c(\Sigma)} < \pmap(\Sigma)$ is the closure of the compactly supported mapping class group with respect to the topology described above, and when $\Sigma$ has an isolated puncture $p$, $\MCG(\Sigma,p)$ is the subgroup of mapping classes that fix $p$. 

 When $\Sigma$ is finite-type, $\MCG(\Sigma)$ is algebraically generated by finitely many Dehn twists \cite{Lickorish}. Infinite-type mapping class groups, sometimes called big mapping class groups, are uncountable groups, so there is no countable algebraic generating set. However, one can consider topological generating sets (countable dense subsets of $\MCG(\Sigma)$) and in \cite{PatelVlamis}, Vlamis and the third author prove that for many infinite-type surfaces, Dehn twists are not sufficient in topologically generating even $\pmap(\Sigma)$. They show that in addition to Dehn twists, a new class of homeomorphisms called \textit{handleshifts} (defined in Section~\ref{sec:backgroundshiftmaps}) are often needed to topologically generate $\pmap(\Sigma)$. In a subsequent paper with Aramayona \cite{APV}, Vlamis and the third author give an algebraic description of $\pmap(\Sigma)$ that will be relevant in Section~\ref{sec:loxo}. When $\Sigma$ is an infinite-type surface with $n > 1$ nonplanar ends, they prove that $\pmap(\Sigma) = \overline{\MCG_c(\Sigma)} \rtimes \mathbb{Z}^{n-1}$, where $\mathbb{Z}^{n-1}$ is generated by $n-1$ handleshifts with disjoint support. In particular, when $\Sigma$ has exactly 2 nonplanar ends (for example when $\Sigma$ is the ladder surface), $\pmap(\Sigma) = \overline{\MCG_c(\Sigma)} \rtimes \mathbb{Z}$ where $\mathbb{Z} = \langle H \rangle$ and $H$ is the standard handleshift, shifting each genus of $\Sigma$ over to the right by one.
 
In this paper we are primarily concerned with mapping classes of intrinsically infinite-type.

\begin{defn}\label{def:intrinsic}
An element  $f \in \MCG(\Sigma)$ is of \textit{intrinsically infinite-type} if $f \notin \overline{\MCG_c(\Sigma)}$. 
\end{defn}

\noindent More specifically, we are interested in how such elements act on a particular graph of arcs called the \textit{relative arc graph}. 

Let $\Sigma$ be a connected, orientable surface with empty boundary, and let $\Pi\subset \Sigma$ be the set of punctures of $\Sigma$, which we assume to be non-empty.  In this subsection, it is convenient to regard $\Pi$ as a set of marked points on $\Sigma$.  By a \textit{proper arc} on $\Sigma$ we mean a map $\alpha\colon [0,1]\to \Sigma$ such that $\alpha^{-1}(\Pi)=\{0,1\}$. We often conflate an arc with its image in $\Sigma$.  An arc is \textit{simple} if it is an embedding when restricted to the open interval $(0,1)$.

The \textit{arc graph} $\mathcal A(\Sigma)$ is the simplicial graph whose vertices are isotopy classes of simple arcs on $\Sigma$, where we only consider isotopies rel endpoints, and two (isotopy classes of) arcs are connected by an edge if they can be realized disjointly away from $\Pi$.  The mapping class group $\MCG(\Sigma)$ acts on $\mathcal A(\Sigma)$ by isometries. Hensel, Przytycki, and Webb \cite{HPW} show that when $\Sigma$ has finite-type, the graph $\mathcal A(\Sigma)$ is infinite diameter and 7--hyperbolic.  On the other hand, when $\Sigma$ is infinite-type with infinitely many punctures, it is straight-forward to see that $\mathcal A(\Sigma)$ has diameter 2, and so this graph is not particularly useful for studying $\MCG(\Sigma)$. 

Assuming that $\Pi$ contains a non-empty set of isolated punctures, Aramayona, Fossas, and Parlier \cite{AFP} construct a particular subgraph of the arc graph which has interesting geometry, even when $\Pi$ is infinite.  We are interested in a special case of this construction, involving a single isolated puncture $p$.

\begin{defn}\label{defn:modarcgraph}
The \textit{relative arc graph} $\mathcal{A}(\Sigma,p)$ is the subgraph of $\mc A(\Sigma)$ spanned by arcs which start and end at $p$.  More precisely, the vertices of $\mathcal{A}(\Sigma,p)$ are isotopy classes of arcs on $\Sigma$ with endpoints on $p$, where we allow only isotopy rel endpoints.  There is an edge between two (isotopy classes of) arcs if they can be realized disjointly away from $p$.
\end{defn}

Aramayona, Fossas, and Parlier show that $\mathcal A(\Sigma,p)$ is connected, has infinite diameter, and is $7$--hyperbolic (see \cite[Theorem~1.1]{AFP}).  While $\MCG(\Sigma)$ does not necessarily act on $\mathcal A(\Sigma,p)$, the subgroup $\MCG(\Sigma,p)$ that fixes the puncture $p$ does act by isometries on this graph. When $\Sigma$ has only one isolated puncture $p$, $\MCG(\Sigma) = \MCG(\Sigma, p)$.

\subsection{Metric spaces and loxodromic isometries}

We now introduce some basics of metric spaces and isometries of a hyperbolic metric space.
Given a metric space $X$, we denote by $d_X$ the distance function on $X$. A map $f\colon X\to Y$ between metric spaces $X$ and $Y$ is a \textit{$(K,C)$--quasi-isometric embedding} if there is are constants $K\geq 1$, $C\geq 0$ such that for all $x,y\in X$, \[\frac1K d_X(x,y)-C\leq d_Y(f(x),f(y))\leq  Kd_X(x,y)+C.\]    A \textit{geodesic} in $X$ is an isometric embedding of an interval into $X$ and a \textit{$(K,C)$--quasi-geodesic} in $X$ is a $(K,C)$--quasi-isometric embedding of an interval into $X$. We call the constants $K,C$ the \textit{quality} of the quasi-geodesic. By an abuse of notation, we often conflate a (quasi-)geodesic and its image in $X$.

\begin{defn}
Given an action by isometries of a group $G$ on a hyperbolic space $X$, an element $g\in G$ is \textit{elliptic} if it has bounded orbits; \textit{loxodromic} if the map $\Z\to X$ given by $n\mapsto g^n x_0$ for some (equivalently, any) $x_0\in X$ is a quasi-isometric embedding; and \textit{parabolic} otherwise.
\end{defn}

Any bi-infinite quasi-geodesic in $X$ which is preserved by a loxodromic isometry $ g\in G$ is called an \textit{axis} of $g$.  An axis always exists; for any $x_0\in X$, the set $\{g^nx_0\mid n\in \Z\}$ is a (discrete) quasi-geodesic preserved by $g$.  If $X$ is a geodesic metric space, in the sense that there exists a geodesic connecting any two points of $X$, then we may construct a continuous quasi-geodesic axis as follows.  Fix a geodesic $[x_0,gx_0]$ from $x_0$ to $gx_0$.  Then $g$ stabilizes the path formed by concatenating the geodesics $g^n[x_0,gx_0]$; this path is a quasi-geodesic axis of $g$ in $X$.  Varying the point $x_0$ will change the quality of the quasi-geodesic.   Let $g^+=\lim_{n\to\infty}g^nx_0$ and $g^-=\lim_{n\to-\infty}g^nx_0$ be points in the Gromov boundary $\partial X$ of $X$.  The \textit{limit set} of $\langle g\rangle$ is the subset $\{g^+,g^-\}\subseteq \partial X$; this set is fixed pointwise by $g$.  It is straightforward to show that the limit set $\{g^+,g^-\}$ does not depend on the choice of $x_0\in X$.
 
\subsection{Shift maps and the biinfinite flute surface}\label{sec:backgroundshiftmaps}

A handleshift was first defined in \cite{PatelVlamis} as follows. Consider the surface $S'$ defined by taking the strip $\mathbb{R}\times [-1,1]$, removing a disk of radius $\frac{1}{2}$ with center $(n, 0)$ for each $n \in \mathbb{Z}$, and attaching a torus with one boundary component to the boundary of each such disk. A handleshift on $S'$ is the homeomorphism that acts like a translation, sending $(x, y)$ in $S$ to $(x +1, y)$ and which tapers to the identity on $\partial S'$. Given a surface of infinite-genus $\Sigma$ with at least two nonplanar ends and a proper embedding of $S'$ into $\Sigma$ so that the two ends of the strip correspond to two distinct ends of $\Sigma$, the handleshift on $S'$ induces a handleshift on $\Sigma$, where the homeomorphism acts as the identity on the complement of $S'$. In this paper, more flexibility is allowed, and we define the following generalization.

\begin{defn}\label{def:shift}
Let $S'$ be the surface defined by taking the strip $\mathbb{R}\times [-1,1]$, removing a closed disk of radius $\frac{1}{4}$ with center $(n, 0)$ for $n \in \mathbb{Z}$, and attaching any fixed topologically non-trivial surface with exactly one boundary component to the boundary of each such disk.  A \textit{shift} on $S'$ is the homeomorphism that acts like a translation, sending $(x,y)$ in $S'$ to $(x +1, y)$ and which tapers to the identity on $\partial S'$. 
\end{defn}

\noindent  Lanier and Loving use two particular cases of this generalization in \cite{LanierLoving}. Naming the full generalization a ``shift" is in line with their paper. Note that it is essential for the same surface to be glued to the boundary component of each disk in order for the shift to be a homeomorphism of the surface. 

As above, given a surface $\Sigma$ with a proper embedding of $S'$ into $\Sigma$ so that the two ends of the strip correspond to two different ends of $\Sigma$, the shift on $S'$ induces a shift on $\Sigma$, where the homeomorphism acts as the identity on the complement of $S'$. Given a shift $h$ on $\Sigma$, the embedded copy of $S'$ in $\Sigma$ is called the \textit{domain} of $h$.  In this paper, we produce special homeomorphisms that can be obtained as a composition of three shift maps on such a surface $\Sigma$ with an isolated puncture $p$ and that are loxodromic  with respect to the action of $\MCG(\Sigma, p)$ on $\mathcal{A}(\Sigma, p)$. Instead of working generally with surfaces that admit shift maps, we begin by letting $S$ be the biinfinite flute surface. Then, $S$ admits shift maps which shift a countable collection of punctures on $S$. To prove Theorem~\ref{thm:loxomain} we first construct mapping classes that are loxodromic with respect to the action of $\MCG(S, p)$ on $\mathcal{A}(S, p)$.  We then use this surface as a template for constructing the desired mapping classes for more general surfaces $\Sigma$ by extending the shift maps on $S$ to shift maps on $\Sigma$ as follows.

\begin{defn}\label{def:typeS}
Let $S$ be the biinfinite flute surface. A surface $\Sigma$ with an isolated puncture $p$ is \textit{of type $\mathcal{S}$} if there exists a proper embedding $S \hookrightarrow \Sigma$ where $S$ contains $p$, the two non-isolated ends of $S$ correspond to distinct ends of $\Sigma$, and such that a countably infinite collection of connected components of $\Sigma \setminus S$ are of the same (nontrivial) topological type. Note that when the components are once-punctured disks, there are countably many isolated punctures of $S$ that remain isolated punctures when embedded in $\Sigma$. Denote this special class of connected components of $\Sigma\setminus S$ by $\mathcal{U}$, so that the elements of $\mathcal{U}$ are all homeomorphic to a fixed surface $\Sigma_0$ with one boundary component. See Figure~\ref{fig:exs} for some examples of surfaces of type $\mathcal{S}$. 
\end{defn}


Given a shift map $h$ on $S$, the support of $h$ is a strip $\mathbb{R}\times [-1,1]$ with countably many punctures. When the set of punctures in the support of $h$ only consists of those corresponding to elements of $\mathcal{U}$, we can glue copies of $\Sigma_0$ onto the punctures of this strip to produce a shift map on a surface $S'$ as in Definition~\ref{def:shift}. The embedding of $S$ in $\Sigma$ therefore gives an embedding of $S'$ in $\Sigma$ and the shift on $S'$ in $\Sigma$ is extended via the identity on $\Sigma\setminus S'$ as usual. From this construction, we immediately have one direction of the following lemma. 	

\begin{lem}\label{lem:equiv}
Given a surface $\Sigma$ with an isolated puncture, $\Sigma$ is of type $\mathcal{S}$ if and only if $\Sigma$ admits shift maps.
\end{lem} 

\begin{proof}
It is left to show that if $\Sigma$ has an isolated puncture $p$ and admits a shift map, then $\Sigma$ is of type $\mathcal{S}$. To see this, we consider the proper embedding of $S'$ into $\Sigma$. Recall that $S'$ is obtained from a punctured strip by gluing on countably many copies of any surface $\Sigma_0$ with exactly one boundary component. Let $\mathcal{T} = \{T_i\}$ denote the corresponding countable collection of subsurfaces homeomorphic to $\Sigma_0$ in $\Sigma$, indexed by $\Z$. 

Note that $E(S')$ is a closed subset of $E(\Sigma)$, as is $X = E(S') \cup \{p\}$, and thus $X^c= E(\Sigma)\setminus X$ is open in $E(\Sigma)$. The second countability of the topology on $E(\Sigma)$ implies that $X^c$ is the union of countably many basis elements. If $X^c$ is in fact clopen in $E(\Sigma)$, then $X^c$ is compact and is therefore a finite union of basis elements. In this case, there exists a simple closed curve $\gamma$ in $\Sigma$ with the following property: there exists a connected component $K$ of $\Sigma \setminus \gamma$ such that the end space of $\overline K = K \cup \gamma$ is exactly $X^c$. In this way, $\gamma$ cuts away the ends of $\Sigma$ that are in $X^c$ (see Figure~\ref{fig:shiftsareS}). We then have that $$\Sigma \setminus \left( \overline K \, \bigcup \, \left(\cup_i T_i \right)\right)$$ is homeomorphic to the biinfinite flute surface and $\Sigma$ is of type $\mathcal{S}$ with $\mathcal{T}$ playing the role of $\mathcal{U}$ in the definition of a surface of type $\mathcal{S}$. 

In general, we can only assume that $X^c$ is open, not clopen, so that it can be expressed as the union of countably many basis elements for the topology. Then, there exists a countable collection of simple closed curves $\{\gamma_i\}$ and a countable collection of connected components $K_i$ of $\Sigma \setminus \gamma_i$ such that the end space of $\bigcup_i \overline K_i =  \bigcup_i K_i \cup \gamma_i$ is exactly $X^c$ (see Figure~\ref{fig:shiftsareS}). In this case, $$\Sigma \setminus \left( \left(\cup_i \overline K_i\right) \, \bigcup \, \left(\cup_i T_i \right)\right)$$ is homeomorphic to the biinfinite flute surface and $\Sigma$ is of type $\mathcal{S}$.
\end{proof}

\begin{figure}[h]
\begin{center}
\begin{overpic}[trim = 1.8in 3.5in 1.15in 2.15in, clip=true, totalheight=0.55\textheight]{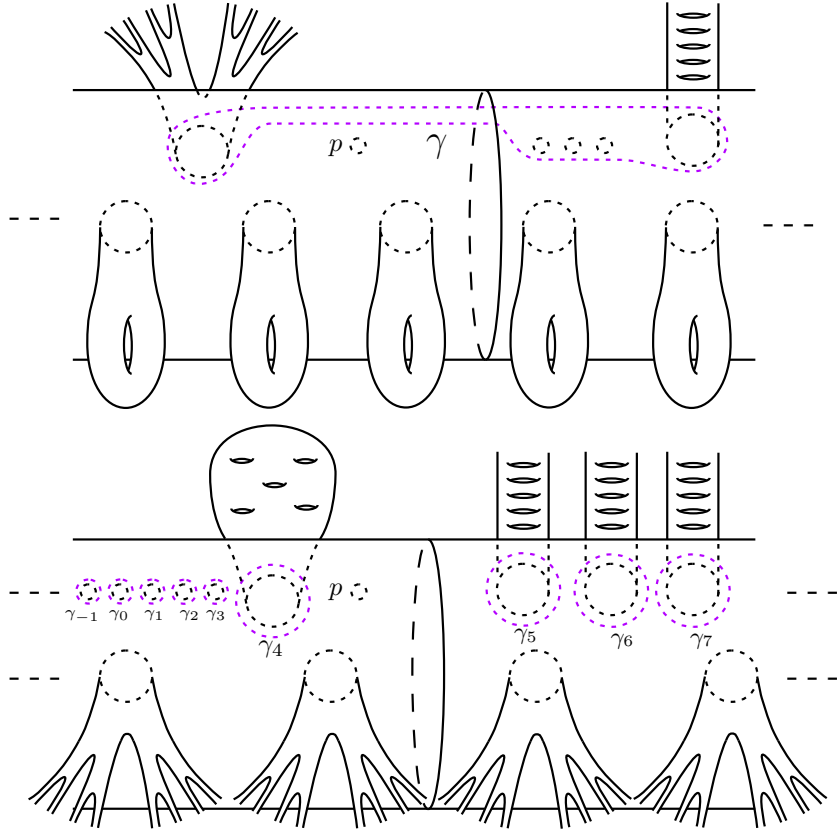}
\put(39,78){$p$}
\put(39,27.5){$p$}
\put(50,78){\LARGE{$\gamma$}}
\put(31,21){$\gamma_4$}
\put(25,25){\tiny{$\gamma_{3}$}}
\put(22,25){\tiny{$\gamma_{2}$}}
\put(18,25){\tiny{$\gamma_{1}$}}
\put(14,25){\tiny{$\gamma_{0}$}}
\put(9,25){\tiny{$\gamma_{-1}$}}
\put(60,22.5){\small{$\gamma_{5}$}}
\put(71,22){\small{$\gamma_{6}$}}
\put(80,22){\small{$\gamma_{7}$}}
\end{overpic}
\caption{Examples showing that surfaces with shift maps are always of type $\mathcal{S}$. For the first surface, only one curve $\gamma$ is needed to cut away the extra topology of $\Sigma$. In the second case, a countable collection of curves $\{\gamma_i\}$ is needed.}\label{fig:shiftsareS}
\end{center}
\end{figure}

Given this equivalent definition for a surface of type $\mathcal{S}$, we can show that this class includes all  infinite-type surfaces with an isolated puncture outside of two sporadic classes that do not admit shift maps. We will need the following definition. 

\begin{defn}
The \textit{Loch Ness Monster} is the infinite-type surface with no planar ends and exactly one non-planar end. An infinite-type surface is a \textit{fluted Loch Ness Monster} if it is obtained from the Loch Ness Monster in one of the two following ways: 1) by deleting a finite, non-zero collection of isolated points, or 2) deleting a countably infinite collection of isolated points accumulating to  exactly one point, which we also delete from the surface, or accumulating onto the end of the Loch Ness Monster.  See Figure~\ref{fig:sporadic} for examples of fluted Loch Ness Monsters. 
\end{defn}


\begin{lem}\label{lem:sporadic}
Let $\Sigma$ be an infinite-type surface with an isolated puncture. Then $\Sigma$ is of type $\mathcal{S}$ unless $\Sigma$ is a flute surface with finite (possibly zero) genus or is a fluted Loch Ness Monster surface. 
\end{lem}

\begin{proof}
Let $\Sigma$ be an infinite-type surface with an isolated puncture $p$. If $\Sigma$ has at least two non-planar ends, then $\Sigma$ admits a shift map (in fact a handleshift). Similarly, if $\Sigma$ has at least two non-isolated planar ends, then $\Sigma$ admits a shift map with these two ends corresponding to the two ends of the strip $S'$ in Definition~\ref{def:shift}. Thus, if $\Sigma$ does not admit a shift map, $\Sigma$ has exactly one non-isolated planar end and finite genus, i.e., a flute surface with finite genus, or has exactly one non-planar end and up to one non-isolated planar end, i.e., a fluted Loch Ness Monster. 
\end{proof}

 Going back to the original definition of a surface of type $\mathcal{S}$, there are a few more notable remarks regarding the relationship between $S$ and $\Sigma$. First, there is not necessarily an embedding of $\MCG(S,p)$ into $\MCG(\Sigma,p)$ since if the support of a shift $h$ on $S$ contains punctures that do not correspond to elements of $\mathcal{U}$, then there may not be a way to extend that shift to $\Sigma$. In particular, if $h$ shifts one puncture $x$ to another puncture $x'$ but the topology of the surfaces glued to $x$ and $x'$ are different, there is no extension of $h$ to a shift of $\Sigma$. This will not affect our arguments since there are countably many punctures of $S$ corresponding to the elements of $\mathcal{U}$ which we move to the \textit{front} of the cylinder for $S$ along with $p$, and we move all other punctures to the \textit{back} of the cylinder. Here we are choosing one non-isolated end of $S$ to correspond to the left direction on the surface, and the other non-isolated end to correspond to moving right on the surface so that there is a well-defined notion of the front and back of $S$. In our constructions, we use  shift maps on $S$ whose support only contains the punctures on the front of $S$ so that all of these shifts extend to $\Sigma$.

Second, and most importantly, we now show that proving Theorem \ref{thm:loxomain} for surfaces $\Sigma$ of type $\mathcal S$ can be reduced to the case of the biinfinite flute $S$. In fact, this is the motivation for the original definition of a surface of type $\mathcal{S}$. 
For simplicity, given a surface $M$ with an isolated puncture $p$ and any points $a,b\in\mathcal A(M,p)$, we write $d_M(a,b)$ for the distance between $a$ and $b$ in $\mathcal A(M,p)$

\begin{lem} \label{lem:AsqiembsAsigma}
Let $\Sigma$ be an infinite-type surface with an isolated puncture $p$ of type $\mathcal{S}$. Then the inclusion of $\mathcal A(S,p)$ into $\mathcal A(\Sigma,p)$ is a $(2,0)$--quasi-isometric embedding.
\end{lem}

\begin{proof}
As $S\subset \Sigma$, it is clear that   $d_{\Sigma}(a,b)\leq d_{S}(a,b)$ for any $a,b\in\mathcal A(S,p)$.   

To obtain the other inequality, let $S_{a,b}\subset S$ be a finite-type subsurface of $S$ which contains $a,b$, the puncture $p$, and has complexity at least $2$.  Note that $S_{a,b}$ is then a finite-type subsurface of $\Sigma$ as well.  Thus by \cite[Corollary 4.3]{AFP} applied to $S_{a,b}\subset \Sigma$ and to $S_{a,b}\subset S$, we have 
\[d_S(a,b)\leq d_{S_{a,b}}(a,b)\leq 2d_{\Sigma}(a,b).\]  Together, these imply that \[d_\Sigma(a,b)\leq d_S(a,b)\leq 2d_\Sigma(a,b),\] completing the proof.
\end{proof}

In particular, let $g\in \MCG(S,p)$ be loxodromic with respect to the action of $\MCG(S,p)$ on $\mc A(S,p)$ with a $(K,C)$--quasi-geodesic axis.  If $g$ can be extended to an element of $\MCG(\Sigma,p)$, then this extension is loxodromic with respect to the action of  $\MCG(\Sigma,p)$ on $\mc A(\Sigma, p)$, and the extension will have a $(2K,C)$--quasi-geodesic axis.


\section{Coding arcs and standard position}\label{sec:codingarcs}
Let  $S$ be the biinfinite flute surface with a distinguished isolated puncture $p$, and let $\{p_i\}_{i\in\Z}$ be any countably infinite discrete collection of punctures on $S$ which exits both ends of the cylinder and does not contain $p$.
As described in Section \ref{sec:backgroundshiftmaps}, we choose one non-isolated end of $S$ to correspond to the left direction and one to correspond to the right direction, which gives a well-defined notion of a front and back of the cylinder for $S$.  We move all of the punctures in $\{p_i\mid i\in \Z\}\cup\{p\}$ to the front of the cylinder for $S$ and all other punctures to the back.   We also move the distinguished puncture $p$ so that it lies to the right of $p_{-1}$ and to the left of $p_0$. We will consider the collection $\{p_i\mid i\in \Z\}\cup \{p\}$ of punctures.  We index this set with $\Z\cup\{P\}$, which we give the ordering consisting of the usual ordering on $\Z$ with the additional requirement that $-1<P<0$.  The index $P$ corresponds to the distinguished puncture $p$.

Fix the simple closed curve $B_0$ bounding the puncture $p_0$ on $S$ shown in Figure \ref{fig:B0}.  More formally, to define $B_0$ we fix a complete hyperbolic metric on $S$ and let $B_0$ be a horocycle at a height sufficiently far out the cusp.   Fix a shift map $H$ on $S$ whose domain contains exactly the collection $\{p_i\}$ for $i \in \mathbb{Z} \cup P$ and which shifts $p_i$ to $p_{i+1}$ for all $i \in \Z \cup P$.
\begin{figure}[h]
\begin{center}
\begin{overpic}[trim = 0.45in 4.75in 1.05in 0.85in, clip=true, width=4in]{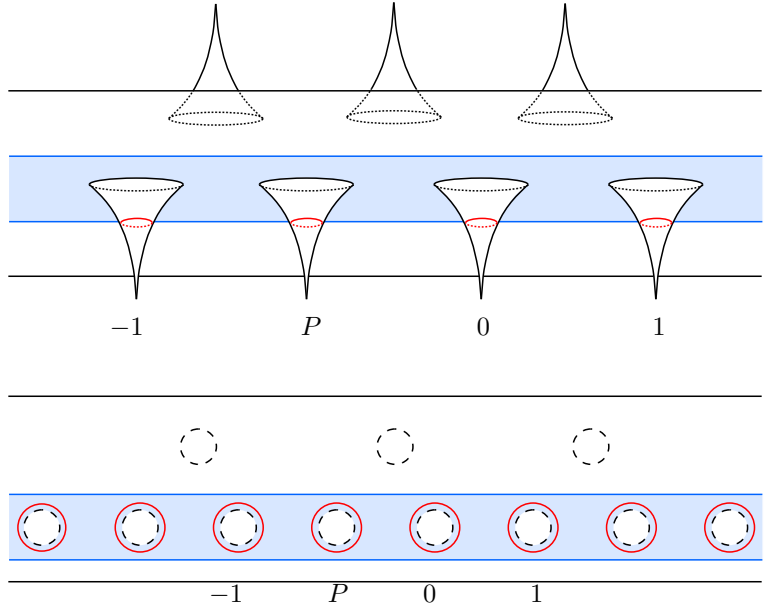}
\put(39, 33){$P$}
\put(62, 33){$0$}
\put(14, 33){$-1$}
\put(85, 33){$1$}
\put(42.5, -2){$P$}
\put(55, -2){$0$}
\put(69, -2){$1$}
\put(27, -2){$-1$}
\end{overpic}
\caption{The curves $B_i$ are in red.  The blue region is the domain of the shift map $H$.}\label{fig:B0}
\end{center}
\end{figure}

\begin{defn}\label{def:B_i}
Define the simple closed curves $B_i:=H^iB_0$ for $i\in \mathbb Z \cup P$.  Then $B_i$ is a simple closed curve bounding the puncture $p_i$, where $p_P=p$.
\end{defn}

\noindent  Our choice of left/right also gives a well-defined notion of an arc passing \textit{over} or \textit{under} a puncture (or equivalently some $B_i$). In all pictures of $S$ throughout the paper, we denote the special puncture $p$ by an ``X," and rather than drawing the punctures $p_i$, we draw the simple closed curves $B_i$ in $S$.   We will use these simple closed curves to put arcs into \textit{standard position} as described later in this section.

\subsection{Coding arcs}\label{sec:thecode}
We use the simple closed curves $B_i$ to  describe a way to code simple arcs on $S$ starting and ending at $p$.  We will use this code to quantify how long two arcs fellow travel, which will be essential for proving the results of this paper.

Suppose that $\gamma$ is an oriented arc on $S$ starting and ending at $p$ such that $\gamma$ can be homotoped to be completely contained on the front of $S$.   We code $\gamma$ as follows. First homotope $\gamma$ so that it is disjoint from all $B_i$ with $i\in \Z \cup P$, with the exception that $\gamma$ starts and ends at the puncture $p$ and therefore intersects $B_P$ exactly twice. The code for $\gamma$ always starts and ends with the character $P_s$ (which stands for ``puncture start") and contains either the character $k_o$ or the character $k_u$, where $k\in \Z\cup\{P\}$, whenever $\gamma$ passes over or under the simple closed curve $B_k$ for $k\in \Z \cup P$. These characters appear in the code for $\gamma$ in the same order in which $\gamma$ passes over/under the curves $B_k$.  For example, since $-1<P<0$, the second character of the code for $\gamma$ must be either $0_o$, $0_u$, $(-1)_o$, $(-1)_u$, $P_o$, or $P_u$, because if $\gamma$ doesn't immediately wrap around $p$ (which would lead to the second character being $P_o$ or $P_u$),  it must pass over or under either $B_0$ or $B_{-1}$ before it can pass over or under $B_k$ for any $k\neq 0,-1$.   Similarly, if the character $2_o$ or $2_u$ appears in the code, each adjacent character must be one of $1_o, 1_u, 2_o, 2_u, 3_o,$ or $3_u$. To simplify notation, we write $k_{o/u}$ to mean that the character could be $k_o$ or $k_u$. We will write $k_{o/u}k_{u/o}$ to mean that the two adjacent characters are either $k_ok_u$ or $k_uk_o$; the $k_{u/o}$ is used to emphasize that the second character has the opposite subscript as the first one. 

\begin{ex}  \label{ex:codes}
Consider the arcs shown in Figure \ref{fig:examplecodes}.  The elements $k\in\Z\cup \{P\}$ shown under $S$ denote the subscript on the simple closed curves $B_k$.   The code for $\alpha$ is $P_s0_o1_u2_o2_u1_u0_uP_s$, the code for $\beta$ is  $P_sP_uP_o0_o1_o2_o2_u1_o0_oP_s$, and the code for $\gamma$ is $P_s(-1)_o(-2)_o(-2)_u(-1)_uP_u0_u1_u1_o0_oP_s$.
\end{ex}

Now suppose $\gamma$ is an oriented arc on $S$ starting and ending at $p$ such that no arc in its homotopy class is contained on the front of $S$.  Since $\gamma$ starts and ends at $p$, which is on the front of the surface, every time $\gamma$ leaves the front of $S$ it must eventually re-enter the front.  We give the code $C$ to any subpath of $\gamma$ which is on the back of $S$.  Up to homotopy, we may assume that each time $\gamma$ exits then enters the front of $S$, it does so ``between" two simple closed curves $B_k$ and $B_{k+1}$. In other words, there is an arc $\gamma'$ in the homotopy class of $\gamma$ whose code contains either $k_{o/u}C(k+1)_{o/u}$ or $k_{o/u}Ck_{o/u}$ each time $\gamma'$ leaves the front of $S$.  We give $\gamma$ the same code as $\gamma'$.  We emphasize that this implies that the code of an arc does not distinguish the behavior of arcs $\gamma$ on the back of  $S$.  

\begin{figure}
\centering
\def\svgwidth{4in}
\Small{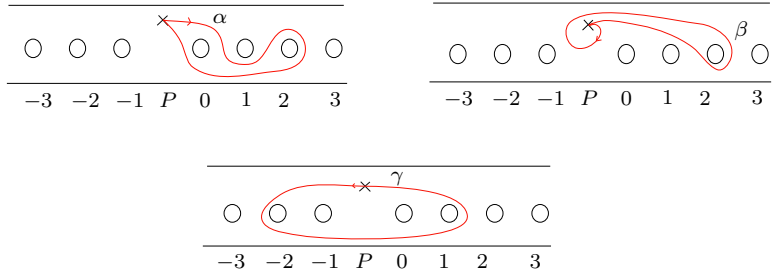}
\caption{Pictured are arcs on the front of the surface $S$.  The X denotes the puncture $p$ and the elements $k\in\Z\cup \{P\}$ shown under $S$ denote the subscript on the simple closed curves $B_k$. The codes for arcs $\alpha,\beta, \gamma$ are given in Example \ref{ex:codes}.}
\label{fig:examplecodes}
\end{figure}

By an abuse of notation, we typically blur the distinction between an arc and its code, writing, for example, $\alpha = P_s0_o1_u2_o2_u1_u0_uP_s$.

\begin{defn}\label{def:reducedcode}
Let $\gamma$ be an oriented arc on $S$ starting and ending at $p$.  A code for $\gamma$ is \textit{reduced} if no two adjacent characters in the code are the same and if the character immediately following the initial $P_s$ or preceding the terminal $P_s$ is not $P_{o/u}$.
\end{defn}

The appearance of repeated characters in the code of an arc indicates backtracking in the arc. 
The following lemma is immediate.

\begin{lem}
If there are two arc $\gamma$ and $\gamma'$, starting and ending at $p$, whose codes differ only by the removal of two adjacent  characters which are equal, i.e., $k_ok_o$ or $k_uk_u$, then $\gamma$ and $\gamma'$ are homotopic.
\end{lem}

\begin{figure}[H]
\centering
\def\svgwidth{4in}
\small{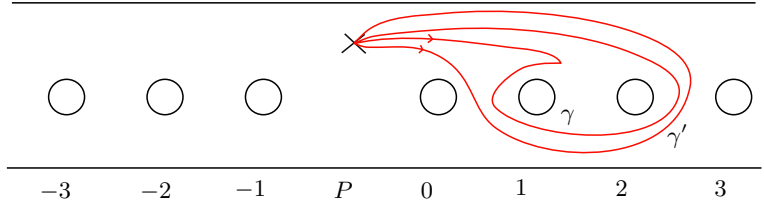}
\caption{The arcs $\gamma$ and $\gamma'$ are homotopic and have the same reduced code: $\gamma'$ is formed from $\gamma$ by the removal of the pair $1_o1_o$.}
\label{fig:reducingcode}
\end{figure}

\begin{ex}
The arcs $\gamma$ and $\gamma'$ in Figure~\ref{fig:reducingcode}, with codes $P_s0_o1_o1_o1_u2_u2_o1_o0_oP_s$ and $P_s0_o1_u2_u2_o1_o0_oP_s$, respectively, are homotopic.  
\end{ex}

\noindent Note that if a triple appears in the code for an arc, it is reduced to a single character according to our convention, as only \emph{pairs} of repeated characters are removed. For example, $P_s0_o1_o1_o1_o1_u0_oP_s$ is reduced to $P_s0_o1_o1_u0_oP_s$.

Each homotopy class of curves on $S$ determines a reduced code, in the sense that any two homotopic curves have the same reduced code.  We write that two codes are equal if they determine homotopic arcs.  For example, we write $$P_s0_o1_o1_o1_u2_u1_u0_uP_s =P_s0_o1_u2_u1_u0_uP_s.$$  The converse of this fact is not true, however, because the code does not encode the behavior of arcs on the back of $S$; hence there can be non-homotopic arcs with the same reduced code.  This will not cause any problems in this paper.

\begin{defn}\label{def:codelength}
The \textit{code length} of an arc $\gamma$, denoted $\ell_c(\gamma)$, is the number of characters in a reduced code for $\gamma$. 
\end{defn}

\begin{conv}\label{conv:alphabet}
When giving the code of an arc for which the numerical values of the characters are unimportant (or unknown), we will use variables in the code.  Our convention is to use Roman letters to represent single characters and Greek letters to represent strings of characters whose length is (possibly) greater than one.  For example, $\ell_c(a_1a_2a_3)=3$ while $\ell_c(a\gamma b)=\ell_c(\gamma)+2$.
\end{conv}

Given a string of characters $\alpha=a_1a_2\ldots a_n$, we denote by $\overline{\alpha}$ the reverse of $\alpha$, so that $\overline{\alpha}=a_na_{n-1}\ldots a_2a_1$.  If $\alpha$ is an arc, then  $\overline{\alpha}$ is the same arc with the opposite orientation.

\subsection{Standard position}\label{sec:standardposition}

In this section, we  describe how to use the code for an arc to find its image under a general class of shifts which we call ``permissible". 

\begin{defn}\label{def:handleshift}
We say a shift shifting to the right is a \textit{right shift}, while a shift shifting to the left is a \textit{left shift}.
A right shift is  \textit{permissible} if its domain $D$ stays on the front of our subsurface and contains a \textit{turbulent region} $(n_1,n_2)$, that is, there exist $n_1,n_2\in\Z\cup\{P\}$ with $n_1<n_2$ such that $D$ contains $B_k$ for all $k\in(-\infty,n_1]\cup[n_2,\infty)$ but does not contain $B_k$ for any $k\in(n_1,n_2)$. We call $(-\infty,n_1)\cup[n_2,\infty)$ the \textit{shift region} of $h$. Analogously, a left shift  is  \textit{permissible} if its domain $D$ stays on the front of our subsurface and contains a \textit{turbulent region} $(n_2,n_1)$, that is, there exist $n_1,n_2\in\Z\cup\{P\}$ with $n_2<n_1$ such that $D$ contains $B_k$ for all  $k\in(-\infty,n_2]\cup[n_1,\infty)$ but  does not contain $B_k$ for any  $k\in(n_2,n_1)$.  The shift region for a left shift is $(-\infty,n_2]\cup(n_1,\infty)$. 
\end{defn}

\begin{conv}\label{conv:leftshift}
Throughout the paper, we will use both left and right shifts.  For notational simplicity, all general results about shifts will be stated for right shifts. All statements of results,  proofs, and figures will make this assumption as well.  However, all of our definitions and results (and their proofs) also hold for left shifts, by modifying any proof for a right shift so that we essentially replace all instances of $n_1$ with $n_2$ and vice versa and replace all instances of the word ``increasing" by the word ``decreasing" and vice versa.
The only subtleties are that:
	\begin{itemize}
	\item We retain the convention that $h(B_{n_1})=B_{n_2}$.

 \item For the shift region intervals $(-\infty,n_1)\cup[n_2,\infty)$ that appear for a right shift, we use $(-\infty,n_2]\cup(n_1,\infty)$ for the left shift.  In particular, the $n_2$ is always contained in the shift region. 
	\end{itemize}
\end{conv}

\begin{rem}
It is worthwhile to mention that Convention \ref{conv:leftshift} is equivalent to simply redefining the order, given by the symbol $<_{rev}$, on $\Z\cup\{P\}$ to be the opposite of the standard meaning of the inequality sign $<$.
For example, in this ``reversed order" we would have $5<_{rev}3$ and so on.
Given this and using the standard meanings for ``increasing" and ``decreasing" with respect to $<_{rev}$, all of the proofs for shifts that shift to the left would go through identically as shifts that shift to the right when one replaces each instance of $<$ with an $<_{rev}$.
Despite the simplicity of this reversed order, we found writing proofs with it to be more confusing to the reader than applying the above convention.
\end{rem}

In order to find the image of an arc using only its code, we will need to consider paths whose endpoints are not on $p$. 
 
\begin{defn} \label{def:backloop}
A \textit{segment} is a simple path with at least one endpoint which is not a puncture, and no endpoints on a puncture other than $p$.  We code a segment in an analogous way as we did arcs in Section \ref{sec:thecode}.  If a segment begins or ends on $p$, then the initial or terminal character of the code is $P_s$, respectively.  Note that a segment can have at most one instance of $P_s$ in its code.  Given a segment $\gamma$, we denote the initial and terminal character of its code by $\gamma^i$ and $\gamma^t$, respectively.  A segment is \textit{supported on} an interval $(a,b)\subset \Z\cup\{P\}$ if the numerical value of every character in its reduced code is contained in $(a,b)$.  A subsegment of $\gamma$ which is supported on an interval $(a,b)$ is denoted $\gamma|_{(a,b)}$, with similar notation for half-open and closed intervals.  A segment is \textit{(strictly) monotone} if the numerical value of the characters in its reduced code are (strictly) monotone as a subset of $\Z\cup\{P\}$. 
A segment with code $C$ is called a \textit{back loop}.
\end{defn}

The notion of left/right on the front of $S$ induces an orientation on strictly monotone segments contained on the front of $S$ in the following way.  If the terminal endpoint of such a segment $\gamma$  is to the right of the initial endpoint, then $\gamma$ is \textit{oriented to the right}.  Similarly, if the terminal endpoint is to the left of the initial endpoint, then $\gamma$ is \textit{oriented to the left}. Since $\gamma$ is strictly monotone, one of the above two possibilities must occur.  We note that single characters of a code represent strictly monotone segments and so can be oriented in this way.

We  will use the code for an arc or segment to find the image of the arc or segment under certain homeomorphisms of $S$.   The process can be complicated.  We now introduce a new way of concatenating strings of characters which will be more suited to finding the image of an arc or segment in certain situations.

\begin{defn}\label{def:effcat} Given two segments $\alpha$ and $\beta$ such that the terminal character of $\alpha$ agrees with the initial character of $\beta$ and such that these two characters have the same orientation, the \textit{efficient concatenation} of $\alpha$ and $\beta$, denoted $\alpha+\beta$, is formed by removing the terminal character of $\alpha$ to form a new string $\alpha'$ and concatenating this new string with $\beta$, resulting in $\alpha'\beta$.
\end{defn}

For example, 
$$P_s0_o+0_o1_o=P_s0_o1_o,$$
and
$$P_s0_o1_o2_o2_u+2_u2_u2_o1_o1_u=P_s0_o1_o2_o2_u2_u2_o1_o1_u=P_s0_o1_u,$$ 
\noindent where the middle term is an unreduced code and the final term is a reduced code.  See Figure \ref{fig:effcat}.  We note that if $\alpha$ and $\beta$ can be efficiently concatenated, then they cannot be concatenated, because $\alpha^t$ and $\beta^i$ have the same orientation.  By a similar reasoning, if two segments can be concatenated, then they cannot be efficiently concatenated.  Throughout the paper, we only (efficiently) concatenate two segments when it is possible.

\begin{figure}
\centering
\begin{overpic}[unit=4pt, width=3in]{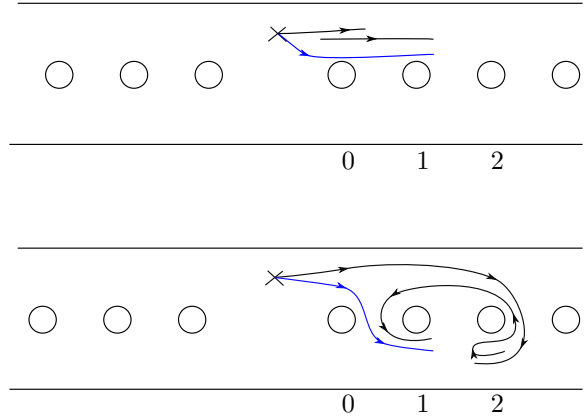}
\put(58,41.5){$0$}
\put(71,41.5){$1$}
\put(84,41.5){$2$}
\put(58,-1){$0$}
\put(71,-1){$1$}
\put(84,-1){$2$}
\end{overpic}
\caption{The two examples of efficient concatenation following Definition \ref{def:effcat}.  In each, the (reduced) code for the blue segment is the efficient concatenation of the codes of the two black segments.}
\label{fig:effcat}
\end{figure}

As written, the code of a segment is not well behaved under homotopy because every segment is homotopically trivial or homotopic into a puncture.  We will introduce a standard position for segments on $S$ with the property that any two segments that are homotopic rel endpoints will, in standard position, have the same reduced code.  Standard position will also allow us to find the image of a segment under a permissible shift using only its code.

\begin{defn}\label{defn:familyS}
 Fix a simple closed curve $S_j$ as in Figure \ref{fig:familyS} for each $j\in \Z\cup \{P\}$.  Let $\scc=\{S_j\mid j\in \Z\cup \{P\}\}$.
\end{defn}

\begin{figure}
\centering
\def\svgwidth{4in}
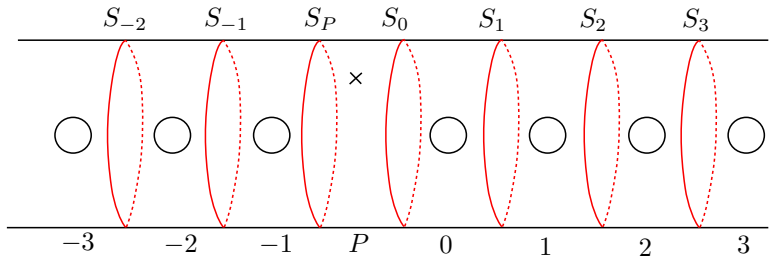
\caption{Some of the simple separating curves in $\scc$ defined in Definition \ref{defn:familyS}.}
\label{fig:familyS}
\end{figure}

For each $j\in\Z\cup\{P\}$, we orient the simple closed curve $B_j$ clockwise and identify $B_j$ with the subset $\mathbb S^1$ of $\mathbb C$ by a homeomorphism which preserves this orientation.  Fix points $b^j_L,b^j_R\in B_j$ corresponding to $-1,1\in\mathbb S^1$, respectively. Here $L$ and $R$ stand for \textit{left} and \textit{right}. 

To describe standard position, we will sometimes move endpoints of segments $\gamma$ to lie on various boundary components.  When we do this, we will use the following convention. Suppose $\gamma^i=k_{o/u}$ and we want to move the initial endpoint of $\gamma$ onto the boundary component $B_k$.  If $\gamma$ is oriented to the right, then we move the initial endpoint of $\gamma$ to  $b^k_L$, and if it is oriented to the left, we move the initial endpoint  to  $b^k_R$.  On the other hand, suppose $\gamma^t=(k')_{o/u}$ and  we want to move the terminal endpoint of $\gamma$ onto the boundary component $B_{k'}$.  If $\gamma$ is oriented to the right, we move the terminal endpoint of $\gamma$ to $b^{k'}_R$, and if it is oriented to the left, we move the terminal endpoint to $b^{k'}_L$. See Figure \ref{fig:changingendpoints}. Moving endpoints in this way does not change the code for $\gamma$.

\begin{figure}
\centering
\def\svgwidth{4in}
\begingroup%
  \makeatletter%
  \providecommand\color[2][]{%
    \errmessage{(Inkscape) Color is used for the text in Inkscape, but the package 'color.sty' is not loaded}%
    \renewcommand\color[2][]{}%
  }%
  \providecommand\transparent[1]{%
    \errmessage{(Inkscape) Transparency is used (non-zero) for the text in Inkscape, but the package 'transparent.sty' is not loaded}%
    \renewcommand\transparent[1]{}%
  }%
  \providecommand\rotatebox[2]{#2}%
  \newcommand*\fsize{\dimexpr\f@size pt\relax}%
  \newcommand*\lineheight[1]{\fontsize{\fsize}{#1\fsize}\selectfont}%
  \ifx\svgwidth\undefined%
    \setlength{\unitlength}{496.800003bp}%
    \ifx\svgscale\undefined%
      \relax%
    \else%
      \setlength{\unitlength}{\unitlength * \real{\svgscale}}%
    \fi%
  \else%
    \setlength{\unitlength}{\svgwidth}%
  \fi%
  \global\let\svgwidth\undefined%
  \global\let\svgscale\undefined%
  \makeatother%
  \begin{picture}(1,0.21404756)%
    \lineheight{1}%
    \setlength\tabcolsep{0pt}%
    \put(0,0){\includegraphics[width=\unitlength,page=1]{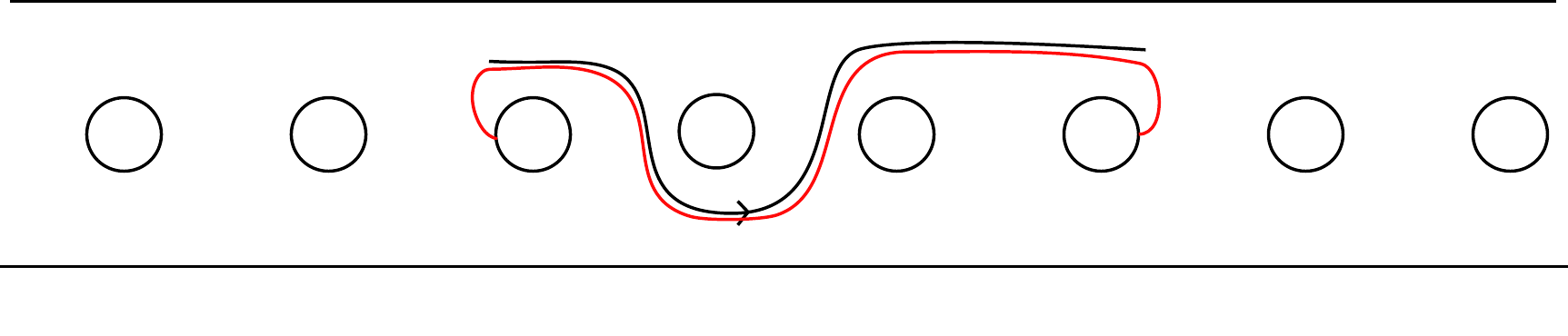}}%
    \put(0.33723146,0.00816987){\makebox(0,0)[lt]{\lineheight{1.25}\smash{\begin{tabular}[t]{l}$k$\end{tabular}}}}%
    \put(0.69519365,0.00425772){\makebox(0,0)[lt]{\lineheight{1.25}\smash{\begin{tabular}[t]{l}$k'$\end{tabular}}}}%
  \end{picture}%
\endgroup%

\caption{Adjusting the endpoints of $\gamma$ (black) results in the red segment.}
\label{fig:changingendpoints}
\end{figure}

It will be useful to understand how a given segment interacts with the domain of a permissible shift.
\begin{defn}
Let $h$ be a permissible right shift with domain $D$ and turbulent region $(n_1,n_2)$.  Then both (the upper and lower) boundary components of $D$ have the same reduced code on $(n_1,n_2)$.  For any $k,k'\in(n_1,n_2)$, we let $\partial D|_{[k,k']}$ be the reduced code of the boundary components of $D$ on the interval $[k,k']$ with similar notation for open and half-open intervals.  If a segment $\gamma$ has the same code as $\partial D$ on an interval $(k_1,k_2)\subset (n_1,n_2)$, we say that $\gamma$ \textit{follows} $\partial D$ or \textit{agrees with} $\partial D$ on that interval. As noted in the convention above, there is an analogous definition when $h$ is a left shift.  When a segment $\gamma$ supported on $[n_1,n_2]$ intersects one component of $\partial D$, we call this a \textit{half crossing}.  When such a $\gamma$ intersects both components of $\partial D$ so that the code for the subsegment of $\gamma$ between these two half crossings is empty, in the sense that this subsegment does not pass over or under $B_k$ for any $k$, we call this a \textit{full crossing}.  See Figure \ref{fig:crossings}.
\end{defn}

\begin{figure}
\centering
\def\svgwidth{4in}
\begingroup%
  \makeatletter%
  \providecommand\color[2][]{%
    \errmessage{(Inkscape) Color is used for the text in Inkscape, but the package 'color.sty' is not loaded}%
    \renewcommand\color[2][]{}%
  }%
  \providecommand\transparent[1]{%
    \errmessage{(Inkscape) Transparency is used (non-zero) for the text in Inkscape, but the package 'transparent.sty' is not loaded}%
    \renewcommand\transparent[1]{}%
  }%
  \providecommand\rotatebox[2]{#2}%
  \newcommand*\fsize{\dimexpr\f@size pt\relax}%
  \newcommand*\lineheight[1]{\fontsize{\fsize}{#1\fsize}\selectfont}%
  \ifx\svgwidth\undefined%
    \setlength{\unitlength}{503.80451737bp}%
    \ifx\svgscale\undefined%
      \relax%
    \else%
      \setlength{\unitlength}{\unitlength * \real{\svgscale}}%
    \fi%
  \else%
    \setlength{\unitlength}{\svgwidth}%
  \fi%
  \global\let\svgwidth\undefined%
  \global\let\svgscale\undefined%
  \makeatother%
  \begin{picture}(1,0.28949208)%
    \lineheight{1}%
    \setlength\tabcolsep{0pt}%
    \put(0,0){\includegraphics[width=\unitlength,page=1]{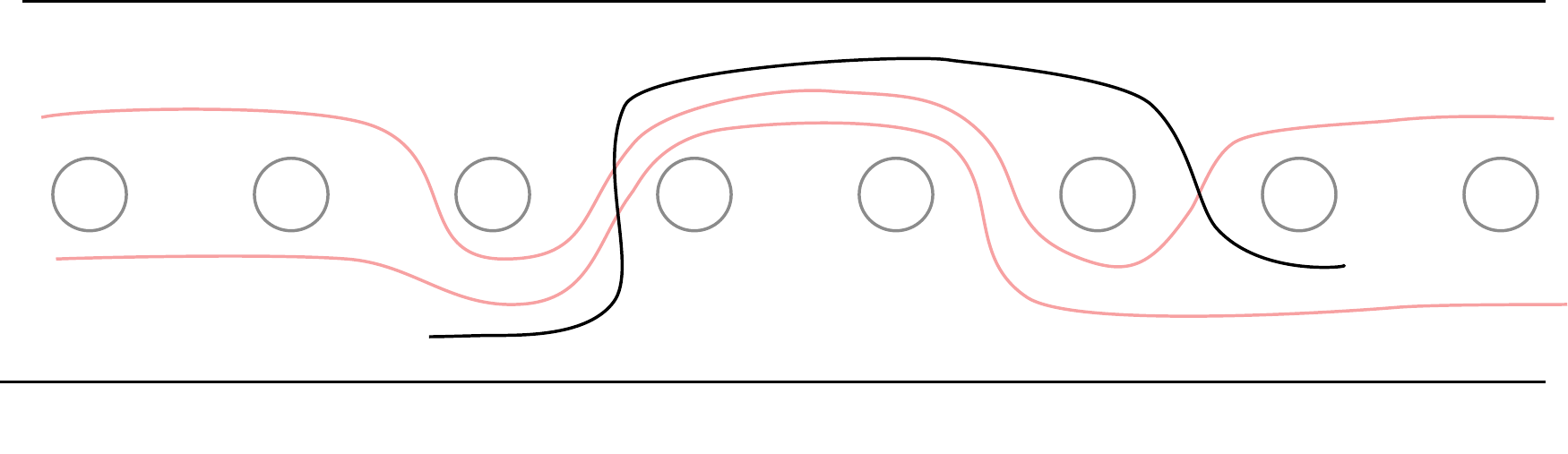}}%
    \put(0.81145014,0.00419852){\makebox(0,0)[lt]{\lineheight{1.25}\smash{\begin{tabular}[t]{l}$n_2$\end{tabular}}}}%
    \put(0.40318418,0.25957473){\makebox(0,0)[lt]{\lineheight{1.25}\smash{\begin{tabular}[t]{l}$\gamma$\end{tabular}}}}%
    \put(0.18141012,0.00923884){\makebox(0,0)[lt]{\lineheight{1.25}\smash{\begin{tabular}[t]{l}$n_1$\end{tabular}}}}%
  \end{picture}%
\endgroup%

\caption{The  domain $D$ for a permissible shift is shown in red.  The segment $\gamma$ follows $\partial D$ on $[n_1+1,n_1+3]$, has a half crossing between $B_{n_2-1}$ and $B_{n_2}$, and has a full crossing between $B_{n_1+1}$ and $B_{n_1+2}$.}
\label{fig:crossings}
\end{figure}
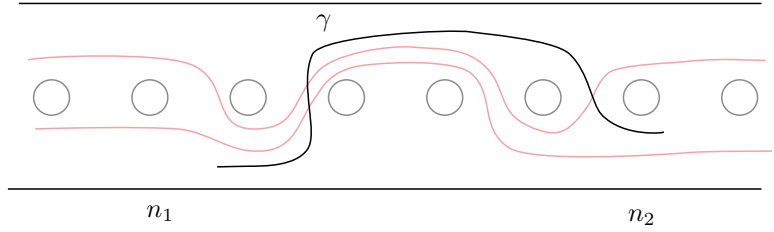

Even though $S$ is a straightforward surface, standard position for segments on $S$ is necessarily complicated. Before introducing it formally in the next two subsections, we briefly give an intuitive idea of standard position in the following remark.  On a first reading of this paper, we strongly suggest the reader read this remark and study Figures \ref{fig:ex1}--\ref{fig:ex4} instead of reading the formal definition of standard position given in Sections \ref{sec:segnobackloop} and \ref{sec:Ccode}.  The reader may then safely skip to Section \ref{sec:segconn}.  If, later in the paper, the image of a particular segment seems counterintuitive, likely this is because we put the segment in standard position before taking its image.  This would be a good time to look back at Sections \ref{sec:segnobackloop} and \ref{sec:Ccode} with that  example of a segment in mind.
\begin{rem}
 Let $h$ be a permissible right shift with domain $D$ and turbulent region $(n_1,n_2)$, and let $\gamma$ be a segment.  To put $\gamma$ in standard position with respect to $h$, we first homotope  subsegments of $\gamma$ that are contained in the  region $(-\infty,n_1)\cup(n_2,\infty)$ to be completely contained in $D$.   For the subsegments of $\gamma$  contained in the  region $[n_1,n_2]$, we homotope $\gamma$ so that crossings are full crossings whenever possible.  We will always be able to make  crossings full except near $n_1$ or $n_2$, because $n_1$ and $n_2$ are where $\gamma$ leaves the turbulent region and enters the shift region, where we have already ensured that it is contained in $D$.  We also homotope $\gamma$ to minimize the number of full crossings.  If $\gamma$ contains a loop $(n_1)_{o/u}(n_1)_{u/o}$ or   $(n_2)_{o/u}(n_2)_{u/o}$, we homotope $\gamma$ so that it crosses one connected component of $\partial D$ immediately before the loop and the other connected component immediately after.  In other words, $\gamma$ has to enter one side of $D$ before the loop and exit the other side of $D$ after the loop. If $\gamma$ contains the character $C$, we homotope it so that the subsegment of $\gamma$ represented by $C$ exits and reenters the front of $S$ at the same place, in the sense that the endpoints of this subsegment lie on the same simple closed curve $S_i$.  Finally, we also homotope each endpoint of $\gamma$ to lie either on the nearest $B_k$ (if it is in the shift region) or  the nearest $S_k$ (if it is in the turbulent region).
\end{rem}

Standard position will be slightly different for those segments which contain back loops.  We first discuss standard position for segments that do not contain back loops.

\subsubsection{Segments without back loops}\label{sec:segnobackloop} Given a segment $\gamma$ whose reduced code does not contain $C$ and a permissible right shift $h$ with domain $D$ and turbulent region $(n_1,n_2)$, we put $\gamma$ into \textit{standard position with respect to $h$} as follows.

If the endpoints of $\gamma$ are not contained in $(n_1,n_2)$ and it is possible to homotope $\gamma$ completely inside of the domain of $h$, we do so, and we move the endpoints of $\gamma$ to lie on the $B_k$ curve numbered by the first and last characters of the code as described above.  

Otherwise,  $\gamma$ can be written as the concatenation $\gamma_1\cdots\gamma_k$ of disjoint connected maximal subsegments  such that each $\gamma_i$ is either:
\begin{enumerate}[(a)]
\item supported on either $(-\infty,n_1]$ or $[n_2,\infty)$; or
\item supported on $(n_1,n_2)$.
\end{enumerate}

We now homotope each $\gamma_i$ individually.
If $\gamma_i$ satisfies (a), then we move the endpoints of $\gamma_i$ onto the $B_k$ curve numbered by the first and last character of the code as above and homotope the interior of $\gamma_i$ to lie completely inside the domain of $h$.  We homotope segments $\gamma_i$ satisfying (b) using the following procedure:

\begin{figure}
\centering
\begin{subfigure}{.6\textwidth}
\centering
\def\svgwidth{2.5in}
\begingroup%
  \makeatletter%
  \providecommand\color[2][]{%
    \errmessage{(Inkscape) Color is used for the text in Inkscape, but the package 'color.sty' is not loaded}%
    \renewcommand\color[2][]{}%
  }%
  \providecommand\transparent[1]{%
    \errmessage{(Inkscape) Transparency is used (non-zero) for the text in Inkscape, but the package 'transparent.sty' is not loaded}%
    \renewcommand\transparent[1]{}%
  }%
  \providecommand\rotatebox[2]{#2}%
  \newcommand*\fsize{\dimexpr\f@size pt\relax}%
  \newcommand*\lineheight[1]{\fontsize{\fsize}{#1\fsize}\selectfont}%
  \ifx\svgwidth\undefined%
    \setlength{\unitlength}{518.21801737bp}%
    \ifx\svgscale\undefined%
      \relax%
    \else%
      \setlength{\unitlength}{\unitlength * \real{\svgscale}}%
    \fi%
  \else%
    \setlength{\unitlength}{\svgwidth}%
  \fi%
  \global\let\svgwidth\undefined%
  \global\let\svgscale\undefined%
  \makeatother%
  \begin{picture}(1,1.28875087)%
    \lineheight{1}%
    \setlength\tabcolsep{0pt}%
    \put(0,0){\includegraphics[width=\unitlength,page=1]{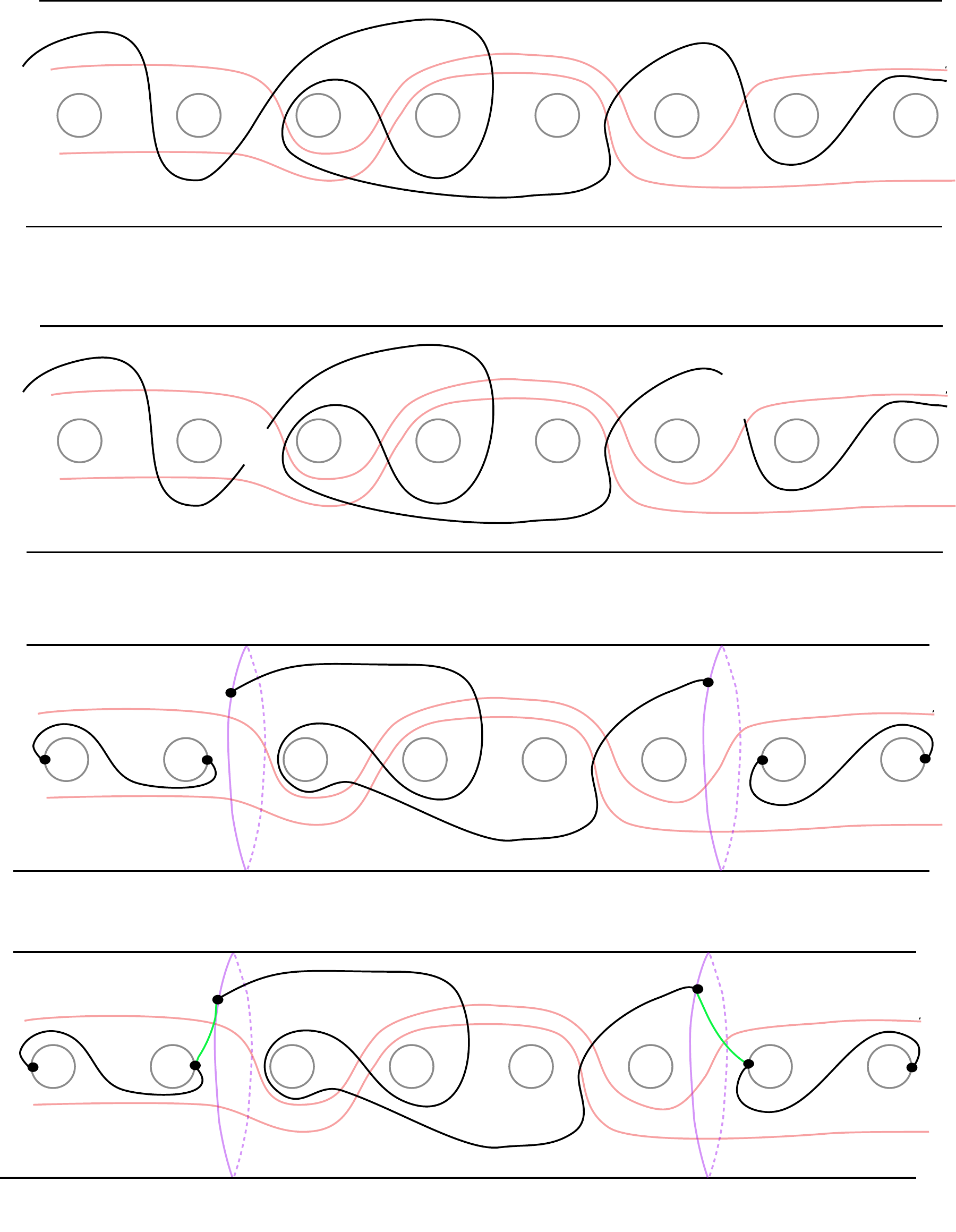}}%
    \put(0.1837133,1.00978281){\makebox(0,0)[lt]{\lineheight{1.25}\smash{\begin{tabular}[t]{l}$\scriptstyle{n_1}$\end{tabular}}}}%
    \put(0.80066475,1.01165808){\makebox(0,0)[lt]{\lineheight{1.25}\smash{\begin{tabular}[t]{l}$\scriptstyle{n_2}$\end{tabular}}}}%
    \put(0.18183805,0.66286486){\makebox(0,0)[lt]{\lineheight{1.25}\smash{\begin{tabular}[t]{l}$\scriptstyle{n_1}$\end{tabular}}}}%
    \put(0.79878953,0.66474014){\makebox(0,0)[lt]{\lineheight{1.25}\smash{\begin{tabular}[t]{l}$\scriptstyle{n_2}$\end{tabular}}}}%
    \put(0.18183805,0.32532315){\makebox(0,0)[lt]{\lineheight{1.25}\smash{\begin{tabular}[t]{l}$\scriptstyle{n_1}$\end{tabular}}}}%
    \put(0.79878953,0.32719843){\makebox(0,0)[lt]{\lineheight{1.25}\smash{\begin{tabular}[t]{l}$\scriptstyle{n_2}$\end{tabular}}}}%
    \put(0.17996282,0.00465839){\makebox(0,0)[lt]{\lineheight{1.25}\smash{\begin{tabular}[t]{l}$\scriptstyle{n_1}$\end{tabular}}}}%
    \put(0.79691431,0.00653367){\makebox(0,0)[lt]{\lineheight{1.25}\smash{\begin{tabular}[t]{l}$\scriptstyle{n_2}$\end{tabular}}}}%
  \end{picture}%
\endgroup%

\caption{Putting $\gamma$ in standard position.}
\end{subfigure}%
\begin{subfigure}{.4\textwidth}
\centering
\def\svgwidth{2.5in}
\begingroup%
  \makeatletter%
  \providecommand\color[2][]{%
    \errmessage{(Inkscape) Color is used for the text in Inkscape, but the package 'color.sty' is not loaded}%
    \renewcommand\color[2][]{}%
  }%
  \providecommand\transparent[1]{%
    \errmessage{(Inkscape) Transparency is used (non-zero) for the text in Inkscape, but the package 'transparent.sty' is not loaded}%
    \renewcommand\transparent[1]{}%
  }%
  \providecommand\rotatebox[2]{#2}%
  \newcommand*\fsize{\dimexpr\f@size pt\relax}%
  \newcommand*\lineheight[1]{\fontsize{\fsize}{#1\fsize}\selectfont}%
  \ifx\svgwidth\undefined%
    \setlength{\unitlength}{526.83915915bp}%
    \ifx\svgscale\undefined%
      \relax%
    \else%
      \setlength{\unitlength}{\unitlength * \real{\svgscale}}%
    \fi%
  \else%
    \setlength{\unitlength}{\svgwidth}%
  \fi%
  \global\let\svgwidth\undefined%
  \global\let\svgscale\undefined%
  \makeatother%
  \begin{picture}(1,0.27642776)%
    \lineheight{1}%
    \setlength\tabcolsep{0pt}%
    \put(0,0){\includegraphics[width=\unitlength,page=1]{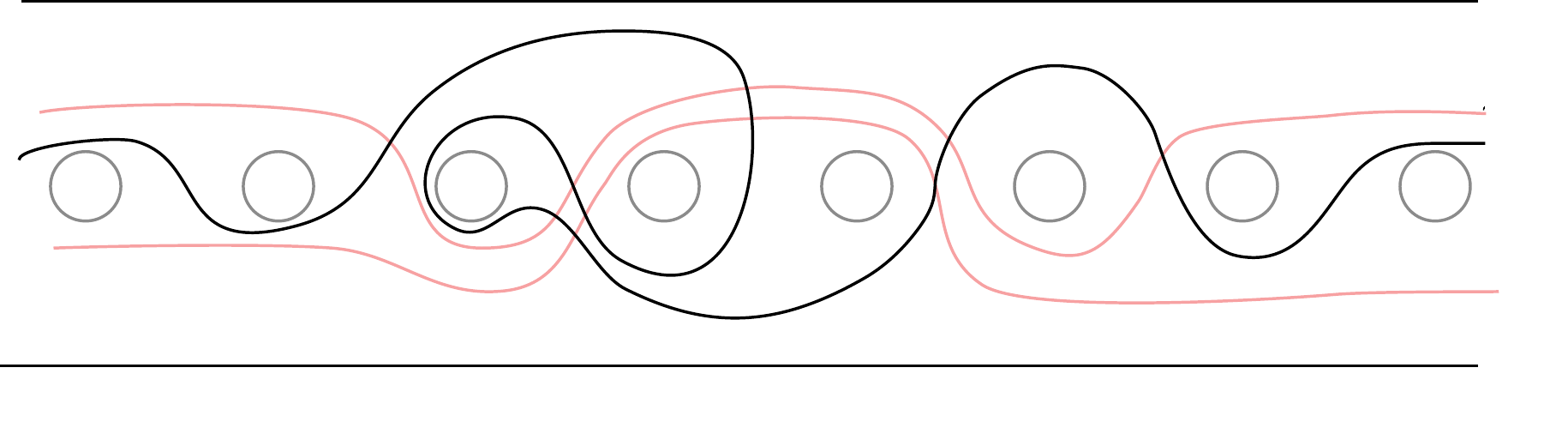}}%
    \put(0.16371007,0.00729506){\makebox(0,0)[lt]{\lineheight{1.25}\smash{\begin{tabular}[t]{l}$\scriptstyle{n_1}$\end{tabular}}}}%
    \put(0.77276053,0.00458216){\makebox(0,0)[lt]{\lineheight{1.25}\smash{\begin{tabular}[t]{l}$\scriptstyle{n_2}$\end{tabular}}}}%
  \end{picture}%
\endgroup%

\caption{A simple segment with the same code as $\gamma$ which we say is in standard position.}
\end{subfigure}
\caption{The  domain $D$ for a permissible shift is shown in red.  In black is a segment $\gamma$ that has no $(n_i)_{o/u}(n_i)_{u/o}$ for $i=1,2$. }
\label{fig:ex1}
\end{figure}
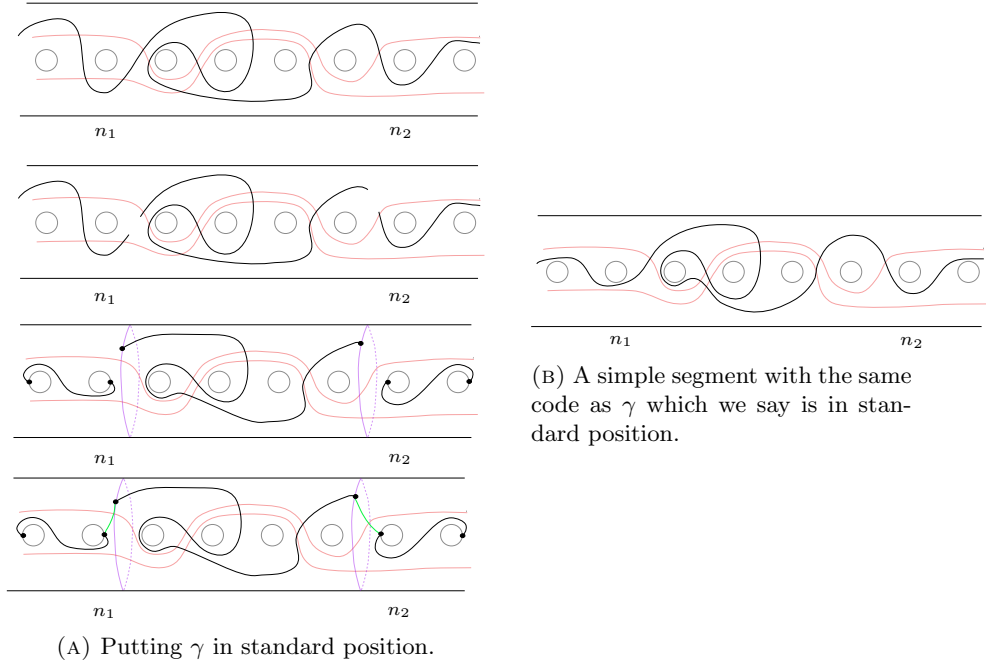

\begin{enumerate}[Step (i):]
\item If the initial character of the segment is $k_{o/u}$,  move the initial endpoint of the segment onto the separating curve $S_k$ if $k_{o/u}$ is oriented to the right and onto $S_{k+1}$ if it is oriented to the left.  If the terminal character of the segment is $k'_{o/u}$, move the terminal endpoint of the segment onto $S_{k'}$ if $k'_{o/u}$ is oriented to the left and onto $S_{k'+1}$ if it is oriented to the right.  Now move the endpoints along the $S_j$ containing them to reduce the number of full and half crossings, if possible, without creating any self-intersections.  In particular, the endpoints should not lie in the domain of $h$. See Figure~\ref{fig:ex1}.

There is one caveat to the rule above. Note that when $\gamma_i$ is type (b), then $\gamma_{i+1}$ and $\gamma_{i-1}$ are type (a), otherwise $\gamma_i$ would not be maximal. In the case that either of these neighboring segments is exactly a loop around $n_1$ or $n_2$ we must adjust the position of the endpoints of $\gamma_i$ along $S_j$. If $\gamma_{i+1} = (n_k)_{o}(n_k)_{u}$, for $k = 1$ or $2$, we require that the terminal endpoint of $\gamma_i$ is above $D$ and require that the terminal endpoint is below $D$ when $\gamma_{i+1}=(n_i)_{u}(n_i)_{o}$. Similarly, if $\gamma_{i-1}=(n_i)_{o}(n_i)_{u}$, for $i = 1$ or $2$, we require that the initial endpoint of $\gamma_i$ is below $D$ and require that the initial endpoint of $\gamma_{i}$ is above $D$ when $\gamma_{i-1}=(n_i)_{u}(n_i)_{o}$. See Figure \ref{fig:ex2}. Note that this repositioning of endpoints can cause additional crossings of $D$ as in Figure~\ref{fig:ex3}, but this is the appropriate configuration for our calculations. 
\item Homotope the segment rel endpoints to make all crossings full.  Since Step (i) ensures that the endpoints of $\gamma_i$ are always outside the domain $D$, this is always possible.
\item Homotope the segment rel endpoints to reduce the number of crossing by removing all bigons that bound disks and have one side on the segment $\gamma_i$  and the other side on $\partial D$.  
\item Finally, there may be a choice of where a full crossing occurs.    If there is such a choice, then it will always be possible to homotope rel endpoints so that the crossing occurs between two adjacent characters  $k_{o/u}$, $(k')_{o/u}$ with $k,k'\in(n_1,n_2)$ such that the $o/u$ pattern of $k$ and/or $k'$ does not match that of $\partial D$, and our convention is to make this choice for the largest possible $k,k'$. For example, in Figure \ref{fig:ex1}, the bottom strand could fully cross the domain between $(n_1+1)_o$ and $(n_1+1)_u$ or between $(n_1+1)_u$ and $(n_1+2)_u$; we make the latter choice.
\end{enumerate}

At this point, we have a collection of disjoint subsegments, each in standard position.  The final step is to connect the endpoints of these segments, in the order they were originally connected, by  segments called \textit{connectors}.  Connectors always occur between characters with numerical value $n_1$ and $n_1 +1$ or $n_2$ and $n_2 -1$. For the purposes of the code, we picture these  segments extended slightly in either direction to overlap with the segments on either side so that the code of a connector will always be a pair $j_{o/u}(j+1)_{o/u}$ or $(j+1)_{o/u}j_{o/u}$.  Thus, if a connector $\alpha$ connects disjoint subsegments $\delta_1$ and $\delta_2$ (in that order), then $\alpha=(\delta_1)^t(\delta_2)^i$ and we can write $\delta_1+\alpha+\delta_2$.  Note that this code is equivalent to the concatenation $\delta_1\delta_2$.  By construction, connectors always have one endpoint inside and one endpoint outside the the domain $D$ of the shift.  After applying this procedure, the resulting segment in standard position may no longer be simple.  However, it will have the same code as our original $\gamma$.

\begin{figure}
\centering
\begin{subfigure}{.6\textwidth}
\centering
\def\svgwidth{2.5in}
\begingroup%
  \makeatletter%
  \providecommand\color[2][]{%
    \errmessage{(Inkscape) Color is used for the text in Inkscape, but the package 'color.sty' is not loaded}%
    \renewcommand\color[2][]{}%
  }%
  \providecommand\transparent[1]{%
    \errmessage{(Inkscape) Transparency is used (non-zero) for the text in Inkscape, but the package 'transparent.sty' is not loaded}%
    \renewcommand\transparent[1]{}%
  }%
  \providecommand\rotatebox[2]{#2}%
  \newcommand*\fsize{\dimexpr\f@size pt\relax}%
  \newcommand*\lineheight[1]{\fontsize{\fsize}{#1\fsize}\selectfont}%
  \ifx\svgwidth\undefined%
    \setlength{\unitlength}{515.90661993bp}%
    \ifx\svgscale\undefined%
      \relax%
    \else%
      \setlength{\unitlength}{\unitlength * \real{\svgscale}}%
    \fi%
  \else%
    \setlength{\unitlength}{\svgwidth}%
  \fi%
  \global\let\svgwidth\undefined%
  \global\let\svgscale\undefined%
  \makeatother%
  \begin{picture}(1,1.25521117)%
    \lineheight{1}%
    \setlength\tabcolsep{0pt}%
    \put(0,0){\includegraphics[width=\unitlength,page=1]{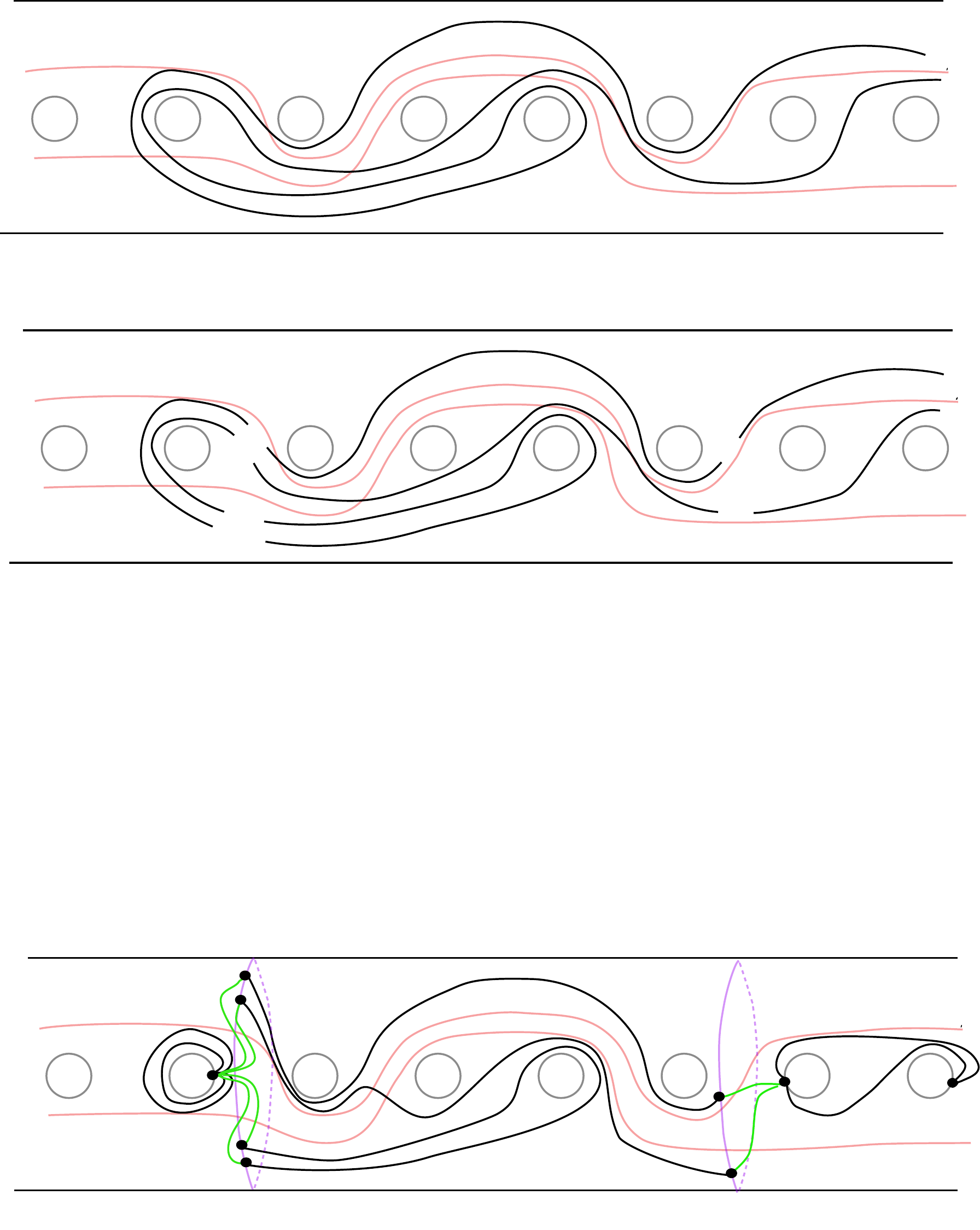}}%
    \put(0.17076857,0.97353775){\makebox(0,0)[lt]{\lineheight{1.25}\smash{\begin{tabular}[t]{l}$\scriptstyle{n_1}$\end{tabular}}}}%
    \put(0.17464994,0.63946383){\makebox(0,0)[lt]{\lineheight{1.25}\smash{\begin{tabular}[t]{l}$\scriptstyle{n_1}$\end{tabular}}}}%
    \put(0,0){\includegraphics[width=\unitlength,page=2]{StdPos2.pdf}}%
    \put(0.16334816,0.32678057){\makebox(0,0)[lt]{\lineheight{1.25}\smash{\begin{tabular}[t]{l}$\scriptstyle{n_1}$\end{tabular}}}}%
    \put(0.18030087,0.00467926){\makebox(0,0)[lt]{\lineheight{1.25}\smash{\begin{tabular}[t]{l}$\scriptstyle{n_1}$\end{tabular}}}}%
  \end{picture}%
\endgroup%

\caption{Putting $\gamma$ in standard position.}
\end{subfigure}
\begin{subfigure}{.35\textwidth}
\centering
\def\svgwidth{2.5in}
\begingroup%
  \makeatletter%
  \providecommand\color[2][]{%
    \errmessage{(Inkscape) Color is used for the text in Inkscape, but the package 'color.sty' is not loaded}%
    \renewcommand\color[2][]{}%
  }%
  \providecommand\transparent[1]{%
    \errmessage{(Inkscape) Transparency is used (non-zero) for the text in Inkscape, but the package 'transparent.sty' is not loaded}%
    \renewcommand\transparent[1]{}%
  }%
  \providecommand\rotatebox[2]{#2}%
  \newcommand*\fsize{\dimexpr\f@size pt\relax}%
  \newcommand*\lineheight[1]{\fontsize{\fsize}{#1\fsize}\selectfont}%
  \ifx\svgwidth\undefined%
    \setlength{\unitlength}{503.80470626bp}%
    \ifx\svgscale\undefined%
      \relax%
    \else%
      \setlength{\unitlength}{\unitlength * \real{\svgscale}}%
    \fi%
  \else%
    \setlength{\unitlength}{\svgwidth}%
  \fi%
  \global\let\svgwidth\undefined%
  \global\let\svgscale\undefined%
  \makeatother%
  \begin{picture}(1,0.29051723)%
    \lineheight{1}%
    \setlength\tabcolsep{0pt}%
    \put(0,0){\includegraphics[width=\unitlength,page=1]{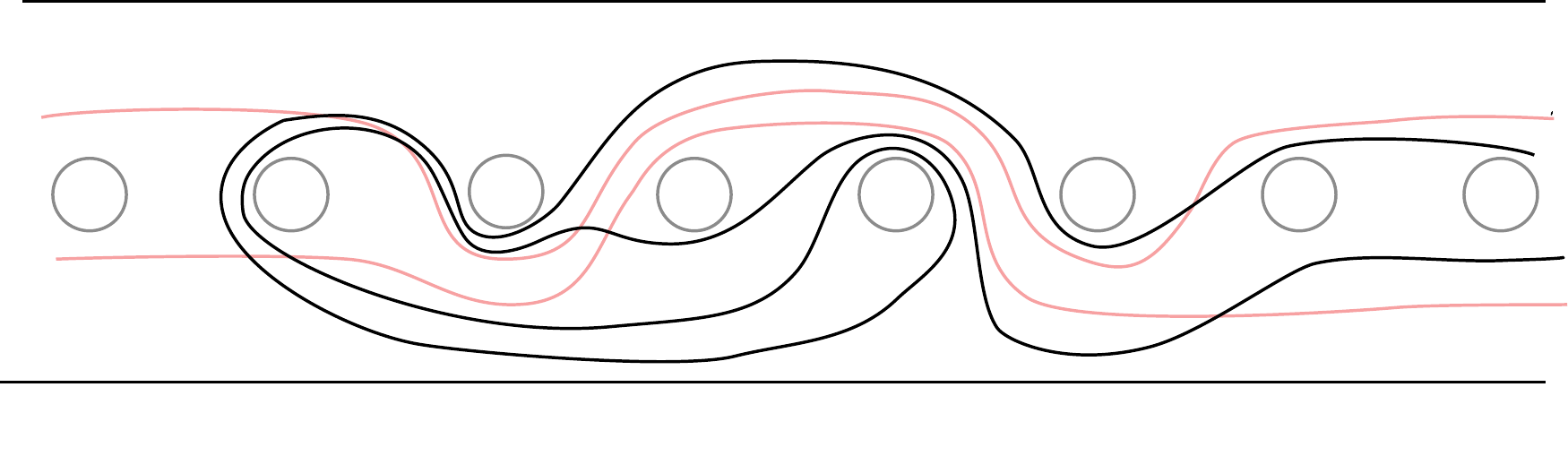}}%
    \put(0.17285986,0.00479166){\makebox(0,0)[lt]{\lineheight{1.25}\smash{\begin{tabular}[t]{l}$\scriptstyle{n_1}$\end{tabular}}}}%
  \end{picture}%
\endgroup%

\caption{A simple segment with the same code as $\gamma$ which we say is in standard position.}
\end{subfigure}
\caption{A segment $\gamma$ that contains multiple copies occurrences of $(n_1)_o(n_1)_u$.}
\label{fig:ex2}
\end{figure}

The following lemma summarizes the above procedure.
\begin{lem}\label{lem:stdpos}
A segment $\gamma$  in standard position with respect to a permissible right shift $h$ which is not completely contained in the domain of $h$ can be written as the efficient concatenation of (possibly empty) segments in the following way.
	\begin{equation}\label{eqn:gammadecomp}
	\gamma=\gamma_1^{tu}+\gamma_1^{c1}+\gamma_1^{sh}+\gamma_1^{c2}+\gamma_2^{tu}+\gamma_2^{c1}+\cdots + \gamma_n^{sh}+\gamma_n^{c2},
	\end{equation}
where for each $i$,
	\begin{itemize}
	\item $\gamma_i^{tu}$ is supported on the turbulent region $(n_1,n_2)$, has been put into standard position following steps (i)--(iv) above, and has both endpoints outside of the domain $D$; 
	\item $\gamma_i^{sh}$ is supported on the shift region $(-\infty,n_1]\cup[n_2,\infty)$, is completely contained in the domain $D$, and has both endpoints on the $B_k$ curves; and
	\item $\gamma_i^{c1},\gamma_i^{c2}$ are connectors, each of which has code length 2, is supported on either $[n_1,n_1+1]$ or $[n_2-1,n_2]$, and has one endpoint on $B_{n_1}$ or $B_{n_2}$ and one endpoint outside the domain $D$ on $S_{n_1+1}$ or $S_{n_2}$.
	\end{itemize}
\end{lem}

We will sometimes use the notation $\gamma_i^{conn}$ for a connector if it is not important if this subsegment is the first connector $\gamma_i^{c1}$ or the second connector $\gamma_i^{c2}$.

It is often  convenient to abuse notation and say that a simple segment is in ``standard position" even if it is not the result of the above procedure because these segments are easier to draw and think about.  What we mean by this is that the  segment intersects the boundary of $D$ in the same place as it would in standard position.   In Figures \ref{fig:ex1}--\ref{fig:ex4}, we give a segment, the steps to put it in standard position, and an example of a simple segment with the same code which we also say is in standard position.

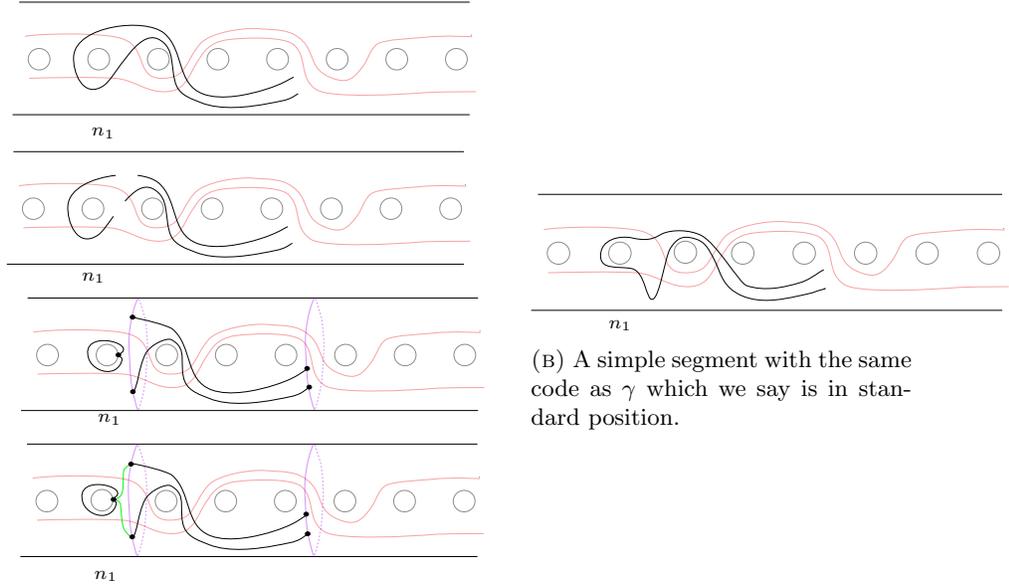
\begin{figure}
\centering
\begin{subfigure}{.6\textwidth}
\centering
\def\svgwidth{2.5in}
\begingroup%
  \makeatletter%
  \providecommand\color[2][]{%
    \errmessage{(Inkscape) Color is used for the text in Inkscape, but the package 'color.sty' is not loaded}%
    \renewcommand\color[2][]{}%
  }%
  \providecommand\transparent[1]{%
    \errmessage{(Inkscape) Transparency is used (non-zero) for the text in Inkscape, but the package 'transparent.sty' is not loaded}%
    \renewcommand\transparent[1]{}%
  }%
  \providecommand\rotatebox[2]{#2}%
  \newcommand*\fsize{\dimexpr\f@size pt\relax}%
  \newcommand*\lineheight[1]{\fontsize{\fsize}{#1\fsize}\selectfont}%
  \ifx\svgwidth\undefined%
    \setlength{\unitlength}{519.20700907bp}%
    \ifx\svgscale\undefined%
      \relax%
    \else%
      \setlength{\unitlength}{\unitlength * \real{\svgscale}}%
    \fi%
  \else%
    \setlength{\unitlength}{\svgwidth}%
  \fi%
  \global\let\svgwidth\undefined%
  \global\let\svgscale\undefined%
  \makeatother%
  \begin{picture}(1,1.21330124)%
    \lineheight{1}%
    \setlength\tabcolsep{0pt}%
    \put(0,0){\includegraphics[width=\unitlength,page=1]{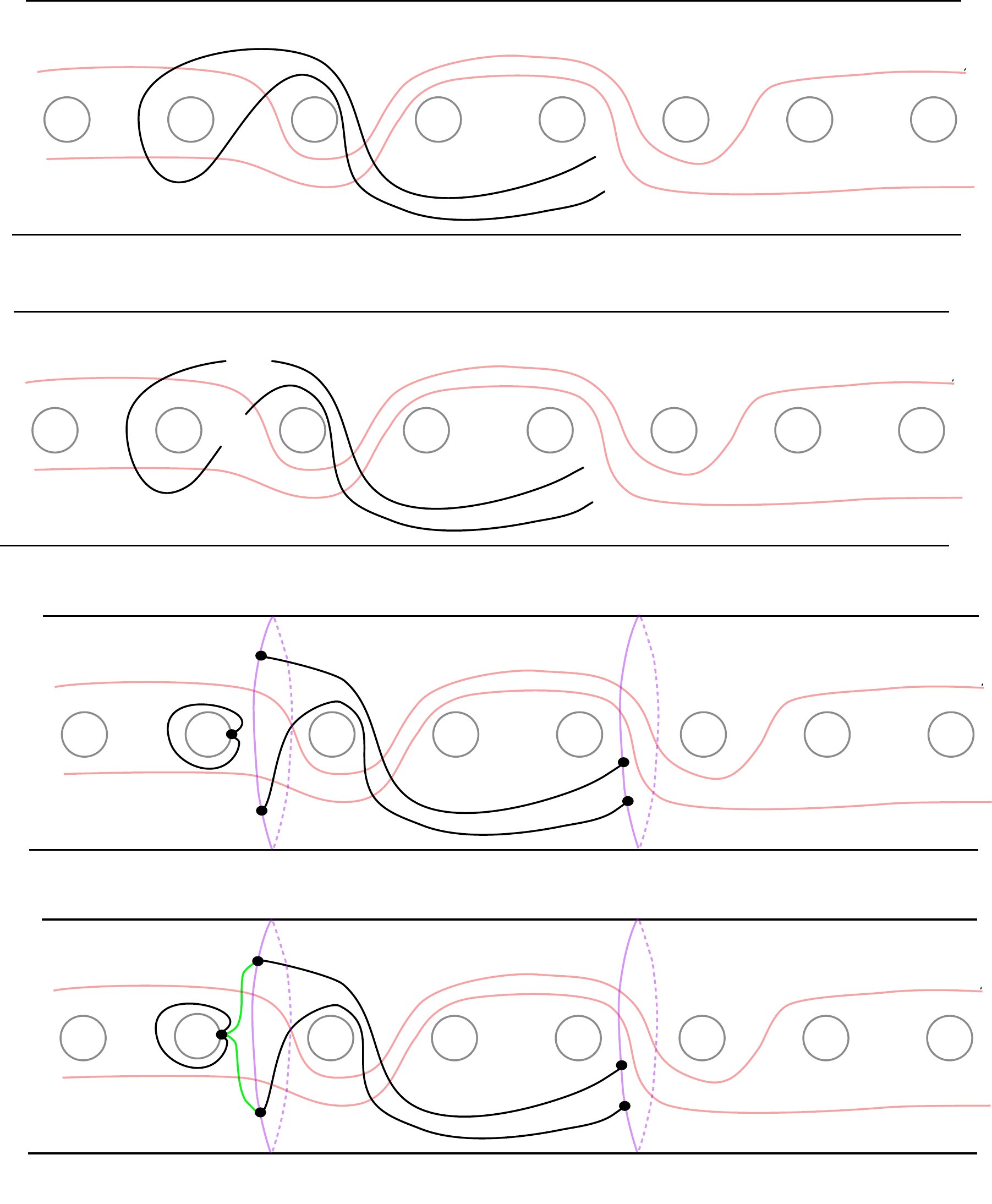}}%
    \put(0.1774858,0.93299294){\makebox(0,0)[lt]{\lineheight{1.25}\smash{\begin{tabular}[t]{l}$\scriptstyle{n_1}$\end{tabular}}}}%
    \put(0.15689752,0.62978397){\makebox(0,0)[lt]{\lineheight{1.25}\smash{\begin{tabular}[t]{l}$\scriptstyle{n_1}$\end{tabular}}}}%
    \put(0.19058744,0.33218993){\makebox(0,0)[lt]{\lineheight{1.25}\smash{\begin{tabular}[t]{l}$\scriptstyle{n_1}$\end{tabular}}}}%
    \put(0.18310074,0.00464952){\makebox(0,0)[lt]{\lineheight{1.25}\smash{\begin{tabular}[t]{l}$\scriptstyle{n_1}$\end{tabular}}}}%
  \end{picture}%
\endgroup%

\caption{Putting $\gamma$ in standard position.}
\end{subfigure}%
\begin{subfigure}{.4\textwidth}
\def\svgwidth{2.5in}
\begingroup%
  \makeatletter%
  \providecommand\color[2][]{%
    \errmessage{(Inkscape) Color is used for the text in Inkscape, but the package 'color.sty' is not loaded}%
    \renewcommand\color[2][]{}%
  }%
  \providecommand\transparent[1]{%
    \errmessage{(Inkscape) Transparency is used (non-zero) for the text in Inkscape, but the package 'transparent.sty' is not loaded}%
    \renewcommand\transparent[1]{}%
  }%
  \providecommand\rotatebox[2]{#2}%
  \newcommand*\fsize{\dimexpr\f@size pt\relax}%
  \newcommand*\lineheight[1]{\fontsize{\fsize}{#1\fsize}\selectfont}%
  \ifx\svgwidth\undefined%
    \setlength{\unitlength}{503.80490626bp}%
    \ifx\svgscale\undefined%
      \relax%
    \else%
      \setlength{\unitlength}{\unitlength * \real{\svgscale}}%
    \fi%
  \else%
    \setlength{\unitlength}{\svgwidth}%
  \fi%
  \global\let\svgwidth\undefined%
  \global\let\svgscale\undefined%
  \makeatother%
  \begin{picture}(1,0.28615664)%
    \lineheight{1}%
    \setlength\tabcolsep{0pt}%
    \put(0,0){\includegraphics[width=\unitlength,page=1]{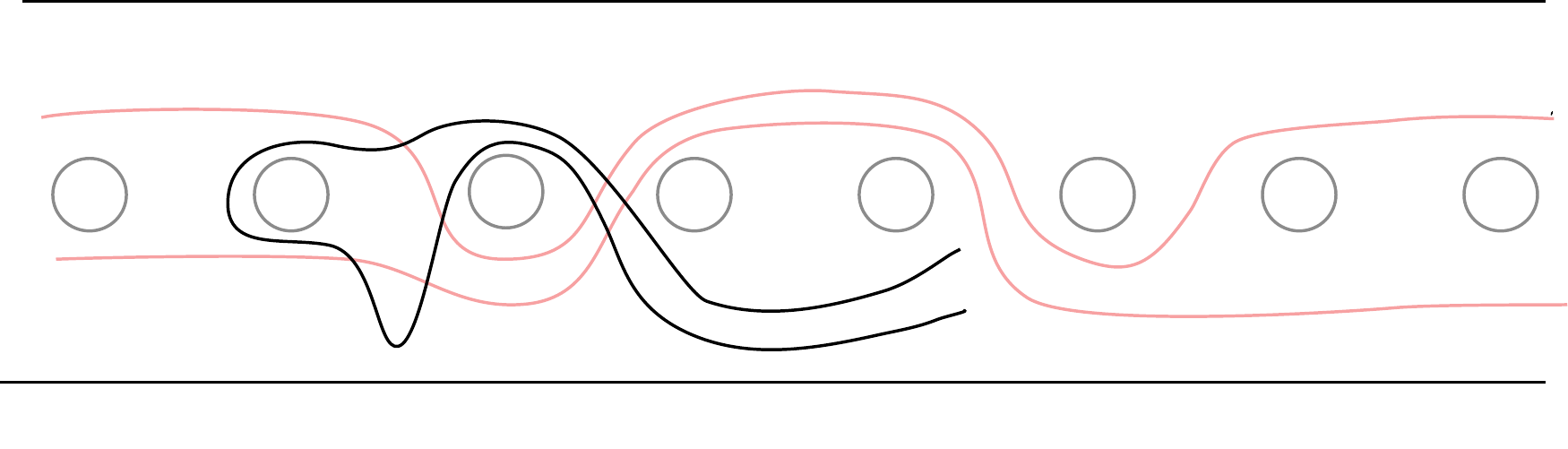}}%
    \put(0.16215022,0.00479166){\makebox(0,0)[lt]{\lineheight{1.25}\smash{\begin{tabular}[t]{l}$\scriptstyle{n_1}$\end{tabular}}}}%
  \end{picture}%
\endgroup%
\caption{A simple segment with the same code as $\gamma$ which we say is in standard position.}
\end{subfigure}
\caption{A segment $\gamma$ that contains $(n_1)_o(n_1)_u$.  Note that the requirements on the endpoints of $\gamma_{i}$ in step (i) of the procedure for standard position require that we have one connector cross the top of $\partial D$ and one connector cross the bottom of $\partial D$.}
\label{fig:ex3}
\end{figure}

\begin{figure}
\begin{centering}
\def\svgwidth{3in}
\begingroup%
  \makeatletter%
  \providecommand\color[2][]{%
    \errmessage{(Inkscape) Color is used for the text in Inkscape, but the package 'color.sty' is not loaded}%
    \renewcommand\color[2][]{}%
  }%
  \providecommand\transparent[1]{%
    \errmessage{(Inkscape) Transparency is used (non-zero) for the text in Inkscape, but the package 'transparent.sty' is not loaded}%
    \renewcommand\transparent[1]{}%
  }%
  \providecommand\rotatebox[2]{#2}%
  \newcommand*\fsize{\dimexpr\f@size pt\relax}%
  \newcommand*\lineheight[1]{\fontsize{\fsize}{#1\fsize}\selectfont}%
  \ifx\svgwidth\undefined%
    \setlength{\unitlength}{524.5091512bp}%
    \ifx\svgscale\undefined%
      \relax%
    \else%
      \setlength{\unitlength}{\unitlength * \real{\svgscale}}%
    \fi%
  \else%
    \setlength{\unitlength}{\svgwidth}%
  \fi%
  \global\let\svgwidth\undefined%
  \global\let\svgscale\undefined%
  \makeatother%
  \begin{picture}(1,0.52782202)%
    \lineheight{1}%
    \setlength\tabcolsep{0pt}%
    \put(0,0){\includegraphics[width=\unitlength,page=1]{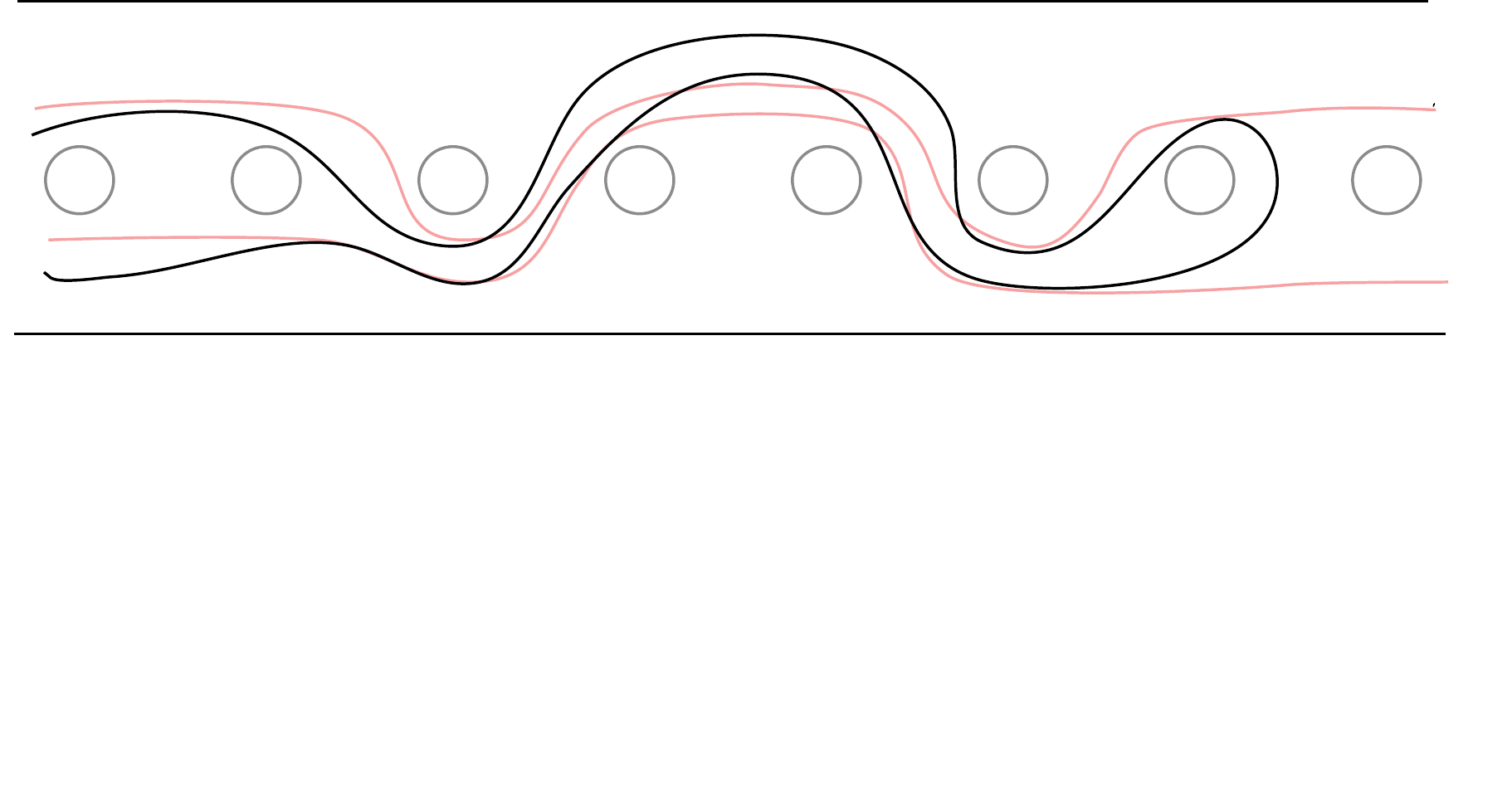}}%
    \put(0.1592621,0.27698327){\makebox(0,0)[lt]{\lineheight{1.25}\smash{\begin{tabular}[t]{l}$\scriptstyle{n_1}$\end{tabular}}}}%
    \put(0.76291688,0.27992791){\makebox(0,0)[lt]{\lineheight{1.25}\smash{\begin{tabular}[t]{l}$\scriptstyle{n_2}$\end{tabular}}}}%
    \put(0,0){\includegraphics[width=\unitlength,page=2]{StdPos4.pdf}}%
    \put(0.16809608,0.00460252){\makebox(0,0)[lt]{\lineheight{1.25}\smash{\begin{tabular}[t]{l}$\scriptstyle{n_1}$\end{tabular}}}}%
    \put(0.77175088,0.00754717){\makebox(0,0)[lt]{\lineheight{1.25}\smash{\begin{tabular}[t]{l}$\scriptstyle{n_2}$\end{tabular}}}}%
  \end{picture}%
\endgroup%

\end{centering}
\caption{A segment $\gamma$ that has endpoints in $(-\infty,n_1]\cup[n_2,\infty)$ and can be homotoped to be completely inside $D$.}
\label{fig:ex4}
\end{figure}

Because each segment in \eqref{eqn:gammadecomp} has fixed endpoints, its image under $h$ is well-defined up to homotopy rel endpoints.  Thus we may use the decomposition of $\gamma$ to find its image under $h$ as follows
\[
h(\gamma)=h(\gamma_1^{tu})+h(\gamma_1^{c1})+h(\gamma_1^{sh})+h(\gamma_1^{c2})+h(\gamma_2^{tu})+h(\gamma_2^{c1})+\cdots + h(\gamma_n^{sh})+h(\gamma_n^{c2}).
\]
Since in standard position $\gamma$  may not be simple, its image under $h$ may not be simple.  However, this code corresponds to a unique (homotopy class of) simple segment with the same endpoints.  It is important that we use \textit{efficient} concatenation when calculating $h(\gamma)$.  Using regular concatenation, the code for $\gamma$ is simply $\gamma_1^{tu}\gamma_1^{sh}\ldots \gamma_n^{tu}\gamma_n^{sh}$.   However, it is not true that $h(\gamma_1^{tu})h(\gamma_1^{sh})\ldots h(\gamma_n^{tu}) h(\gamma_n^{sh})$ is a code for $\gamma$; in fact, much of the time this code does not define a segment.  See Example \ref{ex:imageofcrossing}.  Most of the interesting behavior in the image of an arc or segment under a permissible shift actually comes from the full and half crossings. 
Since each connector contributes a half crossing, they are essential for determining the image of $\gamma$. 

\begin{ex} \label{ex:imageofcrossing}
Consider the permissible shift $h$ shown in Figure \ref{fig:crossings} along with the segment $\gamma$.  In Figure \ref{fig:interestingimage}, we put $\gamma$ in standard position and find its image under the shift.  Using code, we have $\gamma=(n_1+1)_u(n_1+2)_o(n_1+3)_o(n_2-1)_o(n_2)_u$.  Note that every character in this code except the terminal $(n_2)_u$ is fixed by $h$, and $h((n_2)_u)=(n_2+1)_u$.  If we simply compute $h(\gamma)$ character by character, we get $h(\gamma)=(n_1+1)_u(n_1+2)_o(n_1+3)_o(n_2-1)_o(n_2+1)_u$.  However, this is not a well-defined code since $n_2-1$ and $n_2+1$ are not adjacent. On the other hand, using efficient concatenation, we see that 
\begin{align*}
h(\gamma)&= h(\gamma^{tu})+ h(\gamma^{conn})+h(\gamma^{sh}) \\
&=h((n_1+1)_u(n_1+2)_o (n_1+3)_o(n_2-1)_o) + h((n_2-1)_o(n_2)_u) + h((n_2)_u)  \\
&=(n_1+1)_u(n_1+2)_o (n_1+3)_o(n_2-1)_o + (n_2-1)_o(n_2)_o(n_2+1)_u + (n_2+1)_u\\
&= (n_1+1)_u(n_1+2)_o(n_1+3)_o(n_2-1)_o(n_2)_o(n_2+1)_u.
\end{align*}
Note that all but the final pair $(n_2-1)_o(n_2)_u$ are fixed by $h$.  This final pair is a half crossing and in standard position it is a connector.
\end{ex}

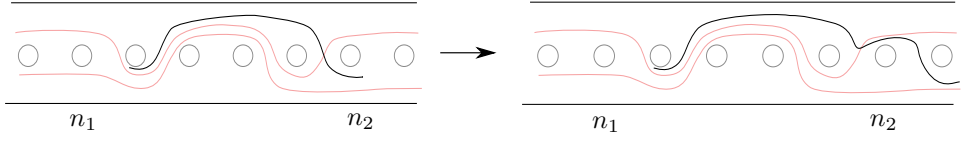
\begin{figure}
\centering
\def\svgwidth{5in}
\begingroup%
  \makeatletter%
  \providecommand\color[2][]{%
    \errmessage{(Inkscape) Color is used for the text in Inkscape, but the package 'color.sty' is not loaded}%
    \renewcommand\color[2][]{}%
  }%
  \providecommand\transparent[1]{%
    \errmessage{(Inkscape) Transparency is used (non-zero) for the text in Inkscape, but the package 'transparent.sty' is not loaded}%
    \renewcommand\transparent[1]{}%
  }%
  \providecommand\rotatebox[2]{#2}%
  \newcommand*\fsize{\dimexpr\f@size pt\relax}%
  \newcommand*\lineheight[1]{\fontsize{\fsize}{#1\fsize}\selectfont}%
  \ifx\svgwidth\undefined%
    \setlength{\unitlength}{602.4019774bp}%
    \ifx\svgscale\undefined%
      \relax%
    \else%
      \setlength{\unitlength}{\unitlength * \real{\svgscale}}%
    \fi%
  \else%
    \setlength{\unitlength}{\svgwidth}%
  \fi%
  \global\let\svgwidth\undefined%
  \global\let\svgscale\undefined%
  \makeatother%
  \begin{picture}(1,0.13483112)%
    \lineheight{1}%
    \setlength\tabcolsep{0pt}%
    \put(0,0){\includegraphics[width=\unitlength,page=1]{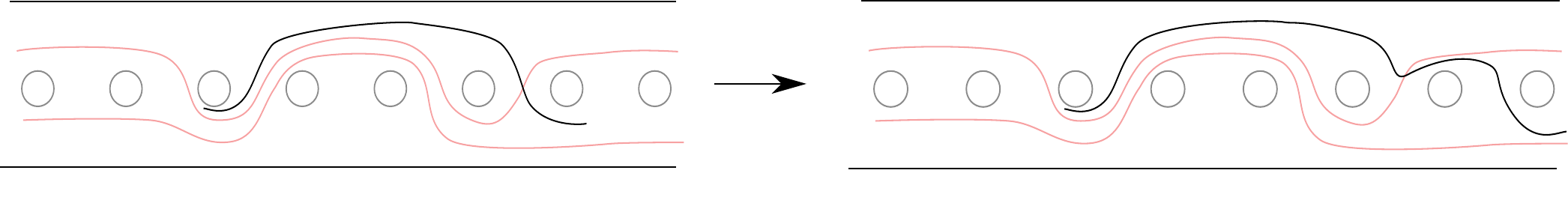}}%
    \put(0.06745915,0.00422893){\color[rgb]{0,0,0}\makebox(0,0)[lt]{\lineheight{1.25}\smash{\begin{tabular}[t]{l}$n_1$\end{tabular}}}}%
    \put(0.35798715,0.00553173){\color[rgb]{0,0,0}\makebox(0,0)[lt]{\lineheight{1.25}\smash{\begin{tabular}[t]{l}$n_2$\end{tabular}}}}%
    \put(0.614642,0.00292611){\color[rgb]{0,0,0}\makebox(0,0)[lt]{\lineheight{1.25}\smash{\begin{tabular}[t]{l}$n_1$\end{tabular}}}}%
    \put(0.90516999,0.00422892){\color[rgb]{0,0,0}\makebox(0,0)[lt]{\lineheight{1.25}\smash{\begin{tabular}[t]{l}$n_2$\end{tabular}}}}%
  \end{picture}%
\endgroup%

\caption{The segment $\gamma$ from Example \ref{ex:imageofcrossing} in standard position (left) and its image under the shift whose domain is shown (right).}
\label{fig:interestingimage}
\end{figure}

In fact, given a segment $\gamma=\gamma^{tu}$ supported on $(n_1,n_2)$, we can further decompose $\gamma$ into subsegments which are disjoint from $D$ and pairs which fully cross $D$ in such a way that makes it straightforward to find its image under $h$.  Write

\begin{equation}\label{eq:turbulentdecomp}
\gamma=\gamma^{d}_1+\gamma^{e}_1+\cdots+\gamma^{d}_s+\gamma^{e}_s,
\end{equation}
where each $\gamma^{d}$ is a maximal subsegment disjoint from $D$, each $\gamma^{e}$ fully crosses $D$, and $\ell_c(\gamma^{d})\ge 2$, $\ell_c(\gamma^{e})=2$, when non-empty.

Using the above decomposition, in an unreduced code we have
$$ h(\gamma)=\gamma^{d}_1+h(\gamma^{e}_1)+\dots+\gamma^{d}_s+h(\gamma^{e}_s).$$
 Every non-empty $\gamma^{e}_j$ will have image which follows $\partial D$, loops around $n_2$, and follows $\partial D$ back, so that 
$$h(\gamma^{e}_j)=\partial D\vert_{[(\gamma^{e}_j)^i,n_2)}(n_2)_{o/u}(n_2)_{u/o}\overline{\partial D\vert_{[(\gamma^{e}_j)^t,n_2)}}.$$ 

As in Example \ref{ex:imageofcrossing}, if we don't use efficient concatenation then we can write  $\gamma=\gamma^d_1\gamma^d_2\ldots \gamma^d_s$.   Applying $h$ to each of these subsegments individually would yield $h(\gamma)=\gamma$, since each of these subsegments is fixed by $h$, which is not the correct image.

\subsubsection{Segments with back loops.}\label{sec:Ccode}
If $\gamma$ is a segment with code equal to $C$,  we require that $\gamma$ has both endpoints on some separating curve $S_i$ in our collection $\mc C$.  We also assume that the endpoints of $C$ lie outside the domain, and, moreover, that $\gamma$ does not intersect $D$.  There are two possibilities for $\gamma$, either $\gamma$ both enters and exits the front of $S$ at the top or bottom or (up to taking inverses) $\gamma$ enters at the top and exits at the bottom of the front of the surface. Recall that we define the top/bottom of the front of $S$ with respect to the notion of right/left on the front of $S$.  In the first case, this implies that the endpoints of $\gamma$ are both above or both below $D$, respectively, while in the second case one will be above and one will be below.  In either case, this convention implies that $\gamma\cap D=\emptyset$ and $h(\gamma)=\gamma$.

Suppose next that $\gamma$ is a segment whose code contains $C$ but also contains other characters.  Suppose for simplicity that the code for $\gamma$ contains a single instance of $C$, so that $\gamma=\tau_1C\tau_2$, where $\tau_i$ does not contain $C$ for $i=1,2$.  We put $\gamma$ in standard position as follows. First note that by definition of the code, we must either have $(\tau_1)^t=(\tau_2)^i$ (if $\tau_1$ and $\tau_2$ have opposite orientations) or $(\tau_1)^t=(\tau_2)^i\pm 1$ (if $(\tau_1)^t$ and $(\tau_2)^i$ have the same orientation). See Figure~\ref{fig:backloopconnectors}.  Without loss of generality, suppose $\tau_1^t$ is oriented to the right. Consider the (disjoint) segments $\tau_1$ and $\tau_2$, put them in standard position as in the previous section, and homotope the endpoints of $C$ as in the previous paragraph.  Note the endpoints of $C$ will lie on the curve $S_{(\tau_1)^t+1}$.  We now have three disjoint segments with codes $\tau_1$, $\tau_2$, and $C$.  We will add (possibly empty) segments called \textit{back loop connectors} from the terminal point of $\tau_1$ to the initial point of $C$ and from the terminal point of $C$ to the initial point of $\tau_2$, respectively, to form a connected segment.  See Figure \ref{fig:backloopconnectors}.  We code these back loop connectors with the characters $C^-,C^+$, respectively, so that we can easily discuss their image.  In particular, by a slight abuse of notation, we replace $C$ in the code for $\gamma$ with $C^-CC^+$.    

\begin{figure}
\begin{centering}
\includegraphics[scale=.6, trim={0in 2.5in 0in 0in}, clip]{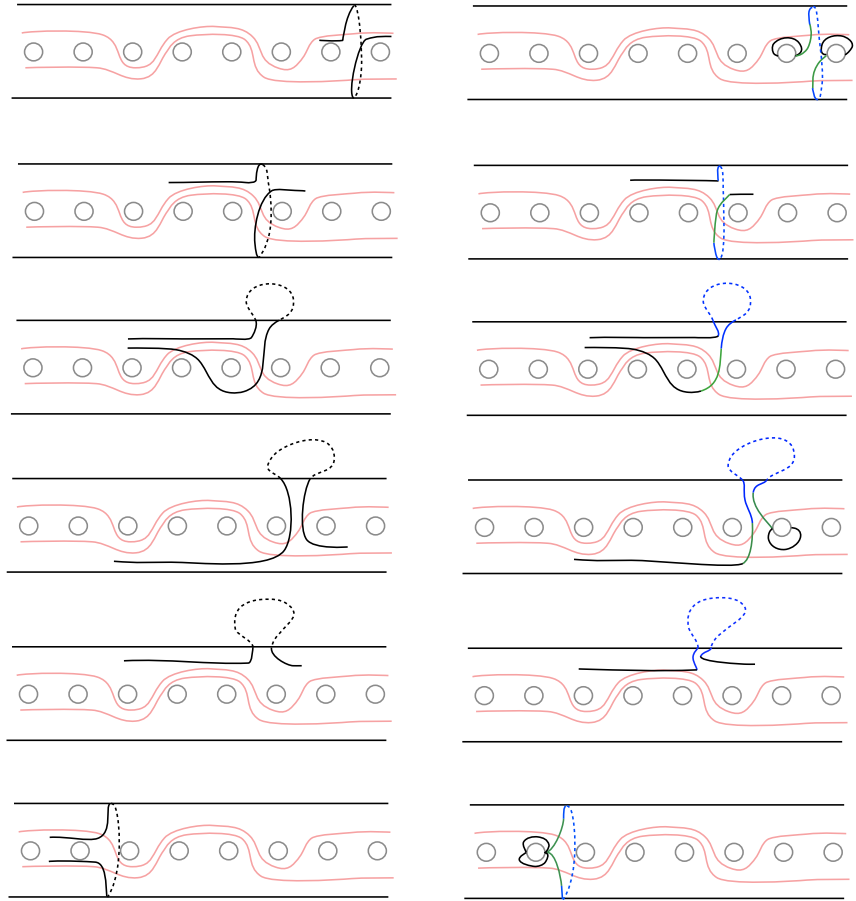}
\end{centering}
\caption{Various segments $\gamma$ containing back loops.  The segments on the right are in standard form.  Back loops are blue and back loop connectors are green.  Note that in the second, third, and fifth examples, at least one of $C^\pm$ is empty.}
\label{fig:backloopconnectors}
\end{figure}

Recall that without loss of generality, $(\tau_1)^t$ is oriented to the right.  If the terminal endpoint of $\tau_1$ lies outside $D$ and on the same side of $D$ as the initial point of the back loop $C$, then the back loop connector $C^-$ is empty.  An analogous statement holds for $C^+$, where we use the initial endpoint of $\tau_2$ in place of the terminal endpoint of $\tau_1$.   Each non-empty back loop connector $C^\pm$ either:
	\begin{enumerate}[(i)]
	\item has one endpoint  on $B_{(\tau_1)^t}$ or $B_{(\tau_2)^i}$
	(depending on whether it is $C^-$ or $C^+$), the other endpoint on $S_{ (\tau_1)^t+1}$, and half-crosses $D$;   or
	\item  has both endpoints on $S_{ (\tau_1)^t+1}$ which are outside of $D$ and fully crosses $D$.
	\end{enumerate}

\subsection{Gaps in segments}\label{sec:segconn}
When we use the code of a segment to find its image under a permissible shift, we first break it into smaller subsegments using standard position.  When we do this, we always use efficient concatenation, so that the codes of the individual pieces overlap in a single character.  The goal of efficient concatenation is to ensure that we do not lose any information about the segment by breaking it into pieces.  However, we need to be careful when we do this.  If we first break a segment into subsegments and then put each subsegment into standard position, it is possible that we will lose some information. In particular, we may cause there to be a ``gap" in the segment.  Based on standard position, these gaps can only occur when breaking a segment in the interior of the turbulent region $(n_1,n_2)$. 

To see this, suppose we break a segment $\gamma=\gamma_1+\gamma_2$ 
into two subsegments such that the numerical value of $\gamma_1^t=\gamma_2^i$ is $j\in(n_1,n_2)$. If we put each $\gamma_i$ into standard position individually, it is possible that $\gamma_1^t$ and $\gamma_2^i$ lie on opposite sides of $D$ (see Figure \ref{fig:gapsinstdpos}).  In this case, we have lost the full crossing between them. Recall that a shift fixes the surface outside of its domain.  In the region $(n_1,n_2)$, a segment in standard position is disjoint from the domain of the shift \textit{except} where there is a full crossing (see the decomposition in equation  \eqref{eq:turbulentdecomp}), so  the full crossings are essential for determining the image of a segment.  Thus we cannot use $\gamma_1$ and $\gamma_2$ to  find the correct image of $\gamma$.

\begin{figure}
\centering
\def\svgwidth{5in}
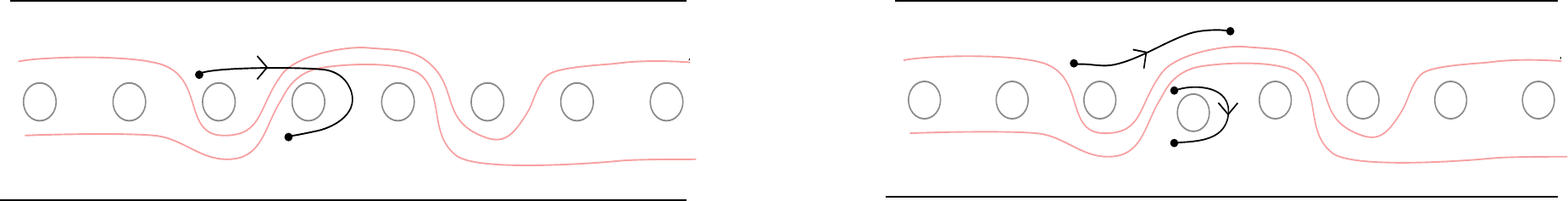
\caption{A segment $\gamma$ (left) is split into subsegments $\gamma=\gamma_1+\gamma_2$ and each $\gamma_i$ is individually put into standard position with respect to the shift $h$ whose domain is shown (right).  This causes a loss of information when we take the images of $\gamma_i$ individually under $h$: since $h(\gamma_i)=\gamma_i$ for $i=1,2$, we have $h(\gamma_1)+h(\gamma_2)=\gamma_1+\gamma_2=\gamma$, but $\gamma$ is clearly not fixed by $h$.}
\label{fig:gapsinstdpos}
\end{figure}

In order to ensure that we do not lose any information when working with a segment and subsegments in the turbulent region, we always first homotope the whole segment $\gamma$ to be a simple segment in standard position.  We then fix this particular representative of the homotopy class of $\gamma$ for the remainder of the time we work with it.  Thus, when we break $\gamma$ into subsegments $\gamma_1$ and $\gamma_2$, we do \textit{not} put these into standard position individually.  To be precise, we choose the endpoints  $\gamma_1^t$ and $\gamma_2^i$ to be the intersection of $\gamma$ and the appropriate separating curves in our collection $\scc$ and we do not allow any further homotopies of $\gamma_1$ or $\gamma_2$.  This will always ensure that there are no gaps between $\gamma_1$ and $\gamma_2$.  

In certain cases, it may be simpler to  break $\gamma$ into subsegments in a different way, and we do this whenever it will avoid technicalities. For example, in Figure \ref{fig:gapsinstdpos}, we may simply choose not to divide $\gamma$ into subsegments at all.  On the other hand, we could also choose to make $\gamma_1$ or $\gamma_2$ longer than is strictly necessary in a particular calculation  in order to avoid a potential loss of information. 

\subsection{Taking inverses of segments}\label{sec:inverses}

In general, if we have a segment $\gamma$ and a subsegment $\zeta$ of $h(\gamma)$, there is not necessarily a subsegment of $\gamma$ which we may call $h^{-1}(\zeta)$.  In other words, not every subsegment of $h(\gamma)$ is the image of a subsegment of $\gamma$.  It is important here that when we think of a subsegment of $\gamma$, we are fixing $\gamma$ in standard position.  That is, we are thinking of a \textit{reduced} code for $\gamma$, rather than any (unreduced) code representing $\gamma$.

For example, consider the  shift shown in Figure \ref{fig:nopreimage}.  Here, the pink subsegment $\zeta_1$ of $h(\gamma)$ is not the image of any subsegment of $\gamma$.  Rather, it is properly contained in the image of the purple subsegment of $\gamma$, specifically because it is in the image of the full crossing of the purple  segment with $D$.
Explicitly, 
$$\gamma=0_o1_u2_o2_u1_u,\quad h(\gamma)=0_o1_u2_o2_u1_u0_o0_u1_u2_u2_o1_u1_o2_o2_u1_u0_u,$$
and one can  see that if the subsegment $\zeta_1=1_u2_o2_u1_u$ of $h(\gamma)$ then there is no subsegment $\gamma'$ of $\gamma$ for which $h(\gamma')=\zeta_1$. 

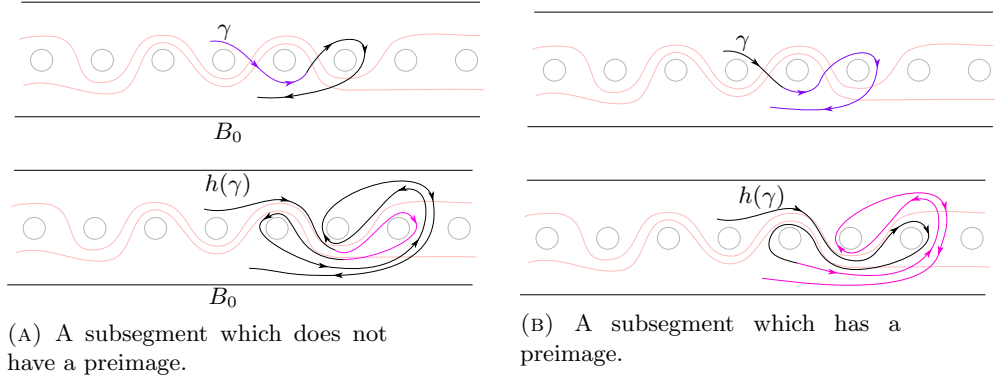
\begin{figure}
\begin{subfigure}{.4\textwidth}
\centering
\def\svgwidth{2.5in}
{\small 
\begingroup%
  \makeatletter%
  \providecommand\color[2][]{%
    \errmessage{(Inkscape) Color is used for the text in Inkscape, but the package 'color.sty' is not loaded}%
    \renewcommand\color[2][]{}%
  }%
  \providecommand\transparent[1]{%
    \errmessage{(Inkscape) Transparency is used (non-zero) for the text in Inkscape, but the package 'transparent.sty' is not loaded}%
    \renewcommand\transparent[1]{}%
  }%
  \providecommand\rotatebox[2]{#2}%
  \newcommand*\fsize{\dimexpr\f@size pt\relax}%
  \newcommand*\lineheight[1]{\fontsize{\fsize}{#1\fsize}\selectfont}%
  \ifx\svgwidth\undefined%
    \setlength{\unitlength}{509.8554bp}%
    \ifx\svgscale\undefined%
      \relax%
    \else%
      \setlength{\unitlength}{\unitlength * \real{\svgscale}}%
    \fi%
  \else%
    \setlength{\unitlength}{\svgwidth}%
  \fi%
  \global\let\svgwidth\undefined%
  \global\let\svgscale\undefined%
  \makeatother%
  \begin{picture}(1,0.63587666)%
    \lineheight{1}%
    \setlength\tabcolsep{0pt}%
    \put(0,0){\includegraphics[width=\unitlength,page=1]{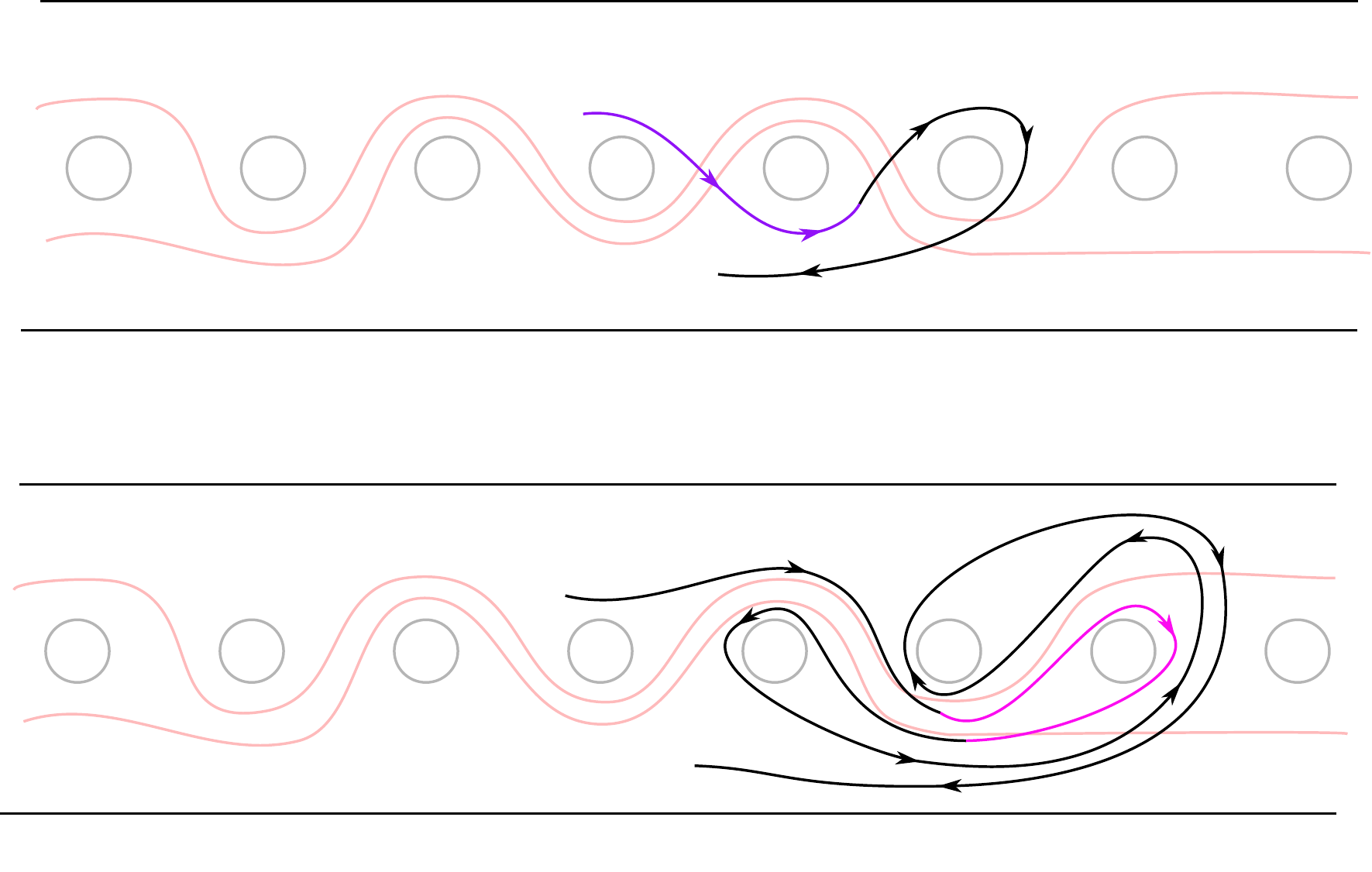}}%
    \put(0.44011752,0.56714672){\makebox(0,0)[lt]{\lineheight{1.25}\smash{\begin{tabular}[t]{l}$\gamma$\end{tabular}}}}%
    \put(0.41343364,0.23931645){\makebox(0,0)[lt]{\lineheight{1.25}\smash{\begin{tabular}[t]{l}$h(\gamma)$\end{tabular}}}}%
    \put(0.43211677,0.3472608){\makebox(0,0)[lt]{\lineheight{1.25}\smash{\begin{tabular}[t]{l}$B_0$\end{tabular}}}}%
    \put(0.42166465,0.00172862){\makebox(0,0)[lt]{\lineheight{1.25}\smash{\begin{tabular}[t]{l}$B_0$\end{tabular}}}}%
  \end{picture}%
\endgroup%
}
\caption{A subsegment which does not have a preimage.}
\label{fig:nopreimage}
\end{subfigure}\hspace{5em}%
\begin{subfigure}{.4\textwidth}
\centering
\def\svgwidth{2.5in}
{\small 
\begingroup%
  \makeatletter%
  \providecommand\color[2][]{%
    \errmessage{(Inkscape) Color is used for the text in Inkscape, but the package 'color.sty' is not loaded}%
    \renewcommand\color[2][]{}%
  }%
  \providecommand\transparent[1]{%
    \errmessage{(Inkscape) Transparency is used (non-zero) for the text in Inkscape, but the package 'transparent.sty' is not loaded}%
    \renewcommand\transparent[1]{}%
  }%
  \providecommand\rotatebox[2]{#2}%
  \newcommand*\fsize{\dimexpr\f@size pt\relax}%
  \newcommand*\lineheight[1]{\fontsize{\fsize}{#1\fsize}\selectfont}%
  \ifx\svgwidth\undefined%
    \setlength{\unitlength}{509.8554bp}%
    \ifx\svgscale\undefined%
      \relax%
    \else%
      \setlength{\unitlength}{\unitlength * \real{\svgscale}}%
    \fi%
  \else%
    \setlength{\unitlength}{\svgwidth}%
  \fi%
  \global\let\svgwidth\undefined%
  \global\let\svgscale\undefined%
  \makeatother%
  \begin{picture}(1,0.59425084)%
    \lineheight{1}%
    \setlength\tabcolsep{0pt}%
    \put(0,0){\includegraphics[width=\unitlength,page=1]{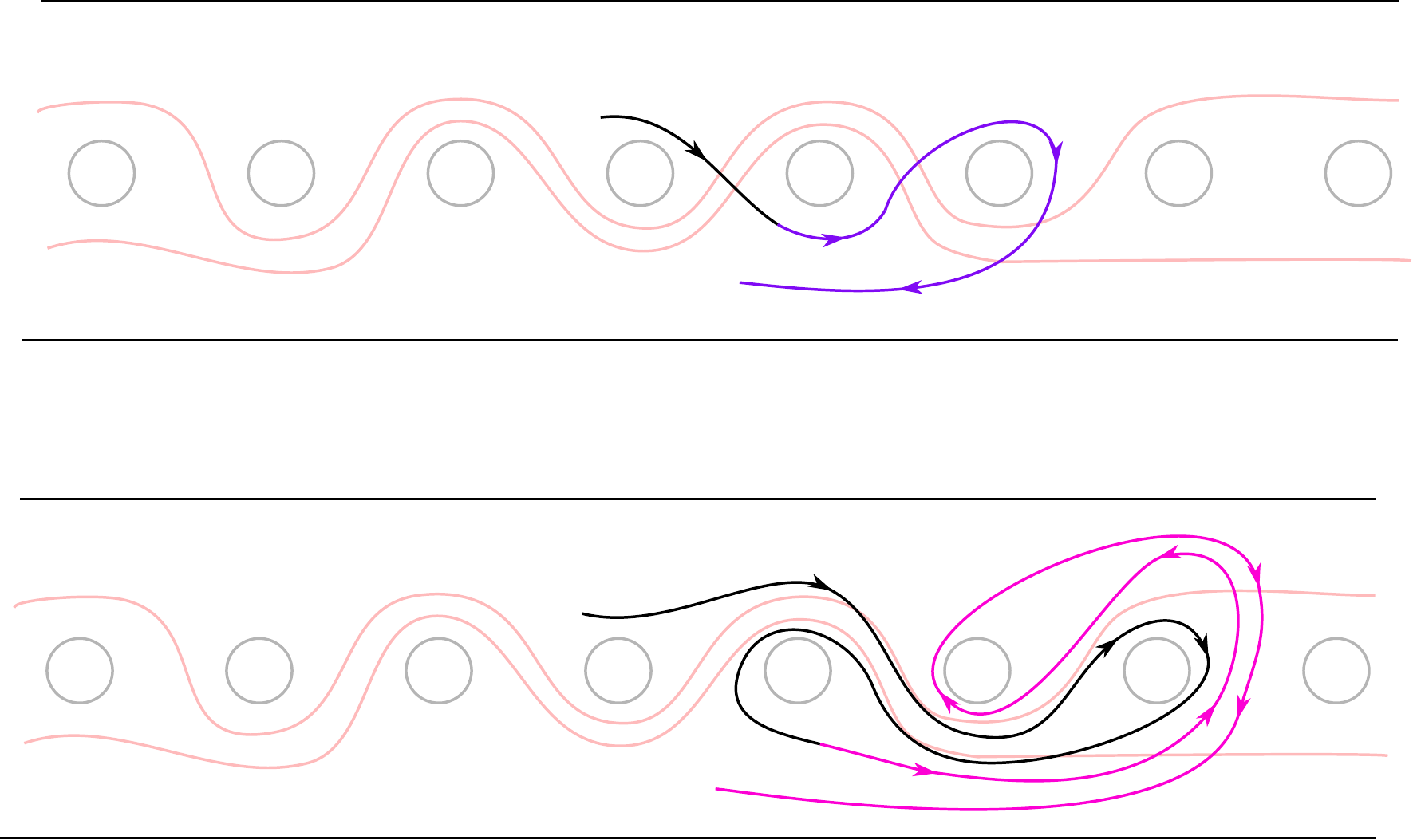}}%
    \put(0.45155348,0.52170892){\makebox(0,0)[lt]{\lineheight{1.25}\smash{\begin{tabular}[t]{l}$\gamma$\end{tabular}}}}%
    \put(0.4572714,0.19387871){\makebox(0,0)[lt]{\lineheight{1.25}\smash{\begin{tabular}[t]{l}$h(\gamma)$\end{tabular}}}}%
  \end{picture}%
\endgroup%
}
\caption{A subsegment which has a preimage.}
\label{fig:preimage}
\end{subfigure}
\caption{A segment $\gamma$ and its image $h(\gamma)$ are shown.  In (A), the pink subsegment of $h(\gamma)$ does not have an inverse; it is a proper subpath of the image of the purple subsegment of $\gamma$.  In (B), the pink subsegment of $h(\gamma)$ has an inverse; it is the image of the purple subsegment of $\gamma$.}
\label{fig:inverses}
\end{figure}

However, we may take an inverse image of a subsegment $\zeta$ of $h(\gamma)$ whenever we \textit{know} that $\zeta$ is the image of a subsegment of $\gamma$. This is the case, for example, when $\zeta=h(\gamma)$; in other words, it is true that $\gamma=h^{-1}(h(\gamma))$. This can also happen  when the initial and terminal characters of a subsegment of $h(\gamma)$ are the images of the initial and terminal characters of a subsegment of $\gamma$.  For example, in Figure \ref{fig:preimage}, a direct computation will show that the initial and terminal characters of $\zeta_2$ (in pink) are the images of the initial and terminal characters of the purple subsegment of $\gamma$.  Therefore $\zeta_2$ is the image of the purple subsegment of $\gamma$, which is precisely the result of calculating $h^{-1}(\zeta_2)$.


\section{Loop Theorem}\label{sec:loop}

As in Section \ref{sec:codingarcs}, we let $S$ be the biinfinite flute surface with a distinguished puncture $p$ and fix the collection of simple closed curves $\{B_i \mid i\in \Z\cup\{P\}\}$ as in Definition~\ref{def:B_i}.
Let $h$ be a permissible shift (see Definition \ref{def:handleshift}) and $k\in\Z\cup\{P\}$.  By an abuse of notation, we may occasionally write $h(k)$, by which we mean that $h(k)$ is the label of $h(B_k)$.  Thus, given any segment whose code is  $k_{o/u}$, we have $h(k_{o/u})=h(k)_{o/u}$.

Recall that a segment is a path which does not have both endpoints on $p$.
\begin{defn} A segment in standard position is \textit{trivial} if it can be homotoped rel endpoints to one of the following:
\begin{enumerate}
\item[(a)] 
a segment contained in one of the separating curves $S_i\in \scc$; or 
\item[(b)] a point.
\end{enumerate}
\end{defn}

\noindent We will use the notation $\emptyset$ to denote the reduced code for a trivial segment. For example, we write $k_ok_o=\emptyset$.

\begin{defn}\label{defn:loop}
A \textit{loop} is a segment that has one of the following forms:
\begin{enumerate}
\item $\delta_1a_{o/u}a_{u/o}\delta_2$ for some $a$ where $\delta_1^i=\delta_2^t$, 
\item $a_{o/u}a_{u/o}$ for some $a$, or
\item $C$.
See Figure \ref{fig:loopex} for examples of loops.
\end{enumerate}
\begin{figure}
\centering
\begin{overpic}[width=3.5in, trim={0.4in 9in 1.1in 0.05in}, clip]{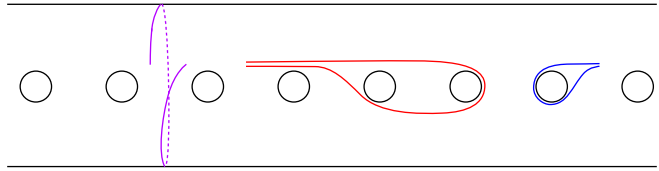}
\end{overpic}
\caption{Some examples of loops from Definition \ref{defn:loop}. The loop in red fits Case (1), the loop in blue fits Case (2), and the loop in purple fits Case (3).}
\label{fig:loopex}
\end{figure}
A loop satisfying (1) is called a \textit{regular loop}, a loop satisfying (2) is called an \textit{over-under loop}, and, as in Definition \ref{def:backloop}, a loop satisfying (3) is called a \textit{back loop}.
Note that regular loops always contain an over-under loop but not all over-under loops can be extended to a regular loop.
A \textit{single loop} is a back loop,  an over-under loop, or a regular loop  such that $\delta_i$ does not contain a loop for $i=1,2$. 
\end{defn}

\noindent With the above definitions, the rest of this section is devoted to analyzing the following question. 

\begin{ques}\label{ques:imageoflooptrivial}
Let $h$ be a permissible shift.  When does $h$ send a loop to a trivial segment? 
\end{ques}

The reason we introduced standard position is to ensure that this question is well defined.  The issue is that homotopies of a loop can change whether or not its image is trivial, even if those homotopies keep the endpoints on a fixed separating closed curve (see Figure \ref{fig:trivvsnon}).  Thus it is important that, given a loop, we first put it in standard position before applying the permissible shift $h$.  This will remove any possible ambiguity in the image of the loop.

\begin{figure}[H]
\centering
\includegraphics[width=4.5in, trim={0in 8.6in 0in .3in},clip]{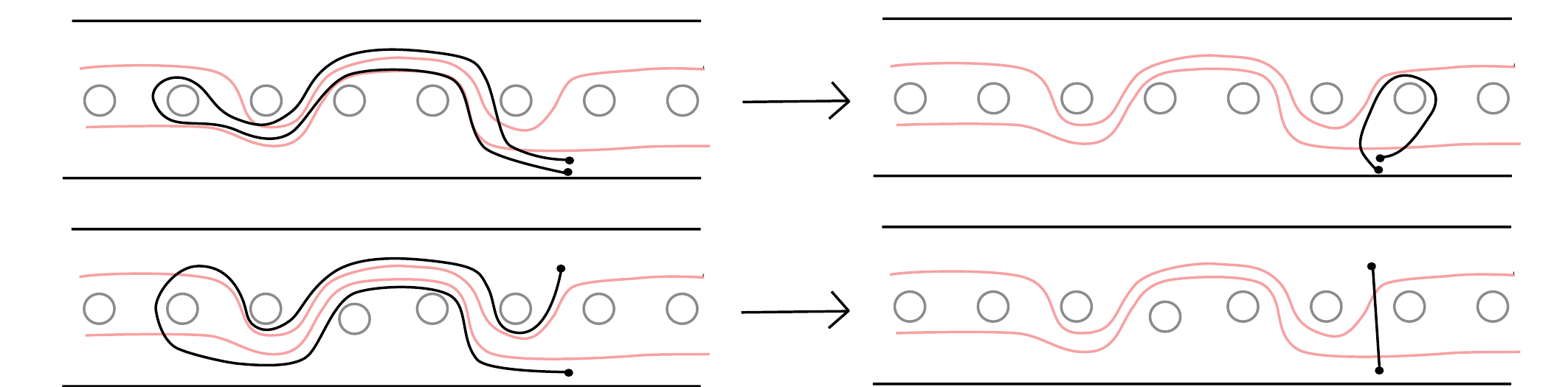}
\caption{The permissible shift $h$ translates to the right.  The first two loops on the left are homotopic via homotopies which keep the endpoints on the same fixed separating curve. However, one image is trivial while the other is not.} 
\label{fig:trivvsnon}
\end{figure}
\begin{figure}[H]
\includegraphics[width=4.5in, trim={0in 5.3in 0in 3.25in},clip]{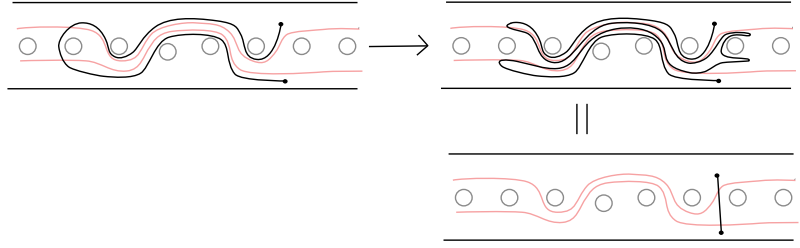}
\caption{ The second loop from  Figure \ref{fig:trivvsnon} is in standard position.  We show why its image under $h$ is trivial.} 
\label{fig:trivexample}
\end{figure}

In this section, we first introduce a kind of cancellation in the image of a segment and its code which we call \textit{cascading cancellation}.  This kind of cancellation will cause technical problems throughout the paper, and much of  Section \ref{sec:arcsthatstartlike} is devoted to understanding how to control it. We then prove the Loop Theorem (Theorem \ref{prop:nontrivialloop}), which answers Question \ref{ques:imageoflooptrivial}.  We end the section with a discussion of several technical consequences of the Loop Theorem which will be useful later.

\subsection{Cascading cancellation}\label{sec:cascading}

\begin{defn}\label{defn:symmetric}
We call an arc $\gamma$ \textit{symmetric} if $\gamma=\delta q_1q_2\overline\delta$ for any characters $q_1$, $q_2$.
In other words, a reduced code for $\gamma$ is palindromic with the exception of the middle two characters.
Note in particular that this implies that $q_1$, $q_2$ have the same numerical value.
\end{defn}

Recall that we find the image of a path with code $\alpha\beta$ under a permissible shift $f$ as follows. Let $q$ be the last character of $\alpha$ and $q'$ be the first character of $\beta$.  Then \[f(\alpha\beta)=f(\alpha)+f(qq')+f(\beta).\]  While $f(\alpha), f(qq'),$ and $f(\beta)$ are all reduced codes, it is possible that the efficient concatenation will cause there to be cancellation.  For example, if $f(\alpha)=1_u1_o2_o3_o$ and $f(qq')=3_o3_o2_o1_o0_oP_o$, then $f(\alpha)+f(qq')=1_u0_oP_o$.  When this type of cancellation occurs, that is, when a character of $f(qq')$ does not appear in a reduced code of the image, we say there is \textit{cancellation involving $f(qq')$}.  In our example, there is cancellation involving $f(qq')$ and $f(\alpha)$.  When it is necessary to be more precise, we may also say there is (respectively, is not) cancellation involving a character $s$, if $s$ appears in the unreduced code but not the reduced code (respectively, appears in both the unreduced code and the reduced code) of the image. Thus in our example, there is no cancellation involving $1_u$ but there is cancellation involving $2_o$. Our goal is to understand, in general, when there is cancellation with a particular character in a path under a permissible shift.

Suppose $\alpha=\gamma q_1$  and $h$ is a permissible shift.  Let $q$ be the terminal character of $\gamma$. It is tempting to believe that if we can show that there is no cancellation involving $h(q_1)$ within $h(qq_1)$, then there is no cancellation involving $h(q_1)$ at all.  However, this is not sufficient, for it is possible that there is ``cascading cancellation."  Before giving a formal definition, we illustrate this phenomenon with an example. 
\begin{ex}
Consider the segment $\delta=3_u2_u1_o0_u$ and the permissible left shift $h$ whose domain is shown in Figure \ref{fig:cascadingcancellation}.  Then 
\[h(\delta)=h(3_u2_u)+h(2_u1_o)+h(1_00_u).\]  We have 
\[h(1_o0_u)=1_o0_u,\]
\[h(2_u1_o)=2_u1_o,\] and
\[h(3_u2_u) = 3_u2_u1_o0_u(-1)_u(-1)_o0_u1_o2_u2_u .\]
Putting this together, we obtain $h(\delta)=3_u2_u1_o0_u(-1)_u(-1)_o.$

\begin{figure}
\centering
\begin{overpic}[width=2.5in]{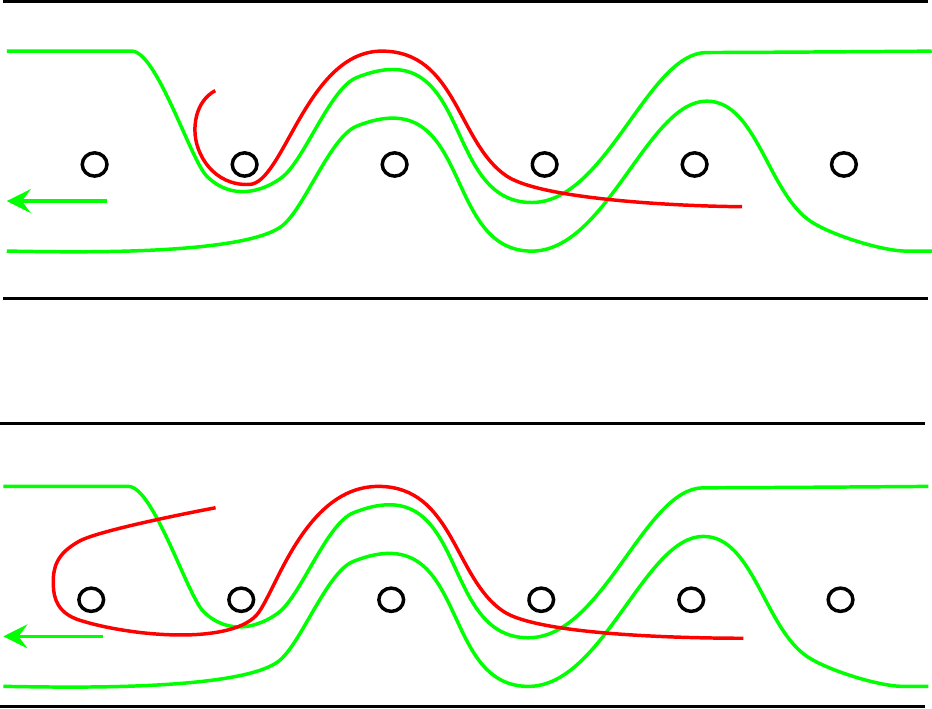}
\put(35,71){$\textcolor{red}{\delta}$}
\put(35,25.5){$\textcolor{red}{h(\delta)}$}
\put(4,-5){$-1$}
\put(25,-5){$0$}
\put(41,-5){$1$}
\put(57,-5){$2$}
\put(73,-5){$3$}
\put(89,-5){$4$}
\put(4,39){$-1$}
\put(25,39){$0$}
\put(41,39){$1$}
\put(57,39){$2$}
\put(73,39){$3$}
\put(89,39){$4$}
\end{overpic}
\caption{Above: the segment $\delta$ in standard position. The domain of the left shift $h$ is shown in green. Below: the image $h(\delta)$.}
\label{fig:cascadingcancellation}
\end{figure}
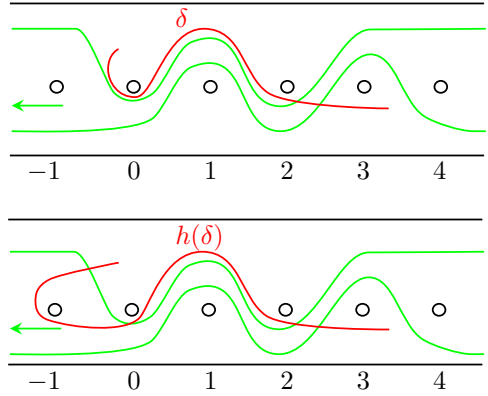

There is no cancellation involving either of the terms when computing $h(2_u1_0)+h(1_o0_u)$.  However, when computing $h(3_u2_u)+h(2_u1_0)$, we see that $h(2_u1_o)$ completely cancels with an initial segment of $h(3_u2_u3)$, and $h(1_o0_u)$ completely cancels with the next subsegment of $h(3_u2_u)$.  Therefore, there is in fact cancellation involving $h(0_u)$ in $h(\delta)$.  
\end{ex}

\begin{defn}
Formally, given an arc $\delta_1\dots \delta_n$ and a permissible shift $f$, we say there is \textit{cascading cancellation involving $f(\delta_n)$} if there is cancellation involving $f(\delta_{n})$ in $f(\delta_1\delta_2\dots \delta_n)$ but there is no cancellation involving $f(\delta_{n})$ in $f(\delta_{n-1}\delta_n)$.  
\end{defn}

Understanding and controlling cascading cancellation is the difficult part of many of the proofs in this paper. The remainder of this subsection is devoted to theorems that will allow us to control cascading cancellation for permissible shifts.   We will often be in the following situation:  there is a segment $\gamma=\gamma_1\gamma_2$ whose image under a shift $h$ we would like to understand and we know that $h(\gamma_1)$ has some desired quality (such as containing a loop, for example).  In order to show that the desirable behavior of $h(\gamma_1)$ persists in $h(\gamma)$, we need to ensure that  there is no  cancellation involving $h(\gamma_1)$ and $h(\gamma_2)$ by controlling cascading cancellation. In Section \ref{sec:arcsthatstartlike}, we will revisit this topic and prove some additional results that allow us to control cascading cancellation for the particular homeomorphisms we construct, which are compositions of shifts.

\begin{lem}\label{lem:generalmonotone}
Let $h$ be a permissible right shift with domain $D$ and turbulent region $(n_1,n_2)$.  Let $\alpha$ be a strictly monotone segment supported on $(n_1,n_2)$.
Then in a reduced code $h(\alpha)$ has $n$ loops around $n_2$, where $n$ is the number of times $\alpha$ fully crosses $D$.
\end{lem}
\begin{proof}
Without loss of generality, we will assume that $\alpha$ is strictly monotone increasing, that is, the numerical value $\alpha^i$ is strictly less than that of $\alpha^t$, as the conclusion is invariant under replacing $\alpha$ by $\overline\alpha$.  As in Section~\ref{sec:segnobackloop}, put $\alpha$ in standard position.   Since no segment in standard position which is supported on $(n_1,n_2)$ will be completely contained in the domain of the shift, we may write
$$\alpha=\alpha^{d}_1+\alpha^{e}_1+\dots+\alpha^{d}_s+\alpha^{e}_s,$$
where each $\alpha_j^{d}$ is a maximal subsegment disjoint from $D$, each $\alpha_j^{e}$ fully crosses $D$, and $\ell_c(\alpha_j^{d})\ge 2$, $\ell_c(\alpha_j^{e})=2$ when non-empty.
Notice that if $\alpha^{d}_j\neq\emptyset$ and $\alpha^{d}_{j+1}\neq\emptyset$, then also $\alpha^{e}_j\neq\emptyset$ by maximality. 

Under the above decomposition, in an unreduced code we have
$$h(\alpha)=\alpha^{d}_1+h(\alpha^{e}_1)+\dots+\alpha^{d}_s+h(\alpha^{e}_s),$$
where every non-empty $\alpha^{e}_j$ will have image
$$h(\alpha^{e}_j)=\partial D\vert_{[(\alpha^{e}_j)^i,n_2)}(n_2)_{o/u}(n_2)_{u/o}\overline{\partial D\vert_{[(\alpha^{e}_j)^t,n_2)}}.$$  Thus each full crossing $\alpha_j^e$ contributes a loop around $n_2$ in an \textit{unreduced} code for $h(\alpha)$.  We must show that such loops persist in a \textit{reduced} code for $h(\alpha)$.

If $\alpha^{e}_j$, $\alpha^{e}_{j+1}\neq\emptyset$ and $\alpha^{d}_{j+1}=\emptyset$  then a calculation shows that in a reduced code
$$h(\alpha^{e}_{j}+\alpha^{d}_{j+1}+\alpha^{e}_{j+1})=h(\alpha^{e}_{j}+\alpha^{e}_{j+1})=h(\alpha^{e}_{j})+h(\alpha^{e}_{j+1}).$$
In particular, there is no cancellation between $h(\alpha^{e}_{j})$ and $h(\alpha^{e}_{j+1})$ and the loops around $n_2$ persist. See Figure~\ref{fig:twofullcross}.

\begin{figure}
\centering
\begin{overpic}[width=3.5in]{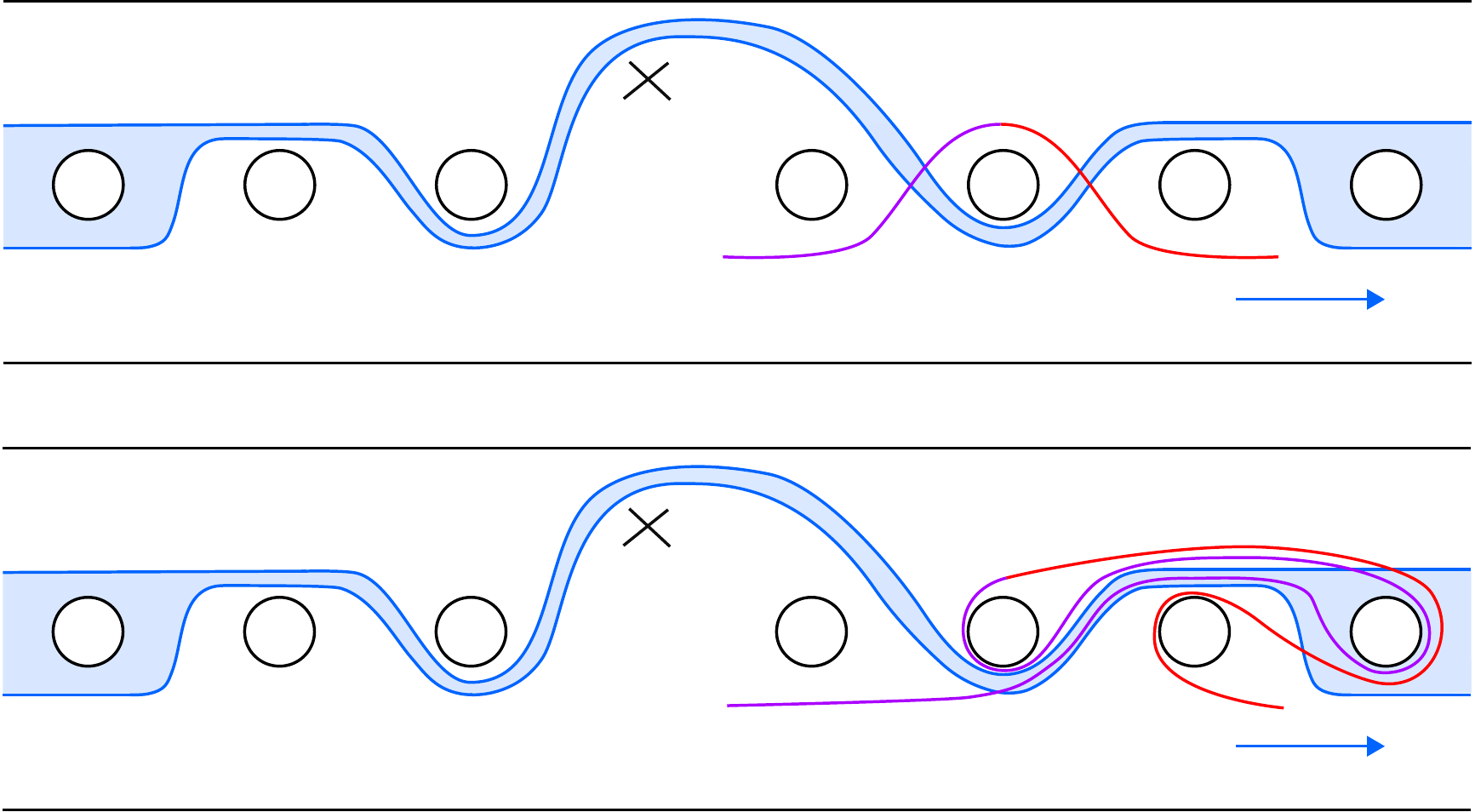}
\end{overpic}
\caption{Above: The case where $\alpha_j^e+\alpha_j^d+\alpha_{j+1}^e = \alpha_j^e+\alpha_{j+1}^e$ in the proof of Lemma \ref{lem:generalmonotone}, where $\alpha_j^e$  is in purple and $\alpha_{j+1}^e$ is in red.  Below: the non-trivial image of this segment under $h$.}
\label{fig:twofullcross}
\end{figure}

 Now assume that $\alpha_{j+1}^d$ is nonempty. The maximality of $\alpha_{j+1}^d$ implies that $\alpha^{e}_{j}$ and $\alpha^{e}_{j+1}$ are nonempty as well. We must show that there is no cancellation between the loops around $n_2$ in $h(\alpha^{e}_{j})$ and $h(\alpha^{e}_{j+1})$ so that both loops persist in a reduced code for $h(\alpha^{e}_{j})+\alpha^{d}_{j+1}+h(\alpha^{e}_{j+1})$. Recall that in standard position, a full crossing occurs between two adjacent characters  $k_{o/u}$, $(k')_{o/u}$ with $k,k'\in(n_1,n_2)$ such that the $o/u$ pattern of $k$ and/or $k'$ does not match that of $\partial D$, and our convention is to make this choice for the largest possible $k,k'$.

The subtlety arises because the code for $h(\alpha^{e}_j)+\alpha^{d}_{j+1}$ may not be reduced if an initial subsegment of $\alpha^{d}_{j+1}=h(\alpha_{j+1}^{d})$ agrees with $\partial D$, in which case $(\alpha_j^e)^t$ agrees with $\partial D$.  If $(\alpha_j^e)^t$ is the character of the full crossing $\alpha^{e}_j$ that does \emph{not} agree with $\partial D$, then this character will block cancellation between $h(\alpha^{e}_{j})$ and $\alpha^{d}_{j+1}$, so that the loops around $n_2$ in $h(\alpha^{e}_{j})$ and $h(\alpha^{e}_{j+1})$ cannot cancel. 
On the other hand, if $(\alpha_{j+1}^e)^i$ is the character of the full crossing $\alpha^{e}_{j+1}$ that does \emph{not} agree with $\partial D$, then this character will block cancellation between the two $n_2$ loops, even if $\alpha_{j+1}^d$ fully cancels in a reduced code for $h(\alpha^{e}_{j})+\alpha^{d}_{j+1}+h(\alpha^{e}_{j+1})$. 

In the last case, where $(\alpha_j^e)^t$ and $(\alpha_{j+1}^e)^i$ both agree with $\partial D$, the largest possible choice of $k,k'$ for the full crossing $\alpha_j^e$ would in fact result in a segment $\alpha^{e}_{j}+\alpha^{d}_{j+1}+\alpha^{e}_{j+1}$ that is fully disjoint from $D$, contradicting the maximality of $\alpha_{j+1}^d$, so that this case cannot occur. See Figure~\ref{fig:disjoint}. 

Thus, there is no cancellation between the loops around $n_2$ in $h(\alpha^{e}_j)$ and $h(\alpha^{e}_{j+1})$, which proves the lemma. 
\end{proof}

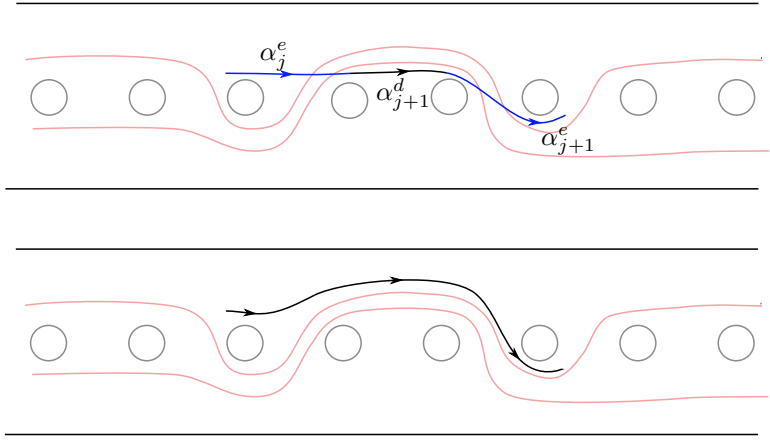
\begin{figure}
\centering
\def\svgwidth{4in}
\begingroup%
  \makeatletter%
  \providecommand\color[2][]{%
    \errmessage{(Inkscape) Color is used for the text in Inkscape, but the package 'color.sty' is not loaded}%
    \renewcommand\color[2][]{}%
  }%
  \providecommand\transparent[1]{%
    \errmessage{(Inkscape) Transparency is used (non-zero) for the text in Inkscape, but the package 'transparent.sty' is not loaded}%
    \renewcommand\transparent[1]{}%
  }%
  \providecommand\rotatebox[2]{#2}%
  \newcommand*\fsize{\dimexpr\f@size pt\relax}%
  \newcommand*\lineheight[1]{\fontsize{\fsize}{#1\fsize}\selectfont}%
  \ifx\svgwidth\undefined%
    \setlength{\unitlength}{504.16271737bp}%
    \ifx\svgscale\undefined%
      \relax%
    \else%
      \setlength{\unitlength}{\unitlength * \real{\svgscale}}%
    \fi%
  \else%
    \setlength{\unitlength}{\svgwidth}%
  \fi%
  \global\let\svgwidth\undefined%
  \global\let\svgscale\undefined%
  \makeatother%
  \begin{picture}(1,0.56644212)%
    \lineheight{1}%
    \setlength\tabcolsep{0pt}%
    \put(0,0){\includegraphics[width=\unitlength,page=1]{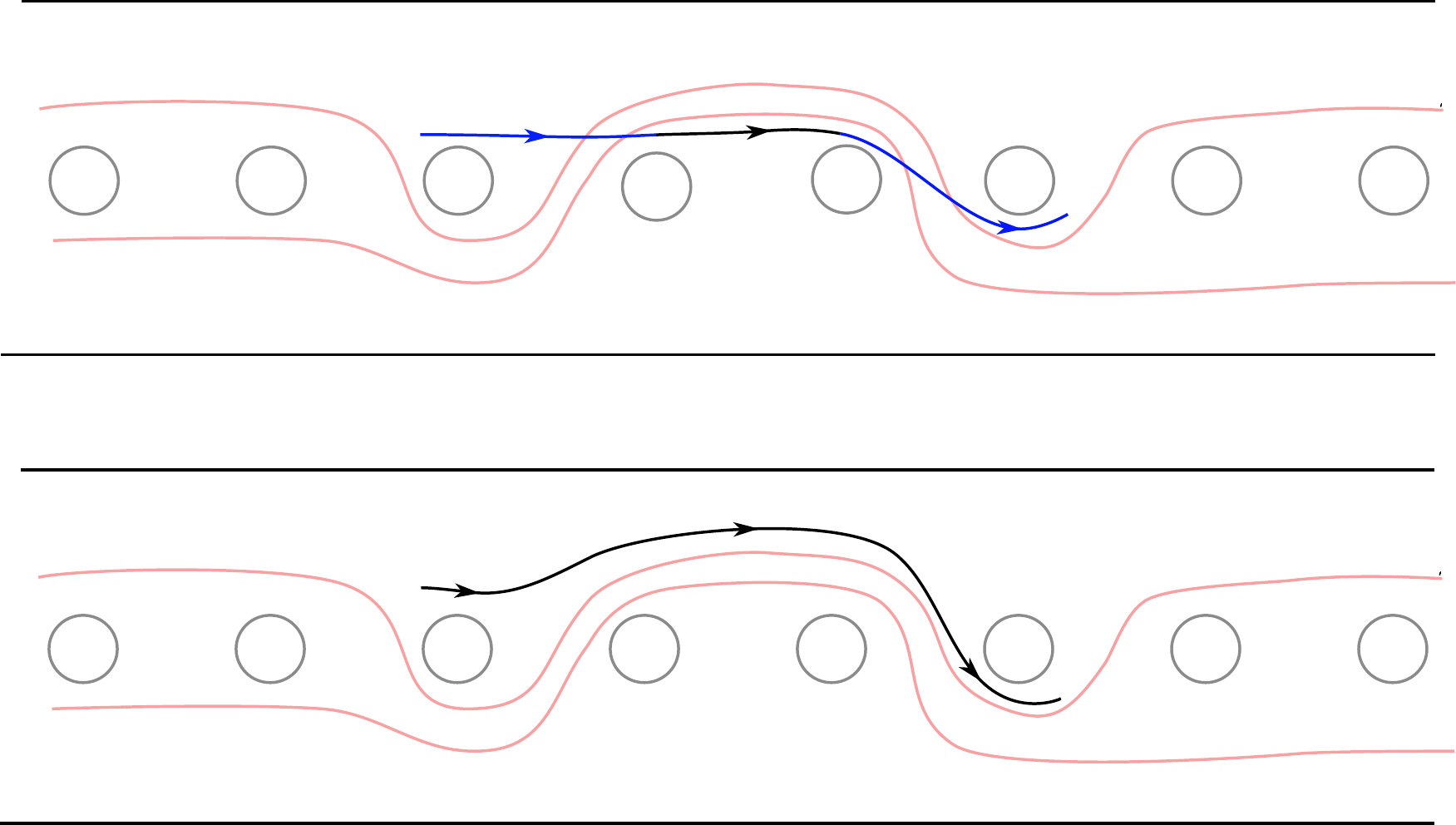}}%
    \put(0.33368125,0.49180684){\makebox(0,0)[lt]{\lineheight{1.25}\smash{\begin{tabular}[t]{l}$\alpha_j^e$\end{tabular}}}}%
    \put(0.70191014,0.38072592){\makebox(0,0)[lt]{\lineheight{1.25}\smash{\begin{tabular}[t]{l}$\alpha_{j+1}^e$\end{tabular}}}}%
    \put(0.48587087,0.43845303){\makebox(0,0)[lt]{\lineheight{1.25}\smash{\begin{tabular}[t]{l}$\alpha_{j+1}^d$\end{tabular}}}}%
  \end{picture}%
\endgroup%

\caption{Above: A picture of $\alpha^{e}_{j}+\alpha^{d}_{j+1}+\alpha^{e}_{j+1}$, not in standard position, where $(\alpha_j^e)^t$ and $(\alpha_{j+1}^e)^i$ both agree with $\partial D$. Below: The same subsegment of $\alpha$ in standard position.}
\label{fig:disjoint}
\end{figure}

\subsection{The Loop Theorem} \label{subsec:loops}

In this subsection, we prove the  Loop Theorem, which describes the form a loop must have if its  image under a shift is trivial.  In particular, the image of any loop which does not have the form as stated in the theorem is non-trivial. In addition, any segment that is not a loop cannot have trivial image under a shift since the numerical values the initial and terminal characters, and thus their images, differ. 

\begin{thm}[Loop Theorem] \label{prop:nontrivialloop}
Let $h$ be a permissible right shift with turbulent region $(n_1,n_2)$.  Suppose $\beta$ is a non-trivial loop such that $h(\beta)=\emptyset$.  Then either $\beta=k_o\delta k_o$ or $\beta=k_u\delta k_u$ where 
\begin{enumerate}[(i)]
\item $k\in (n_1,n_2)$, 
\item $\delta=\gamma(n_1)_{o/u}(n_1)_{u/o}\bar{\gamma}$, and
\item $\gamma$ follows $\partial D$ between $k$ and $n_1$.
\end{enumerate}
\end{thm}

\begin{proof}
Put $\beta$ in standard position. We will consider the image $h(\beta)$.   By the discussion in Section \ref{sec:inverses}, we have $\beta=h^{-1}(h(\beta))$.  By assumption, $h(\beta)$ is trivial, and therefore it is homotopic rel endpoints to either a segment contained in the separating curve $S_k$  for some $k\in[n_1,n_2)$ or a point.  If $h(\beta)$ is homotopic rel endpoints to a point, then $\beta=h^{-1}(h(\beta))$ is also homotopic rel endpoints to a point, in which case $\beta$ is trivial, which is a contradiction.  So suppose that $h(\beta)$ is homotopic rel endpoints to  an embedded subsegment $\sigma$ of the separating curve $S_k$  for some $k\in(n_1,n_2]$. Then  $\beta$ is homotopic rel endpoints to $h^{-1}(\sigma)$.

Let $D$ denote the domain of $h$.  Since $\beta$ is in standard position, the endpoints of $\sigma$ (which are the endpoints of $\beta$) are not contained in $D$.  There are two possibilities.  If $\sigma\cap D=\emptyset$, then $\beta$ is homotopic rel endpoints to $h^{-1}(\sigma)=\sigma$ and so $\beta$ is trivial, which is a contradiction.  On the other hand, if $\sigma\cap D\neq \emptyset$, then $\sigma$ must fully cross $D$.   A direct computation shows that, up to reversing orientation,
\[
h^{-1}(\sigma)=\overline{\partial D}|_{[k,n_1)}(n_1)_{o/u}(n_1)_{u/o}\partial D|_{(n_1,k]},
\]
for $k\in(n_1,n_2)$, as desired.
\end{proof}

We note one immediate consequence of Theorem \ref{prop:nontrivialloop}.
\begin{cor}\label{cor:compoundloop}
Suppose $\beta$ is any of the following:
\begin{enumerate}
\item a loop containing more than one single loop;
\item a loop containing a back loop; or
\item a loop with endpoints outside of the turbulent region.
\end{enumerate}
Then $h(\beta)$ is non-trivial.
\end{cor}

Some of these consequences may be surprising, so we now give a more intuitive (and informal) discussion of why the image of loops as in the corollary are non-trivial.  

We first consider the case of segments which contain more than one single loop.  At first glance, it may seem that if a segment $\beta$ is composed of loops which each satisfy the conclusion of Theorem \ref{prop:nontrivialloop}, then $h(\beta)$ will be trivial.  The problem arises in how these loops fit together to form $\beta$.  Suppose we have two loops, $\beta_1$ and $\beta_2$ which each satisfy the conclusion of Theorem \ref{prop:nontrivialloop}.  If the numerical value of $\beta_1^t$ is not the same as that of $\beta_2^i$, then in order to ``connect" $\beta_1$ to $\beta_2$, we must add a segment between the terminal point of $\beta_1$ and the initial point of $\beta_2$.  This segment is either disjoint from $D$, in which case it is fixed by $h$, or it fully crosses $D$.  In the former case, this segment will appear in the image of $\beta$, while in the latter case, the image of this full crossing will be non-trivial by Lemma \ref{lem:generalmonotone} applied to the full crossing. If the numerical values are the same and $\beta = \beta_1\beta_2$, then $\beta$ will be trivial. On the other hand, if there is some segment connecting the terminal point of $\beta_1$ to the initial point of $\beta_2$, then as before, the image of this segment will not be trivial and will force $h(\beta)$ to be non-trivial.  See Figure \ref{fig:NontrivialImage} for an example of this.

Suppose next that $\beta$ contains a back loop $C$.  Intuitively, it seems reasonable that if $\beta=C\beta'\bar C$ and $\beta'$ is a loop as in Theorem \ref{prop:nontrivialloop}, then $h(\beta)$ is trivial.  However, this is not the case.  To see this, notice that if we put $\beta'$ in standard position then it will have one endpoint above $D$ and one below $D$.  However, $C$ and $\bar C$ are the same loop.  If $C$ is either above or below $D$, then this will cause $\beta$ to have a full crossing between $\beta'$ and one of $C$ or $\bar C$.  The image of this full crossing will be non-trivial by Lemma \ref{lem:generalmonotone} applied to the full crossing, which will prevent any cancellation between $h(C)=C$ and $h(\bar C)=\bar C$.  On the other hand, if $C$ exits at the top and enters at the bottom, then $\bar C$ exits at the bottom and enters at the top.  This will again force $\beta$ to have a full crossing between $\beta'$ and one of $C$ or $\bar C$, preventing cancellation between the images of the back loops.

Finally, if $\beta$ is a loop whose endpoints are outside the turbulent region, say both in $[n_2,\infty)$, and $\beta\vert_{(n_1,n_2)}$ has the form in the conclusion of Theorem \ref{prop:nontrivialloop}, it seems possible that $h(\beta)$ is trivial.  For example, suppose that $\beta=(n_2)_o\beta'(n_2)_o$, where $\beta'$ is as in Theorem \ref{prop:nontrivialloop}.  Then since $h(\beta')$ is trivial, $h(\beta')$ will be a segment contained in the separating curve $S_{n_2}$.
On the other hand, the images of the connectors $(n_2)_o(\beta')^i$ and $(\beta')^t(n_2)_o$ will each be of the form $(n_2+1)_o(n_2)_{o/u}(n_2-1)_{o/u}$ or $(n_2-1)_{o/u}(n_2)_{o/u}(n_2+1)_o$.  In particular, we have $h(\beta)=(n_2+1)_o(n_2)_{o/u}(n_2)_{u/o}(n_2+1)_o$, which is non-trivial.  

\begin{figure}
\centering
\begin{overpic}[width=3in]{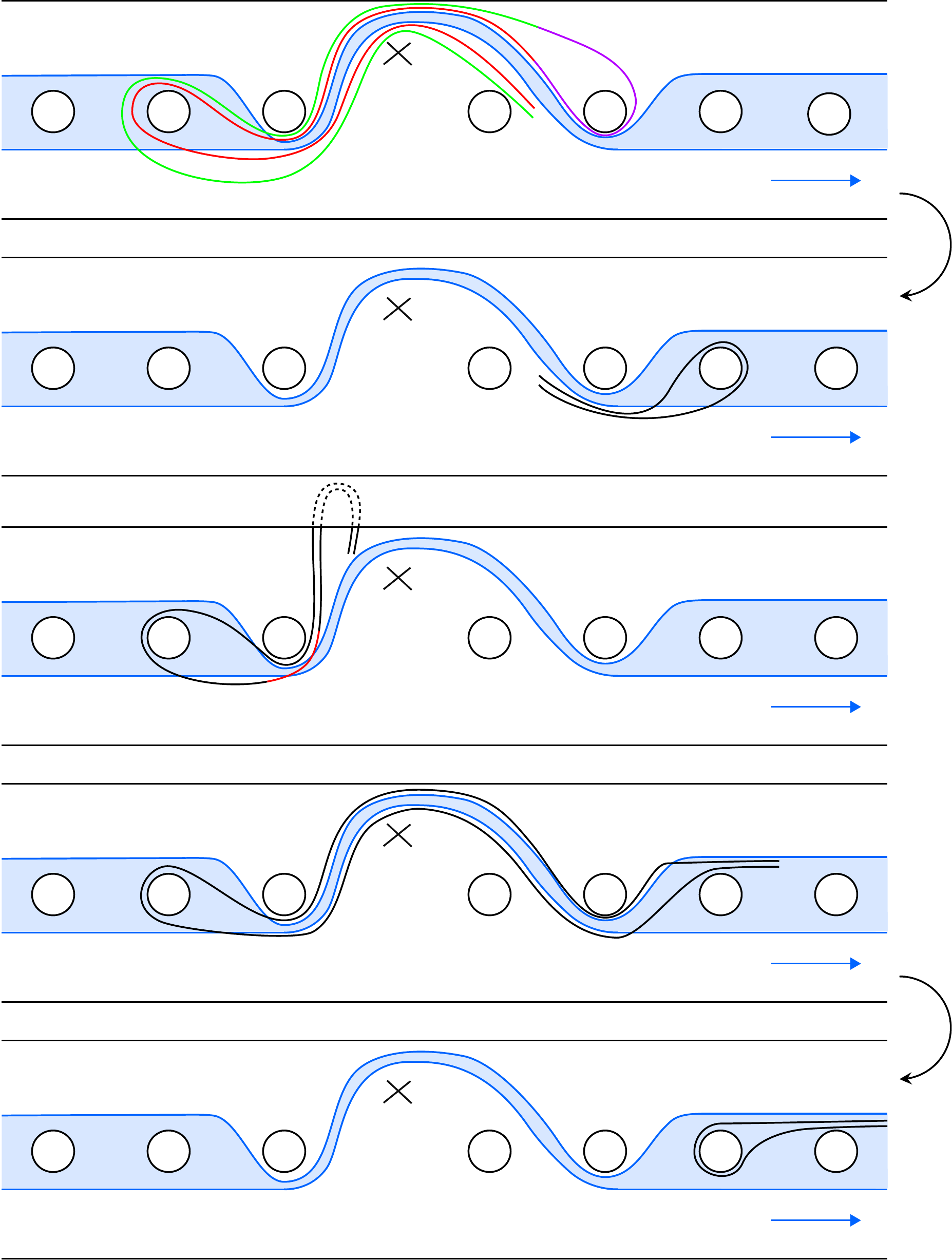}
\put(15,90){\tiny$\textcolor{red}{\beta_1}$}
\put(28,90){\tiny$\textcolor{green}{\beta_2}$}
\put(24,96){\tiny$\beta$}
\put(54,65){\tiny$h(\beta)$}
\put(22.5,58.5){\tiny$C$}
\put(25.75,58.5){\tiny$\overline{C}$}
\put(49,24){\tiny$\beta$}
\put(65,13){\tiny$h(\beta)$}
\end{overpic}
\caption{Examples of loops whose images are non-trivial.  From top to bottom, these correspond to Corollary \ref{cor:compoundloop} (1), (2), and (3).}
\label{fig:NontrivialImage}
\end{figure}
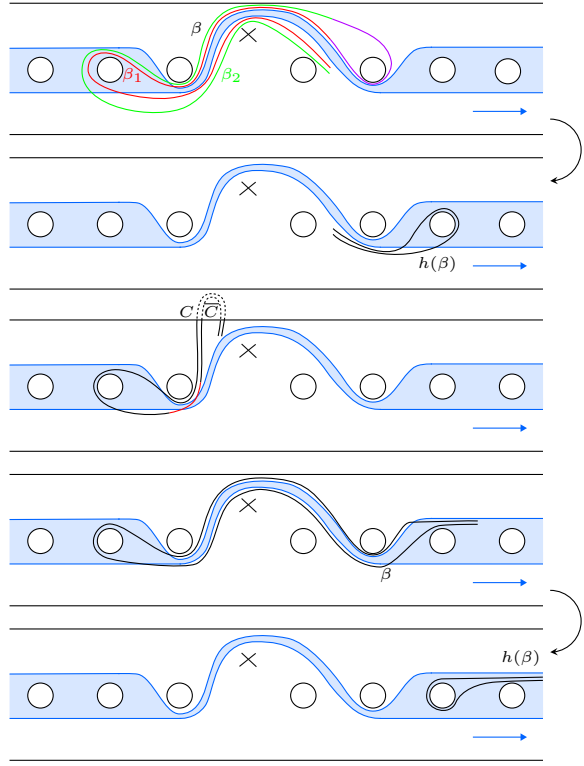




\subsection{Consequences of the Loop Theorem}
In this subsection, we record several (technical) consequences of Theorem \ref{prop:nontrivialloop} which will be useful in later sections.
The first  shows that loops in the shift region persist under the image of shifts.

\begin{lem} \label{lem:imageof2o2ufamily}
Let $h$ be a permissible right shift and $\gamma$ any non-trivial, simple arc with no back loops.
If $k\in (-\infty,n_1)\cup[n_2,\infty)$ and $k_{o/u}k_{u/o}$ appears in a reduced code for $\gamma$, then there is no cancellation involving $h(k_{o/u}k_{u/o})$ in an unreduced code for $h(\gamma)$.
In particular, such a pair $k_{o/u}k_{u/o}$ will yield a pair $(k+1)_{o/u}(k+1)_{u/o}$ in a reduced code for $h(\gamma)$.
\end{lem}

\begin{proof}
We argue for $k_ok_u$ without loss of generality. Fix $k$ as above and put $\gamma$ into standard position. If in standard position $\gamma$ is completely contained in the domain of $h$, then the result is clear, so suppose this is not the case.  We assume for contradiction that there is cancellation with $h(k_ok_u)$ and, without loss of generality, that it is with $h(k_u)=(k+1)_u$, as otherwise we replace $\gamma$ with $\overline\gamma$ and argue identically for $k_o$ instead. Let $\gamma'$ be the minimal subsegment of $\gamma$ beginning with $k_ok_u$ so that $h(\gamma')$ had cancellation with $h(k_u)$. The only way cancellation in an unreduced code can involve $h(k_u)$ is if a subsegment of the form $(k+1)_o(k+1)_u\delta(k+1)_u$ appears in an unreduced $h(\gamma')$ where a reduced code for $\delta$ is trivial. Since $k\in (-\infty,n_1)\cup[n_2,\infty)$, any $(k+1)_u$ in $h(\gamma')$ must be the image under $h$ of $k_u$. Therefore, $\gamma' = k_ok_u\eta k_u$ where $h(k_u\eta k_u) = (k+1)_u \delta (k+1)_u$, which has trivial reduced code. This contradicts Theorem~\ref{prop:nontrivialloop} as $k \notin (n_1, n_2)$.
\end{proof}

The second consequence of our loop theorems shows that characters of a segment that lie in the turbulent region which disagree with the domain of a shift persist in a reduced code of the image of the segment under that shift.

\begin{lem}\label{lem:bpersists}
Let $h$ be a permissible right shift with domain $D$.  Let $\delta$ be a simple segment whose support intersects $(n_1,n_2)$ non-trivially, and suppose $\delta$ contains a character $b$ with numerical value in $(n_1,n_2)$ which disagrees with $\partial D$.  Then $b$ persists in a reduced code for $h(\delta)$.  

In other words, if $\delta=\delta_1b\delta_2$, then $h(\delta)=\sigma_1b\sigma_2$, where $\sigma_1b=h(\delta_1b)$ and $b\sigma_2=h(b\delta_2)$.
\end{lem}

\begin{proof}
Write $b=k_{o/u}$, where by assumption $k\in(n_1,n_2)$. 

\begin{claim}\label{claim:bfromb}
Since $b$ disagrees with $\partial D$, any occurrence of $b$ in $h(\delta)$ must also appear in $\delta$.  
\end{claim}
\begin{proof}Since $b$ disagrees with $\partial D$, the segment $\delta$ cannot be homotoped rel endpoints to be completely contained in $D$.  
Thus in standard position, $\delta$  can be written as the efficient concatenation of subsegments which are disjoint from $D$, subsegments of length two which are either full or half crossings, and subsegments supported on $(-\infty,n_1]\cup[n_2,\infty)$ (see Lemma \ref{lem:stdpos}).  We consider the images of each type of subsegment in turn.  The subsegments which are disjoint from $D$ are fixed by $h$. The image of subsegments supported on $(-\infty,n_1)\cup[n_2,\infty)$ have empty intersection with $(n_1,n_2)$, and so cannot contain $b$.  The remaining subsegments are half crossings, full crossings, or segments that include a character whose numerical value is $n_1$.  Any subsegment of the image of any of these subsegments supported on $(n_1,n_2)$ agrees with $\partial D$ and so cannot contain $b$.  This proves the claim.
\end{proof}

Now suppose towards a contradiction that $b$ does not persist in $h(\delta)$.  Then there must be another instance of $b$ in an \textit{unreduced code} for $h(\delta)$ which cancels with $b$.  By the claim, $\delta$ must contain a subsegment of the form $b\sigma b$ whose image is trivial.
Theorem \ref{prop:nontrivialloop} then implies that $b\sigma b=\overline{\partial D|_{(n_1,k]}}(n_1)_{o/u}(n_1)_{u/o}\partial D|_{(n_1,k]}$.  However, $b=k_{o/u}$ disagrees with $\partial D$ by assumption, and therefore $b\sigma b$ cannot have this form, which is a contradiction.
\end{proof}


\section{Constructing the homeomorphisms}\label{sec:arcinit}

As in Section \ref{sec:codingarcs}, let $S$ be the biinfinite flute  with an isolated puncture $p$, and let $\{B_i\mid i\in\Z\cup P\}$ be the simple closed curves from Definition \ref{def:B_i} bounding a collection of punctures $\{p_i\mid i\in \Z\cup P\}$, where $p_P=p$. As in that section, we move the punctures $\{p_i\mid i\in \Z\cup P\}$ to the front of $S$ and all other punctures to the back of $S$.  

In this section, we define a countable collection of elements $\{g_n\}_{n\in\N}$ in $\MCG(S,p)$.  In the following sections we will show that these elements are of intrinsically infinite type and are loxodromic with respect to the action of $\MCG(S,p)$ on $\mc A(S,p)$.
  
\begin{defn}\label{def:g_n}
For each $n\in\N$, define  
\begin{equation}\label{eqn:defgn}
g_n:=h^{(n)}_3\circ h^{(n)}_2\circ h_1,
\end{equation}
 where $h_1$ and $h_i^{(n)}$ for $i=2, 3$ are permissible shifts defined as follows.  Let $h_1$ be the right shift whose domain includes all simple closed curves $B_j$ for $j\in \mathbb Z\setminus\{0\}$  and passes under $p$ and $B_0$.  Let $h_2^{(n)}$ be the left shift whose domain includes only $\{B_j\mid j\in (-\infty,-n-1]\cup[n+1,\infty)\}$, passes over $p$ and $B_0$, and, when $j\in [-n,-1]\cup[1,n]$,  passes  under $B_j$ for $j$ odd and over $B_j$ for $j$ even.    Finally, let $h_3^{(n)}$ be the right shift whose domain includes only $\{B_j\mid j\in (-\infty,-n-1]\cup[n+1,\infty)\}$ and passes under $p$ and over $B_j$ for all $j\in[-n,P)\cup(P,n]$.  See Figure~\ref{fig:shifts}. \end{defn}

\begin{figure}
\centering
\begin{subfigure}{0.4\linewidth}
\begin{overpic}[width=2.5in]{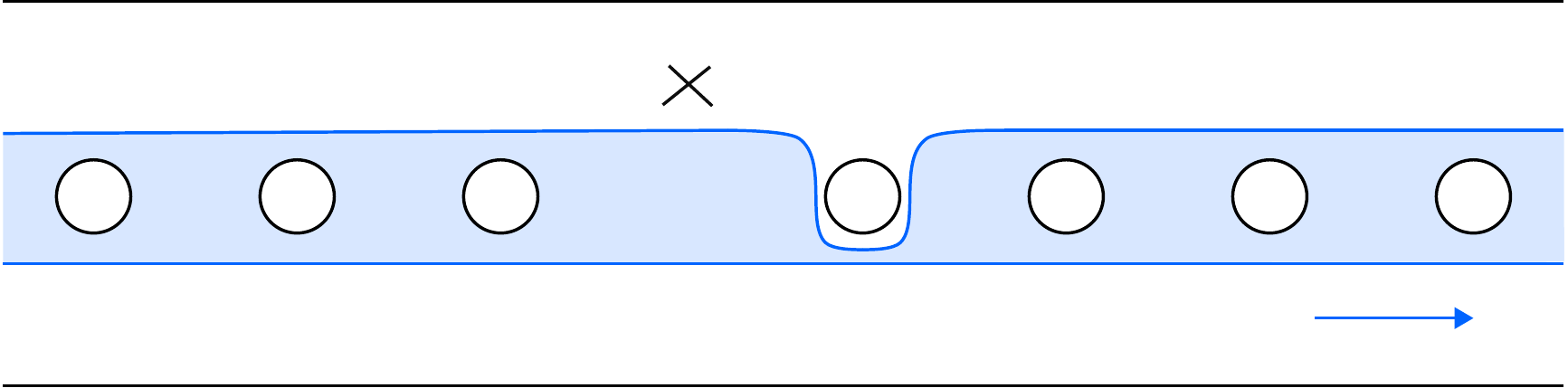}
\put(3,5){\tiny{$-3$}}
\put(15.5,5){\tiny{$-2$}}
\put(29,5){\tiny{$-1$}}
\put(42.5,5){\tiny$P$}
\put(54.5,5){\tiny{$0$}}
\put(67.5,5){\tiny{$1$}}
\put(79.5,3){\color{blue}\tiny{$h_1$}}
\end{overpic}
\caption{The shift map $h_1$.}
\label{fig:theshiftsh1}
\end{subfigure}\\
\par\bigskip
\begin{subfigure}{0.4\linewidth}
\centering
\begin{overpic}[width=1.4\textwidth]{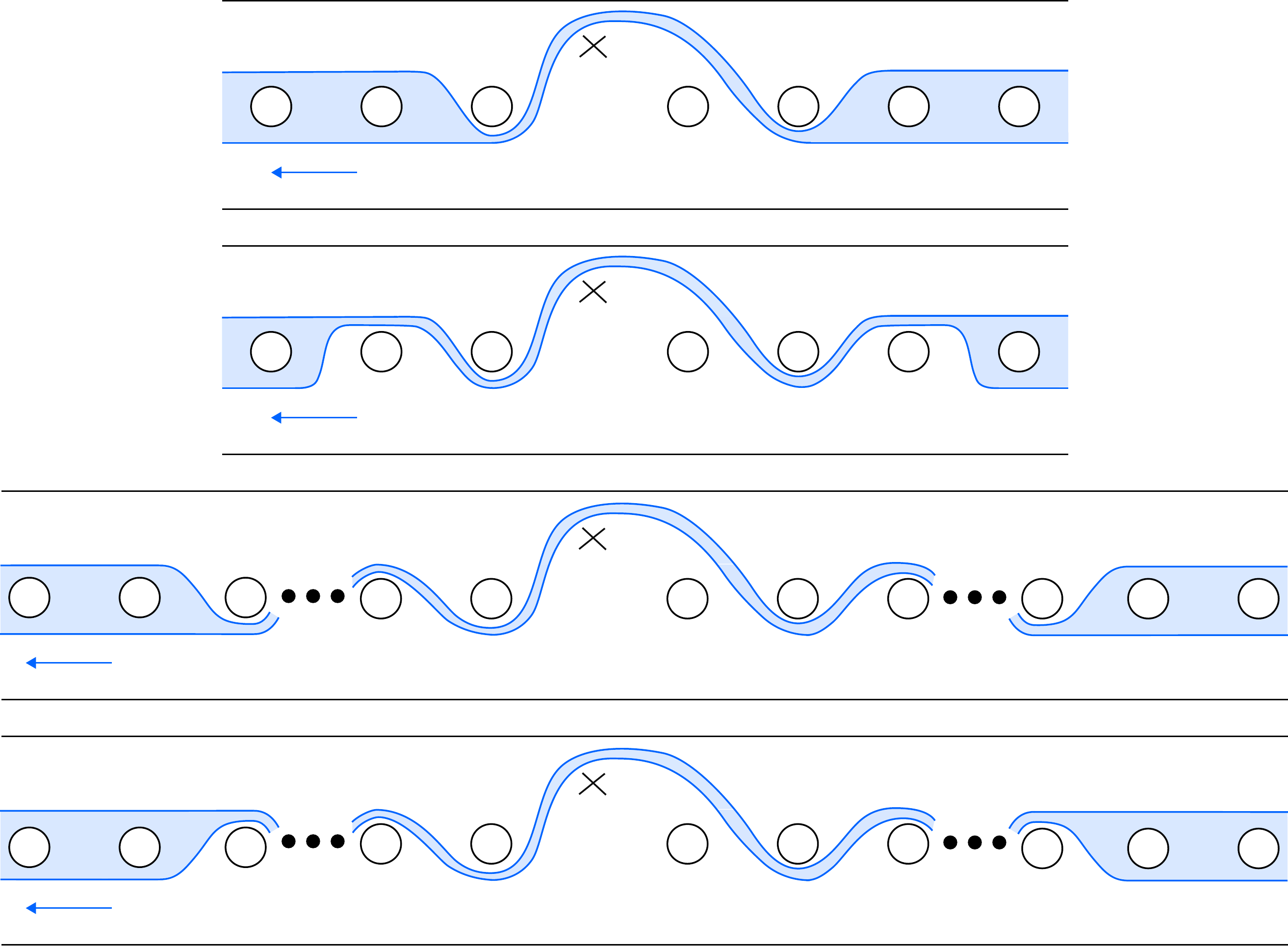}
\put(53,60){\tiny{$0$}}
\put(61.5,60){\tiny{$1$}}
\put(70,60){\tiny{$2$}}
\put(78.5,60){\tiny{$3$}}
\put(45,60){\tiny{$P$}}
\put(36,60){\tiny{$-1$}}
\put(28.5,58.75){\color{blue}\tiny{$h_2^{(1)}$}}
\put(53,41){\tiny{$0$}}
\put(61.5,41){\tiny{$1$}}
\put(70,41){\tiny{$2$}}
\put(78.5,41){\tiny{$3$}}
\put(45,41){\tiny{$P$}}
\put(36,41){\tiny{$-1$}}
\put(28.5,40){\color{blue}\tiny{$h_2^{(2)}$}}
\put(53,22){\tiny{$0$}}
\put(61.5,22){\tiny{$1$}}
\put(70,22){\tiny{$2$}}
\put(86.25,22){\tiny{$n+1$}}
\put(80.25,22){\tiny{$n$}}
\put(45,22){\tiny{$P$}}
\put(36,22){\tiny{$-1$}}
\put(27.5,22){\tiny{$-2$}}
\put(16.5,22){\tiny{$-n$}}
\put(9.5,20.75){\color{blue}\tiny{$h_2^{(n)}$}}
\put(53,3){\tiny{$0$}}
\put(61.5,3){\tiny{$1$}}
\put(70,3){\tiny{$2$}}
\put(86.25,3){\tiny{$n+1$}}
\put(80.25,3){\tiny{$n$}}
\put(45,3){\tiny{$P$}}
\put(36,3){\tiny{$-1$}}
\put(27.5,3){\tiny{$-2$}}
\put(16.5,3){\tiny{$-n$}}
\put(9.5,1.75){\color{blue}\tiny{$h_2^{(n)}$}}
\end{overpic}
\caption{The shift map $h_2^{(n)}$ for various $n$. In order from top to bottom, the first two pictures are the case of $n=1,2$ (respectively) and the bottom two are for general $n$ odd and general $n$ even (respectively).}
\label{fig:theshiftsh2}
\end{subfigure}\hfill
\begin{subfigure}{0.4\linewidth}
\centering
\begin{overpic}[width=1.4\textwidth]{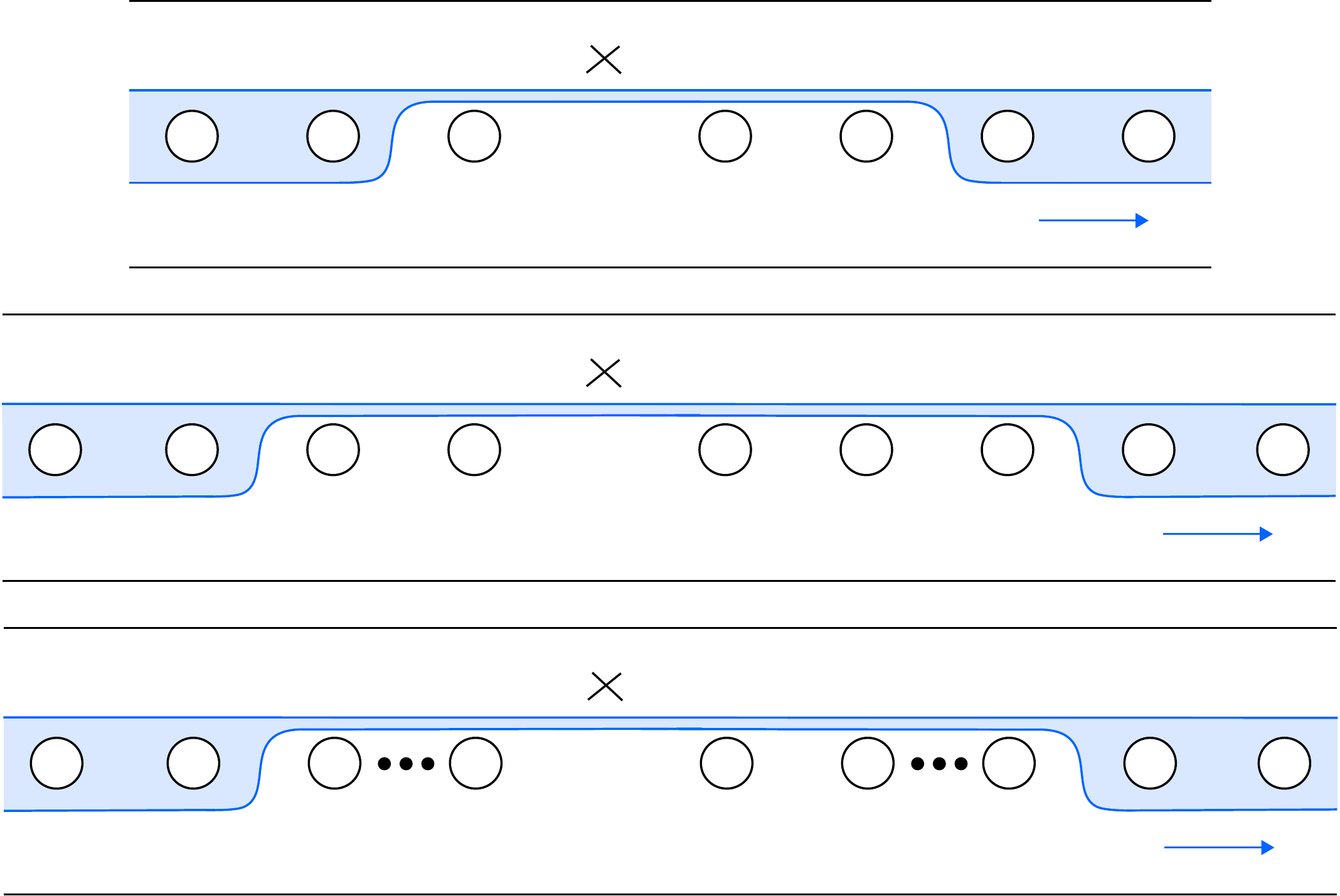}
\put(12,51){\tiny{$-3$}}
\put(23,51){\tiny{$-2$}}
\put(33.5,51){\tiny{$-1$}}
\put(53.5,51){\tiny{$0$}}
\put(64,51){\tiny{$1$}}
\put(44,51){\tiny{$P$}}
\put(71,49){\color{blue}\tiny{$h_3^{(1)}$}}
\put(1.5,27.5){\tiny{$-4$}}
\put(12,27.5){\tiny{$-3$}}
\put(23,27.5){\tiny{$-2$}}
\put(33.5,27.5){\tiny{$-1$}}
\put(53.5,27.5){\tiny{$0$}}
\put(64,27.5){\tiny{$1$}}
\put(75,27.5){\tiny{$2$}}
\put(44,27.5){\tiny{$P$}}
\put(80,25.5){\color{blue}\tiny{$h_3^{(2)}$}}
\put(0,4){\tiny{$-n-2$}}
\put(10,4){\tiny{$-n-1$}}
\put(22.5,4){\tiny{$-n$}}
\put(33.5,4){\tiny{$-1$}}
\put(53.5,4){\tiny{$0$}}
\put(64,4){\tiny{$1$}}
\put(75,4){\tiny{$n$}}
\put(44,4){\tiny{$P$}}
\put(79,2){\color{blue}\tiny{$h_3^{(n)}$}}
\end{overpic}
\caption{The shift map $h_3^{(n)}$ for various $n$. In order from top to bottom, the first two pictures are the case of $n=1,2$ and the bottom picture is the case of general $n$.}
\label{fig:theshiftsh3}
\end{subfigure}
\caption{The various shift maps used in the construction of $g_n=h_3^{(n)}\circ h_2^{(n)}\circ h_1$.}
\label{fig:shifts}
\end{figure}
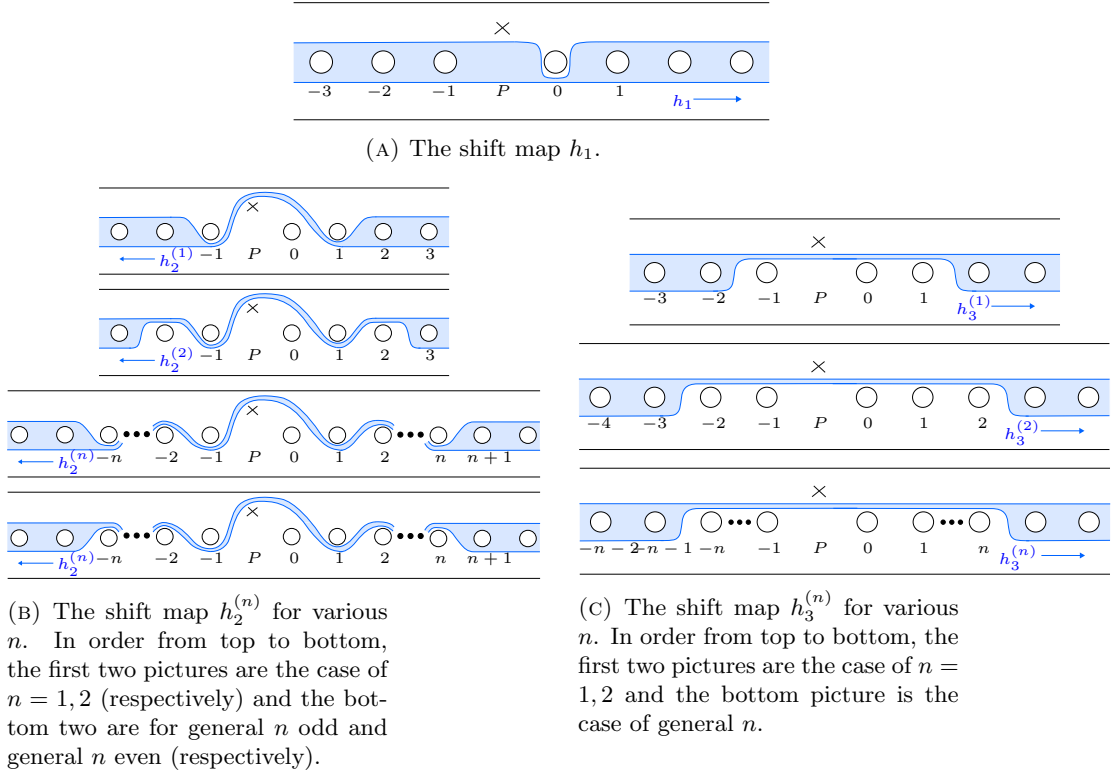

In the language of the previous sections, we have that $n_1 = -1$ and $n_2 = 1$ for $h_1$. For $h_2^{(n)}$, $n_1 = n+1$ and $n_2 = -n-1$. For $h_3^{(n)}$, $n_1 = -n-1$ and $n_2 = n+1$.

In order to prove that each $g_n$ is loxodromic with respect to the action of $\MCG(S,p)$ on $\mc A(S,p)$, we introduce a collection of arcs $\{\alpha_i^{(n)}\mid i\in \Z\}$ which are invariant under $g_n$.  We will show in Proposition \ref{prop:loxonS} that for $i\geq 0$, these arcs form a quasi-geodesic half-axis for $g_n$ in $\mc A(S,p)$.

\begin{defn}\label{def:alpha_i}
For each fixed $n$, we  define a sequence of arcs $\{\alpha_i^{(n)}\}_{i=-\infty}^\infty$ in $\mathcal{A}(S,p)$ as the images under successive applications of $g_n$ of the fixed initial arc $\alpha_0=P_s0_o0_uP_s$.
That is to say that $\alpha^{(n)}_i=g^i_n(\alpha_0)$ for any $i\in \Z$.  See Figure \ref{fig:alphain=1} for the case $n=1$ and Figure \ref{fig:alphain} for general $n\geq 1$.  
\end{defn}

\begin{figure}
\centering
\def\svgwidth{3in}
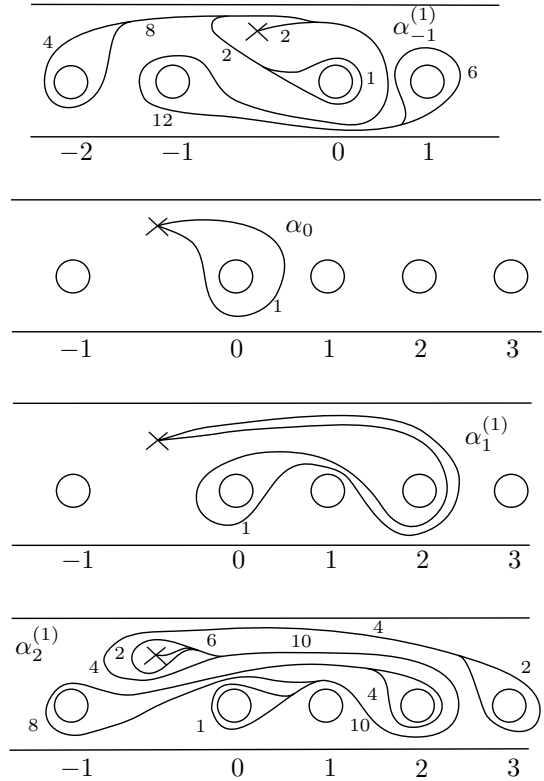
\caption{The arcs $\alpha_i^{(1)}$ for $i=-1,0,1,2$ drawn as train tracks for simplicity.}
\label{fig:alphain=1}
\end{figure}

When $n=1$, the arcs $\alpha_i^{(1)}$ are the most straightforward and thus most useful for building intuition.  We suggest that the reader keep these arcs in mind while reading the remainder of the paper.  We now give the code for $\alpha_1^{(1)}$ and $\alpha_2^{(1)}$, which the reader can compare to Figure~\ref{fig:alphain=1}:
\begin{align*}
\alpha_1^{(1)} =& P_s0_o1_o2_o2_u1_o0_o0_u1_o2_u2_o1_o0_oP_s \\
\alpha_2^{(1)} =& P_s0_o1_o2_o2_u1_o0_oP_u(-1)_u(-1)_oP_u0_o1_o2_o2_u1_o0_oP_u(-1)_o(-1)_uP_u0_o1_o2_u2_o1_o \\
&0_oP_uP_o0_o1_o2_o3_o3_u2_o1_o0_oP_oP_u0_o1_o2_o2_u 
 1_o0_oP_u(-1)_u(-1)_oP_u 0_o1_o2_u2_o1_o0_o\\ 
 &P_u(-1)_o(-1)_u  P_u0_o1_o2_u2_o1_o0_oP_uP_o0_o1_o2_o2_u1_o0_o0_u 1_o2_u2_o1_o0_oP_oP_u0_o1_o2_o \\
& 2_u1_o0_oP_u(-1)_u(-1)_oP_u0_o  1_o2_o2_u1_o0_oP_u(-1)_o(-1)_uP_u0_o1_o2_u2_o1_o0_oP_uP_o0_o \\
&1_o2_o3_u3_o  2_o1_o0_oP_oP_u0_o1_o2_o2_u1_o0_oP_u(-1)_u(-1)_oP_u0_o1_o2_u2_o1_o0_oP_u(-1)_o(-1)_u \\
& P_u0_o1_o2_u2_o1_o0_oP_s
\end{align*}

While it is possible to compute the images of arcs under $g_n=h_3^{(n)}\circ h_2^{(n)} \circ h_1$ by hand, we have also written a computer program to implement this.  The interested reader should contact the authors for more details.

\begin{figure}
\centering
\def\svgwidth{4in}
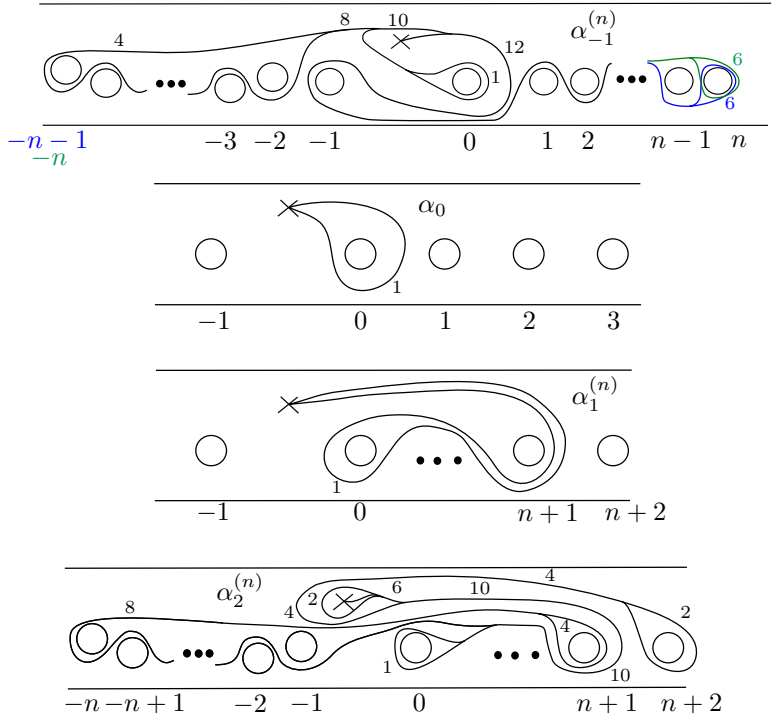
\caption{The arcs $\alpha_i^{(n)}$ for $i=-1,0,1,2$ and $n\geq 1$.  In $\alpha_{-1}^{(n)}$, the blue appears when $n$ is odd, while the green appears when $n$ is even.}
\label{fig:alphain}
\end{figure}

\begin{rem} In light of how complicated the arcs $\alpha_i^{(n)}$ are when $i\neq 0,1$, it may be surprising that only single loops can have trivial image under each of the shifts $h_1,h_2^{(n)},$ and $h_3^{(n)}$ (see Theorem \ref{prop:nontrivialloop} and Corollary \ref{cor:compoundloop}).  Figure \ref{fig:alpha-1toalpha0} gives the intermediate steps so that one can see how the image under $g_1$ of a complicated arc such as $\alpha^{(1)}_{-1}$ becomes the straightforward arc $\alpha_0$.   The figure shows the collection of single loops with trivial image at each stage of computing $g_1(\alpha^{(1)}_{-1})=\alpha_0$. A similar phenomenon happens for $n>1$.
\end{rem}

\begin{figure}
\centering
\def\svgwidth{4in}
{\small 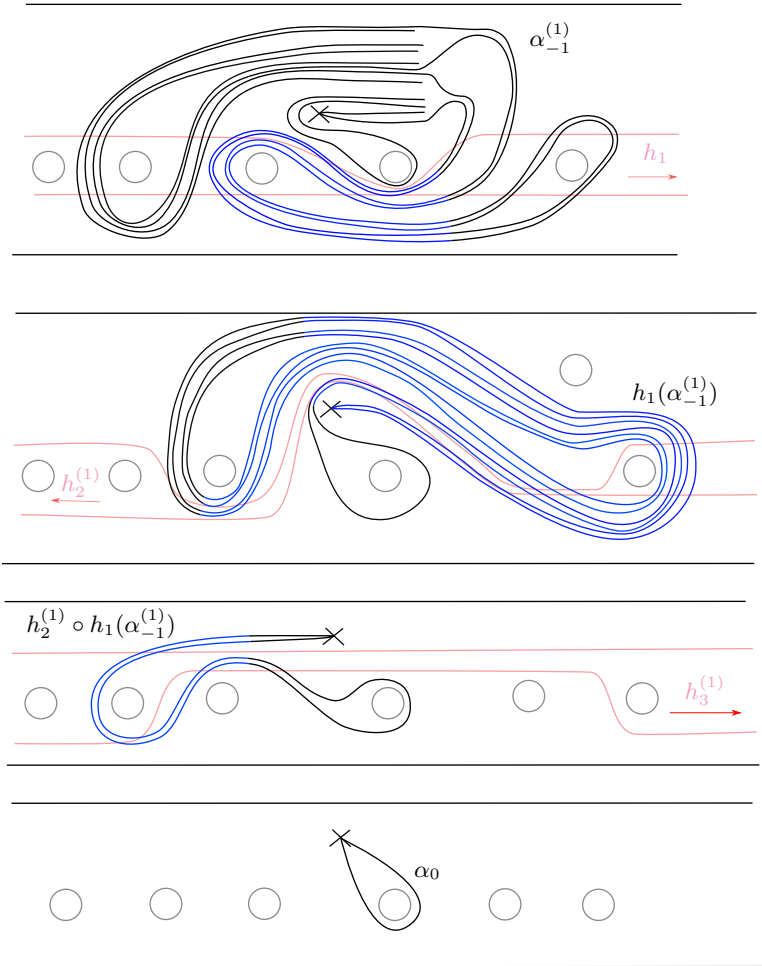}
\caption{Each of the strands in blue are single loops whose  image is trivial under the  shift whose domain is shown in red.  Each blue strand fits the form of Theorem \ref{prop:nontrivialloop}.}
\label{fig:alpha-1toalpha0}
\end{figure}

\medskip
 Our method for using the arcs $\alpha_i^{(n)}$ to prove that $g_n$ is a loxodromic isometry of $\mc A(S,p)$ is inspired by the foundational work of Bavard in \cite{Bavard}.  Bavard constructs an element $f$ of the mapping class group of the plane minus a Cantor set which is loxodromic with respect to the associated relative arc graph.  We give a brief outline of Bavard's methods here.  

Bavard constructs a collection of simple paths $\beta_i$ which start at an isolated puncture and end at some point of the Cantor set. These paths are invariant under the action of a chosen homeomorphism $f$ and have the property that $\beta_{i+1}=f(\beta_i)$.  They are constructed so that for all $i$, the path $\beta_{i+1}$ begins by following the same path as $\beta_i$.  Roughly, if a path $\gamma$ begins by following the same path as $\beta_i$ then Bavard says the path $\gamma$ `begins like' $\beta_i$ (see \cite[Section~2]{Bavard}). Bavard uses this definition to define a function from the vertex set of a certain graph of paths (defined similarly to the relative arc graph but with paths instead of arcs) to $\Z_{\ge 0}$ by sending a path $\gamma$ to the maximal $i\in \Z_{\ge 0}$ so that $\gamma$ begins like $\beta_i$. 
Aramayona, Fossas, and Parlier \cite{AFP} show that the relative arc graph is quasi-isometric to the graph whose vertex set is isotopy class of paths with \textit{at least one} endpoint on the distinguished isolated puncture.   Therefore this function can then be used to estimate distances in the relative arc graph and ultimately to show that the collection of paths $\{\beta_i\mid i\in\Z_{\geq 0}\}$ forms a geodesic half-axis for the element $f$.  The key fact that Bavard uses is that if $\delta$ is a path which begins like $\beta_i$ and $\gamma$ is any path disjoint from $\delta$, then $\gamma$ begins like $\beta_{i-1}$ \cite[Lemma~2.4]{Bavard}.

Our arcs $\{\alpha_i^{(n)}\}_{i\in\Z}$ do not satisfy  the same property as Bavard's paths. One notable difference is that our arcs start and end at the puncture $p$.  Because of this, for a fixed $n$, $\alpha_{i+1}^{(n)}$ does not begin by following \textit{the entirety} of $\alpha_i^{(n)}$.  However, we will show that  it does begin by following the \textit{first half} of $\alpha_i^{(n)}$.  In light  of this, we modify Bavard's notion of `begins like' and make the following definition.

\begin{defn}\label{def:startslike}
Fix an $n\in \N$. An arc $\delta$ \textit{starts like} $\alpha^{(n)}_i$ if the maximal initial or the maximal terminal segment of $\delta$ which agrees with an initial or terminal segment of $\alpha^{(n)}_i$ (not necessarily respectively) has code length at least $\lfloor\frac12\ell_c(\alpha^{(n)}_i)\rfloor$. Recall that code length was defined in Definition \ref{def:codelength}.
\end{defn}

In Section \ref{sec:arcsthatstartlike}, we investigate the properties of arcs that start like $\alpha_i^{(n)}$ for some $i$, and in Section \ref{sec:highways} we use these properties to prove Theorem \ref{thm:startslike}, which is an analog of \cite[Lemma~2.4]{Bavard}.  In our more general context, proving this result is quite a bit more involved than the proof of \cite[Lemma~2.4]{Bavard}.  One reason for this is that the behavior of our arcs $\alpha_i^{(n)}$ is much more complicated than the paths from \cite{Bavard}.


\section{Arcs that start like $\alpha_i^{(n)}$}\label{sec:arcsthatstartlike}

The main goal of this section is to prove that the images of arcs which start like $\alpha^{(n)}_i$ under $g_n$ must start like $\alpha^{(n)}_{i+1}$ (see Corollary \ref{cor:inductcor}).
This will follow from Proposition \ref{lem:imageofinit} and Theorem \ref{lem:induct}.
Proposition \ref{lem:imageofinit} states that the image of the first half of $\alpha_i^{(n)}$ is the first half of $\alpha_{i+1}^{(n)}$.
We need to be careful about what we mean by this because of the possibility that we cause there to be a gap when we break $\alpha_i^{(n)}$ into its first and second half (see Section \ref{sec:segconn}).   Given a segment $\gamma$, let $\mathring\gamma$ be the initial subsegment of $\gamma$ with code length $\lfloor\frac12\ell_c(\gamma)\rfloor$.
For a fixed $i$ and $n$, we have $\alpha_i^{(n)}=\mathring\alpha_i^{(n)}+\alpha'$ for some $\alpha'$.  As in Section \ref{sec:segconn}, it is possible that when we put $\mathring\alpha_i^{(n)}$ and $\alpha'$ (or their images) into standard position individually with respect to one of the shifts $h_1,h_2^{(n)}$, or $h_3^{(n)}$, we will  lose  information when we take its image.  In fact, this does happen when we put $h_2^{(n)}(h_1(\mathring\alpha_i^{(n)}))$ and $h_2^{(n)}(h_1(\alpha'))$ into standard position with respect to $h_3^{(n)}$ (see Figure \ref{fig:gapinh_2}).   We avoid this issue by extending $\mathring\alpha_{i}^{(n)}$ by one character and considering $\mathring\alpha_{i}^{(n)}0_u$.  Note that $0_u$ is fixed by $g_n$.  Thus when we say that the image of the first half of $\alpha_i^{(n)}$ is the first half of $\alpha_{i+1}^{(n)}$ we precisely mean that $g_n(\mathring\alpha_{i}^{(n)}0_u)=\mathring\alpha_{i+1}^{(n)}0_u$ as a subsegment of $g_n(\alpha_{i}^{(n)})=\alpha_{i+1}^{(n)}$, in a reduced code.  In other words, $g_n$ sends everything before the central $0_u$ in $\alpha_i^{(n)}$ to $\mathring\alpha_{i+1}^{(n)}$.
\begin{rem}
When dealing with the arcs $\alpha_i^{(n)}$, the floor function in $\mathring\alpha_i^{(n)}$ is actually unnecessary, as $\ell_c(\alpha_i^{(n)})$ will always be even with the central two characters being $0_o0_u$.
However, at this point in the paper we have not proven this fact, which is a consequence of Proposition \ref{lem:imageofinit}, so we will not use it in what follows.
\end{rem}

\begin{figure}
\centering
\def\svgwidth{4in}
{\small 
\begingroup%
  \makeatletter%
  \providecommand\color[2][]{%
    \errmessage{(Inkscape) Color is used for the text in Inkscape, but the package 'color.sty' is not loaded}%
    \renewcommand\color[2][]{}%
  }%
  \providecommand\transparent[1]{%
    \errmessage{(Inkscape) Transparency is used (non-zero) for the text in Inkscape, but the package 'transparent.sty' is not loaded}%
    \renewcommand\transparent[1]{}%
  }%
  \providecommand\rotatebox[2]{#2}%
  \newcommand*\fsize{\dimexpr\f@size pt\relax}%
  \newcommand*\lineheight[1]{\fontsize{\fsize}{#1\fsize}\selectfont}%
  \ifx\svgwidth\undefined%
    \setlength{\unitlength}{582.91978792bp}%
    \ifx\svgscale\undefined%
      \relax%
    \else%
      \setlength{\unitlength}{\unitlength * \real{\svgscale}}%
    \fi%
  \else%
    \setlength{\unitlength}{\svgwidth}%
  \fi%
  \global\let\svgwidth\undefined%
  \global\let\svgscale\undefined%
  \makeatother%
  \begin{picture}(1,0.55742626)%
    \lineheight{1}%
    \setlength\tabcolsep{0pt}%
    \put(0,0){\includegraphics[width=\unitlength,page=1]{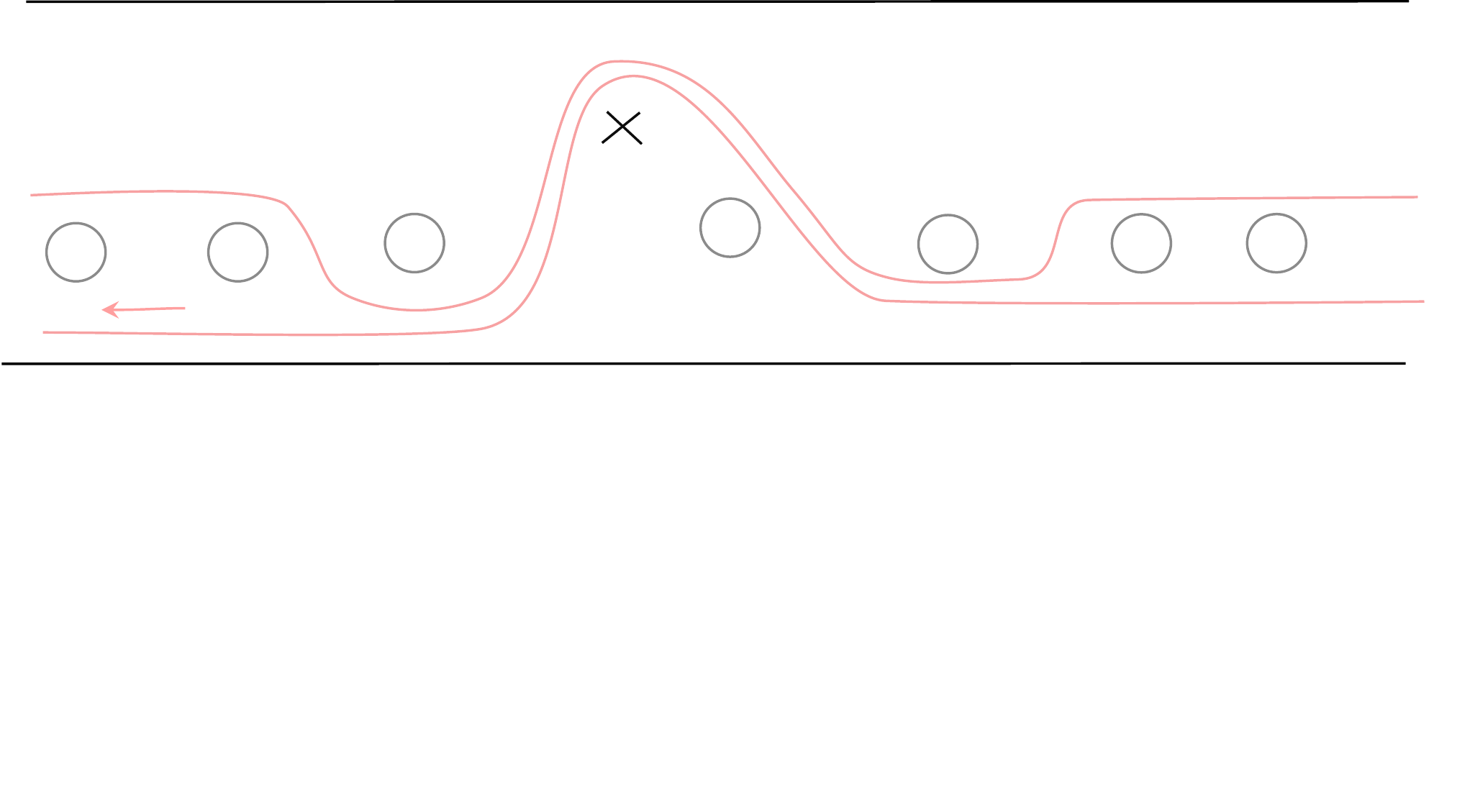}}%
    \put(0.0804057,0.3682342){\color[rgb]{0,0,0}\makebox(0,0)[lt]{\lineheight{1.25}\smash{\begin{tabular}[t]{l}$h_2^{(1)}$\end{tabular}}}}%
    \put(0,0){\includegraphics[width=\unitlength,page=2]{gapinh_2.pdf}}%
    \put(0.78234873,0.51632032){\color[rgb]{0,0,0}\makebox(0,0)[lt]{\lineheight{1.25}\smash{\begin{tabular}[t]{l}$h_1(\mathring\alpha_1^{(1)})$\end{tabular}}}}%
    \put(0.45905641,0.34204685){\color[rgb]{0,0,0}\makebox(0,0)[lt]{\lineheight{1.25}\smash{\begin{tabular}[t]{l}$h_1(0_o0_u)$\end{tabular}}}}%
    \put(0,0){\includegraphics[width=\unitlength,page=3]{gapinh_2.pdf}}%
    \put(0.46359432,0.01780791){\color[rgb]{0,0,0}\makebox(0,0)[lt]{\lineheight{1.25}\smash{\begin{tabular}[t]{l}$h_2^{(1)}\circ h_1(0_o0_u)$\end{tabular}}}}%
    \put(0.60287724,0.21045263){\color[rgb]{0,0,0}\makebox(0,0)[lt]{\lineheight{1.25}\smash{\begin{tabular}[t]{l}$h_2^{(1)}\circ h_1(\mathring\alpha_1^{(1)})$\end{tabular}}}}%
    \put(0.92116649,0.08747729){\color[rgb]{0,0,0}\makebox(0,0)[lt]{\lineheight{1.25}\smash{\begin{tabular}[t]{l}$h_3^{(1)}$\end{tabular}}}}%
  \end{picture}%
\endgroup%
}
\caption{Here we consider $h_1(\mathring\alpha_10_u)=h_1(\mathring\alpha_1)+h_1(0_o0_u)$ and take the images of each subsegment separately.  The images are drawn as train tracks for clarity, but we omit the weights as they are immaterial.  When we put the images $h_2^{(n)}(h_1(\mathring \alpha_1))$ and $h_2^{(n)}(h_1(0_o0_u))=0_o0_u$ in standard position individually with respect to $h_3^{(n)}$ (the second picture), there is a gap in the segment between $p$ and $B_1$, which would have been a full crossing.  This results in a loss of information in the image under $h_3^{(n)}$.}
\label{fig:gapinh_2}
\end{figure}

\begin{prop} \label{lem:imageofinit}
For any $n$ and any $i$, we have  $g_n(\mathring\alpha_{i}^{(n)}0_u)=\mathring\alpha_{i+1}^{(n)}0_u$  in a reduced code.
\end{prop}

In order to prove Proposition \ref{lem:imageofinit}, we need to understand how to control cascading cancellation for our homeomorphisms $g_n$ (see Section~\ref{sec:cascading} for the definition and a discussion of cascading cancellation).
The following lemma is almost a direct consequence of Lemma \ref{lem:imageof2o2ufamily}, the difference being that $g_n$ is not itself a shift.

\begin{lem}\label{cor:imageof2o2ufamily}
For all $k\geq n+1$, if $k_{o/u}k_{u/o}$ appears in a reduced code for $\alpha^{(n)}_i$, then there is no cancellation involving $g_n(k_{o/u}k_{u/o})$ in an unreduced code for $\alpha^{(n)}_{i+1}$. 
In particular, such a pair $k_{o/u}k_{u/o}$ will yield $(k+1)_{o/u}(k+1)_{u/o}$ in a reduced code for $\alpha^{(n)}_{i+1}$.
\end{lem}
\begin{proof}
The union of the turbulent regions for $h_1$, $h_2^{(n)}$, and $h_3^{(n)}$ is $[-n-1,n+1]$.
      Therefore we may apply Lemma \ref{lem:imageof2o2ufamily} three times once we remark that $h_2^{(n)}$ shifts left, $h_1(k_{o/u}k_{u/o})=(k+1)_{o/u}(k+1)_{u/o}$, and $h_2^{(n)}$ fits the hypothesis of Lemma \ref{lem:imageof2o2ufamily}, as $k+1>n+1$.
\end{proof}

We can now use this control on cascading cancellation to prove Proposition \ref{lem:imageofinit}.

\begin{proof}[Proof of Proposition \ref{lem:imageofinit}]
As $\alpha_i^{(n)}$ is symmetric, the only way this proposition could be false is if there is cancellation in $g_n(\alpha_i^{(n)})$ involving the image of the central $0_o0_u$ (which is also $0_o0_u$) in $\alpha_i^{(n)}$.  Equivalently, we need to show that there is no cancellation involving the final $0_o0_u$ in $g_n(\mathring\alpha_{i}^{(n)}0_u)$.  We begin by considering the case $i=1$, where 
\begin{equation}\label{eqn:alpha1init}
\mathring\alpha_{1}^{(n)}0_u=P_s0_o1_o2_o\dots(n+1)_o(n+1)_u n_o(n-1)_o\dots 2_o1_o0_o0_u,
\end{equation}
and all of the characters that are not displayed are $k_o$ for the appropriate $k$. 
By Lemma \ref{cor:imageof2o2ufamily}, there is no cancellation in $g_n(\alpha_1^{(n)})$ involving $g_n((n+1)_o(n+1)_u)$, and so there can be no cascading cancellation involving both the final $0_o0_u$ and the image of any character before $(n+1)_o(n+1)_u$ in \eqref{eqn:alpha1init}.  Thus it suffices to consider $g_n(\mu0_u)$, where \[\mu=(n+1)_o(n+1)_u n_o(n-1)_o\dots 2_o1_o0_o.\]
A direct computation yields, in a reduced code,
\begin{equation} \label{eqn:2o2urepeatsfamily}
g_n(\mu0_u)=(n+2)_o(n+2)_u(n+1)_on_o\dots (n-1)_on_o\mu0_u,
\end{equation}
 and from this computation we can see that there is no cancellation in $g_n(\mathring\alpha_{1}^{(n)})$ involving the final characters, $0_o0_u$.  Thus $g_n(\mathring\alpha_{1}^{(n)}0_u)=\mathring\alpha_{2}^{(n)}0_u$ in a reduced code.

By \eqref{eqn:2o2urepeatsfamily}, we  also see that $g_n(\mu0_u)$ also ends in $\mu0_u$.  Hence, to   show that the result holds for any index $i$, we simply repeatedly apply $g_n$ and Lemma \ref{cor:imageof2o2ufamily} as in the previous paragraph.
\end{proof}

We can deduce from Proposition \ref{lem:imageofinit} that if an arc $\gamma$ starts like $\alpha_i^{(n)}$, then an \textit{unreduced code} for its image $g_n(\gamma)$ will start like $\alpha_{i+1}^{(n)}$.  We will  use the following technical theorem to show that this is true for a \textit{reduced} code for $g_n(\gamma)$, as well.

\begin{thm}\label{lem:induct}
Let $\gamma$ be a simple arc of the form $\gamma=P_s\eta_10_o\eta_2P_s$ in standard position.
Assume the following two conditions hold:
\begin{enumerate}
\item The numerical value of $\eta_2^i$ is at most $0$, i.e., $\eta_2^i$ is either $0_u$ or $P_{o/u}$.
\item Either the first two characters of $\eta_2$ are not a loop around $P$ or, if they are, the initial segment of $\eta_2$ is given by $P_uP_o0_o1_o$.
\end{enumerate}
Then there is no cancellation with the initial $0_o$ in a reduced code for $g_n(0_o\eta_2P_s)$.
\end{thm}

We note that the first condition means that the segment $0_o$ is oriented to the left.

\begin{rem}\label{rem:induct}
Theorem \ref{lem:induct} is written with a particular orientation in mind, but such an orientation is arbitrary.
That is to say, the exact same statement is true applied to the image of $\overline{\gamma}$ under $g_n$.
We will use both the original and this ``reverse" version of Theorem \ref{lem:induct} later in the paper, so we make  its statement precise:
Suppose, as in Theorem \ref{lem:induct}, that $\gamma=P_s\eta_10_o\eta_2P_s$ and
\begin{enumerate}
\item The numerical value of $\eta_1^t$ is at most $0$;
\item Either the final two characters of $\eta_1$ are not a loop around $P$ or the final segment of $\eta_1$ is given by $1_o0_oP_oP_u$.
\end{enumerate}
Then there is no cancellation with the terminal $0_o$ in a reduced code for $g_n(P_s\eta_10_o)$.
\end{rem}

We will prove Theorem \ref{lem:induct} through a series of lemmas.
Fix an arc $\gamma$ which satisfies the hypothesis of Theorem \ref{lem:induct}. We must show that there is no cancellation with the initial $0_o$ in a reduced code for $g_n(0_o\eta_2P_s)$.
For contradiction, suppose there is cancellation with the initial $0_o$ in a reduced code for $g_n(0_o\eta_2P_s)$. Consider the subsegment of $\gamma$ given by $0_o\delta$, where $0_o\delta$ is the minimal subsegment such that $g_n(0_o\delta)$ has cancellation with $0_o$. 

\begin{lem}\label{lem:nobackloops}
The segment $\delta$ does not contain any back loops.
\end{lem}

\begin{proof}
Suppose towards a contradiction that $\delta$ contains at least one back loop and write $\delta$ in standard form as
$$\delta=\tau_1 C_1^-C_1C_1^+\tau_2\dots \tau_sC_s^-C_sC_s^+\tau_{s+1},$$
where the $\tau_i$ are possibly empty for $i\leq s+1$. 
Then
$$g_n(\delta)=g_n(\tau_1C_1^-) C_1g_n(C_1^+\tau_2C_2^-)C_2\dots g_n(C_{s-1}^+\tau_sC_s^-)C_sg_n(C_s^+\tau_{s+1}).$$

In order to have cascading cancellation involving $0_o$, there must be an instance of $0_o$ in $g_n(C_s^+\tau_{s+1})$ and the initial subsegment of $g_n(\delta)$ which ends immediately before that $0_o$ must have trivial image.  By the minimality of our choice of $\delta$, this initial subsegment must contain $g_n(C_i)=C_i$ for all $i=1,\dots, s$.  Since its image is trivial, there must be cancellation involving each $C_i$.  The only way that this can occur is if there is some $j$ such that $C_j=\overline{C_{j+1}}$ and $C_jg_n(C_j^+\tau_jC_{j+1}^-)\overline{C_j}=\emptyset$. Consequently, we must have $g_n(C_{j}^+\tau_jC_{j+1}^-)=\emptyset$.  
 This implies that exactly one of the following holds:
\begin{enumerate}
\item $h_1(C_{j}^+\tau_jC_{j+1}^-)=\emptyset$,
\item $h_1(C_{j}^+\tau_jC_{j+1}^-)\neq\emptyset$ in a reduced code while $h^{(n)}_2(h_1(C_{j}^+\tau_kC_{j+1}^-))=\emptyset$,
\item $h^{(n)}_2(h_1(C_{j}^+\tau_jC_{j+1}^-))\neq\emptyset$ in a reduced code while $h^{(n)}_3(h^{(n)}_2(h_1(C_{j}^+\tau_jC_{j+1}^-)))=\emptyset$.
\end{enumerate}
Let $i$ correspond to which of the above cases we are in and let $h=h^{(n)}_i$.  Define
$$\tau'_j=\begin{cases}
C_{j}^+\tau_jC_{j+1}^-,&h=h_1,\\
h_1(C_{j}^+\tau_jC_{j+1}^-),&h=h^{(n)}_2,\\
h^{(n)}_2(h_1(C_{j}^+\tau_jC_{j+1}^-)),&h=h^{(n)}_3.\end{cases}$$
By assumption, $\tau'_j\neq\emptyset$ and $h(\tau_j')=\emptyset$.

By the Loop Theorem (Theorem \ref{prop:nontrivialloop}), this implies that $\tau_j'$ is a loop which begins and ends in the turbulent region and does not fully cross $D$ in $(n_1,n_2)$.  Moreover, in $\delta$, the path $\tau_j'$ is preceded by $C_j$ and followed by $C_{j+1}=\overline{C_j}$ and so one of the back loop connectors $C_j^{+}$ or $C_{j+1}^-$ must fully cross $D$.  However this is a contradiction, since these back loop connectors must occur in $(n_1,n_2)$.  See Figure \ref{fig:NoBackLoops}. 
\begin{figure}
\centering
\begin{overpic}[width=4in, trim={.6in 8.95in .75in .25in}, clip]{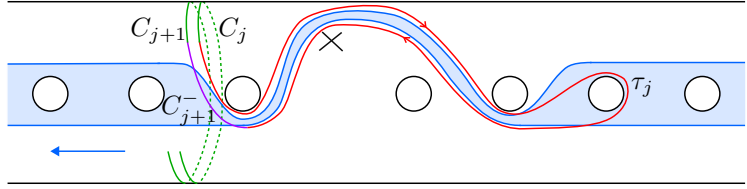}
\put(83,13){$\tau_j$}
\put(29,20){$C_j$}
\put(17.5,20){$C_{j+1}$}
\put(21.5,9.75){$C^-_{j+1}$}
\end{overpic}
\caption{Situation in the proof of Lemma \ref{lem:nobackloops}.}
\label{fig:NoBackLoops}
\end{figure}
\end{proof}

Notice that $0_o$ does not appear in the image under $g_n$ of any segment supported on $(-\infty,-n-1)\cup[n+1,\infty)$.  Therefore, if $\delta$ contains any subsegment supported in this region then, by minimality,  it must also contain a subsegment of the form $q_{o/u}q_{u/o}$ for $q\ge n+1$ or $q<-n-1$; in other words, the segment $\delta$ must ``turn around" in the shift region.  Moreover, by Lemma \ref{lem:nobackloops}, $\delta$ can have no back loops, and hence Lemma \ref{lem:imageof2o2ufamily} implies that such a $q_{o/u}q_{u/o}$ blocks cascading cancellation, so this cannot occur.
Thus, we conclude that a reduced code for $0_o\delta $ is supported on $[-n-1,n+1)$. 

We will now show that that there is no cancellation involving $0_o$ under each of $h_1,h_2^{(n)},$ and $h_3^{(n)}$.  By the assumptions of the theorem, the $0_o$ in $0_o\delta$ is oriented to the left.  If there is no cancellation involving $0_o$ under a shift, then $0_o$ persists in the image and is still oriented to the left.  This implies that the character after $0_o$ in the image still satisfies condition (1) of the theorem, that is, it is either $0_u$ or $P_{o/u}$.

A straightforward calculation  shows that $h_1(0_o\delta)$ cannot cancel $0_o$ for any $\delta $. This follows from the fact that $\partial D_1$ contains $0_u$, not $0_o$, where $D_1$ is the domain for $h_1$. 
Write $0_o\delta'$ for the reduced code of  $h_1(0_o\delta)$, which is another segment completely contained in the region $(-n-1,n+1]$ with no back loops. The first character of $\delta'$ is still either $0_u$ or $P_{o/u}$.

It remains to rule out any cancellation involving $0_o$ under $h_2^{(n)}$ and $h_3^{(n)}$.  We do this in the following two lemmas.  

\begin{lem}\label{lem:h_2}
There is no cancellation involving $0_o$ under $h_2^{(n)}$.
\end{lem}

\begin{proof}
For contradiction, assume that there is cancellation involving $0_o$, and let $\zeta$ be the minimal subsegment of $0_o\delta'$ which has cancellation with $0_o$ under $h_2^{(n)}$.

If $\zeta^t =0_o$, then $\zeta$ is a loop whose image is trivial.  By the Loop Theorem (see Theorem \ref{prop:nontrivialloop}), $\zeta$ must be of the form  $\partial D_2|_{[0,n_1)}(n_1)_{o/u}(n_1)_{u/o}\overline{\partial D_2|_{[0,n_1)}}$, where $D_2$ is the domain of $h^{(n)}_2$.  
In particular,  $\zeta^i=0_o$ and this $0_o$ is oriented to the right.  However, as noted above, the initial $0_o$ of $0_o\delta$ is oriented to the left, which contradicts that $\zeta$ is an initial subsegment of $0_o\delta'$.

Therefore, $\zeta=0_o\zeta'a_1a_2$, where the  $0_o$ which cancels with $\zeta^i=0_o$ appears in $h_2^{(n)}(a_1a_2)$ and is not the terminal character.  This is the case exactly when $a_1a_2$ is either a full or half crossing.  Moreover, in order for $0_o$ to appear in the image of $a_1a_2$, we must have that the numerical values of $a_1$ and $a_2$ are greater than or equal to zero, since $h_2^{(n)}$ shifts to the left.

There are two cases to consider: either $(\zeta')^i=0_u$ or $(\zeta')^i=P_{o/u}$. 
If $(\zeta')^i=0_u$, then $\zeta'$ contains a character which disagrees with $\partial D_2$, namely $0_u$.
On the other hand, if $(\zeta')^i=P_{o/u}$, then it must be oriented to the left. Since the numerical values of $a_1$ and $a_2$ are at least $0$ and $P<0$, $\zeta'$  must ``turn around" and  contain an over-under loop in turbulent region.  In particular, it must again contain a character which disagrees with $\partial D_2$. In either case,  call the character which disagrees with the boundary of the domain $b$. 

By Lemma \ref{lem:bpersists}, $b$ persists in a reduced code for $h_2^{(n)}(\zeta')$ and such a $b$ must block any cascading cancellation.  However,  by assumption,  there is cancellation between two instances of  $0_o$ in an \textit{unreduced code} for $h_2^{(n)}(\zeta')$, one which precedes $b$ and one which succeeds $b$.
This is a contradiction.
\end{proof}

It remains only to check for cancellation under $h_3^{(n)}$.
\begin{lem}\label{lem:h_3}
There is no cancellation involving $0_o$ under $h_3^{(n)}$.
\end{lem}

\begin{proof}
The argument is similar to the proof of Lemma \ref{lem:h_2}.
For contradiction assume that there is cancellation involving $0_o$ and write $\beta=0_o\delta''$ for the minimal subsegment of a reduced code for $h_2^{(n)}(0_o\delta')$ which has cancellation with $0_o$ under $h_3^{(n)}$.  Note that $(\delta'')^i$ is still either $0_u$ or $P_{o/u}$.   By the same reasoning as in the proof of Lemma \ref{lem:h_2}, $\beta=0_o\beta'c_1c_2$, and $c_1c_2$ is a full or half crossing.
Recall that for the shift $h_3^{(n)}$, $n_1=-n-1$, $n_2=n+1$, and $h_3^{(n)}$ shifts to the right.

If $\beta'c_1$ contains a character which does not agree with $\partial D_3$, then the contradiction follows by applying Lemma~\ref{lem:bpersists} as in the proof of Lemma \ref{lem:h_2}.  However, it is possible that $\beta'c_1$ does not contain such a character.  This occurs exactly when $\beta=\overline{\partial D_3|_{(n_1,0]}}(n_1)_{o/u}$.   This possibility did not arise when considering $h_2$ because $h_2$ is a left shift, and so $n_1>0$.

In this case, $\beta$ must begin with $0_oP_u(-1)_o$.  We will now show that $\beta$ could not be the image of $0_o\delta$ under $h_2^{(n)}\circ h_1$.  In the notation above, $(-1)_o$ is in the image of $\xi=0_o\delta'$ under $h_2^{(n)}$.  Since $(-1)_o$ disagrees with $\partial D_2$, there must be an instance of $(-1)_o$ in $\delta'$ by Claim \ref{claim:bfromb}.   Let $\xi'=0_o\sigma(-1)_o$ be the subsegment of $\xi$ so that $h_2^{(n)}(\xi')=0_oP_u(-1)_o$.  We can find $\xi$ directly by computing $(h_2^{(n)})^{-1}(0_oP_u(-1)_o)$.  Recall that we cannot always take the preimage of a segment under a permissible shift (see the discussion in Section \ref{sec:inverses}).  However, in this case we are able to take the preimage precisely because Claim \ref{claim:bfromb} ensures that $0_oP_u(-1)_o$ is  the image of a subsegment of $h_2^{(n)}(\xi)$. 
Computing $(h_2^{(n)})^{-1}(0_oP_u(-1)_o)$ results in
\[
\xi'=0_oP_u\partial D_2|_{[P,n+1)}(n+1)_u(n+1)_o\overline{\partial D_2|_{[P,n+1)}}(-1)_o.
\]
 In particular, $\xi'$ starts with $0_oP_uP_o0_o1_u$. 
To conclude the proof, we will show that $\xi'$ cannot occur as the image under $h_1$ of the segment $0_o\delta$.

Assume for contradiction that $h_1(0_o\delta) = \xi'$. Since $P_o$ does not agree with $\partial D_1$, there must be an instance of $P_o$ in $\delta$ by Claim \ref{claim:bfromb}.  Let $0_o\tau P_o$ be the initial subsegment of $0_o\delta$ so that $h_1(0_o\tau P_o)=0_oP_uP_o$.  As above, we may directly compute the preimage of $0_oP_uP_o$ under $h_1$.  This shows that $0_o\tau P_o=0_oP_uP_o$.  The assumptions of the theorem then imply that $0_o\delta$ starts with $0_oP_uP_o0_o1_o$, where $1_o$ is oriented to the right.  Since $\xi'$ starts with $0_oP_uP_o1_u$, the only way $h_1(0_o\delta)=\xi'$ is if there is cancellation with the terminal $1_o$ in  $0_oP_uP_o0_o1_o$ when we apply $h_1$ to $0_o\delta$. Since $h_1$ shifts to the right and $\xi'$ contains $(-1)_o$, it must be the case that $0_o\delta$  contains a loop $k_{o/u}k_{u/o}$ with $k\geq 1$.  In particular, $0_o\delta$ starts with either $0_oP_uP_o0_o1_o1_u$ (if $k=1)$ or $0_oP_uP_o0_o1_o\eta k_{o/u}k_{u/o}$ where $\eta$ is strictly monotone increasing (if $k>1$).  Since $n_2=1$ for $h_1$,  Lemma~\ref{lem:imageof2o2ufamily} then implies that this $k_{o/u}k_{u/o}$ blocks cancellation  between the terminal $1_o$ in $0_oP_uP_o0_o1_o$ and any characters in $0_o\delta$ that appear \emph{after} $k_ok_u$.  Since the segment $\eta$ of $0_o\delta$ between $1_o$ and $k_ok_u$ is either empty (if $k=1$) or a strictly monotone increasing segment (if $k>1$), using the fact that $h_1$ shifts to the right, it is straightforward to check that there can be no cancellation with the terminal $1_o$ in $0_oP_uP_o1_o$ when we apply $h_1$, which is a contradiction.  This completes the proof.
\end{proof}

Lemmas~\ref{lem:nobackloops}, \ref{lem:h_2}, and \ref{lem:h_3} show that $g_n(0_o\delta)$ has no cancellation with $0_o$, which completes the proof of Theorem \ref{lem:induct}.

As promised at the beginning of this section, we  have the following corollary.

\begin{cor}\label{cor:inductcor}
For any fixed $n$ and any $i\ge 0$, $\alpha^{(n)}_{i+1}$ starts like $\alpha^{(n)}_{i}$.  More generally, for any $i\geq 1$, if $\gamma$ starts like $\alpha^{(n)}_i$ then $g_n(\gamma)$ starts like $\alpha^{(n)}_{i+1}$.
\end{cor}

\begin{proof}
First, a direct computation shows that for any fixed $n$, $\alpha^{(n)}_1$ starts like $\alpha^{(n)}_0$.
For all $i\geq 1$, we will prove both statements simultaneously, using strong induction.  We will show that for each $n$ and any $i\geq 1$, if $\gamma$ is a simple arc that starts like $\alpha_i$, then $g_n(\gamma)=\mathring\alpha^{(n)}_{i+1}\eta'$ for some reduced $\eta'$ which has the form of $\eta_2$ in the hypotheses of Theorem \ref{lem:induct}.

For the base case $i=1$, suppose $\gamma$ starts like $\alpha^{(n)}_1$ so that we may write
\[
\gamma=\mathring\alpha^{(n)}_1\eta,
\]
 in a reduced code.  Applying $g_n$ to both sides of this equation, we have, in an \textit{unreduced} code, 
 \[g_n(\gamma)=g_n(\mathring\alpha^{(n)}_{1})\eta',
 \]
for some reduced $\eta'$.  By Proposition \ref{lem:imageofinit}, $g_n(\mathring\alpha_{1}^{(n)})=\mathring\alpha_{2}^{(n)}$ in a reduced code.  Thus to conclude that $g_n(\gamma)$ starts like $\alpha_2^{(n)}$, we need to show that $g_n(\mathring\alpha_{1}^{(n)})=\mathring\alpha_{2}^{(n)}$ persists in a reduced code for $g_n(\gamma)$.  Since $\mathring\alpha^{(n)}_1$ ends with $0_o$, it suffices to show that $\eta$ has the form of $\eta_2$ in the hypotheses of Theorem \ref{lem:induct} so that there is no cancellation with this $0_o$.  We have $\mathring\alpha_1^{(n)}=P_s0_o1_o2_o\dots(n+1)_o(n+1)_u n_o(n-1)_o\dots 2_o1_o0_o$, and so the terminal $0_o$ is oriented to the left.  It follows that $\eta$ satisfies Theorem \ref{lem:induct}(1). Moreover, since $\gamma$ is simple and starts like $\alpha_1^{(n)}$, if $\eta$ begins with a loop around $p$, it must begin with $P_uP_o0_o1_o$ (see Figure \ref{fig:etaPoPu}).  Thus $\eta$ satisfies Theorem \ref{lem:induct}(2) and so Theorem  \ref{lem:induct} shows that there is no cancellation in $g_n(\gamma)$ between the final $0_o$ in $g_n(\mathring\alpha_{1}^{(n)})=\mathring\alpha_{2}^{(n)}$ and $\eta'$.  Therefore, $g_n(\gamma)$ starts like $\alpha_2^{(n)}$.  In particular, a direct computation shows that $\alpha_2^{(n)}$ starts like $\alpha_1^{(n)}$ and so by setting $\gamma=\alpha_2^{(n)}$, we have shown that $g_n(\alpha_2^{(n)})=\alpha_3^{(n)}$ starts like $\alpha_2^{(n)}$.

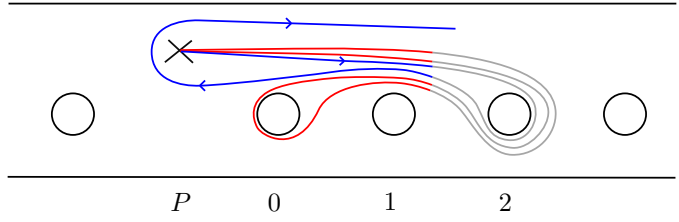
\begin{figure}
\centering
\def\svgwidth{3.5in}
\begingroup%
  \makeatletter%
  \providecommand\color[2][]{%
    \errmessage{(Inkscape) Color is used for the text in Inkscape, but the package 'color.sty' is not loaded}%
    \renewcommand\color[2][]{}%
  }%
  \providecommand\transparent[1]{%
    \errmessage{(Inkscape) Transparency is used (non-zero) for the text in Inkscape, but the package 'transparent.sty' is not loaded}%
    \renewcommand\transparent[1]{}%
  }%
  \providecommand\rotatebox[2]{#2}%
  \newcommand*\fsize{\dimexpr\f@size pt\relax}%
  \newcommand*\lineheight[1]{\fontsize{\fsize}{#1\fsize}\selectfont}%
  \ifx\svgwidth\undefined%
    \setlength{\unitlength}{375.02143213bp}%
    \ifx\svgscale\undefined%
      \relax%
    \else%
      \setlength{\unitlength}{\unitlength * \real{\svgscale}}%
    \fi%
  \else%
    \setlength{\unitlength}{\svgwidth}%
  \fi%
  \global\let\svgwidth\undefined%
  \global\let\svgscale\undefined%
  \makeatother%
  \begin{picture}(1,0.31720825)%
    \lineheight{1}%
    \setlength\tabcolsep{0pt}%
    \put(0,0){\includegraphics[width=\unitlength,page=1]{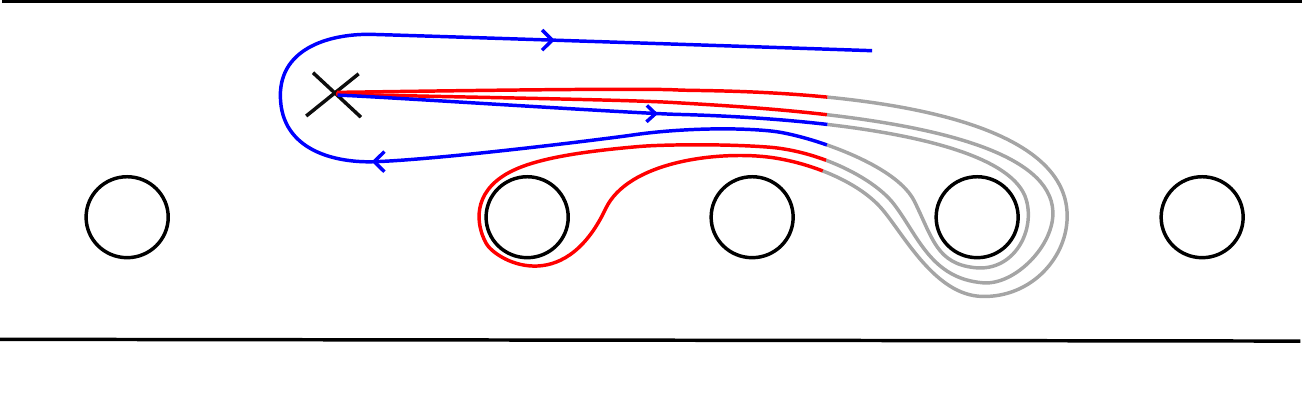}}%
    \put(0.24330655,0.00579545){\makebox(0,0)[lt]{\lineheight{1.25}\smash{\begin{tabular}[t]{l}$P$\end{tabular}}}}%
    \put(0.38931277,0.0056403){\makebox(0,0)[lt]{\lineheight{1.25}\smash{\begin{tabular}[t]{l}$0$\end{tabular}}}}%
    \put(0.56248891,0.00698276){\makebox(0,0)[lt]{\lineheight{1.25}\smash{\begin{tabular}[t]{l}$1$\end{tabular}}}}%
    \put(0.73432275,0.0056403){\makebox(0,0)[lt]{\lineheight{1.25}\smash{\begin{tabular}[t]{l}$2$\end{tabular}}}}%
  \end{picture}%
\endgroup%

\caption{The arc $\alpha_1^{(n)}$ is in red; the grey portions may vary depending on $n$.  One possibility for $\gamma$ and the subsegment $\eta$ is shown in blue; the grey portions may vary depending on $n$.}
\label{fig:etaPoPu}
\end{figure}

To finish the base case, we will now show that $\eta'$ has the form of $\eta_2$ in the hypotheses of Theorem \ref{lem:induct}.  
A direct computation shows that $\mathring \alpha_2^{(n)}$ ends with $(n+1)_o(n+1)_u n_o(n-1)_o\dots 2_o1_o0_o$, and so that $\eta'$ satisfies Theorem \ref{lem:induct}(1).  As above, since $g_n(\gamma)=\mathring\alpha_{2}^{(n)}\eta'$ is simple and $\alpha_2^{(n)}$ (and therefore $\mathring\alpha_2^{(n)}$) starts like $\alpha_1^{(n)}$, it follows that if $\eta'$ starts with a loop around $p$, it must begin with $P_uP_o0_o1_o$, so $\eta'$ satisfies Theorem \ref{lem:induct}(2).  This completes the base case.

Now assume the results hold for all $j<i$.  
For the induction step, suppose that $\gamma=\mathring\alpha_i^{(n)}\eta$ for some $\eta$ which has the form of $\eta_2$ in the hypotheses of Theorem \ref{lem:induct}.  The argument for this case goes through exactly as in the base case except we need one additional argument to show that $\eta'$ has the form of $\eta_2$ in the hypotheses of Theorem \ref{lem:induct}, where  $g_n(\gamma)=\mathring\alpha_{i+1}^{(n)}\eta'$ in a reduced code.  By the proof of Proposition \ref{lem:imageofinit}, $\mathring\alpha_{i+1}^{(n)}$ ends with $\mu=(n+1)_o(n+1)_u n_o(n-1)_o\dots 2_o1_o0_o$ and therefore $\eta'$ satisfies Theorem \ref{lem:induct}(1).  To see that $\eta'$ satisfies Theorem \ref{lem:induct}(2), notice that $\alpha_{i+1}^{(n)}$ starts like $\alpha_i^{(n)}$, which starts like $\alpha_{i-1}^{(n)}$, etc., so that $\alpha_{i+1}^{(n)}$ (and therefore $\mathring\alpha_{i+1}^{(n)}$) starts like $\alpha_1^{(n)}$.  In particular, $\gamma$ must start like $\alpha_1^{(n)}$.  Since $\gamma$ is simple, if $\eta'$ starts with a loop around $p$, it must begin with $P_uP_o0_o1_o$ and Theorem \ref{lem:induct}(2) is satisfied.  This completes the induction step, and the result is proved.
\end{proof}


\section{Highways in arcs}\label{sec:highways}

In this section, we introduce and examine the prevalence of certain segments of the code for $\alpha_i^{(n)}$ that we call \textit{highways}.  The presence of highways forces arcs disjoint from $\alpha_i^{(n)}$ to have very specific initial and terminal subsegments.
This will be instrumental in proving Theorem \ref{thm:startslike}, which shows that if $\delta$ is any arc which starts like $\alpha_i^{(n)}$ and $\gamma$ is an arc disjoint from $\delta$, then $\gamma$ starts like $\alpha_{i-1}^{(n)}$, provided $i$ is large enough.
In Section \ref{sec:loxo}, we will use Theorem \ref{thm:startslike}  to show that the arcs $\alpha_i^{(n)}$ lie on a quasi-geodesic in the modified arc graph.

In Section \ref{sec:highwayprelim}, we will give general preliminary definitions and results for general arcs, and then in Section \ref{sec:highwayalpha} we will analyze highways in the arcs $\alpha_i^{(n)}$.

\subsection{Preliminaries on highways}\label{sec:highwayprelim}
Recall our convention that all arcs are assumed to be simple and start and end at $p$. 

\begin{defn}
Given an arc $\delta=P_sq_1\eta q_2P_s$, where $q_1, q_2$ are single characters which are not $C$ and $\eta$ is a segment, we say that $\delta$ \textit{has highways} if either $q_1P_{o/u}P_{u/o}q_1$ or $q_2P_{o/u}P_{u/o}q_2$ is a subsegment of $\delta$.
\end{defn}

\noindent The following lemma is an almost immediate corollary of the definition and the fact that $\delta$ is simple. 

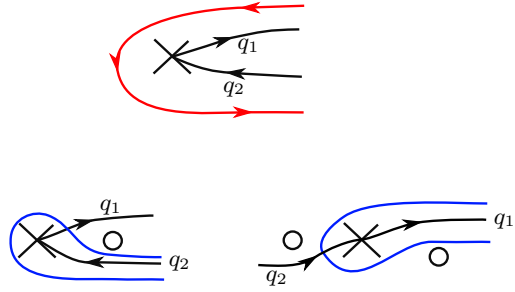
\begin{figure}
\centering
\def\svgwidth{3in}
\begingroup%
  \makeatletter%
  \providecommand\color[2][]{%
    \errmessage{(Inkscape) Color is used for the text in Inkscape, but the package 'color.sty' is not loaded}%
    \renewcommand\color[2][]{}%
  }%
  \providecommand\transparent[1]{%
    \errmessage{(Inkscape) Transparency is used (non-zero) for the text in Inkscape, but the package 'transparent.sty' is not loaded}%
    \renewcommand\transparent[1]{}%
  }%
  \providecommand\rotatebox[2]{#2}%
  \newcommand*\fsize{\dimexpr\f@size pt\relax}%
  \newcommand*\lineheight[1]{\fontsize{\fsize}{#1\fsize}\selectfont}%
  \ifx\svgwidth\undefined%
    \setlength{\unitlength}{242.78217962bp}%
    \ifx\svgscale\undefined%
      \relax%
    \else%
      \setlength{\unitlength}{\unitlength * \real{\svgscale}}%
    \fi%
  \else%
    \setlength{\unitlength}{\svgwidth}%
  \fi%
  \global\let\svgwidth\undefined%
  \global\let\svgscale\undefined%
  \makeatother%
  \begin{picture}(1,0.49161726)%
    \lineheight{1}%
    \setlength\tabcolsep{0pt}%
    \put(0,0){\includegraphics[width=\unitlength,page=1]{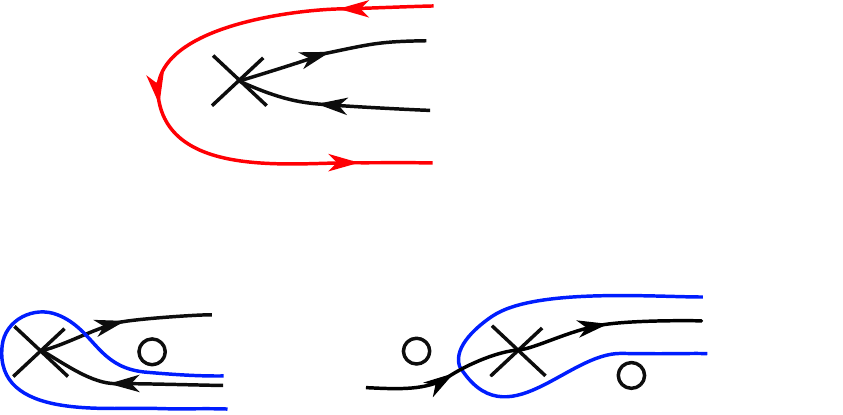}}%
    \put(0.15789078,0.12732254){\makebox(0,0)[lt]{\lineheight{1.25}\smash{\begin{tabular}[t]{l}\small{$q_1$}\end{tabular}}}}%
    \put(0.27791638,0.02967462){\makebox(0,0)[lt]{\lineheight{1.25}\smash{\begin{tabular}[t]{l}\small{$q_2$}\end{tabular}}}}%
    \put(0.3981901,0.40840824){\makebox(0,0)[lt]{\lineheight{1.25}\smash{\begin{tabular}[t]{l}\small{$q_1$}\end{tabular}}}}%
    \put(0.37295909,0.33457602){\makebox(0,0)[lt]{\lineheight{1.25}\smash{\begin{tabular}[t]{l}\small{$q_2$}\end{tabular}}}}%
    \put(0.44744064,0.00414305){\makebox(0,0)[lt]{\lineheight{1.25}\smash{\begin{tabular}[t]{l}\small{$q_2$}\end{tabular}}}}%
    \put(0.84521493,0.1038105){\makebox(0,0)[lt]{\lineheight{1.25}\smash{\begin{tabular}[t]{l}\small{$q_1$}\end{tabular}}}}%
  \end{picture}%
\endgroup%

\caption{At the top is an arc where $q_1 = q_2$ and the red segment shows that this arc has highways. Below are two arcs where $q_1 \neq q_2$ and the blue segments must appear if the arc has highways. However, this arrangement contradicts the simplicity of each arc (see Lemma \ref{lem:highwayssymm}) so neither arc has highways.}
\label{fig:highways}
\end{figure}

\begin{lem}\label{lem:highwayssymm}
If $\delta$ is an arc that has highways, then
$$\delta=P_s\tau\dots\overline{\tau} P_s,$$
for some segment $\tau$ with $\ell_c(\tau)>0$.
That is to say, the first part of the code for $\delta$ always overlaps with the reverse of the last part of the code in at least two characters, one of which is $P_s$.
\end{lem}

\begin{proof} It suffices to consider the case when $\tau$ is a single character, so that $\tau=a=\overline{\tau}$ for some $a$.  Since $\delta=P_sq_1\eta q_2P_s$, if the conclusion does not hold, then $q_1\neq q_2$.  In this case, either the subsegment $P_sq_1$ intersects $q_2P_oP_uq_2$ or the subsegment $q_2P_s$ intersects $q_1P_oP_uq_1$, contradicting the fact that $\delta$ is simple. See Figure \ref{fig:highways}.
\end{proof}

Recall that in the code for an arc, the character $P_s$ does not correspond to any subsegment.  Since the first and last characters of every arc are always $P_s$,   we use the first \textit{two} characters of $\delta$ in the statement of the above lemma to ensure that there is an initial subsegment of $\delta$ which fellow travels a terminal subsegment.

In the future, we will need to use a refined notion of highways to constrain the beginnings of certain arcs.
For this we define a notion of right lane and left lane.  See Figure \ref{fig:lanes} for examples and non-examples.
\begin{defn}\label{defn:leftlane}
Given an arc $\delta$ that has highways, a subsegment $\lambda$ of $\delta$ is called a \textit{left lane} if one of the following holds:
\begin{enumerate}
\item $\lambda=P_oP_u\gamma$, where $\gamma$ does not contain $C$ and $\gamma$ is maximal with respect to the property that the code for $P_s\gamma$ coincides with the initial $\ell_c(P_s\gamma)$ many characters in $\delta$ and $\overline{\gamma}P_s$ coincides with the terminal $\ell_c(P_s\gamma)$ many characters in $\delta$; or
\item $\lambda=\overline\gamma P_uP_o$ is the reverse of the segment of the form in (1).
\end{enumerate}
\end{defn}

We note that the reason that there are two possibilities for a left lane is that we want the definition to be independent of the orientation of $\delta$.

\begin{defn}\label{defn:rightlane}
Given an arc $\delta$ that has highways, a segment $\rho$ of the code is called a \textit{right lane} if one of the following holds:
\begin{enumerate}
\item $\rho=P_uP_o\gamma$, where $\gamma$ does not contain $C$ and $\gamma$ is maximal with respect to the property that the code $P_s\gamma$ coincides with the initial $\ell_c(P_s\gamma)$ many characters in $\delta$ and $\overline{\gamma}P_s$ coincides with the terminal $\ell_c(P_s\gamma)$ many characters in $\delta$; or
\item $\rho=\overline\gamma P_oP_u$ is the reverse of the segment from (1).
\end{enumerate}
\end{defn}

\begin{figure}
\begin{subfigure}[t]{.4\linewidth}
\centering
\includegraphics[width=2in]{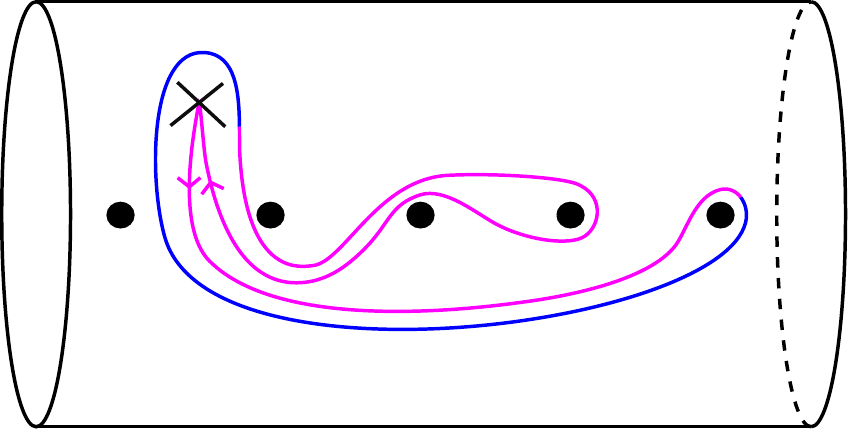}
\caption{The blue segment is \textit{not} a left lane since it does not coincide with the terminal subsegment of the arc. }
\end{subfigure}\hfill
\begin{subfigure}[t]{.4\linewidth}
\centering
\includegraphics[width=2in]{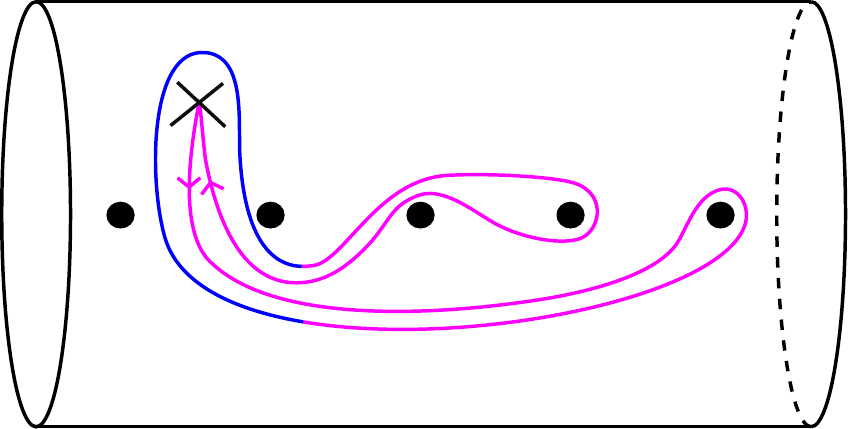}
\caption{The blue segment contains both left and right lanes. There is a left lane via Case (2) of Definition~\ref{defn:leftlane} and a right lane via Case (1) of Definition~\ref{defn:rightlane}.}
\end{subfigure}

\begin{subfigure}[t]{.4\linewidth}
\centering
\includegraphics[width=2in]{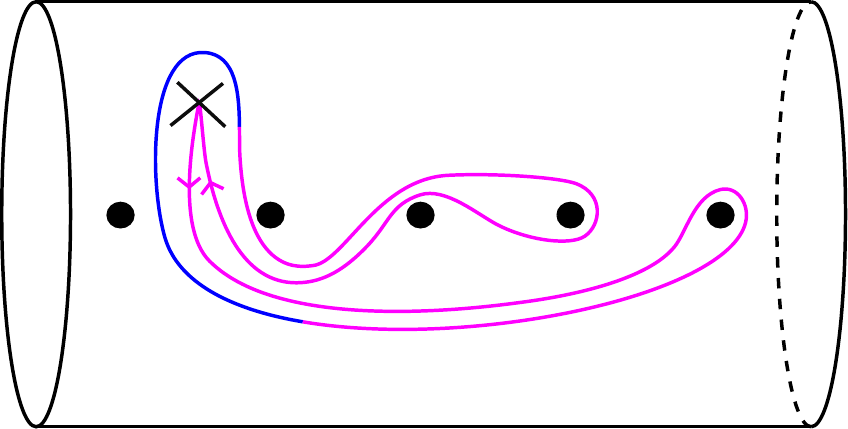}
\caption{The blue segment  contains two left lanes via Case (2) of Definition~\ref{defn:leftlane}.}
\end{subfigure}\hfill
\begin{subfigure}[t]{.4\linewidth}
\centering
\includegraphics[width=2in]{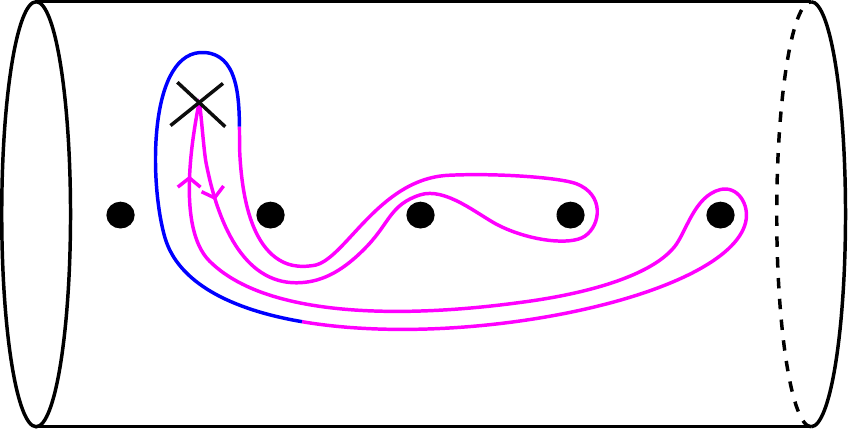}
\caption{The blue segment  contains two left lanes via Case (1) of Definition~\ref{defn:leftlane} (notice the orientation on the arc has been changed).}
\end{subfigure}
\caption{Some examples and non-examples of left and right lanes  are shown in blue. }
\label{fig:lanes}
\end{figure}

%

\noindent If $\delta$ is a symmetric arc (see Definition \ref{defn:symmetric}), then  in Definitions \ref{defn:leftlane} and \ref{defn:rightlane} it suffices to check the overlap on just the initial part of the code for $\delta$. However, in the general case, checking the overlap with both the initial and terminal parts of the code for $\delta$ is necessary.

\begin{defn}
Let $\delta$ be an arc with highways.  
The \textit{lane length} $L(\xi)$ of a left or right lane $\xi$ of $\delta$ is defined to be 
$$L(\xi)=\ell_c(\xi)-1.$$
We denote the collection of all left lanes of $\delta$ by $\mathcal{L}$ and similarly of all right lanes of $\delta$ by $\mathcal{R}$.
\end{defn}

Through the rest of the section, we fix a closed topological disc $D_p$  of sufficiently small radius with center at the puncture $p$ such that $D_p$ has empty intersection with each $B_i$ where $i\in \Z$, contains $B_P$, and has empty intersection with each separating curve $\{S_i\}$ from Section \ref{sec:segnobackloop}. Moreover, we will only work with homotopy representatives of $\delta$ which have reduced code and the property that any pair $P_{o/u}P_{u/o}$ lies inside of $D_p$, while any other segments remain outside, except for the two that come from the initial and terminal two characters of $\delta$ (see Figure \ref{fig:highwaydisk}).
Throughout the section, when we further homotope $\delta$ we will only do so relative to $D_p$, so one can assume that the set $\delta\cap D_p$ is pointwise fixed.

A left or right lane $\xi$ is called \textit{innermost} if every oriented straight line segment with initial point at the puncture and terminal point on the boundary circle of $D_p$ intersects the $P_{o/u}P_{u/o}$ at the beginning (resp. end) of $\xi$ before it intersects any other lane of the same type (left or right).  If the oriented line segment does not intersect any lane, then this condition is vacuously satisfied. See Figure \ref{fig:innerlane}.
In particular, innermost left and right lanes are the lanes which are closest to an initial and terminal subsegments of $\delta$.

\begin{figure}
\centering
\begin{overpic}[width=4in, trim={1.5in 7.5in .9in 1.6in}, clip]{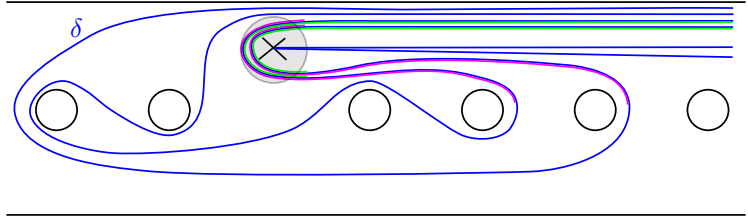}
\put(10,25){\textcolor{blue}{$\delta$}}
\end{overpic}
\caption{In blue is (a portion) of an arc $\delta$.  The shaded disk around the puncture is $D_p$.  This example has two right lanes (in green), a portion of which are shown, and two left lanes (in pink), which are shown in their entirety.  The $P_{o/u}P_{u/o}$ portion of these lanes are contained in $D_p$ while all other strands of $\delta$ are disjoint from $D_p$.}
\label{fig:highwaydisk}
\end{figure}

\begin{figure}
\centering
\def\svgwidth{2in}
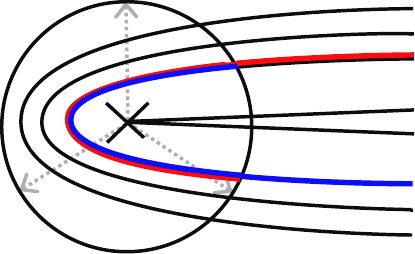
\caption{The blue and red segments are the initial portions of innermost left and right lanes, respectively.  The dotted gray segments intersect these segments before intersecting the initial or terminal portion of any other lane.}
\label{fig:innerlane}
\end{figure}

We then have the following lemma.

\begin{lem}\label{lem:stayinyourlane}
Let $\delta$ be an arc that has highways.
Let $\lambda\in \mathcal{L}$ and $\rho\in \mathcal{R}$ denote the innermost left and right lanes of $\delta$, respectively.
Then $L(\lambda)\ge L(\lambda')$ for all $\lambda'\in\mathcal{L}$ and  $L(\rho)\ge L(\rho')$ for all $\rho'\in\mathcal{R}$.
Moreover, writing $\lambda=P_uP_o\beta_l$ or its reverse and $\rho=P_oP_u\beta_r$ or its reverse, then if $\gamma$ is an arc disjoint from $\delta$, then one of the following holds:
\begin{enumerate}
\item $\gamma$ has initial code $P_s\beta_l$ and terminal code $\overline{P_s\beta_l}$,
\item $\gamma$ has initial code $P_s\beta_r$ and terminal code $\overline{P_s\beta_r}$.
\end{enumerate}
\end{lem}

Intuitively, this lemma states that the arc $\gamma$ must ``stay in a lane" of the arc $\delta$; see Figure \ref{fig:stayinyourlane}.  Note that we do not require $\gamma$ to be distinct from $\delta$, so this statement applies to $\delta$ as well. Indeed, the arc $\delta$ can always be homotoped to be disjoint from itself while retaining the same code.

\begin{figure}
\centering
\includegraphics[width=2.5in, trim={3in 7.75in 3in 2.0in}, clip]{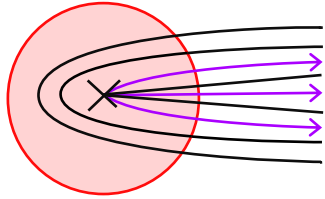}
\caption{In black is (a portion) of an arc $\delta$. The red disk is $D_p$. Any arc $\gamma$ which is disjoint from $\delta$ must have initial and terminal segment which follows one of the three purple strands emanating from the puncture and continues to fellow travel $\delta$.}
\label{fig:stayinyourlane}
\end{figure}

\begin{proof}
The hypothesis that $\delta$ has highways implies that $\mathcal{L},\mathcal{R}\neq\emptyset$, and hence an innermost left and right lane exist.
We prove the first statement for the innermost left lane $\lambda$; an identical proof works for $\rho$.

Since $\lambda$ is innermost, if there exists a left lane $\lambda'\in\mathcal{L}$ such that $L(\lambda')>L(\lambda)$, then  we get a contradiction to the simplicity of $\delta$, as $\lambda$ must intersect either $\lambda'$ or an initial segment of $\delta$.  See Figure \ref{fig:innerlanelongest}.

\begin{figure}[H]
\centering
\begin{overpic}[width=3in, trim={3.25in 8.25in 1in 1.5in}, clip]{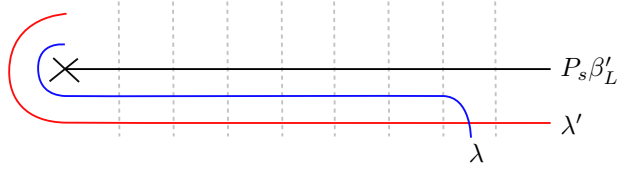}
\put(97,4){$\lambda'$}
\put(97,14){$P_s\beta_L'$}
\put(81,-1){$\lambda$}
\end{overpic}
\caption{A schematic of the two left lanes: $\lambda$, which is innermost, and $\lambda'$, which is not.  If $\lambda'$ fellow travels an initial segment of $\delta$ for longer than $\lambda$ does, then $\lambda$ must intersect either $\lambda'$ (pictured) or the initial segment of $\delta$.}
\label{fig:innerlanelongest}
\end{figure}

For the final statement, fix an arc $\gamma$ that is disjoint from $\delta$.
It must be the case that  the initial and terminal subsegments of $\gamma$ are each contained in a single connected component of $D_p\setminus {\delta}$ (not necessarily the same component).
Moreover these initial and terminal subsegments must begin at the puncture $p$ and therefore must also fellow travel $\lambda$, $\rho$, and the initial/terminal parts of $\delta$ (see Figure \ref{fig:stayinyourlane}).  
Therefore, by the same reasoning as in the previous paragraph, we see that $\gamma$ satisfies conclusion (1) of the lemma if $\ell_c(\beta_l)\le\ell_c(\beta_r)$ and satisfies conclusion (2) of the lemma if $\ell_c(\beta_r)\le \ell_c(\beta_l)$.
\end{proof}

\subsection{Highways for the arcs $\alpha_i^{(n)}$}\label{sec:highwayalpha}

The main goal of this section is to apply the technology of the previous section to show that any arc disjoint from an arc which starts like $\alpha_i^{(n)}$ starts like $\alpha_{i-1}^{(n)}$, provided $i$ is large enough (see Theorem \ref{thm:startslike}).
To do this, we will show that $\alpha^{(n)}_i$ has highways for all $i\ge 2$ and all $n$, beginning by proving the result for $i=2$.
For fixed $i$ and $n$, define $\chi_i^{(n)}$ to be all of $\mathring{\alpha}_i^{(n)}$ except the initial $P_s$, so that 
\begin{equation}\label{eqn:defofchi_i}
\mathring{\alpha}_i^{(n)}=P_s\chi_i^{(n)}.
\end{equation}

\begin{prop}\label{prop:highwayconstruction}
The segment $\alpha^{(n)}_2$ contains a subsegment of the form $\overline{\chi}_1^{(n)}P_uP_o\chi_1^{(n)}$ that appears as the final subsegment of $\mathring\alpha_2^{(n)}$.
In particular, $\alpha^{(n)}_2$ has highways whose lanes contain $\chi_1^{(n)}$. 
\end{prop}

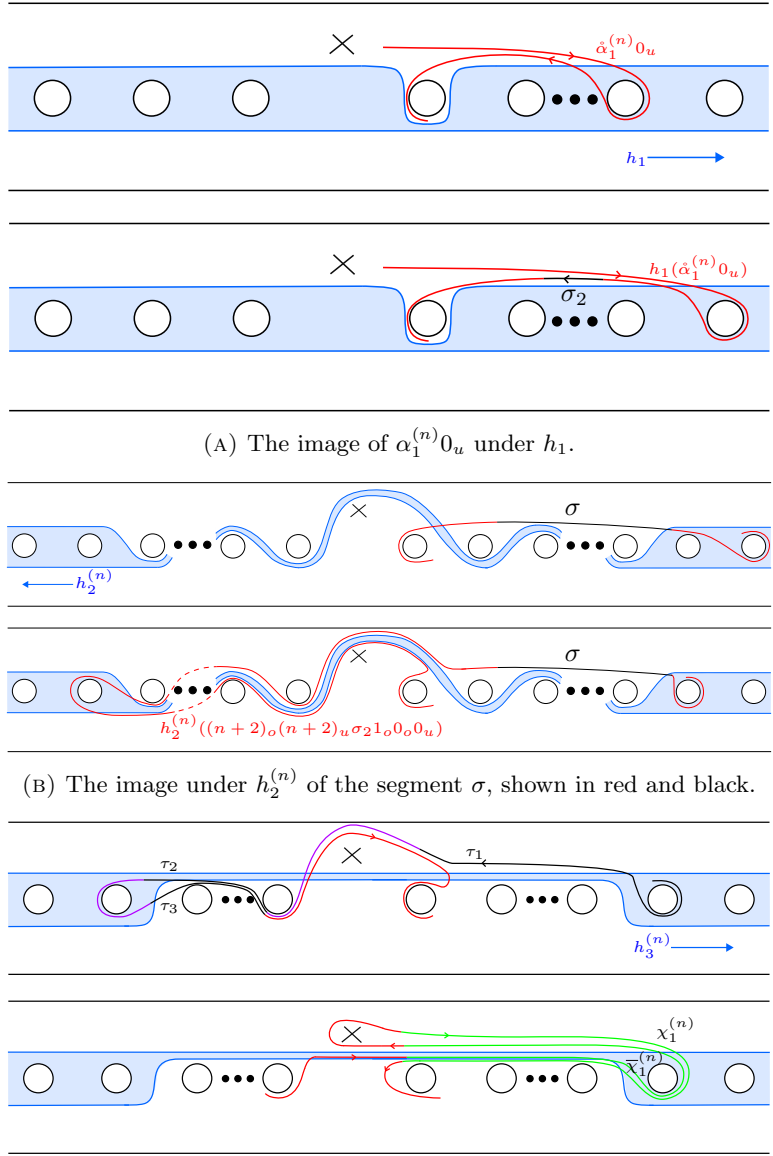
\begin{figure}
\centering
\begin{subfigure}{\linewidth}
\centering
\begin{overpic}[width=4in]{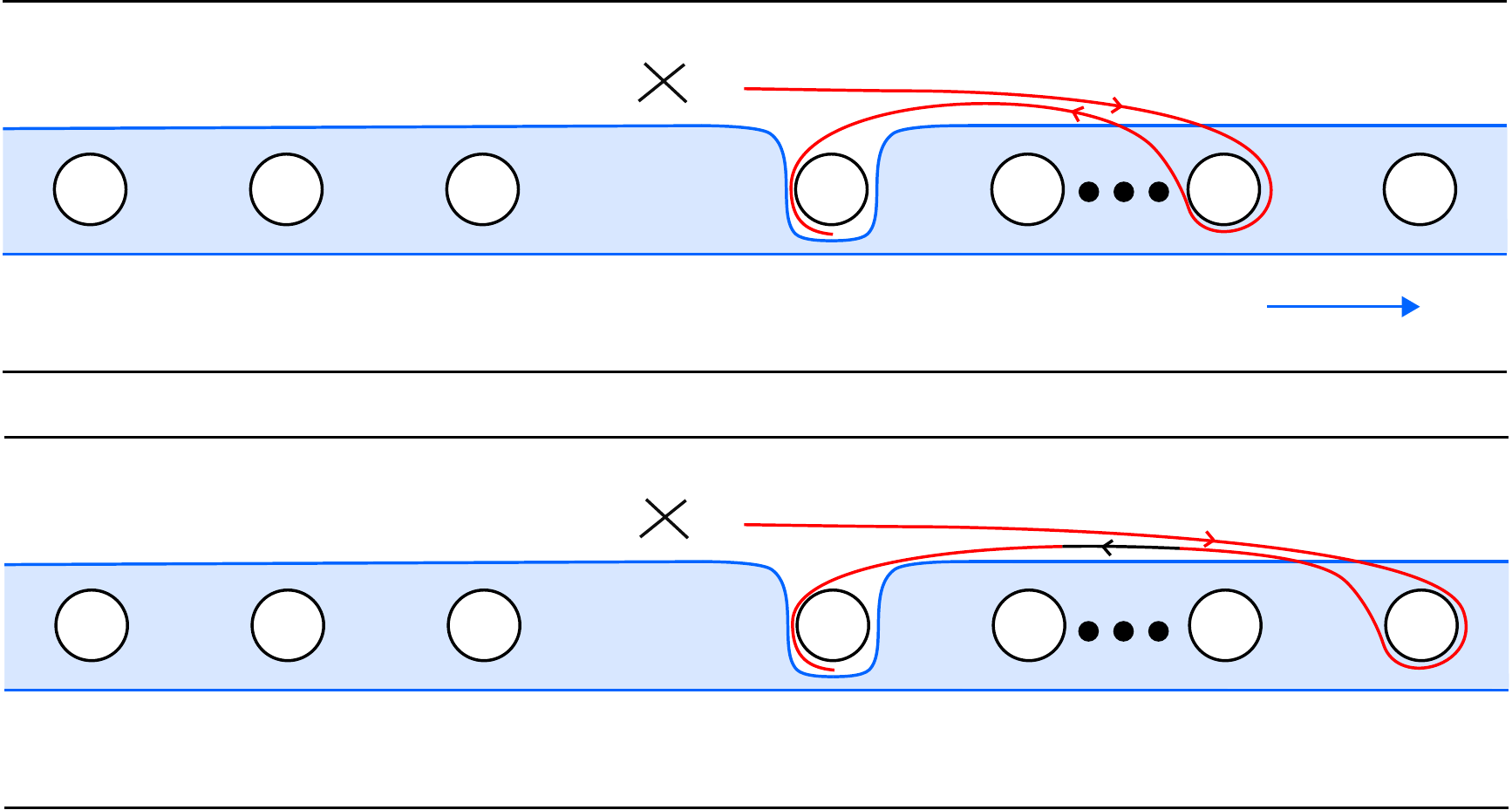}
\put(81,32.5){\tiny$\textcolor{blue}{h_1}$}
\put(77,47){\tiny$\textcolor{red}{\mathring\alpha_1^{(n)}0_u}$}
\put(84,18){\tiny$\textcolor{red}{h_1(\mathring\alpha_1^{(n)}0_u)}$}
\put(72.5,14.5){$\tiny\sigma_2$}
\end{overpic}
\caption{The image of $\alpha_1^{(n)}0_u$ under $h_1$.}
\label{fig:panel1}
\end{subfigure}
\par\bigskip
\begin{subfigure}{\linewidth}
\centering
\begin{overpic}[width=4in]{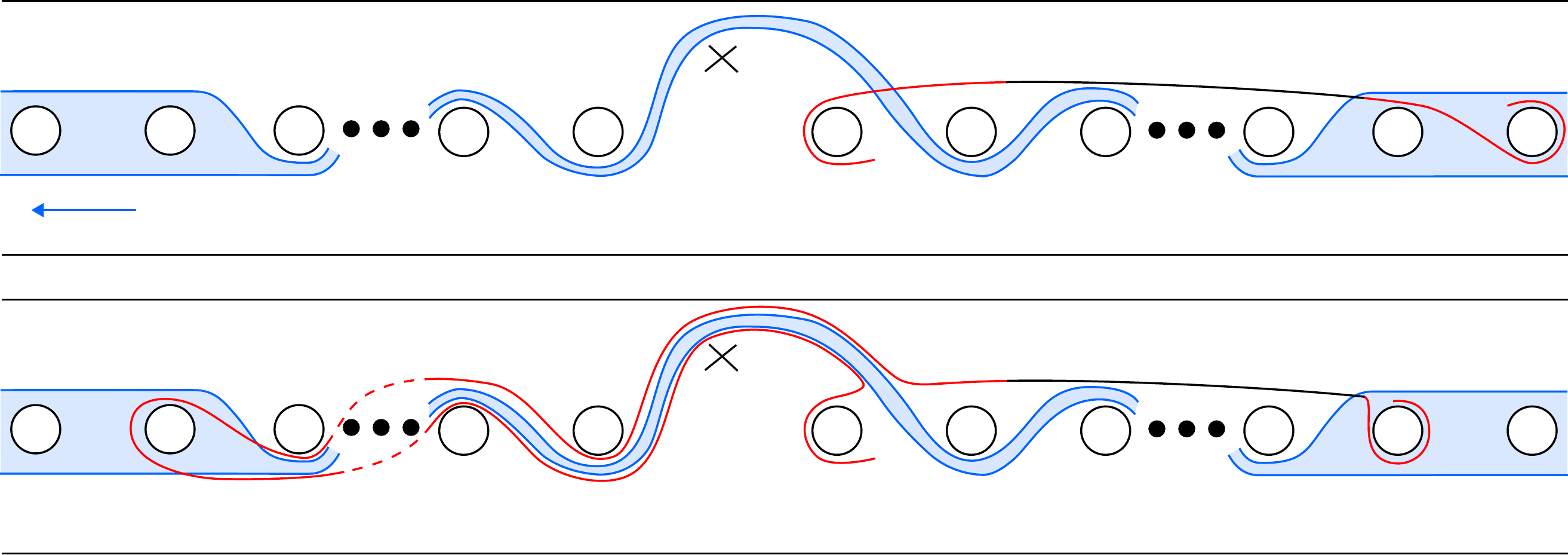}
\put(9,21.5){\tiny$\textcolor{blue}{h_2^{(n)}}$}
\put(73,31){$\tiny\sigma$}
\put(73,12){$\tiny\sigma$}
\put(20,2.5){\tiny$\textcolor{red}{h_2^{(n)}((n+2)_o(n+2)_u\sigma_21_o0_o0_u)}$}
\end{overpic}
\caption{The image under $h_2^{(n)}$ of the segment $\sigma$, shown in red and black.}
\label{fig:panel2}
\end{subfigure}
\par\bigskip
\begin{subfigure}{\linewidth}
\centering
\begin{overpic}[width=4in]{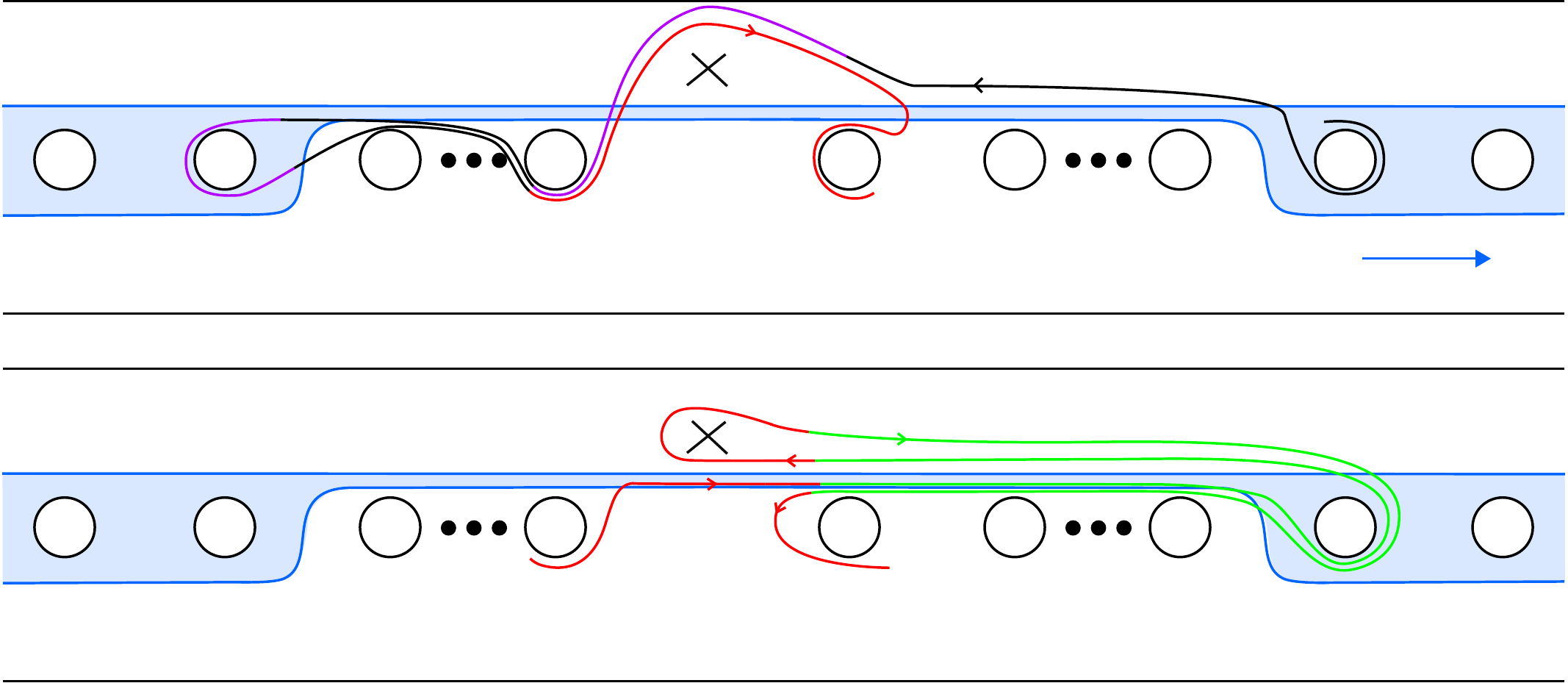}
\put(60,39){\tiny$\textcolor{black}{\tau_1}$}
\put(20,37.5){\tiny$\textcolor{black}{\tau_2}$}
\put(20,32){\tiny$\textcolor{black}{\tau_3}$}
\put(82,26.5){\tiny$\textcolor{blue}{h_3^{(n)}}$}
\put(81,11){\tiny\textcolor{black}{$\overline{\chi}_1^{(n)}$}}
\put(85,15.5){\tiny$\textcolor{black}{\chi_1^{(n)}}$}
\end{overpic}
\caption{The two full crossings in red give rise to $\chi_1^{(n)}$ and $\overline{\chi}_1^{(n)}$, which are shown in green,  in the proof of Proposition \ref{prop:highwayconstruction} upon applying $h_3^{(n)}$.}
\label{fig:panel3}
\end{subfigure}
\caption{The direct computations done in the proof of Proposition \ref{prop:highwayconstruction}.}
\end{figure}

\begin{proof}
We prove this by carefully examining the code for $\alpha^{(n)}_1$, which we assume is in standard position relative to the disk $D_p$.
As in Section \ref{sec:arcsthatstartlike}, we will retain the $0_u$ which immediately follows $\mathring\alpha_1^{(n)}$ when we compute images.  We will show that $\overline{\chi}_1^{(n)}P_uP_o\chi_1^{(n)}0_u$  appears as the final subsegment of $\mathring\alpha_2^{(n)}0_u$, which then implies the statement of the proposition. 

By Proposition \ref{lem:imageofinit}, $\mathring\alpha_2^{(n)}0_u=g_n(\mathring\alpha_1^{(n)}0_u)$.
Recall from \eqref{eqn:alpha1init} that
\begin{equation}\label{eqn:alpha1ring}
\mathring\alpha_{1}^{(n)}0_u=P_s0_o1_o\dots(n+1)_o(n+1)_u n_o(n-1)_o\dots 2_o1_o0_o0_u.
\end{equation}
Defining $\sigma:=\emptyset$ if $n=1$ or $\sigma:=n_o(n-1)_o\dots 2_o$ if $n>1$, a computation shows that 
\begin{align*}
h_1(\mathring\alpha_{1}^{(n)}0_u)&=P_s0_o1_o2_o\dots(n+2)_o(n+2)_u (n+1)_on_o\dots 2_o1_o0_o0_u,\\
&=P_s0_o1_o2_o\dots(n+2)_o(n+2)_u(n+1)_o\sigma1_o0_o0_u.
\end{align*}
See also Figure \ref{fig:panel1}.
We will show that the  image of the $1_o0_o0_u$ at the end of this last equation under $h_3^{(n)}\circ h_2^{(n)}$ produces the requisite segment in $\alpha_2^{(n)}$.

A direct computation shows that
$$h_2^{(n)}((n+2)_o(n+2)_u(n+1)_o\sigma1_o0_o0_u)=(n+1)_o(n+1)_u\sigma h_2^{(n)}(1_o0_o0_u),$$
since all of $\sigma$ is disjoint from $D_2$.
In standard position,  $1_o0_o0_u$ contains a full crossing, and  we compute that 
$$h_2^{(n)}(1_o0_o0_u)=1_o\overline{\partial D_2\vert_{(-n-1,0]}}(-n-1)_o(-n-1)_u\partial D_2\vert_{(-n-1,0]}0_o0_u,$$
in an unreduced code.
We thus have the decomposition
\begin{align}
h_2^{(n)}((n+2)_o&(n+2)_u\sigma1_o0_o0_u)\nonumber\\
&=(n+1)_o(n+1)_u\sigma1_o\overline{\partial D_2\vert_{(-n-1,0]}}(-n-1)_o(-n-1)_u\partial D_2\vert_{(-n-1,0]}0_o0_u\nonumber\\
&=\tau_1P_o(-1)_u\tau_2(-n-1)_o(-n-1)_u\tau_3(-1)_uP_o0_u,\label{eqn:h2backhalf}
\end{align}
where each $\tau_i$ is defined by the second equality. See Figure \ref{fig:panel2}.
None of $\tau_1,\tau_2,\tau_3$ fully cross $D_3$ in standard position.
As $(-1)_uP_o0_u$ fully crosses $D_3$ twice, we compute  that 
\begin{align}
h_3^{(n)}((-1)_uP_o0_u)=(-1)_uP_u&0_o\partial D_3\vert_{(0,n]}(n+1)_u(n+1)_o\overline{\partial D_3\vert_{(0,n]}} 0_oP_uP_o0_o\partial D_3\vert_{(0,n]}\nonumber\\
(n+1)_o&(n+1)_u\overline{\partial D_3\vert_{(0,n]}}0_o0_u,\nonumber\\
=(-1)_uP_u&\overline{\chi}_1^{(n)}P_uP_o\chi_1^{(n)}0_u\label{eqn:fullcrossingh3}.
\end{align}
Hence we see that $\overline{\chi}_1^{(n)}P_uP_o\chi_1^{(n)}0_u$ is contained in an \textit{unreduced} code for $g_n(\mathring\alpha_{1}^{(n)}0_u)$.
It remains to show that this segment persists in a \textit{reduced} code for $g_n(\mathring\alpha_{1}^{(n)}0_u)=\mathring{\alpha}_2^{(n)}0_u$.  

Proposition \ref{lem:imageofinit} shows that if $\overline{\chi}_1^{(n)}P_uP_o\chi_1^{(n)}0_u$ persists in a reduced code for $\mathring{\alpha}_20_u$ then it persists in a reduced code for $\alpha_2$. 
To check if $\overline{\chi}_1^{(n)}P_uP_o\chi_1^{(n)}0_u$ persists in a reduced code for $\mathring{\alpha}_20_u$, we need only consider a reduced code for the image under $g_n$ of the characters following the $(n+1)_o(n+1)_u$ in \eqref{eqn:alpha1ring}, because Corollary \ref{cor:imageof2o2ufamily} shows that the  pair $(n+1)_o(n+1)_u$ blocks cascading cancellation between characters on either side.  This is precisely the segment whose image under $h_2^{(n)}$ is given by \eqref{eqn:h2backhalf}.

We will use direct computation to show that the segment $\overline{\chi}_1^{(n)}P_uP_o\chi_1^{(n)}0_u$ is the final segment of a reduced code for
\begin{equation}\label{eqn:h3image}
h_3^{(n)}(\tau_1P_o(-1)_u\tau_2(-n-1)_o(-n-1)_u\tau_3(-1)_uP_o0_u).
\end{equation}
Before making this computation, we point out the relevant pieces. 
Each $(-1)_uP_o$ (or its reverse) in \eqref{eqn:h3image} fully crosses $D_3$ and has image given by  the first  half of \eqref{eqn:fullcrossingh3}, up to reversing the orientation.
The terminal $P_o0_u$  also  fully crosses $D_3$, and so we may similarly compute its image.
The segments $\tau_2$, $\tau_3$ are invariant under $h_3^{(n)}$.  The image of $\tau_1$ is $h_3^{(n)}(\tau_1)=(n+2)_o(n+2)_u(n+1)_o\sigma 1_o0_o$. Moreover, the four characters $\tau_2^t(-n-1)_o(-n-1)_u\tau_3^i$ form a loop which fully crosses $D_3$ exactly once. (Note that $\tau_2^t=\tau_3^i=(-n)_{o/u}$, where the choice of $o/u$ depends on the parity of $n$.)
Combining all of these remarks, we find the following \textit{reduced} code (see Figure \ref{fig:panel3}):
\begin{align}
h_3^{(n)}(\tau_1P_o(-1)_u&\tau_2(-n-1)_o(-n-1)_u\tau_3(-1)_uP_o0_u)=\nonumber\\
&(n+2)_o(n+2)_u(n+1)_o\sigma 1_o0_o 
P_oP_u0_o\dots(n+1)_o(n+1)_un_o\dots 0_o\nonumber\\
&P_u\overline{\partial D_2\vert_{(-n,-1]}}(-n)_u(-n)_o\partial D_3\vert_{(-n,n+1)}
(n+1)_o(n+1)_u\overline{\partial D_2\vert_{(-n,n]}}\nonumber\\
&(-n)_o(-n)_u
\partial D_2\vert_{(-n,-1]}P_u{\uwave{\partial D_3\vert_{[0,n+1)}(n+1)_u
(n+1)_o\overline{\partial D_3\vert_{(P,n+1)}}\nonumber }}\\ 
&{\uwave{P_uP_o\partial D_3\vert_{[0,n+1)}(n+1)_o(n+1)_u\overline{\partial D_3\vert_{[1,n+1)}}0_o0_u}}\nonumber
\end{align}
In particular, we see the requisite segment at the end of this string (underlined).  Precisely, we have 
\[\chi_1^{(n)}=\partial D_3\vert_{[0,n+1)}(n+1)_o(n+1)_u\overline{\partial D_3\vert_{[1,n+1)}}0_o.\]
This completes the proof of Proposition \ref{prop:highwayconstruction}.
\end{proof}
The next corollary shows that $\alpha_i^{(n)}$ has highways whose lanes contain $\chi_{i-1}^{(n)}$.  Notice that $g_n(\chi_{i-1}^{(n)})=\chi_i^{(n)}$ by Proposition \ref{lem:imageofinit}.

\begin{cor}\label{cor:highwayconstruction}
When $i\ge 2$, $\alpha^{(n)}_i$ contains a subsegment of the form $\overline{\chi}_{i-1}^{(n)}P_uP_o\chi_{i-1}^{(n)}$  that appears as the final subsegment of $\mathring\alpha_i^{(n)}$.
In particular, $\alpha^{(n)}_i$ has highways for all $i\ge 2$ with lanes containing $\chi_{i-1}^{(n)}$.
\end{cor}
\begin{proof}
As above, we will retain the $0_u$ which immediately follows $\mathring\alpha_i^{(n)}$ when we compute images.  We will show by induction that $g_n(\mathring\alpha_{i-1}^{(n)}0_u)$ ends with $\overline{\chi}_{i-1}^{(n)}P_uP_o\chi_{i-1}^{(n)}0_u$ in a reduced code, which will show that $\overline{\chi}_{i-1}^{(n)}P_uP_o\chi_{i-1}^{(n)}$ appears as the final subsegment of $\mathring\alpha_i^{(n)}$.   
Since $g_n(\mathring\alpha_{i-1}^{(n)}0_u)=\mathring\alpha_i^{(n)}0_u$ by Proposition \ref{lem:imageofinit}, 
this will imply that such a segment persists in $\alpha_i^{(n)}$.

The base case $i=2$ was shown in Proposition \ref{prop:highwayconstruction}, so we proceed to the induction step.
Using the induction hypothesis, write
$$\mathring\alpha^{(n)}_{i-1}0_u=P_s\eta_1\overline{\chi}_{i-2}^{(n)}P_uP_o\chi_{i-2}^{(n)}0_u,$$
for some subsegment $\eta_1$.
A calculation shows that in an unreduced code
\begin{align*}
g_n(\overline{\chi}_{i-2}^{(n)}P_uP_o\chi_{i-2}^{(n)}0_u)&=g_n(\overline{\chi}_{i-2}^{(n)}+0_oP_uP_o0_o+\chi_{i-2}^{(n)}0_u),\\
&=g_n(\overline{\chi}_{i-2}^{(n)})+g_n(0_oP_uP_o0_o)+g_n(\chi_{i-2}^{(n)}0_u),\\
&=\overline{\chi}_{i-1}^{(n)}P_uP_o\chi_{i-1}^{(n)}0_u,
\end{align*}
where one verifies that the last line is in fact reduced.
We therefore reduce to showing that $\overline{\chi}_{i-1}^{(n)}P_uP_o\chi_{i-1}^{(n)}0_u$ persists as the terminal substring of $\mathring\alpha_i^{(n)}0_u$.
For this, we must prove that no characters in an unreduced code for $g_n(P_s\eta_10_o)$ cancel with this terminal $0_o$, which is $(\overline{\chi}_{i-1}^{(n)})^i$.

Recall from \eqref{eqn:defofchi_i} that $\chi_{i-1}^{(n)}$ is by definition all of $\mathring\alpha_{i-1}^{(n)}$ except the initial $P_s$.  In particular, its terminal character $0_o=(\chi_{i-1}^{(n)})^t$ is  oriented to the left.   Thus the initial $0_o$ of the reverse segment  $\overline{\chi}_{i-1}^{(n)}$ is  oriented to the right.  Since we are considering the segment $P_s\eta_10_o=P_s\eta_1(\overline{\chi}_{i-1}^{(n)})^i$, this implies that $(\eta_1)^t$ has numerical value at most $0$. 
Moreover, simplicity of $\mathring\alpha^{(n)}_{i-1}$ shows that if $\eta_1$ ends by looping around $p$, then it must have $1_o0_oP_oP_u$ as its final segment.
Theorem \ref{lem:induct} combined with Remark \ref{rem:induct} then shows that there is no cancellation with $(\overline{\chi}_{i-1}^{(n)})^i=0_o$ in a reduced code for $g_n(P_s\eta_10_o)$.
\end{proof}

The following corollary follows immediately from Corollary~\ref{cor:inductcor}, Corollary \ref{cor:highwayconstruction}, our definition of $\chi^{(n)}_i$ in \eqref{eqn:defofchi_i}, and  the definition of lane length.

\begin{cor}\label{lem:lanelength}
Fix $i\ge 2$ and let $\mathcal{L}$ and $\mathcal{R}$ denote the collections of left and right lanes in $\alpha^{(n)}_i$, respectively. Then
$$\displaystyle n_L=\max_{\lambda\in\mathcal{L}}L(\lambda)\geq \ell_c(\mathring\alpha^{(n)}_{i-1})$$
and 
$$\displaystyle n_R=\max_{\rho\in\mathcal{R}}L(\rho)\geq \ell_c(\mathring\alpha^{(n)}_{i-1}).$$

\end{cor}

\noindent We are now in a position to prove the main result of the section.

\begin{thm} \label{thm:startslike}
If $\delta$ is an arc which starts like $\alpha^{(n)}_i$ for some $i\ge 2$ and $\gamma$ is any arc disjoint from $\delta$, then $\gamma$ starts like $\alpha_{i-1}^{(n)}$.
\end{thm}
\begin{proof}
Proposition \ref{prop:highwayconstruction} and Corollary \ref{cor:highwayconstruction} show that $\mathring\alpha_i^{(n)}$ contains the segment $\overline{\chi}_{i-1}^{(n)}P_uP_o\chi_{i-1}^{(n)}$ and hence $\mathring\alpha_i^{(n)}$ has highways for all $i\ge 2$.

As $\delta$ starts like $\alpha^{(n)}_i$, the first part of its code is identical to that of $\mathring\alpha^{(n)}_i$ and therefore $\delta$ also has highways and also contains the segment $\overline{\chi}_{i-1}^{(n)}P_uP_o\chi_{i-1}^{(n)}$.
By Lemma \ref{lem:stayinyourlane}, the innermost left lane $P_uP_o\alpha_l$ (or its reverse) and the innermost right lane $P_oP_u\alpha_r$ (or its reverse) in $\delta$ each have lane length at least $\ell_c(\mathring\alpha_{i-1}^{(n)})$.
Hence the first $\ell_c(\mathring\alpha_{i-1}^{(n)})-1$ characters of these lanes immediately following $P_{o/u}P_{u/o}$ must agree with $\chi_{i-1}^{(n)}$ or its reverse.
As $\gamma$ is disjoint from $\delta$, the moreover statement of Lemma \ref{lem:stayinyourlane} gives that the code for $\gamma$ must begin with $P_s\alpha_l$ or $P_s\alpha_r$.   Consequently it must begin with $\mathring\alpha_{i-1}^{(n)}=P_s\chi_{i-1}^{(n)}$, as required.
\end{proof}


\section{Loxodromic elements for $\mathcal{A}(\Sigma,p)$}\label{sec:loxo}

In this section, we will first conclude the proof of Theorem \ref{thm:loxomain} and then go on to explore some properties of our loxodromic elements $g_n$, including identifying an explicit geodesic axis for $g_n$, as well as describing the limit points of $g_n$ on the boundary of the modified arc complex $\mathcal{A}(\Sigma,p)$ for a surface $\Sigma$ of type $\mathcal{S}$. 

\subsection{The proof of Theorem \ref{thm:loxomain}}

Let $S$ be the biinfinite flute surface we defined in Section~\ref{background} with isolated puncture $p$.    Recall that $\mathcal{A}(S,p)$ denotes the modified arc graph of $S$, as in Definition \ref{defn:modarcgraph}. Let $\alpha_0=P_s0_o0_uP_s\in\mathcal A(S,p)$ be the arc from Definition \ref{def:alpha_i}.  

For each $n$, consider the map 
\begin{equation}\label{eqn:startslike}
\phi_n\colon \mathcal A(S,p)\to \Z_{\geq 0},
\end{equation}
defined by setting $\phi_n(\delta)$ equal to the largest $i\geq 0$ such that $\delta$ starts like $\alpha_i^{(n)}$.  If there is no such $i$ then set $\phi_n(\delta)=0$.

\begin{lem} \label{cor:dist}
For any $\gamma,\delta\in\mathcal A(S,p)$, we have $d_{\mathcal A(S,p)}(\gamma,\delta)\geq|\phi_n(\gamma)-\phi_n(\delta)|$.
\end{lem}

\begin{proof}
It follows immediately from Theorem \ref{thm:startslike} that if  $d_{\mathcal A(S,p)}(\gamma,\delta)=1$, then $|\phi_n(\gamma)-\phi_n(\delta)|\leq 1$ for any $n$.  The result then follows from the subadditivity of the absolute value.
\end{proof}

Let  $\{g_n\}_{n\in\mathbb N}$ be as in Definition \ref{def:g_n}, and let $\{\alpha_i^{(n)}\}_{i\in\mathbb Z}$ be as in Definition \ref{def:alpha_i}. We first show that the elements $g_n\in \MCG(S,p)$ are loxodromic with respect to the action on $\mathcal A(S,p)$.

\begin{prop}\label{prop:loxonS}
For each $n\in \mathbb N$, the homeomorphism $g_n\in\MCG(S,p)$ is a loxodromic isometry of $\mathcal A(S,p)$ with a $(2,0)$--quasi-geodesic axis $\{\alpha_i^{(n)}\}_{i\in\mathbb Z}$.
\end{prop}

\begin{proof}
We first show that the map $\mathbb Z_{\geq 0}\to \langle g_n\rangle \alpha_0\subset \mathcal A(S,p)$ defined by $i\mapsto g_n^i(\alpha_0)$ is a $(2,0)$--quasi-isometry.    In other words, we will show that $\{\alpha_i^{(n)}\}_{i\geq 0}$ is a quasi-geodesic half-axis for $g_n$ along which $g_n$ acts as translation.  

 Let $\phi_n$ be the map defined in \eqref{eqn:startslike}.  By Lemma \ref{cor:dist}, we have that $d_{\mathcal A(S,p)}(\gamma,\delta)\geq|\phi_n(\gamma)-\phi_n(\delta)|$ for any $\gamma,\delta\in \mathcal A(S,p)$.  Since $\phi_n(\alpha^{(n)}_i)=i$ for all $i\geq 0$, this implies that 
 \begin{equation}\label{eqn:lowerbd} 
 d_{\mathcal A(S,p)}(\alpha_0,\alpha_i^{(n)})\geq i.
 \end{equation} 
       
   Consider the arc  $\beta=P_s(-1)_o(-1)_uP_s$.  Then $\beta$ is disjoint from both $\alpha^{(n)}_0$ and $\alpha^{(n)}_1$.  Since $g_n$ is a homeomorphism, $g_n^i(\beta)$ is disjoint from both $\alpha^{(n)}_i$ and $\alpha^{(n)}_{i+1}$, and thus $d_{\mathcal A(S,p)}(\alpha_i^{(n)},\alpha^{(n)}_{i+1})\leq 2$ for all $i$. Therefore, for all $i$,  
\begin{equation}\label{eqn:upperbd}
d_{\mathcal A(S,p)}(\alpha_0,\alpha_i^{(n)})\leq 2i.
\end{equation}
Together, \eqref{eqn:lowerbd} and \eqref{eqn:upperbd} show that $\{\alpha_i^{(n)}\}_{i\geq 0}$ is a $(2,0)$--quasi-geodesic half-axis for $g_n$. 

Since $\{\alpha_i^{(n)}\}_{i\geq 0}$ is an unbounded orbit of $g_n$, we can see that $g_n$ is not elliptic, and since $g_n$ acts as translation along this quasi-geodesic half-axis, $g_n$ cannot be parabolic.  Thus we may conclude that $g_n$ is a loxodromic isometry of $\mathcal A(S,p)$.  
\end{proof}

Recall from Definition \ref{def:typeS} that a surface $\Sigma$ with an isolated puncture $p$ is of type $\mathcal{S}$ if there exists a proper embedding $S \hookrightarrow \Sigma$ where $S$ contains $p$, the two non-isolated ends of $S$ correspond to distinct ends of $\Sigma$, and with the property that  either there are countably many connected components of $\Sigma \setminus S$ of the same (nontrivial) topological type or countably many isolated punctures of $S$ remain isolated punctures when embedded in $\Sigma$.  We denote this special class of connected components by $\mathcal{U}$, so that the elements of $\mathcal{U}$ are all homeomorphic to a fixed surface $\Sigma_0$. Recall that this definition ensures that the shift maps $h_1, h_2^{(n)}, $ and $h_3^{(n)}$ we are interested in on $S$ extend to shift maps on $\Sigma$.  In particular, $g_n= h_3^{(n)}\circ h_2^{(n)}\circ h_1$ extends to a homeomorphism of $\Sigma$.  When we reference $g_n$ below, we will try to be specific about when we are considering $g_n$ as an element of $\MCG(S, p)$ versus  an element of $\MCG(\Sigma, p)$. Recall from Definition \ref{def:intrinsic} that a homeomorphism $f \in \MCG(\Sigma, p)$ is intrinsically infinite-type if $f \notin \overline{\MCG_c(\Sigma)}$. 

Theorem \ref{thm:loxomain} is then a direct consequence of the following theorem.

\begin{thm}\label{thm:gnlox}
Let $\Sigma$ be a surface of type $\mc S$. For each $n\in \mathbb N$, the homeomorphism $g_n\in\MCG(\Sigma,p)$ is a loxodromic isometry of $\mathcal A(\Sigma,p)$ with a $(4,0)$--quasi-geodesic axis $\{\alpha_i^{(n)}\}_{i\in\mathbb Z}$.  Moreover, $g_n$ is intrinsically infinite-type.
\end{thm}

\begin{proof}
Fix $n\in \mathbb N$.  By Proposition \ref{prop:loxonS}, $g_n$ is loxodromic with respect to the action of $\MCG(S,p)$ on $\mathcal A(S,p)$.  Moreover, $g_n$ extends to an element of $\MCG(\Sigma,p)$ and so $g_n$ acts on by isometries $\mathcal A(\Sigma,p)$ as well.   By Lemma~\ref{lem:AsqiembsAsigma}, there is a $(2,0)$--quasi-isometric embedding $\mathcal A(S,p)\hookrightarrow \mathcal A(\Sigma,p)$.    Therefore the image of the $(2,0)$--quasi-geodesic half-axis for $g_n$ constructed in Proposition \ref{prop:loxonS} is a $(4,0)$--quasi-geodesic half-axis in $\mathcal A(\Sigma,p)$.  It is clear that $g_n$ stabilizes this $(4,0)$--quasi-geodesic half-axis and so the arcs $\{\alpha_i^{(n)}\}_{i\in \mathbb Z_{\geq 0}}$ form a quasi-geodesic half-axis for $g_n$ in $\mathcal A(\Sigma,p)$.  Therefore, $g_n$ is loxodromic with respect to the action of $\MCG(\Sigma,p)$ on $\mathcal A(\Sigma,p)$.

We now show that $g_n \notin \overline{\MCG_c(\Sigma)}$. 
If $\Sigma_0$ has a nontrivial end space, then $g_n \notin \pmap(\Sigma)$ since $g_n$ translates the elements of $\mathcal{U}$. Note that $\overline{\MCG_c(\Sigma)} < \pmap(\Sigma)$ so that $g_n$ is of intrinsically infinite-type in this case. 

Now suppose that the end space of $\Sigma_0$ is trivial, so that $\Sigma_0$ is a finite-genus surface with one boundary component. Therefore, $S \cup \mathcal{U} = \Sigma'$ is homeomorphic to an infinite-genus surface with two non-planar ends and a countable number of planar ends. Note that $\Sigma \setminus \Sigma'$ consists of all of the additional topology of $\Sigma$ that is irrelevant to our construction of shift maps. In this way, the planar ends of $\Sigma'$ cut away the irrelevant topology of $\Sigma$. 

By \cite[Corollary 6]{APV}, $\pmap(\Sigma') = \overline{\MCG_c(\Sigma')} \rtimes \mathbb{Z}(h')$, where $h'$ is the standard handleshift on $\Sigma'$, which also corresponds to a handleshift on $\Sigma$ by extending $h'$ via the identity on $\Sigma\setminus \Sigma'$. For simplicity of notation, we drop the subscript on $g_n$ in what follows since the argument does not depend on $n$.  Theorem~\ref{thm:pAshift} below tells us that when $g$ is considered as an element of $\MCG(S, p)$, there is a compactly supported mapping class $\phi$ such that $g= \phi h$, where   $h$ is the right shift on $S$ shifting the punctures corresponding to the elements of $\mathcal{U}$. As an element of $\pmap(\Sigma', p)$, and therefore of $\MCG(\Sigma, p)$, $g$ is thus equal to $\phi \cdot (h')^m$ where $\phi \in \MCG_c(\Sigma') < \MCG_c(\Sigma) $ and $m$ is the genus of $\Sigma_0$. By the proof of \cite[Proposition 6.3]{PatelVlamis}, $(h')^m \notin \overline{\MCG_c(\Sigma)}$, so that $g \notin \overline{\MCG_c(\Sigma)}$ as well. 
\end{proof}

\subsection{Alternate description of $g_n$}
In this section, we show that our mapping classes $g_n$ can be written as the composition of a pseudo-Anosov on a finite-type subsurface and a standard shift.  Moreover, we will show that $g_n$ is the composition of \textit{the same} pseudo-Anosov on a \textit{fixed} finite-type subsurface and a standard shift whenever $n>3$; however, this subsurface is embedded in $\Sigma$ in different ways for different $n$, yielding distinct elements of $\MCG(\Sigma,p)$.

Let $h$ be the shift map on our subsurface $S$ that translates the punctures that correspond to elements of $\mathcal{U}$, that is, the right shift whose domain contains the simple closed curves $B_i$ for all $i\in \Z$ and maps $B_i$ to $B_{i+1}$. 

Without loss of generality, for each $n$ we may modify our separating curves $S_{-n}$ and $S_{n+1}$ so that  a connected component $\Pi_n'$ of $S\setminus(S_{-n}\cup S_{n+2})$ is a sphere with $2$ boundary components and $n+4$ punctures, one of which is $p$.  This amounts to pushing any extra topology on the back of $S$ outside of this subsurface.
We now further modify $\Pi'_n$ to form a subsurface $\Pi_n$ for each $n$ as follows.  
Let $R_1^{(n)}, R_2^{(n)}, R_3^{(n)}$ be the simple closed curves as shown in Figure \ref{fig:subsurfacepA}.  In particular, 
\begin{itemize}
\item $R^{(n)}_1$ encloses all $B_i$ with $-n\leq i\leq -1$ and $i $ odd;  
\item  $R^{(n)}_2$ encloses all $B_i$ with $-n\leq i\leq n$ and $i$ even; and
\item $R^{(n)}_3$ encloses all $B_i$ with $3\leq i\leq n$ and $i$ odd.
\end{itemize}
Note that $R^{(1)}_i$ is empty for all $i=1,2,3$, $R^{(2)}_i$ is empty when $i=1,3$, and $R^{(3)}_i$ is empty when $i=3$.  For all $n\geq 4$, $R^{(n)}_i$ is non-empty for each $i=1,2,3$.

\begin{figure}
\centering
\begin{overpic}[width=3in, trim={.25in 3.5in .25in 1.5in}, clip]{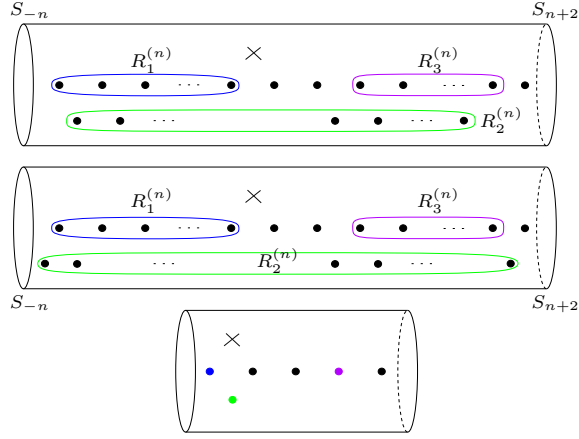}
\put(2,74){$\scriptstyle S_{-n}$}
\put(2,23){$\scriptstyle S_{-n}$}
\put(23,65){$\scriptstyle R_1^{(n)}$}
\put(73,65){$\scriptstyle R_3^{(n)}$}
\put(23,40.5){$\scriptstyle R_1^{(n)}$}
\put(73,40.5){$\scriptstyle R_3^{(n)}$}
\put(84,54.5){$\scriptstyle R_2^{(n)}$}
\put(45,30){$\scriptstyle R_2^{(n)}$}
\put(93,74){$\scriptstyle S_{n+2}$}
\put(93,23){$\scriptstyle S_{n+2}$}
\end{overpic}
\caption{The curves $R^{(n)}_i$ when $n$ is odd and even (top and middle, respectively), and the surface $\Pi=\Pi_n$ when $n\geq 3$ (bottom). In $\Pi$, the blue, green, and purple punctures correspond to where $\Pi_n'$ was cut along the curves $R_1^{(n)}, R_2^{(n)}, R_3^{(n)}$, respectively.}
\label{fig:subsurfacepA}
\end{figure}

\begin{defn}
Let $\Pi_n$ be the component of $\Pi_n'\setminus\left(R^{(n)}_1\cup R^{(n)}_2 \cup R^{(n)}_3\right)$ which contains the puncture $p$.  
\end{defn}

The surface $\Pi_n$ is a sphere with 2 boundary components and some number of punctures: five punctures if $n=1$, six punctures if $n=2$, and seven punctures if $n\geq 3$.  Notice that the $\Pi_n$ are homeomorphic  for all $n\geq 3$.  However, the embedding $\iota_n\colon \Pi_n\to\Sigma$  are different for distinct $n$.  For any $f\in \pmap(\Pi_n,p)$, let $\overline f:=\iota_n\circ f$.

\begin{thm}\label{thm:pAshift}
For each $n\geq 1$, there is a pseudo-Anosov $\phi_n\in\pmap(\Pi_n)$ so that 
$g_n=\overline\phi_n h$.  Moreover, for all $n,n'\geq 3$, $\Pi_n=\Pi_{n'}$ and $\phi_n=\phi_{n'}$ as elements of $\pmap(\Pi_n)$.  However, $\overline \phi_n$ and $\overline \phi_{n'}$ are distinct elements of $\MCG(\Sigma,p)$ since the embeddings $\iota_n$ are distinct.
\end{thm}

\begin{proof}
We define $\overline \phi_n:=g_nh^{-1}$ for all $n\geq 1$.
It is clear that $\overline\phi_n$ stabilizes the subsurface $\Pi_n$ and is the identity on $\Sigma\setminus \Pi_n$.  Let $\phi_n$ be the restriction of $\overline\phi_n$ to $\pmap(\Pi_n)$, so that $\iota\circ \phi_n=\overline\phi_n$.
We will show that $\phi_n$ is pseudo-Anosov.  To do this, we will apply \cite[Lemma~3.1]{Verberne}, which states that a mapping class  is pseudo-Anosov if it preserves a large, generic, birecurrent train track whose associated transition  matrix  is Perron--Frobenius. We will construct such a train track $\tau_n$ for each $n$.

The cases $n=1,2$ are slightly different and we will deal with them separately.  For all $n\geq 3$, the surfaces $\Pi_n$ are homeomorphic and we will build a single train track  which will satisfy the above conditions.  Notice that $\tau_n$ is a train track on the surface $\Pi_n$, so to show that $\tau_n$ is large, generic, and birecurrent it suffices to consider $\Pi_n$.  However, since $\phi_n$ is defined as a restriction of $\overline \phi_n$, which is a product of elements of $\MCG(\Sigma,p)$ which are not supported on $\Pi_n$, we must consider the different embeddings of $\Pi_n$ into $\Sigma$ in order to show that $\tau_n$ is preserved by $\phi_n$ and to calculate the transition matrix of $\tau_n$.

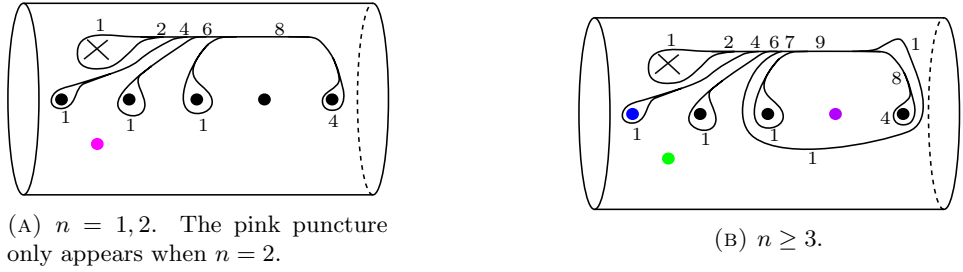
\begin{figure}
\centering
\begin{subfigure}{.4\columnwidth}
\begin{overpic}[width=2in]{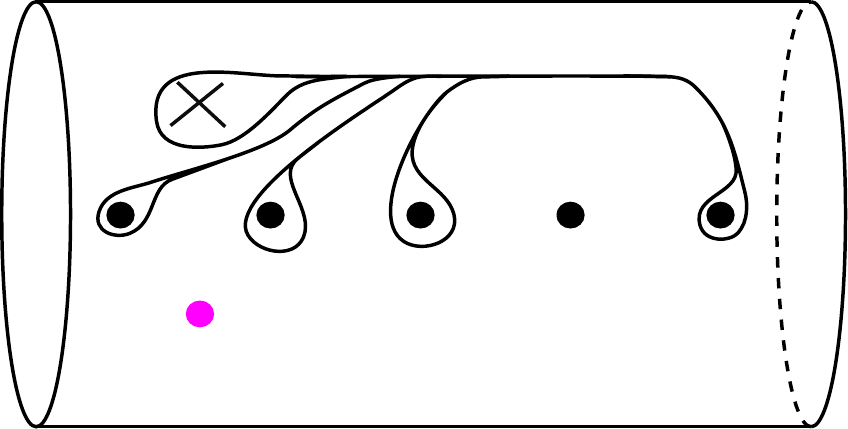}
\put(23,43){$\scriptstyle 1$}
\put(14,19){$\scriptstyle 1$}
\put(31,17){$\scriptstyle 1$}
\put(50,17){$\scriptstyle 1$}
\put(84,18){$\scriptstyle 4$}
\put(39,42){$\scriptstyle 2$}
\put(45,42){$\scriptstyle 4$}
\put(51,42){$\scriptstyle 6$}
\put(70,42){$\scriptstyle 8$}
\end{overpic}
\caption{ $n=1,2$.  The pink puncture only appears when $n=2$. }
\end{subfigure}\hfill%
\begin{subfigure}{.4\columnwidth}
\begin{overpic}[width=2in]{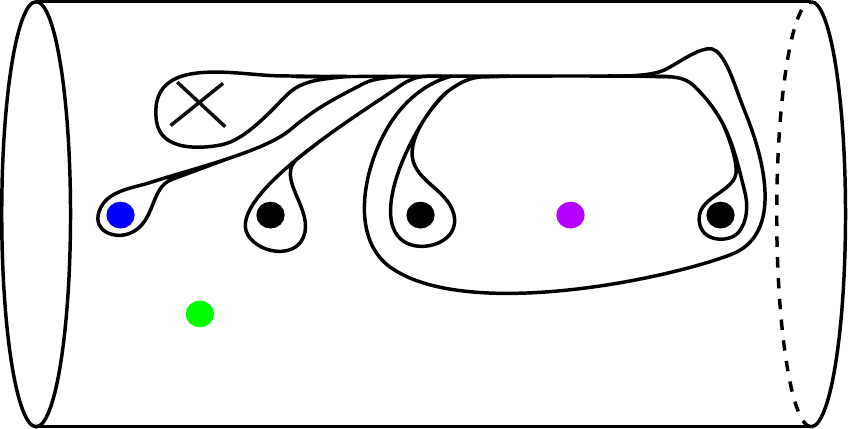}
\put(23,43){$\scriptstyle 1$}
\put(14,18.5){$\scriptstyle 1$}
\put(32,17){$\scriptstyle 1$}
\put(50,18){$\scriptstyle 1$}
\put(87,42){$\scriptstyle 1$}
\put(60,12){$\scriptstyle 1$}
\put(79,22.5){$\scriptstyle 4$}
\put(82,33){$\scriptstyle 8$}
\put(38,42.5){$\scriptstyle 2$}
\put(45,42.5){$\scriptstyle 4$}
\put(50,42.5){$\scriptstyle 6$}
\put(54,42.5){$\scriptstyle 7$}
\put(62,42.5){$\scriptstyle 9$}
\end{overpic}
\caption{$n\geq 3$. }
\end{subfigure}
\caption{The train tracks $\tau_n$ on the finite-type subsurface $\Pi_n$.}
\label{fig:tau}
\end{figure}

The train tracks $\tau_n$ for $n=1,2$ and for $n\geq 3$ are shown on $\Pi_n$ in Figure \ref{fig:tau}.  For all $n$, the following hold.  All switches are trivalent and so $\tau_n$ is generic. Each complementary region is a once-punctured disk or a polygon and so $\tau_n$ is large.  The weights on each branch of $\tau_n$ are positive and so $\tau_n$ is recurrent.  Moreover, the finite collections of simple closed curves $\{\gamma^{(n)}_i\}$ in Figure \ref{fig:gammai} is such that  each branch of $\tau_n$ is intersected transversely and \textit{efficiently} by some $\gamma_i^{(n)}$, i.e., $\gamma^{(n)}_i\cup\tau_n$ has no complementary bigon regions for any $i$. Therefore $\tau_n$ is transversely recurrent.  Since $\tau_n$ is recurrent and transversely recurrent, it is birecurrent. 

\begin{figure}
\centering
\begin{subfigure}{.4\columnwidth}
\centering
\def\svgwidth{2.75in}
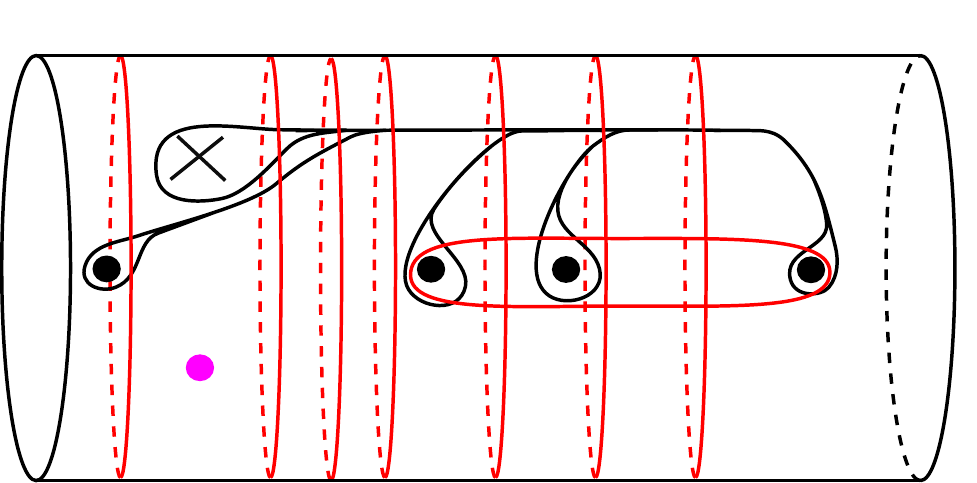
\caption{$n=1,2$}
\end{subfigure}\hfill
\begin{subfigure}{.4\columnwidth}
\centering
\def\svgwidth{2.75in}
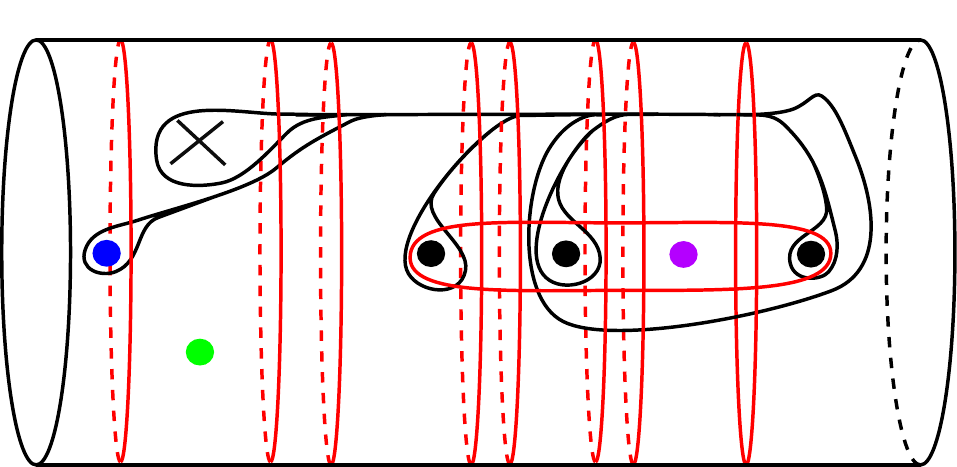
\caption{$n\geq 3$}
\end{subfigure}
\caption{A collection of simple closed curves $\{\gamma^{(n)}_i\}$ such that   each branch of $\tau_n$ is intersected transversely and efficiently by some $\gamma_i^{(n)}$.  As in Figure \ref{fig:tau}, the pink puncture only appears when $n=2$.}
\label{fig:gammai}
\end{figure}

Figures \ref{fig:phitau} and \ref{fig:phitaun} show that $\tau_n$  is preserved by $\phi_n$ for $n=1,2$ and $n\geq 3$, respectively.    It is immediate from these figures that the matrix $A_n$ associated to $\tau_n$ is
\[
A_1=A_2=\begin{pmatrix} 5 & 6 & 0 & 2 \\ 6 & 9 & 0 & 2 \\ 10 & 10 & 2 & 3 \\ 6 & 6 & 1 & 2 \end{pmatrix},\]  
and 
\[
A_n=\begin{pmatrix} 5 & 6 & 0 & 2 & 0 \\ 6 & 9 & 0 & 2 & 0 \\ 10 & 10 & 2 & 2 & 1 \\ 6 & 6 & 1 & 1 & 1 \\  6 & 6 & 1 & 2 & 0 \end{pmatrix},\]
when $n\geq 3$.

\begin{figure}
\centering
\def\svgwidth{5in}
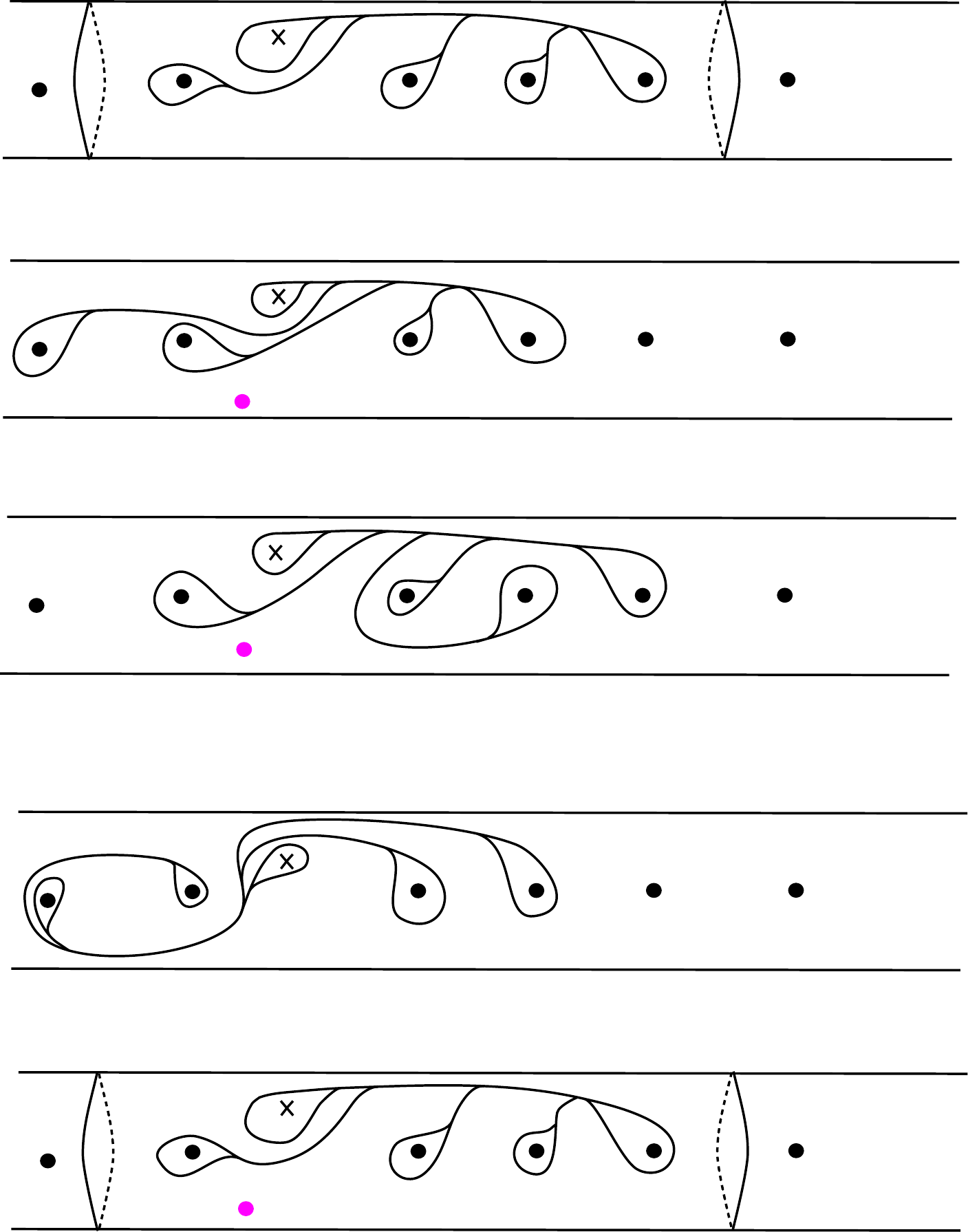 	
\caption{The train track $\tau_n$ is preserved by $\phi_n=g_nh^{-1}$ for $n=1,2$; the pink puncture only appears when $n=2$.  The weights in the picture are used to calculate the matrix $A_n$ associated to $\tau_n$.}
\label{fig:phitau}
\end{figure}

\begin{figure}
\centering
\def\svgwidth{5in}
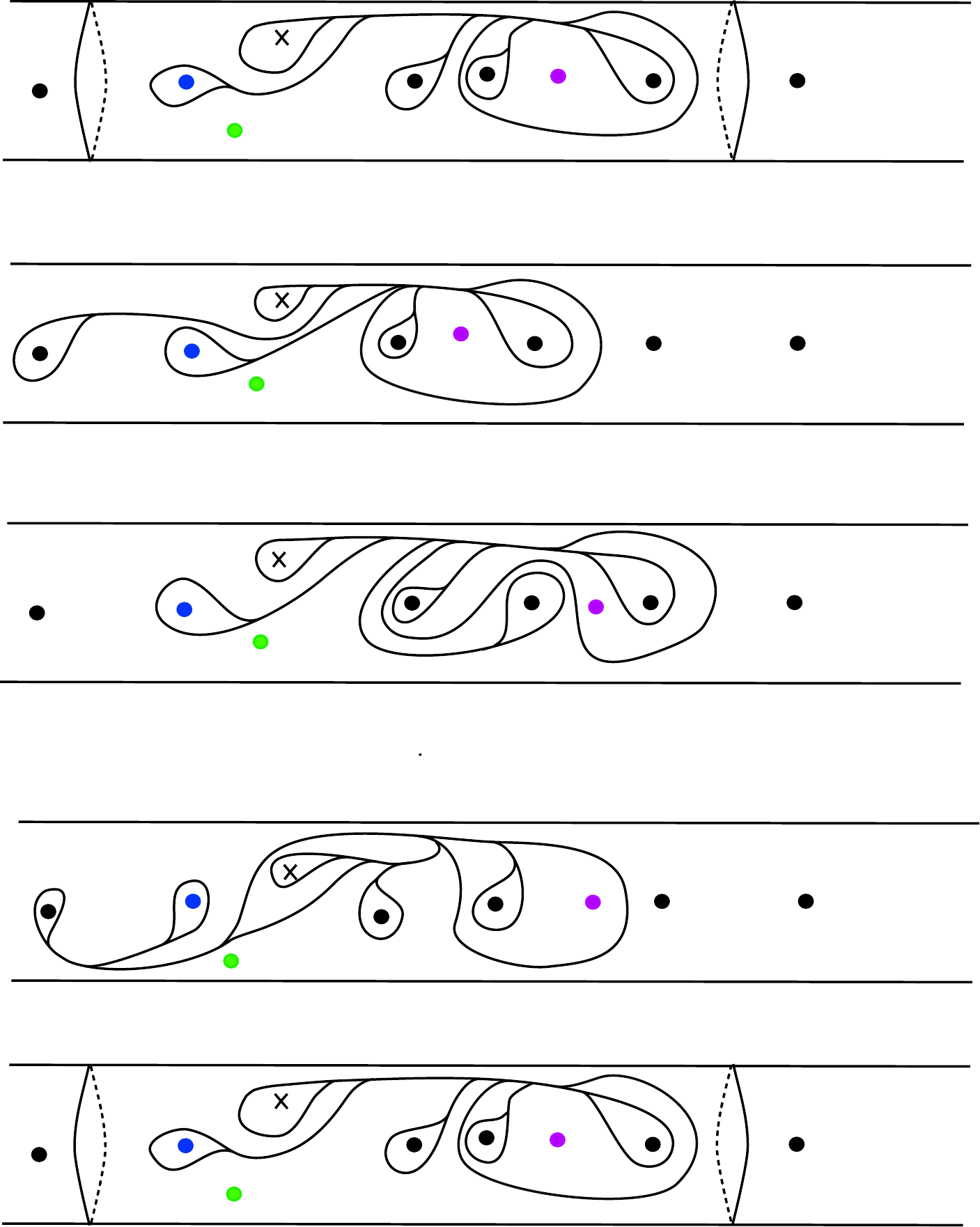
\caption{The train track $\tau_n$ is preserved by $\phi_n=g_nh^{-1}$, when $n\geq 3$.  The weights in the picture show how to calculate the matrix $A_n$ associated to $\tau_n$.  For ease of notation, we often write the weights each branch as a vector in the variable $x_1,\dots, x_5$.  For example, the label $(2\, 1\, 2\, 4\, 5)$ corresponds to the weight $2x_1+x_2+2x_3+4x_4+5x_5$.}
\label{fig:phitaun}
\end{figure}

For each $n$,  $(A_n)^2$ has all positive entries, hence $A_n$ is Perron--Frobenius.  
We conclude that $\phi_n$ is pseudo-Anosov for all $n$ by \cite[Lemma~3.1]{Verberne}.
\end{proof}

While it is not necessary for this paper, it is interesting to note that for all $n\in\N$, the top eigenvalue of $A_n$ is $\frac92+\frac{\sqrt{41}}{2}+\sqrt{\frac{59+9\sqrt{41}}{2}}$, which is associated to a unique irrational lamination on $\Pi_n$ that is carried by $\tau_n$ and fixed by $\phi_n$.

We say that a homeomorphism of $\operatorname{Map}(\Sigma,p)$ is a \textit{pseudo-Anosov shift} if it can be written as the composition of a pseudo-Anosov on a finite-type subsurface containing $p$ and a standard shift.  The results of this section inspire the following questions.

\begin{ques}
 When is the composition of shift maps a pseudo-Anosov shift?
\end{ques}
\begin{ques}
Does every pseudo-Anosov shift act loxodromically on $\mathcal A(\Sigma,p)$?
\end{ques}


\subsection{Geodesic axes}
The proof of Theorem \ref{thm:loxomain} shows that, for each $n$, the sequence $(\alpha_i^{(n)})_{i\in \Z}$ is a $(4,0)$--quasi-geodesic axis for $g_n$ in $\mathcal A(\Sigma,p)$. In this section, we find a geodesic axis for  $g_n$ in $\mathcal A(\Sigma,p)$. 

\begin{thm}
For each $n\in\N$, there is a geodesic axis for $g_n$ in $\mathcal{A}(\Sigma,p)$.
Furthermore, $g_n$ has translation length $1$ on this axis.
\end{thm}

\begin{proof}

As $d_{\mc A(\Sigma,p)}(\gamma,\delta)\leq d_{\mc A(S,p)}(\gamma,\delta)$ for any arcs $\gamma,\delta\in \mathcal A(S,p)$, the image of a geodesic under the inclusion $\mathcal A(S,p)\hookrightarrow \mathcal A(\Sigma,p)$ is still a geodesic. Thus it suffices to construct a geodesic axis for $g_n$ in $\mathcal A(S,p)$.
Toward this goal, define
$$\beta_0^{(n)}=P_s(-1)_o(-2)_o\dots(-n-1)_o(-n-1)_u(-n)_o\dots(-2)_o(-1)_oP_s\in\mathcal A(S,p),$$
and let
\[
\beta_i^{(n)}=g_n^i(\beta_0^{(n)}).
\]
Since the arcs $\beta_i^{(n)}$ are the orbit of $\beta_0^{(n)}$ under $\langle g_n\rangle$ and $g_n$ is a loxodromic isometry, it follows that they form a quasi-geodesic axis in $\mc A(\Sigma, p)$ for $g_n$.  
We will show that $\beta_0^{(n)}$ is disjoint from $\beta_1^{(n)}$, which will prove that $(\beta_i^{(n)})_{i\in \Z}$ is a geodesic axis for $g_n$ and that $g_n$ has translation length one on this axis. 
In fact, it suffices to show that $\beta_1^{(n)}$ does not contain $P_o$ or $k_{o/u}$ for any $k\leq -n-1$, as it then follows that $\beta_0^{(n)}$ is disjoint from $\beta_1^{(n)}$; see Figure \ref{fig:betais}.  

\begin{figure}
\centering
\begin{overpic}[width=4in, trim={.4in 7.5in 1in 1.6in}, clip]{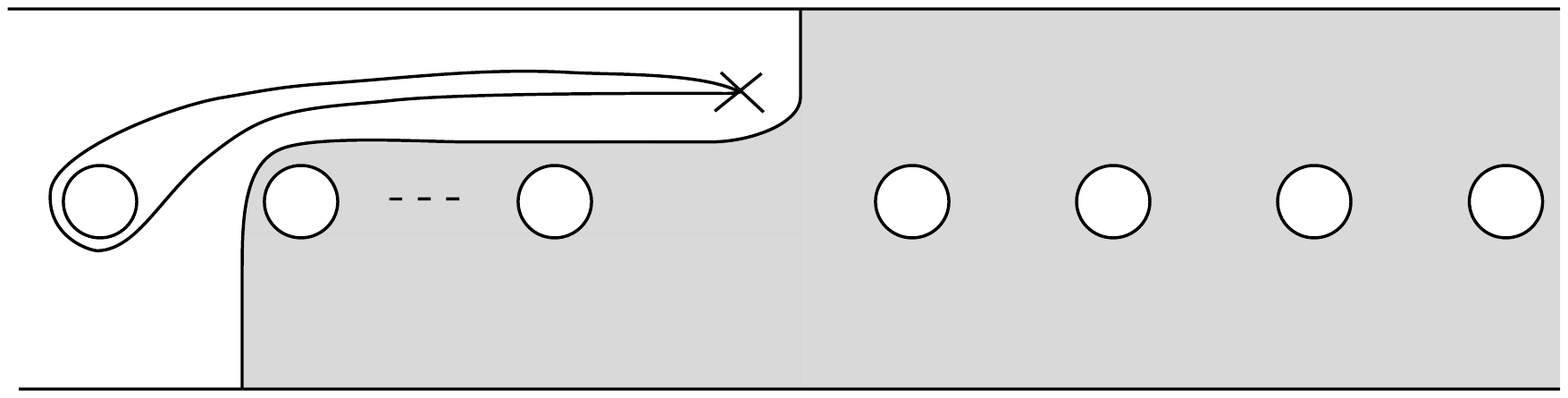}
\put(1,-2){$-n-1$}
\put(4,19){$\beta_0^{(n)}$}
\end{overpic}
\caption{If an arc $\gamma$ does not contain $P_o$ or $k_{o/u}$ for any $k\leq -n-1$, then it must lie in the shaded region of $S$.  In particular, $\gamma$ must be disjoint from $\beta_0^{(n)}$.}
\label{fig:betais}
\end{figure}

Applying $h_1$ to $\beta_0^{(n)}$ yields 
\begin{equation}\label{eqn:h1beta}
h_1(\beta_0^{(n)})=P_s(-1)_o(-2)_o\dots(-n)_o(-n)_u(-n+1)_o\dots(-2)_o(-1)_oP_s.
\end{equation}
  Since all of $h_1(\beta_0^{(n)})$ is to the left of the puncture, the image under $h_2^{(n)}$ (which shifts to the left) will not contain $P_o$.  Moreover, since $P_o$ disagrees with the code for the domain of $h_3^{(n)}$, neither will $\beta_1^{(n)}$.  Thus it remains to show that $\beta_1^{(n)}$ does not contain $k_{o/u}$ for any $k\leq -n-1$.  

Recall that $h_3^{(n)}$ shifts to the right and has shift region $(-\infty,-n-1]\cup[n+1,\infty)$.  Thus, any instance of $k_{o/u}$ with  $k\leq -n-1$ in $\beta^{(n)}_1$ must be the image of  $(k-1)_{o/u}$  in $h_2^{(n)}(h_1(\beta_0^{(n)}))$.  Similarly, such a $(k-1)_{o/u}$ must be the image of $k_{o/u}$ in $h_1(\beta_0^{(n)})$, since $h_2^{(n)}$ shifts to the left.  However, by \eqref{eqn:h1beta}, $h_1(\beta_0^{(n)})$ does not contain $j_{o/u}$ for $j\leq -n-1$, and we conclude that $\beta_1^{(n)}$ does not contain $k_{o/u}$ for any $k\leq -n-1$.
\end{proof}

\subsection{Limit points of the $g_n$}
Since the relative arc graph $\mathcal A(\Sigma,p)$ is a hyperbolic metric space, it has a well-defined Gromov boundary.  This boundary was described by Bavard--Walker \cite{BavardWalker2, BavardWalker}.  
In this section, we describe the limit set of $g_n$ on $\partial\mathcal{A}(\Sigma,p)$ in terms of Bavard and 
Walker's characterization of the boundary.

\subsubsection{The Gromov boundary of $\mathcal A(\Sigma,p)$}
We begin by recalling some definitions from \cite{BavardWalker2, BavardWalker}.
 It is important to mention that in \cite{BavardWalker2, BavardWalker}, the word \textit{loop} is used for what we call an arc in this paper. Any time we mention a result from one of these two papers, we will convert it to our terminology. 
Let $E(\Sigma)$ denote the space of ends of $\Sigma$, which necessarily contains our preferred puncture $p$.

Fix any hyperbolic metric (of the first kind) on $\Sigma$, as in \cite[Theorem 3.0.1]{BavardWalker2}. 
For a fixed lift of $p$ to the universal cover $\mathbb{H}^2$, which necessarily lies on $\partial \mathbb H^2$, there exists a parabolic subgroup $G<\pi_1(\Sigma)$ stabilizing this lift. Define $\widehat{\Sigma}=\mathbb{H}^2/G$ to be the intermediate cover associated to this parabolic subgroup.
The space $\widehat{\Sigma}$ is called the \textit{conical cover} of $\Sigma$.  This cover has  boundary $\mathbb S^1$ and contains a preferred lift $\widehat{p}$ of $p$ which comes from our fixed choice in the universal cover.  Let $\pi\colon \widehat{\Sigma}\to \Sigma$ be the natural quotient map, let $\widehat{\beta}$ be any geodesic ray from $\widehat{p}$ to $\partial \widehat{\Sigma}$, and let $\beta=\pi(\widehat{\beta})$. Thus $\widehat\beta$  has one endpoint on $\widehat p$ and the other endpoint somewhere in $\partial\widehat\Sigma\simeq\mathbb S^1$.  The other endpoint may be a lift of $p$ that is not our chosen $\widehat p$, (a lift of) a point in $E(\Sigma)\setminus\{p\}$, or a point which is neither.  If $\beta$ is simple, then in the first case  $\beta$ is an arc\footnote{As noted above, Bavard--Walker call this a \textit{loop}.}, in the second case $\beta$ is a \textit{short ray}, and in the last case $\beta$ is a \textit{long ray}. 
Equivalently, an embedding $\beta\colon(0,1)\to \Sigma$ is a short ray  if it can be continuously extended to a map $\widehat{\beta}\colon[0,1]\to \Sigma\cup E(\Sigma)$ such that $\widehat{\beta}\vert_{(0,1)}=\beta$, $\widehat\beta(0)=p$, and $\widehat\beta(1)\in E(\Sigma)\setminus\{p\}$ (see \cite[Section 5.1.1]{BavardWalker2}), and $\beta$ is a long ray if it is neither an arc nor a short ray.

Bavard and Walker construct a graph involving all three kinds of rays, which they use to describe the Gromov boundary  $\partial \mathcal A(\Sigma,p)$ of the relative arc graph.
\begin{defn}
The \textit{completed ray graph} $\mathcal{R}(\Sigma,p)$ is the graph whose vertices are isotopy classes of simple arcs, short rays, and long rays and whose edges correspond to homotopically disjoint isotopy classes.
\end{defn}

By definition, $\mathcal{A}(\Sigma,p)$ embeds in $\mathcal{R}(\Sigma,p)$, but the following theorem shows that something stronger is true.  Recall that a clique is a complete graph.

\begin{thm}[{\cite[Theorem 5.7.1]{BavardWalker2}}]
The connected component of $\mathcal{R}(\Sigma,p)$ containing all arcs is quasi-isometric to $\mathcal{A}(\Sigma,p)$.  All other connected components are cliques.
\end{thm}

A particular type of long ray will be important in the description of the Gromov boundary of $\mathcal A(\Sigma,p)$.  A long ray $\beta$ is \textit{$k$--filling} if $k$ is the minimum natural number such that there exists an arc $\beta_0$ and long rays $\beta_1,\dots,\beta_k=\beta$ such that $\beta_i\cap\beta_{i+1}=\emptyset$ for all $i\ge 0$.  In other words, a long ray is $k$--filling if it is distance exactly $k$ from an arc in $\mathcal R(\Sigma,p)$.

\begin{defn}
A long ray $\beta$ is said to be \textit{high-filling} if both of the following hold:
\begin{enumerate}
\item $\beta$ intersects every short ray; and 
\item $\beta$ is not $k$--filling for any $k\in\mathbb{N}$.
\end{enumerate}
\end{defn}

All of the vertices of the connected components that form cliques are high-filling rays; accordingly, such cliques are called \textit{high-filling cliques}.

\begin{thm}[{\cite[Theorem~6.3.1]{BavardWalker2}}]
The set of high-filling cliques is in bijection with the Gromov boundary $\partial \mathcal A(\Sigma,p)$ of the relative arc graph.
\end{thm}

\subsubsection{The limit set of $g_n$}

In \cite[Section 7.1]{BavardWalker2}, Bavard and Walker prove that to any $f\in\MCG(\Sigma,p)$ acting loxodromically on $\mathcal{A}(\Sigma,p)$, there exists an attracting clique of high-filling rays $C^+(f)$ and a repelling clique of high-filling rays $C^-(f)$ in $\mathcal{R}(\Sigma,p)$ that correspond to the attracting and repelling limit points of $f$ in $\partial \mathcal A(\Sigma,p)$, respectively. 
The cliques $C^+(f)$ and $C^-(f)$ have the same (finite) number of vertices, called the \textit{weight of $f$}, denoted by $w(f)$ (see \cite[Definition 7.1.3]{BavardWalker2}).

Following \cite[Example 2.7.1]{BavardWalker}, we have the following lemma.

\begin{lem}
For any $n\geq 1$, the homeomorphism $g_n$ constructed  in Theorem \ref{thm:loxomain} satisfies $w(g_n)=1$.
\end{lem}

\begin{proof}
For notational simplicity, fix $n$ and define $g=g_n$ and $\alpha_i=\alpha_i^{(n)}$.
It suffices to prove that the attracting clique $C^+(g)$ consists of a single high-filling ray.

For each $i$, let $\widehat{\alpha}_i$ be a lift of $\alpha_i$ to the conical cover $\widehat{\Sigma}$.  Then $\widehat\alpha_i$ is a simple geodesic arc with one endpoint on $\widehat{p}$ and the other on some $x_i\in\partial\widehat{\Sigma}$.
By \cite[Lemma~5.2.2]{BavardWalker2}, the set of endpoints of arcs, short rays, and long rays is compact in $\partial\hat\Sigma$, and hence there exists a subsequence $(x_{i_k})$ of $(x_i)$ which limits to a point $x_\infty\in\partial\widehat{\Sigma}$. 
Let $\widehat\alpha_\infty$ be the geodesic ray from $\widehat{p}$ to $x_\infty$ and $\alpha_\infty=\pi(\widehat\alpha_\infty)$, where $\pi\colon\widehat{\Sigma}\to \Sigma$ is the covering map.
The construction of the conical cover $\widehat{\Sigma}$ shows that $\alpha_\infty$ has an infinite-length code with initial segment $\mathring\alpha_{i_k}$ for any $k$.  Since $\ell_c(\mathring\alpha_{i_k})\to\infty$ as $k\to\infty$ and $\mathring\alpha_{i_k}$ has an initial segment equal to $\mathring\alpha_{i_{k-1}}$, this uniquely determines the entire infinite code. 
We claim that $\alpha_\infty$ is a high-filling ray and, moreover, that $\alpha_\infty$ is the sole vertex in $C^+(g)$.

That $\alpha_\infty$ is a ray follows from the fact that the set of endpoints of arcs in $\partial\widehat{\Sigma}$ are isolated \cite[Lemma~5.2.2]{BavardWalker2}.
To see that $\alpha_\infty$ is high-filling, it therefore suffices to show that it intersects every other ray (short or long).
It follows from the same proof as Theorem \ref{thm:startslike} that any ray $\beta$ which is disjoint from $\alpha_{i_k}$ must begin like $\alpha_{i_k-1}$. 
In particular, we see that the first $\ell_c(\mathring{\alpha}_{i_{k-1}})$ characters in a code for $\beta$ must agree with $\mathring{\alpha}_{i_{k-1}}$.
Since $\ell_c(\mathring{\alpha}_{i_{k-1}})<\ell_c(\mathring{\alpha}_{i_{k}})$ for all $k$, taking $k\to\infty$ shows  that any ray $\beta$ which is disjoint from $\alpha_\infty$ must have identical code and hence indeed must be exactly $\alpha_\infty$.

Finally, since $\alpha_\infty$ intersects every other ray, it must, in particular, intersect any other high-filling ray.  By \cite[Lemma 2.7.8]{BavardWalker}, the connected component of any high-filling ray is a clique of high-filling rays, and thus $\alpha_\infty$ is its own connected component in $\mathcal{R}(\Sigma,p)$.
Hence $w(g)=1$, completing the proof.
\end{proof}

We close this section with the remark that if the weight of the limit points of $g_n$ were not all the same, then Bavard--Walker give a method for constructing non-trivial quasimorphisms \cite[Theorem 9.1.1]{BavardWalker2}. 
However, since this is not the case, we must use a different method for showing that the space of quasimorphisms is infinite-dimensional, which is related to Bavard's original proof for the arc graph from \cite{Bavard}.  We do so in Section \ref{sec:qms}.


\section{Infinite-type quasimorphisms}\label{sec:qms}

A \textit{quasimorphism} of a group $G$ is a map $f\colon G\to\R$ such that there exists a real constant $C$ for which $|f(ab)-f(a)-f(b)|\le C$ for all $a,b\in G$.
The set of quasimorphisms forms a vector space $V(G)$ over $\R$, and, moreover,  bounded functions and homomorphisms both form subspaces of $V(G)$.
Let $\widetilde{QH}(G)$ denote the quotient of $V(G)$ by the subspace spanned by bounded functions and homomorphisms. We call $\widetilde{QH}(G)$ the \textit{space of quasimorphisms} of $G$.
The goal of this section is to use the elements constructed in Theorem \ref{thm:loxomain} to prove Theorem \ref{thm:infiniteqms}, which we restate for the convenience of the reader.

\QM*

Several of the ideas in this section closely follow  the strategy and ideas of Bavard \cite{Bavard}, though the production of the elements which give rise to our quasimorphisms differs.
We begin by studying a specific subclass of arcs with the goal of describing a particular homotopy operation and intersection pairing on them. We then use this intersection pairing and a theorem of Bestvina--Fujiwara to prove the theorem.

\subsection{An intersection pairing on symmetric arcs and first properties}\label{sec:intpairing}
Recall that our surface $\Sigma$ has an embedding of the biinfinite flute surface $S$ such that $\Sigma\setminus S$ has infinitely many connected components, a countable collection of which are homeomorphic to a fixed surface $\Sigma_0$.
Moreover, the complement of each arc $\alpha_i$ separates $\Sigma$ into two components, one of which is homeomorphic to $\intr(\Sigma_0)$, the interior  of the fixed surface $\Sigma_0$.
Using $\Sigma_0$, we define $\mathcal{SA}$ to be the set of simple, symmetric arcs $\delta$ (see Definition \ref{defn:symmetric})  such that $\Sigma\setminus\delta$ has two connected components, one of which is homeomorphic to $\intr(\Sigma_0)$.
Notice that $\alpha^{(n)}_i\in\mathcal{SA}$ for all $i\in\Z$, $n\in\N$ and that $\mathcal{SA}$ is a $\MCG(\Sigma,p)$--invariant subset of the set of all arcs.

Since $p$ is isolated, we again fix the small once-punctured disk $D_p$ containing $p$ as in Section \ref{sec:highways}.
This disk is homeomorphic to the closed unit disk minus an interior point.
As in that section, given any element $\delta$ in $\mathcal{SA}$, we put $\delta$ in standard position so that its intersections with $\partial D_p\approx S^1$ are all transverse.
Let $x_s$ and $x_t$ be the initial and terminal point where $\delta$ intersects $\partial D_p$. Let $\lambda_0$ and $\rho_0$ be the subsegments of $\delta$ which connect $x_t$ and $x_s$  to $p$, respectively.

We will modify $\delta$ to form a particular simple closed curve as follows.  We can replace  $\rho_0\cup\lambda_0$ with either $a^+$ or $a^-$, as shown  in Figure \ref{fig:punctureddisk}, forming two distinct simple closed curves, $\delta^+$ and $\delta^-$, respectively.  One of these two curves bounds a surface homeomorphic to $\intr(\Sigma_0)$; in Figure \ref{fig:punctureddisk}, this curve is $\delta^-$. Fixing a hyperbolic metric on the surface, we let $B_\delta$ be the geodesic representative of this curve.

\begin{figure}
\centering
\begin{overpic}[width=3in, trim={2in 7in 2in .5in}, clip]{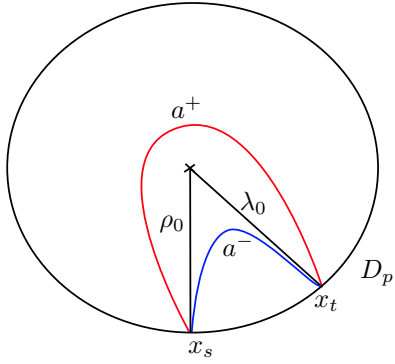}
\put(83,21){$D_p$}
\put(48,30){$\rho_0$}
\put(62,33){$\lambda_0$}
\put(50,49){$a^+$}
\put(59,25){$a^-$}
\put(53,8){$x_s$}
\put(75,16){$x_t$}
\end{overpic}
\caption{A schematic of $D_p$ and the two arcs $a^-$ and $a^+$.}
\label{fig:punctureddisk}
\end{figure}

We now choose the homotopy representative of $\delta$ that will allow us to define an intersection pairing.  
As $\delta$ is symmetric, there exists an arc
\[
\delta'=r_\delta B_\delta r_\delta^{-1}
\] in the homotopy class of $\delta$, where $r_\delta$ is a simple ray from the puncture to $B_\delta$ that intersects $B_\delta$ only at its endpoint.  Intuitively, one can  think of the arc $\delta'$  as being constructed from $\delta$ by ``zipping" the initial and terminal portions of the arc together for as long as they fellow travel to form $r_\delta$.   Particular examples of $r_\delta$ and $\delta'$ are given in Figure \ref{fig:intpairingonback}.
\begin{figure}
\centering
\begin{overpic}[width=3in]{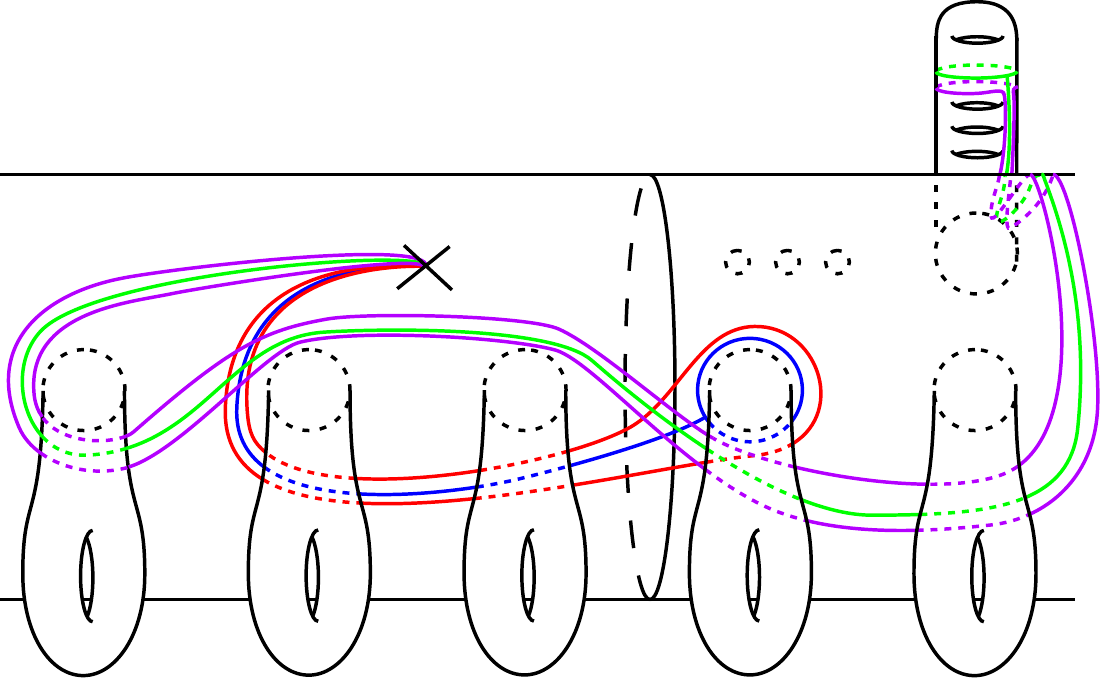}
\end{overpic}
\caption{Two examples of arcs in $\mathcal{SA}$ (in the case that $\Sigma_0$ is a one-holed torus) and their corresponding zippings. The purple and red arcs are zipped to the green and blue arcs, respectively.}
\label{fig:intpairingonback}
\end{figure}

We are now in a position to define the intersection pairings. 
\begin{defn}
We define a map $I^\pm\colon \mathcal{SA}\times\mathcal{SA}\to\Z_{\ge 0}$ as follows.
Let $I^+(\delta,\gamma)$ be the number of positively oriented intersections between minimal position representatives for the homotopy classes of $r_\delta$ and $r_\gamma$ that do not occur on  $B_\delta$ or $B_\gamma$. Here we require that the homotopy fixes the puncture and fixes each of $B_\delta$ and $B_\gamma$ setwise (though not necessarily pointwise).  We define  $I^-(\delta,\gamma)$ similarly using negative intersections.
\end{defn}
Notice that $I^+(-,-)$ and $I^-(-,-)$ are not necessarily symmetric in their arguments.  However, it is straightforward to verify that the relationship
$$I^+(\delta,\gamma)=I^-(\gamma,\delta)$$
holds for any $\delta,\gamma\in\mathcal{SA}$.

\begin{figure}
\centering
\begin{overpic}[width=5in]{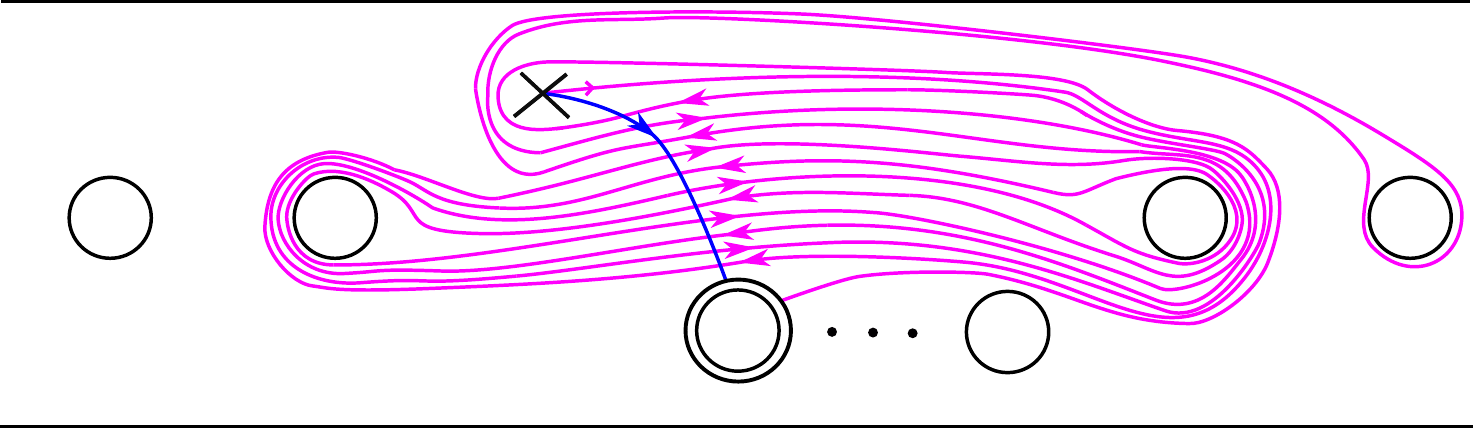}
\put(5,-3){$-2$}
\put(21,-3){$-1$}
\put(49,-3){$0$}
\put(66,-3){$n+1$}
\put(78,-3){$n+2$}
\put(94,-3){$n+3$}
\end{overpic}
\caption{The oriented arc $r_{\alpha_0^{(n)}}$ is shown in blue and the oriented arc $r_{\alpha_2^{(n)}}$ is shown in pink. Each line of $r_{\alpha_0^{(n)}}$ and $r_{\alpha_2^{(n)}}$ represents two strands of $\alpha_0^{(n)}$ and $\alpha_2^{(n)}$, respectively.}
\label{fig:alpha2dagger}
\end{figure}
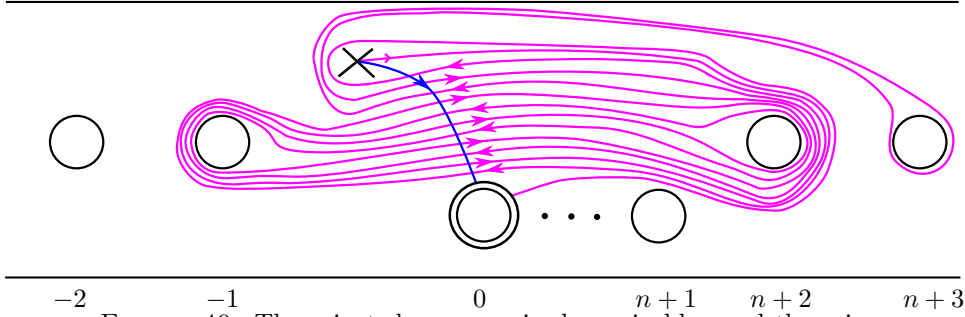

For the remainder of the section we will fix an $n\in\N$ and use the notation that $g=g_n$, $\alpha_i=\alpha_i^{(n)}$, and $\varphi=\varphi_n$ is the ``starts like" function from Section \ref{sec:loxo}.

\begin{ex}\label{ex:intpairing}
One can readily compute from Figure \ref{fig:alpha2dagger} that we have the following relations:
$$5=I^-(\alpha_0,\alpha_2)=I^+(\alpha_2,\alpha_0)\quad\textrm{and} \quad 6=I^+(\alpha_0,\alpha_2)=I^-(\alpha_2,\alpha_0).$$
These calculations will be relevant later in the section.
\end{ex}

We now collect some useful properties of the intersection pairing and its interaction with $\MCG(\Sigma,p)$.
These preliminary facts are inspired by Bavard \cite{Bavard}, where  similar statements are shown in Bavard's context.

\begin{lem}\label{lem:MCGinv}
The intersection pairing is mapping class group invariant.  That is, for any $\delta,\gamma\in\mathcal{SA}$ and any $f\in\MCG(\Sigma,p)$
$$I^{\pm}(f(\delta),f(\gamma))=I^{\pm}(\delta,\gamma).$$
\end{lem}
\begin{proof}
This is immediate from the fact that $\MCG(\Sigma,p)$ is orientation preserving and preserves $\mathcal{SA}$.
\end{proof}

Recall that in Section \ref{sec:loxo}, we defined the ``starts like" function
$$\varphi \colon \mathcal A(S,p)\to \Z_{\geq 0},$$
by setting $\phi(\delta)$ equal to the largest $i\geq 0$ such that $\delta$ starts like $\alpha_i$.  We now extend $\varphi$ to all of $\mathcal{A}(\Sigma,p)$ by setting $\phi(\delta)=0$ if $\delta$ does not have a homotopy representative that is contained in $S$.  We continue to call this extension $\phi$.

\begin{lem}\label{lem:startslikecontra}
If $\delta,\gamma\in\mathcal{SA}$ are arcs such that
$2+\varphi(\delta)\le\varphi(\gamma),$
then
$$6\le I^{-}(\gamma,\delta).$$
\end{lem}
\begin{proof}
By the mapping class group invariance of Lemma \ref{lem:MCGinv}, the quantities $I^\pm(\alpha_i,\alpha_j)$ depend only on $j-i$.
As $\alpha_{i+1}$ starts like $\alpha_i$, the pairing $I^-(\alpha_i,\alpha_j)$ must be monotonically increasing in the difference $j-i$.
In particular, if $2+i\le j,$
then 
$$6= I^{-}(\alpha_2,\alpha_0)\le I^{-}(\alpha_j,\alpha_i).$$
If  $\phi(\delta) = i$ and $\phi(\gamma)= j$, then  $\delta$ starts like $\alpha_i$ and $\gamma$ starts like $\alpha_j$, so the arcs $\gamma$ and $\delta$ must have at least as many negatively oriented intersections as $\alpha_i$ and $\alpha_j$.  Thus we have 
\[
6\le I^{-}(\alpha_j,\alpha_i) \leq  I^{-}(\gamma,\delta),
\]
and the result follows.   
\end{proof}


\subsection{Production of ``infinite-type" quasimorphisms}
We now use the intersection pairing from the previous subsection to show that the elements $g_n$ give rise to non-trivial quasimorphisms.
For this, we need the following theorem of Bestvina and Fujiwara. We explain the two conditions on $h_1$, $h_2$ after the statement.

\begin{thm}[{\cite[Theorem 1]{BestvinaFujiwara}}]\label{thm:BestFuj}
Suppose that a group $G$ has a non-elementary action by isometries on a $\delta$--hyperbolic graph $X$.
If there exist independent loxodromic elements $h_1,h_2\in G$ such that $h_1\not\sim h_2$, then the space of quasimorphisms is infinite dimensional.
\end{thm}

Two loxodromic isometries $h_1$ and $h_2$ are \textit{independent} if their limit sets in the boundary $\partial X$ of $X$ are disjoint.  For the second condition, fix constants $K\geq 1$ and $K'\geq 0$ so that $h_i$ has a $(K,K')$--quasi-geodesic axis $\ell_i$ in $X$ for $i\in\{1,2\}$. A fundamental property of $\delta$--hyperbolic spaces is that there exists $B = B(K,K',\delta)$ such that any two finite $(K,K')$–quasi-geodesics with common endpoints are within distance $B$ of each other. Define an equivalence relation on elements $h_1,h_2\in G$ so that $h_1\sim h_2$ if the following holds:  for any arbitrarily long segment $L$  of $\ell_1$, there exists an $f\in G$ such that $f(L)$ is contained in the $B$--neighborhood of $\ell_2$ and the map $f\colon L\to f(L)$ is orientation-preserving with respect to the $h_i$--orientation on $\ell_i$ for $i\in\{1,2\}$.
For the definition of the $h_i$--orientation on $L$ and $f(L)$, see \cite[Page 72]{BestvinaFujiwara}.

We now recall some arguments from \cite[Section 4.3]{Bavard} which, when adapted into our language, show that $g\not\sim g^{-1}$ for our loxodromic isometries $g=g_n$.
Fix $B\geq 1$ to be the constant defined above for all $(4,0)$--quasi-geodesics in $\mathcal{A}(\Sigma,p)$. 
Let $\ell=\{g^i(\alpha_0)\}_{i\in\Z}$, so that $\ell$ is a $(4,0)$--quasi-geodesic axis of $g$ by Theorem \ref{thm:gnlox}. We then have the following statements that are similar to \cite[Lemmas 4.6 \& 4.7]{Bavard}. We supply the proofs for the reader's convenience.

\begin{lem}\label{lem:Bavard}
Let $L$ be a subpath of $\ell$ from $\alpha_i$ to $\alpha_j$ for $0<i<j$. Let $f\in\MCG(\Sigma,p)$  be such that $d_{\mathcal{A}(\Sigma,p)}(\alpha_i,f(\alpha_j))\le B$ and such that $f (L)\subset N_B(\ell)$ with the opposite orientation.
If $j-i>8B+3$, then there exists some $k$ such that $i \leq k < j$ and $\varphi(f(\alpha_{k+2}))\le\varphi(f(\alpha_k))-2$.
\end{lem}
\begin{proof}

Since $d_{\mathcal{A}(\Sigma,p)}(f(\alpha_j), \alpha_i)\le B$, we conclude by Lemma~\ref{cor:dist} that $\varphi(f (\alpha_j))\le i+B$. Since $f(L)\subset N_B(\ell)$ with the opposite orientation, we may apply \cite[Lemma 4.4]{Bavard} to $L$ and the reverse of $f(L)$ to conclude that $d_{\mathcal{A}(\Sigma,p)}(\alpha_j,f(\alpha_i))\le 3B$.  Note that \cite[Lemma 4.4]{Bavard} is stated for geodesics, but the exact same proof goes through for quasi-geodesics. Again applying Lemma~\ref{cor:dist}, we see that $\varphi(f(\alpha_i))\ge j-3B$.

Suppose towards a contradiction that $\varphi(f(\alpha_{k+2}))>\varphi(f(\alpha_k))-2$ for all $1 \leq k < j$.  Equivalently, $\varphi(f(\alpha_{k+2}))\ge\varphi(f(\alpha_k))-1$ for every $k$.
Then, if $j-i$ is even,
$$i+B\ge \varphi(f(\alpha_j))\ge\varphi(f(\alpha_i))-\frac{j-i}{2}\ge j-3B-\frac{j-i}{2},$$ where the second inequality is obtained by applying $\varphi(f(\alpha_{k+2}))\ge\varphi(f(\alpha_k))-1$ repeatedly starting with $j = k+2$.
If $j-i$ is odd, then by the same reasoning we have
$$i+B\ge \varphi(f(\alpha_j))\ge\varphi(f(\alpha_{i+1}))-\frac{j-i-1}{2}.$$ 
Since 
$$|\varphi(f(\alpha_i))-\varphi(f(\alpha_{i+1}))|\le d(f(\alpha_i),f(\alpha_{i+1}))=d(\alpha_i,\alpha_{i+1})=2,$$
it follows that
$$i+B\ge\varphi(f(\alpha_{i+1}))-\frac{j-i-1}{2}\ge\varphi(f(\alpha_{i}))-\frac{3}{2}-\frac{j-i}{2}\ge j-3B-\frac{3}{2}-\frac{j-i}{2}.$$
Hence we conclude that, in either case, 
$$4B+\frac{3}{2}\ge \frac{j-i}{2},$$
which contradicts that $j-i>8B+3$.
Thus there must be some $k$ for which $\varphi(f(\alpha_{k+2}))\le\varphi(f(\alpha_k))-2$, as required.
\end{proof}

\begin{prop}\label{prop:nosim}
For any segment $L$ of $\ell$ whose length is greater than $32B+12$ and any $f\in\MCG(\Sigma,p)$, if $f(L)\subset N_B(\ell)$ then $f(L)$ has the same orientation as $\ell$.
\end{prop}

\begin{proof}
After possibly increasing the length of $L$, we may assume that $L$ is a subpath of $\ell$ from $\alpha_i$ to $\alpha_j$ for some $i<j$.  Since the length of $L$ is greater than $32B+12$ and $\ell$ is a $(4,0)$--quasi-geodesic edge path by Theorem \ref{thm:gnlox}, we have that $j-i>8B+3$.  By precomposing $f$ with a suitable power of $ g$, we can and do assume that $i,j> 0$.

Assume for contradiction that $f(L)$ has the opposite orientation as $\ell$.
By Example \ref{ex:intpairing} and Lemma \ref{lem:MCGinv} we have that
\begin{equation}\label{eqn:I-02}
I^{-}(f(\alpha_{k}),f(\alpha_{k+2}))=I^{-}(\alpha_{k}, \alpha_{k+2})=I^{-}(\alpha_0,\alpha_{2})=5
\end{equation}
for all $k\in\Z$.
On the other hand, by Lemma \ref{lem:Bavard} there is some fixed index $i\leq k<j$ for which
\begin{equation}\label{eqn:badstartslike}
2+\varphi(f(\alpha_{k+2}))\le \varphi(f(\alpha_{k})).
\end{equation}
Applying Lemma \ref{lem:startslikecontra} to  $f(\alpha_k)$ and $f(\alpha_{k+2})$ for this index $k$ shows that
$$6\leq I^{-}(f (\alpha_{k}),f(\alpha_{k+2}))=I^{-}( \alpha_{k}, \alpha_{k+2})=I^{-}(\alpha_0, \alpha_{2}).$$
However, this contradicts \eqref{eqn:I-02}, and so we conclude that $f(L)$ has the same orientation as $L$.
\end{proof}

\noindent Proposition~\ref{prop:nosim} implies that $g \not\sim g^{-1}$ since the axis $\ell$ has opposite orientations for $g$ and $g^{-1}$. Additionally, since conjugate elements are equivalent under the relationship $\sim$, we have the following immediate corollary of Proposition \ref{prop:nosim}.

\begin{cor}\label{cor:nosim}
For fixed $n\in\N$, the loxodromic elements $g=g_n$ have the property that $g\not\sim hg^{-1}h^{-1}$ for any $h\in\MCG(\Sigma,p)$.
\end{cor}

\noindent With this in hand, we can prove the main result of this section.

\begin{proof}[Proof of Theorem~\ref{thm:infiniteqms}]
Fix $n\in\N$, and continue to use the notation that $g=g_n$.
By Theorem \ref{thm:BestFuj} and Corollary \ref{cor:nosim} it  suffices to show that there exists an $h\in\MCG(\Sigma,p)$ such that $g$ and $hg^{-1}h^{-1}$ are independent loxodromic elements. For this we can use any $h \in \MCG(\Sigma,p)$ which does not fix the limit set of $g$ (which is the same as the limit set of $g^{-1}$). For example, fix any finite-type subsurface $\Pi_0$ with boundary of sufficient complexity which contains the puncture $p$, and take any pseudo-Anosov $h\in\MCG(\Pi_0,p)$. Extending $h$ by the identity outside of $\Pi_0$, we may consider $h$ as an element of $\MCG(\Sigma,p)$. By Lemma \ref{lem:AsqiembsAsigma}, there is a $(2,0)$--quasi-isometric embedding $\iota\colon\mathcal{A}(\Pi_0,P)\hookrightarrow\mathcal{A}(\Sigma,p)$. Since $h$ is loxodromic with respect to the action of $\MCG(\Pi_0,p)$ on  $\mc A(\Pi_0,p)$, it is therefore also loxodromic with respect to the action of $\MCG(\Sigma,p)$ on  $\mathcal{A}(\Sigma,p)$. In addition, out of all such pseudo-Anosovs, there is a choice of $h$ whose limit points are different from the limit points $\xi_{\pm}$ of $g$. (Note that the quasi-isometric embedding $\iota$ ensures that two pseudo-Anosovs with distinct limit points in the boundary of $\mc A(\Pi_0,p)$ also have distinct limit points in the boundary of $\mathcal{A}(\Sigma,p)$.) Therefore, the limit points $h(\xi_{\pm})$ of $hg^{-1}h^{-1}$ are distinct from those of $g$.
In particular, $g$ and $hg^{-1}h^{-1}$ are independent loxodromic elements, as required.
\end{proof}


\section{Convergence to a geodesic lamination}\label{sec:lamination}

The goal of this section is to prove Theorem \ref{thm:geodlam}, which we restate for the convenience of the reader. As in \cite{Saric}, we equip $\Sigma$ with its unique conformal hyperbolic metric. Additionally, we require that with this metric, $\Sigma$ is equal to its convex core, which is equivalent to eliminating hyperbolic funnels and half-planes in $\Sigma$.  This is necessary in order to consider geodesic laminations on an infinite-type surface; see~\cite{Saric}.
 
\GL*


\subsection{Geodesic laminations}
We begin by reviewing some facts about geodesic laminations on infinite-type surfaces.  For a complete treatment of the subject, we refer the reader to \cite{Saric}.
\begin{defn}
A \textit{geodesic lamination} $\lambda$ on $\Sigma$ is a foliation of a closed subset of $\Sigma$ by complete geodesics.  
\end{defn}

Fix a locally finite geodesic pants decomposition $\{P_n\}$ of $\Sigma$ and a train track $\Theta$ on $\Sigma$ constructed as in \cite[Section~4]{Saric}.  Denote by $\tilde\Theta$ the lift of $\Theta$ to $\tilde\Sigma$, the universal cover of $\Sigma$.  An \textit{edge path} of $\tilde \Theta$ is a finite, infinite, or bi-infinite sequence of edges of $\tilde\Theta$ such that consecutive edges meet smoothly at each vertex.  Every bi-infinite edge path has two distinct accumulation points on $\partial_\infty\tilde\Sigma$ by \cite[Proposition~4.5]{Saric}.

Given a bi-infinite edge path $\tilde\gamma$ of $\tilde \Theta$, let $G(\tilde\gamma)$ be the geodesic of $\tilde \Sigma$ whose endpoints on $\partial_\infty\tilde\Sigma$ are the two distinct endpoints of $\tilde\gamma$.  A geodesic $g$ of $\tilde \Sigma$ is \textit{weakly carried} by $\tilde\Theta$ if there exists a bi-infinite path $\tilde\gamma$ in $\tilde\Theta$ such that $G(\tilde\gamma)=g$.

We now gather various results from \cite{Saric} which will be useful in what follows.  These are all analogous to the situation for finite-type surfaces (see, for example, \cite{Bonahon, PennerHarer, Thurston}).  The first gives a correspondence between geodesics weakly carried by $\tilde\Theta$ and bi-infinite edge paths. 
\begin{prop} [{\cite[Proposition~4.7]{Saric}}] \label{prop:pathtogeod}
There is a one-to-one correspondence between bi-infinite edge paths of $\tilde\Theta$ and geodesics weakly carried by $\tilde\Theta$.
\end{prop}

Let $\gamma$ be an edge path in $\Theta$ and $\tilde\gamma$  a single component of the lift of $\gamma$ to $\tilde\Sigma$.
Then, as above, we construct a geodesic $G(\tilde\gamma)$ with the same endpoints as $\tilde\gamma$ and denote by $G(\gamma)$ its projection to $\Sigma$.
We then say that the geodesic $G(\gamma)$ is \textit{weakly carried by $\Theta$} if $G(\tilde\gamma)$ is weakly carried by $\tilde \Theta$.  A geodesic lamination $\lambda$ on $\Sigma$ is \textit{weakly carried by $\Theta$} if every geodesic of $\lambda$ is weakly carried by $\Theta$.
The next result gives a correspondence between geodesic laminations and certain families of bi-infinite edge paths.

\begin{prop} [{\cite[Proposition~4.11]{Saric}}] \label{prop:pathtolamin}
The set of geodesic laminations on $\Sigma$ that are weakly carried by $\Theta$ is in one-to-one correspondence with the families $\Gamma$ of bi-infinite edge paths of $\tilde\Theta$  that satisfy:
	\begin{itemize}
	\item any two bi-infinite edge paths $\gamma$ and $\gamma'$ in $\Gamma$ do not cross; and 
	\item if $\gamma$ is a bi-infinite edge path such that for any finite edge subpath there is a bi-infinite edge path in $\Gamma$ that contains it, then $\gamma\in \Gamma$.
	\end{itemize}
\end{prop}

The following proposition describes when a sequence of geodesics carried by $\Theta$ converge.  Let $G(\tilde\Sigma)$ be the set of unoriented geodesics on $\tilde\Sigma$.

\begin{prop} [{\cite[Proposition~4.9]{Saric}}] \label{prop:convergence}
Let $f_n,f\in G(\tilde \Sigma)$ be weakly carried by a train track $\tilde\Theta$, and denote by $\tilde\gamma_n,\tilde\gamma$ the corresponding bi-infinite edge paths in $\tilde\Theta$.  Then $f_n$ converges to $f$ as $n\to\infty$ if and only if for each finite subpath $\tilde\gamma'$ of $\tilde\gamma$  there exists $n_0\geq 0$ such that $\tilde\gamma'$ is contained in the path $\tilde\gamma_n$ for all $n\geq n_0$.
\end{prop}

The final proposition we will need shows that every geodesic lamination is weakly carried by a train track.
\begin{prop} [{\cite[Proposition~4.12]{Saric}}]
Every geodesic lamination $\lambda$ on a hyperbolic surface $X$ is weakly carried by a train track $\Theta$ that is constructed as above starting from a fixed locally finite geodesic pants decomposition. 
\end{prop}

\subsection{Construction of the train track $\Theta$}  We now construct a train track on our surface $\Sigma$. In the next section, we will define our simple closed curve $c_0$ and show that $g^i_n(c_0)$ converges as $i\to\infty$ to a geodesic lamination  that is weakly carried by this train track.

First, fix a simple closed curve $\gamma_p'$ on $S$ so that the bounded component of $S\setminus \gamma_p$ is the pair of pants  with cuffs $\gamma_p'$, $p_{-1}$, and $p$ (recall that punctures are allowed to be cuffs).  Let $\gamma_p$  be the image of $\gamma_p'$ under the embedding $S\hookrightarrow \Sigma$.   We construct our first pair of pants $Q_p$ to have cuffs $\gamma_p$, $B_{-1}$, and $p$. Next choose two simple closed separating curves $\gamma_{-1,l}$ and $\gamma_{-1,r}$ on the embedded copy of $S$ in $\Sigma$ so that one component of $\Sigma\setminus (\gamma_{-1,l}\cup\gamma_{-1,r})$ contains two pairs of pants: $Q_p$ and a pair of pants $Q_{-1}$ with cuffs $\gamma_{-1,l},\gamma_{-1,r}$, and $\gamma_p$.  Let $\gamma_{-1,l}$ be the curve which bounds the connected component of $\Sigma\setminus (\gamma_{-1,l}\cup\gamma_{-1,r})$ containing $B_{-2}$.  For each $i\in \Z\setminus\{-1\}$, choose two  simple closed separating curves $\gamma_{i,l}$ and $\gamma_{i,r}$ on the embedded copy of $S$ in $\Sigma$ so that one component of $\Sigma \setminus(\gamma_{i,l}\cup\gamma_{i,r})$ contains a pair of pants, $Q_i$, with cuffs $\gamma_{i,l},\gamma_{i,r}$, and $B_i$.  Similarly, let $\gamma_{i,l}$ be the curve which bounds the connected component of $\Sigma\setminus (\gamma_{-1,l}\cup\gamma_{-1,r})$ containing $B_{i-1}$.

Next, for all $i\in\Z$, consider the component $C_i$ of $\Sigma\setminus(\gamma_{i,r}\cup\gamma_{i+1,l})$ which does not contain $\gamma_{i,l}$ for any $i$.  If $C_i$ is a cylinder, that is, if $\gamma_{i,r}$ is homotopic to $\gamma_{i+1,l}$, then modify $\gamma_{i,r}$ and $Q_i$ so that $\gamma_{i,r}=\gamma_{i+1,l}$.  If $C_i$ is a pair of pants, let $Q_{i,r}=C_i$.  If $C_i$ is neither a pair of pants nor a cylinder, fix a simple closed curve $\delta_i$ so that one component of $C_i\setminus \delta_i$ is a pair of pants $Q_{i,r}$ with cuffs $\gamma_{i,r},  \gamma_{i+1,l}$, and $\delta_i$.  See Figure \ref{fig:pantsdecomp}.

\begin{figure} 
\centering
\begin{overpic}[width=4in, trim={1.7in 7.55in 1in 1.65in},clip]{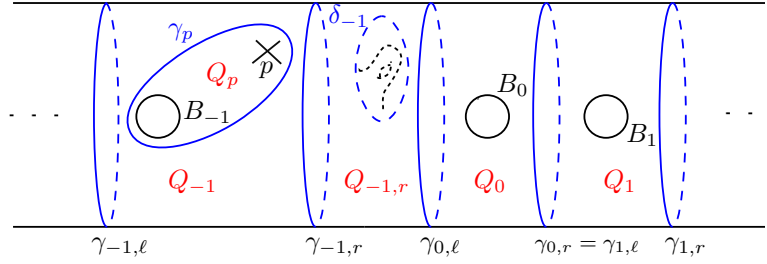}
\put(22,26){\textcolor{blue}{$\gamma_p$}}
\put(43,27.5){\textcolor{blue}{$\delta_{-1}$}}
\put(12,-2){$\gamma_{-1,\ell}$}
\put(40,-2){$\gamma_{-1,r}$}
\put(55,-2){$\gamma_{0,\ell}$}
\put(70,-2){\footnotesize$\gamma_{0,r}=\gamma_{1,\ell}$}
\put(87,-2){$\gamma_{1,r}$}
\put(24,15){$B_{-1}$}
\put(65,18.5){$B_{0}$}
\put(82,12){$B_{1}$}
\put(27,20){\textcolor{red}{$Q_p$}}
\put(22,6){\textcolor{red}{$Q_{-1}$}}
\put(45,6){\textcolor{red}{$Q_{-1,r}$}}
\put(62,6){\textcolor{red}{$Q_{0}$}}
\put(79,6){\textcolor{red}{$Q_{1}$}}
\put(34,21){$p$}
\end{overpic}
\caption{The cuffs of the pants decomposition are in blue and black, and the pairs of pants are labeled in red.  Anything dotted occurs on the back of the embedded copy of $S$ in $\Sigma$.}
\label{fig:pantsdecomp}
\end{figure}

We replace each cuff with a geodesic representative of the same homotopy class; by a slight abuse of notation, we continue to call the resulting geodesic pairs of pants $Q_i$, $Q_p$, and $Q_{i,r}$.  This is a locally finite geodesic pants decomposition of a subsurface of the embedded copy of $S$ in $\Sigma$, and we extend it to a locally finite geodesic pants decomposition of $\Sigma$, which we denote $\mc Q$.

We next construct the train track $\Theta$ as follows. The specific connectors on the front of the surface for all $Q_i$, $Q_p$, and $Q_{i,r}$ are as in Figure \ref{fig:cuffs}.  For all other pairs of pants, we choose any connectors that satisfy the conditions of \cite{Saric}.  

We code a subset of the connectors and cuff segments of $\Theta$ which lie on the front of $S$ in the following way (see Figure \ref{fig:cuffs}).
\begin{itemize}
\item In each pair of pants $Q_i$ with $i\neq -1$, denote the connectors from $\gamma_{i,l}$ and $\gamma_{i,r} $ to $B_i$ by $i_L$ and $i_R$, respectively.  The two segments of the cuff $B_0$ are denoted $0_o,0_u$, as in our standard code.
\item In each pair of pants $Q_{i,r}$, denote the single connector on the front of $S$ by $i_{RR}$. 
\item In the pair of pants $Q_p$, denote the connector which has both endpoints on $\gamma_p$ by $P_u$.  This connector divides the cuff $\gamma_p$ into two segments, denote these by  $P_o$ and $(-1)_u$, as in Figure \ref{fig:cuffs}.  We will not code the second connector or the segments on the cuff $B_{-1}$ in this pair of pants.
\item In the pair of pants $Q_{-1}$, denote the connectors from $\gamma_{-1,l}$ and $\gamma_{-1,r}$ to $\gamma_p$ by $(-1)_L$ and $(-1)_R$, respectively.
\end{itemize}

\begin{figure} 
\centering
\begin{overpic}[width=4in, trim={1.7in 7.55in 1in 1.65in},clip]{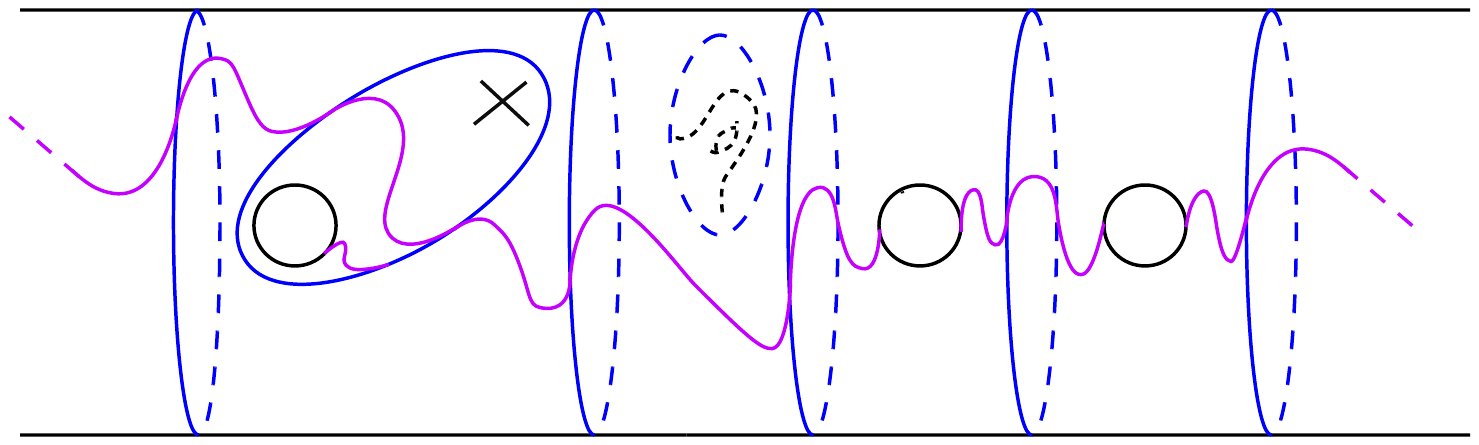}
\put(22,25){$\scriptstyle\gamma_p$}
\put(12,-1){$\scriptstyle\gamma_{-1,\ell}$}
\put(40,-1){$\scriptstyle\gamma_{-1,r}$}
\put(55,-1){$\scriptstyle\gamma_{0,\ell}$}
\put(70,-1){$\scriptstyle\gamma_{0,r}=\gamma_{1,\ell}$}
\put(87,-1){$\scriptstyle\gamma_{1,r}$}
\put(29,28){$\scriptstyle P_o$}
\put(28.75,20){$\scriptstyle P_u$}
\put(22.5,20){$\scriptstyle-1_o$}
\put(18,9){$\scriptstyle-1_u$}
\put(14,21){$\scriptstyle-1_L$}
\put(31,11){$\scriptstyle-1_R$}
\put(23,17){$\scriptstyle B_{-1}$}
\put(66,12){$\scriptstyle B_{0}$}
\put(82,12){$\scriptstyle B_{1}$}
\put(43,8){$\scriptstyle-1_{RR}$}
\put(58.5,10.5){$\scriptstyle0_{L}$}
\put(74,10){$\scriptstyle1_{L}$}
\put(66,18.5){$\scriptstyle0_{R}$}
\put(82,18.5){$\scriptstyle1_{R}$}
\put(62.5,19){$\scriptstyle0_{o}$}
\put(62.5,10.5){$\scriptstyle0_{u}$}
\put(78.5,19){$\scriptstyle1_{o}$}
\put(78.5,10.5){$\scriptstyle1_{u}$}
\put(34.3,21.9){$\scriptstyle p$}
\end{overpic}
\caption{A portion of the train track $\Theta$.  The connectors are in purple, and the labels are the codes for certain connectors and cuffs.}
\label{fig:cuffs}
\end{figure}

\noindent If a simple closed curve is carried by this train track and does not intersect any subpaths without a code, then we call this a \textit{$\Theta$--code} for the given simple closed curve. 
We say a $\Theta$--code is \textit{reduced} if no two adjacent characters are the same.

\subsection{The simple closed curves $c_i$}\label{sec:ci}
For the remainder of this section, we fix $n\in \mathbb N$ and use the notation $g=g_n$ and $\alpha_i=\alpha_i^{(n)}$.

In \cite{Saric}, is it assumed that the relevant geodesic laminations do not contain geodesics that run out a cusp at one (or both) ends (see the discussion before \cite[Proposition 4.12]{Saric}). It is impossible for our sequence $\alpha_i$ to converge to a geodesic lamination of this type.  In this section, we describe how to associate simple closed curves $c_i$ to each $\alpha_i$ so that $g(c_{i-1})=c_i$.  In the following subsection, we prove that they converge to a geodesic lamination as in \cite{Saric}.  

Let $D_p$ be a small disk around the puncture $p$ which is invariant under our homeomorphism $g$, as in Section \ref{sec:highwayprelim}.   As before, we choose $D_p$ small enough so that, for each $i$, $\alpha_i\cap D_p$ consists of two segments, one starting at $p$ and one ending at $p$, with endpoints $z_{i,1},z_{i,2}$ on $\partial D_p$.  Up to homotopy, we may assume without loss of generality that $z_{i,1}=z_{j,1}$ and $z_{i,2}=z_{j,2}$ for all $i,j$.  To form the simple closed curves $c_i$, we start with $\alpha_i$ and remove $\alpha_i\cap \intr(D_p)$.  We then add an arc of $\partial D_p$ from $ z_{i,1}$ to $z_{i,2}$; there two possible choices of arc, and we choose the one so that one connected component of $S\setminus c_0$ contains $p$ and $p_0$.
Note that this is the opposite choice than the one made in Section \ref{sec:intpairing}.
See Figure \ref{fig:curvesci}.
\begin{figure} 
\centering
\begin{overpic}[width=3in, trim={.5in 6.25in 0in  0in },clip]{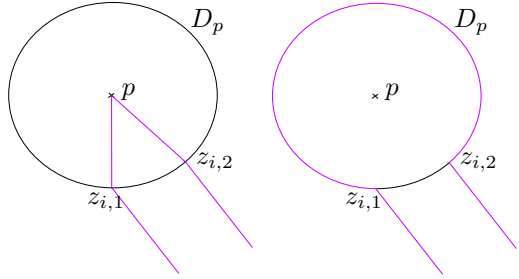}
\put(37,44){$D_p$}
\put(83,44){$D_p$}
\put(25,32){$p$}
\put(71,32){$p$}
\put(19,13){$z_{i,1}$}
\put(64,13){$z_{i,1}$}
\put(38,20){$z_{i,2}$}
\put(84,20){$z_{i,2}$}
\end{overpic}
\caption{Forming $c_i$ (right) from $\alpha_i$ (left).  The initial and terminal segments of $\alpha_i$ are in purple (left), and the corresponding segment of $c_i$ is in purple (right).}
\label{fig:curvesci}
\end{figure}

For example, given the train track in Figure \ref{fig:cuffs}, a $\Theta$--code of $c_0$ is 
\begin{equation}\label{eqn:c0code}
P_o(-1)_R(-1)_{RR}0_L0_o0_u0_L(-1)_{RR}(-1)_RP_u.
\end{equation}
Note that since curves do not have a well-defined starting point, any cyclic permutation of this $\Theta$--code for $c_0$ is also a $\Theta$--code for $c_0$.  For the rest of the section, we fix the starting point $P_o$ for the $\Theta$--code for $c_0$ as in \eqref{eqn:c0code}.

Our curves $c_i$ and the arcs $\alpha_i$ agree outside of $D_p$, and since $D_p$ is invariant under $g$, it follows that $g(c_i)=c_{i+1}$.  For each $i$, we fix the starting point $P_o$ for a $\Theta$--code for $c_i$ in such a way that applying $g$ to the $\Theta$--code for $c_i$ yields the $\Theta$--code for $c_{i+1}$.

\subsection{Proof of Theorem \ref{thm:geodlam}}

We will show that the simple closed curves $c_i$ defined in the previous section converge to a geodesic lamination $T$ on $\Sigma$.  Our strategy is as follows. We will first fix a lift $\tilde\gamma_i$ of each curve $c_i$ in $\tilde \Sigma$ and show that the sequence $\tilde\gamma_i$ converges to some $\tilde\gamma$.
We then show that if $f_i= G(\tilde\gamma_i)$ and $f= G(\tilde\gamma)$ are the corresponding geodesics that are weakly carried by $\tilde\Theta$, then $\lim_{i\to\infty} f_i=f$.  Finally, we take the geodesic lamination $T$ to be the image of $f$ in $\Sigma$.

For each $i$, let $\ell_\Theta(c_i)$ be the length of any reduced $\Theta$--code for $c_i$.  Recall that $\ell_c(\alpha_i)$ is the code length of $\alpha_i$, as in Definition \ref{def:codelength}.

\begin{lem} \label{lem:Thetacode}
For each $i$, $\ell_c(\alpha_i)\leq \ell_\Theta(c_i)\leq5\ell_c(\alpha_i)$.  
\end{lem}
\begin{proof}
It is clear that the $\Theta$--code for $c_i$ is at least as long as the code for $\alpha_i$.  
For the second inequality, notice that each instance of $k_{o/u}$ with $k\neq-1,P$ in $\alpha_i$ is replaced with one of the following strings, depending on what precedes/follows the character $k_{o/u}$ and on the chosen pants decomposition: $k_Lk_{o/u}k_R$, $k_Lk_{o/u}$, $k_Rk_{o/u}$, $(k-1)_{RR}k_Lk_{o/u}k_R$, $(k-1)_{RR}k_Lk_{o/u}k_Rk_{RR}$, or $k_Lk_{o/u}k_Rk_{RR}$.  If $k=-1$, then $(-1)_{o/u}$ is replaced with $(-1)_L(-1)_{o/u}$, $(-1)_{o/u}(-1)_R$, $(-2)_{RR}(-1)_L(-1)_{o/u}$, $(-1)_{o/u}(-1)_R(-1)_{RR}$, or $(-2)_{RR}(-1)_L(-1)_{o/u}(-1)_{RR}$.  Finally, if $k=P$, then each $P_{o/u}$ remains the same and $P_s$ is replaced with $P_{o/u}$.  Therefore, each character in the code for $\alpha_i$ is replaced with at most 5 characters in the $\Theta$--code for $c_i$, which gives the upper bound.
\end{proof}

For each $i$, let $\ell_i=\ell_\Theta(c_i)$, and consider the $\Theta$--code for $c_i$  
\begin{equation}\label{eqn:cicode}
c_i=c^i_1c^i_2\dots c^i_{\ell_i}.
\end{equation}
For each $i$, fix the lift $\tilde \gamma_i$ of $c_i$ that is the periodic bi-infinite edge path 
\[
\tilde\gamma_i=(\dots, b_{-1}^i,b_{0}^i,b_1^i,b_2^i,\dots),
\]
 with period $\ell_i$, where  $b_j^i=c_j^i$ for each $1\leq j\leq \ell_i$, and  $b_j^i=b_{j-\ell_i}^i$  for all $j$. 
The codes for $\alpha_i$ and $\alpha_{i+1}$ agree on the first $\frac12\ell_c(\alpha_i)$ characters. Thus it follows from Lemma \ref{lem:Thetacode}  that the $\Theta$--codes of $c_i$ and $c_{i+1}$ defined by \eqref{eqn:cicode} agree on at least the first $L_i=\left\lfloor\frac{1}{10}\ell_i\right\rfloor$ characters. Therefore, $b_j^i=b_j^{i+1}$ for all $1\leq j\leq L_i$. 

We define a bi-infinite path 
\[
\tilde\gamma=(\dots,d_{-1},d_{0},d_1,d_2,\dots),
\] as follows. Intuitively, our goal is to define $\tilde\gamma$ so that it agrees with each $\tilde\gamma_i$ from $d_{-L_i}$ to $d_{L_i}$.  In the first step, we  define the characters $d_{-L_0}$ to $d_{L_0}$ of $\tilde \gamma$ so that they agree with $\tilde \gamma_0$.  In the second step, we define the characters $d_{-L_1}$ to $d_{L_1}$ of $\tilde \gamma$ so that they agree with $\tilde \gamma_1$. The key point here is that $\tilde\gamma_0$ and $\tilde\gamma_1$ agree on the characters of $\tilde\gamma$ that we have already defined in the first step.  Thus we are not redefining $d_{i}$ if $-L_0\leq i\leq L_0$.  Rather, these characters remain, and the additional information from the second step is the definition of $d_i$ if $-L_1\leq i < -L_0$ or $L_0<i\leq L_1$.  We then continue this process.

Formally, this is equivalent to the following definition. For each $i\geq 0$ and each $1\leq j\leq L_i$, define  
\[
d_1=b_1^i, \qquad\dots \qquad d_j=b^i_j,\qquad \dots \qquad d_{L_i}=b_{L_i}^i,
\] 
and define 
\[d_{0}=b_{1}^i,\qquad\dots \qquad d_{-j+1}=b_j^i,\qquad\dots\qquad d_{-L_i+1}=b_{L_i}^i.\]  For each $i$ and all $1\leq j\leq L_i$, since $b^i_j=b^{i+1}_j$, there is no conflict with the previous defined edges as $i$ increases.

By construction, $\tilde\gamma$ is a bi-infinite path in $\tilde\Theta$.  Let $f_i= G(\tilde\gamma_i)$ and $f= G(\tilde\gamma)$ be the corresponding geodesics which are weakly carried by $\tilde \Theta$,  the existence of which is guaranteed by Proposition \ref{prop:pathtogeod}.

\begin{lem}
$\lim_{i\to\infty}f_i=f$.
\end{lem}

\begin{proof}

This is  almost immediate from the construction of $\tilde\gamma$.  Fix any finite subpath $T\subset \tilde\gamma$.  Then $T$ is supported on  $[-k,l]$ for some $k,l\geq 1$.  Let $L=\max\{k,l\}$, and fix $N$ such that $L_N\geq L$.  Such an $N$ exists since  $\ell_c(\alpha_i)\to\infty$ implies that $\lim_{i\to\infty}L_i=\infty$.  Then by construction $T$ appears in all $\tilde\gamma_i$ with $i\geq N$. Convergence follows by Proposition \ref{prop:convergence}.
\end{proof}

Let $T$ be the image of $f$ in $\Sigma$.  Figure \ref{fig:T} shows the   train track which weakly carries the lamination, that is, the image of $\tilde{\gamma}$ in $\Sigma$.

\begin{figure} 
\centering
\begin{overpic}[width=5in]{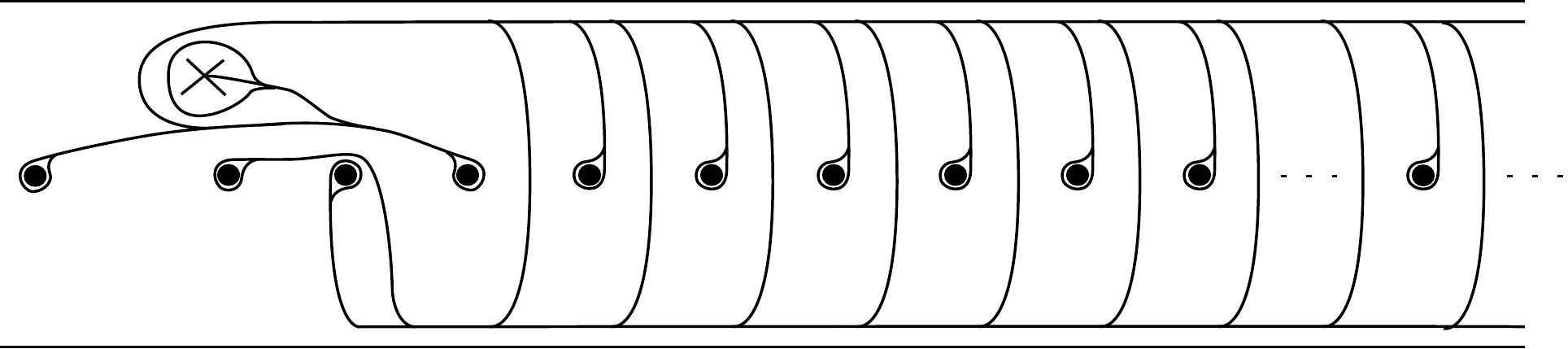}
\put(0,-2.5){$-1$}
\put(14,-2.5){$0$}
\put(21.5,-2.5){$1$}
\put(29,-2.5){$2$}
\put(36.5,-2.5){$3$}
\put(44.5,-2.5){$4$}
\put(52.5,-2.5){$5$}
\put(60,-2.5){$6$}
\put(67.5,-2.5){$7$}
\put(75.5,-2.5){$8$}
\put(89.5,-2){$n$}
\end{overpic}
\caption{The train track on $\Sigma$ which weakly carries the lamination $T$ when $n=1$. Note that the train track is, in fact, contained on the front of the embedded copy of $S$ in $\Sigma$.}
\label{fig:T}
\end{figure}

\begin{lem}\label{lem:geodlam}
$T$ is a geodesic lamination on $\Sigma$.
\end{lem}
\begin{proof}
This follows immediately from Proposition \ref{prop:pathtolamin} applied to $\Gamma=\{\tilde\gamma\}$.
\end{proof}

Finally, Theorem \ref{thm:geodlam}  follows from  Lemma  \ref{lem:geodlam} 


\bibliographystyle{abbrv}
\bibliography{Biblio}

\end{document}